\documentclass{memo-l}

\usepackage[stable]{footmisc}
\usepackage{}
\usepackage[stable]{footmisc}
\usepackage{color,eucal,enumerate,mathrsfs}
\usepackage[normalem]{ulem}
\usepackage{amsmath,amssymb,epsfig,bbm}
\usepackage{footmisc}[stable]

\usepackage[latin1]{inputenc}
\usepackage[english]{babel}

\newtheorem{theorem}{Theorem}[chapter]
\newtheorem{lemma}[theorem]{Lemma}
\newtheorem{proposition}[theorem]{Proposition}
\newtheorem{corollary}[theorem]{Corollary}

\theoremstyle{definition}
\newtheorem{definition}[theorem]{Definition}

\theoremstyle{remark}
\newtheorem{remark}[theorem]{Remark}

\newcommand{\D}{\mathbb{D}}

\newcommand{\N}{\mathbb{N}}
\newcommand{\Q}{\mathbb{Q}}
\newcommand{\R}{\mathbb{R}}
\newcommand{\Z}{\mathbb{Z}}
\newcommand{\mm}{{\mbox{\boldmath$m$}}}
\newcommand{\nn}{{\mbox{\boldmath$n$}}}

\newcommand{\aalpha}{{\mbox{\boldmath$\alpha$}}}

\newcommand{\ggamma}{{\mbox{\boldmath$\gamma$}}}

\newcommand{\eeta}{{\mbox{\boldmath$\eta$}}}

\newcommand{\ppi}{{\mbox{\boldmath$\pi$}}}

\newcommand{\sggamma}{{\mbox{\scriptsize\boldmath$\gamma$}}}

\newcommand{\sppi}{{\mbox{\scriptsize\boldmath$\pi$}}}
\newcommand{\sfd}{{\sf d}}

\newcommand{\sfh}{{\sf h}}

\newcommand{\Id}{{\rm Id}}                          % Identita
\newcommand{\Kliminf}{K\kern-3pt-\kern-2pt\mathop{\rm lim\,inf}\limits}  % Kuratowski liminf di insiemi
\newcommand{\supp}{\mathop{\rm supp}\nolimits}   % supporto 
\newcommand{\Lip}{\mathop{\rm Lip}\nolimits}          %Lipschitz constant
\renewcommand{\d}{{\mathrm d}}

\newcommand{\restr}[1]{\lower3pt\hbox{$|_{#1}$}}
\newcommand{\la}{\left<}                  % brackets
\newcommand{\ra}{\right>}
\newcommand{\eps}{\varepsilon}  
\newcommand{\nchi}{{\raise.3ex\hbox{$\chi$}}}
\newcommand{\weakto}{\rightharpoonup}
\newcommand{\limi}{\varliminf}
\newcommand{\lims}{\varlimsup}

\newcommand{\fr}{\penalty-20\null\hfill$\blacksquare$}                      %quadratino nero alla fine del remark, se non vi piace, la cosa migliore e' `svuotare' la macro, cosi' non bisogna intervenire sul testo

\newcommand{\gopt}{{\rm{OptGeo}}}                   % displacement plans ottimali
\newcommand{\prob}[1]{\mathscr P(#1)}                   % misure di probabilita
\newcommand{\probt}[1]{\mathscr P_2(#1)}                   % misure di probabilita con momento secondo finito
\newcommand{\e}{{\rm{e}}}                           % mappa di valutazione, bisogna mettere `a mano' il tempo t
\newcommand{\geo}{{\rm Geo}}                      % spazio delle geodetiche

\renewcommand{\mm}{\mathfrak m}                                %misura di riferimento

\renewcommand{\nn}{\mathfrak n}                                %misura di riferimento

\newcommand{\weakgrad}[1]{|\nabla #1|_w} % gradient ottenuto integrando lungo "quasi tutte" le geodetiche

\newcommand{\bd}{{\mathbf\Delta}}

\newcommand{\s}{{\rm S}}

\renewcommand{\b}{{\rm b}}
\newcommand{\X}{{\rm F}}

\newcommand{\h}{{\sfh}}
\newcommand{\mau}{{\sf T}}
\newcommand{\mad}{{\sf S}}

\renewcommand{\ae}{{\textrm{\rm{-a.e.}}}}
\newcommand{\vol}{{\rm Vol}}

\newcommand{\dom}{{\rm DH}}
\newcommand{\lip}{{\rm lip}}
\newcommand{\CD}{{\sf CD}}
\newcommand{\RCD}{{\sf RCD}}

\renewcommand{\u}{\mathcal U}
\renewcommand{\weakgrad}[1]{|D #1|}

\numberwithin{section}{chapter}
\numberwithin{equation}{chapter}

\makeindex

\begin{document}

\frontmatter

\title{The splitting theorem in non-smooth context}

%    Remove any unused author tags.

%    author one information
\author{Nicola Gigli}
\address{Universit\'e de Nice}
\curraddr{SISSA, via Bonomea 265, Trieste, 34136, Italy}
\email{ngigli@sissa.it}
\thanks{}

%    \date is required; it is the date received by the editor.
\date{}

\subjclass[2020]{Primary 51Fxx, 58-XX, 53Cxx, 46-XX}
%    Recognition of the 2010 edition of the Mathematics Subject
%    Classification requires a version of amsbook.cls from July 2009
%    or later.  If "2010" is not recognized, please upgrade.

\keywords{Splitting, Ricci curvature, non-smooth geometry}

%\dedicatory{Dedication text (use \\[2pt] for line break if necessary)}

\begin{abstract}
We prove that an infinitesimally Hilbertian $CD(0,N)$ space containing a line splits as the product of $\R$ and an infinitesimally Hilbertian $CD(0,N-1)$ space. By `infinitesimally Hilbertian' we mean that the Sobolev space $W^{1,2}(X,\sfd,\mm)$, which in general is a Banach space, is an Hilbert space. When coupled with a curvature-dimension bound,  this condition is known to be stable with respect to measured Gromov-Hausdorff convergence.\footnote{This never-before-published article was first posted as arxiv:1302.5555 on Feb 22, 2013 and an earlier version was posted as cvgmt.sns.it/paper/2079/ on Jan 28, 2013. Here we keep the paper true to the original arxiv post with no revisions other than reformatting for the journal.

Appendix C was mostly extracted from the previous unpublished preprint of the author: cvgmt.sns.it/paper/1801/ that was posted online a year earlier on March 13, 2012. The scope of such preprint was to point out how it was more natural to expect a link between the Bochner inequality and the `starred version' of the CD condition rather than the original `unstarred' CD(K,N).
}

\end{abstract}

\maketitle

\tableofcontents

%    Include unnumbered chapters (preface, acknowledgments, etc.) here.

\chapter*{Prologue\except{toc}{\\} {\small by Luigi Ambrosio}}

\markboth{\MakeUppercase{Prologue}}{\MakeUppercase{Prologue}}

% Nicola Gigli's paper ``The splitting theorem in non-smooth context''
% was originally posted on the arxiv in 2013 yet is only being published
% for the first time in this 2024 issue of {\em Memoirs of the American
%   Mathematical Society}.

Gigli's splitting theorem is by now a classical tool used in many papers dealing with the theory of $\RCD(K,N)$ metric measure spaces.  This prologue to his paper restates his results using modern terminology and contains a brief survey of some of these many applications.

Recall that the Riemannian splitting theorem as proven by Cheeger-Gromoll in \cite{Cheeger-Gromoll-splitting} states that a smooth Riemannian manifold, $(M,g)$, with nonnegative Ricci curvature that contains a line,
\begin{equation*}
\gamma: {\mathbb R} \to M \text{ with }d\big(\gamma(t),\gamma(s)\big)=|t-s| \quad \forall s,t \in {\mathbb R},
\end{equation*}
must split isometrically.  That is, there exists a Riemannian manifold, $(M',g')$, with
nonnegative Ricci curvature
such that
\begin{equation*}
M=M'\times {\mathbb R} \text{ and } d\big((x',t),(y',s)\big)=\sqrt{d'(x',y')^2+|t-s|^2\,}.
\end{equation*}
This result was extended to measured Gromov-Hausdorff limits of sequences of manifolds with almost nonnegative Ricci curvature by Cheeger-Colding in \cite{Cheeger-Colding96} and was a key tool in their subsequent analysis of measured Gromov-Hausdorff limits of Riemannian manifolds with uniform lower bounds on their Ricci curvature and uniform upper bounds on their dimension.   As Gigli's paper provides a thorough history of splitting theorems, we will limit ourselves to just this one paragraph.

The main theorem in Gigli's 2013 paper appearing in this volume of {\em Memoirs} is an extension of the splitting theorem to the class of {infinitesimally Hilbertian} $\CD(0,N)$ metric measure spaces.   Recall that a metric measure space, $(X,d, m)$, satisfying the curvature dimension condition, $\CD(K,N)$, has a synthetic lower Ricci curvature bound, $K$, and upper dimension bound, $N$, as defined by Lott-Villani \cite{Lott-Villani09} and Sturm \cite{Sturm06I}-\cite{Sturm06II}.
However, a metric measure space which only satisfies the $\CD(0,N)$ condition need not satisfy the splitting theorem.  Cordero-Erausquin, Villani, and Sturm discovered that, in fact, ${\mathbb R}^N$ endowed with any norm and the Lebesgue measure is a $\CD(0,N)$ metric measure space (cf. the last theorem in \cite{Villani09}).
For example, the metric measure space, $(X,d,m)$, defined~by
\begin{equation*}
X={\mathbb R}^2 \text{ with } d\left((x,s),(y,t)\right)=|x-y|+|s-t| \text{ and }  m={\mathcal L}\times {\mathcal L}
\end{equation*}
is a $\CD(0,N)$ space
which contains a line, but fails the Pythagorian equality.

In 2011, Ambrosio-Gigli-Savar\'e had introduced the notion of a metric measure space satisfying the Riemannian curvature dimension condition $\RCD(K,\infty)$
as a $\CD(K,\infty)$ metric measure space satisfying the Riemannian like {infinitesimally Hilbertian} hypothesis that the Sobolev space, $W^{1,2}(X,d, m)$, is a Hilbert space \cite{AmbrosioGigliSavare11-2}.   This Sobolev space
is defined as the domain of the Cheeger energy in $L^2(X,m)$.
Note that ${\mathbb R}^2$ endowed with the max norm and Lebesgue measure  fails to satisfy this hypothesis because its Sobolev space is only a Banach space without an inner product structure.  This Ambrosio-Gigli-Savar\'e paper was built in part upon work on the heat flow on
$CD(0,N)$ spaces by Gigli in \cite{Gigli10}, on Alexandrov spaces by Gigli-Kuwada-Ohta's \cite{Gigli-Kuwada-Ohta10} and their previous collaboration in \cite{AmbrosioGigliSavare11}.
Then in 2012, Gigli completed a paper on the {infinitesimally Hilbertian} hypothesis on more general metric measure spaces proving, in particular, the Laplace comparison theorem in that setting \cite{Gigli12} .   Gigli's 2013 paper within this volume of {\em Memoirs} provides a detailed review of this background material.

In 2013, the notion of $\RCD(K,N)$ metric measure spaces was not yet fully estabilished for finite values of $N$.
For those interested in the later development of the definitions of $\RCD(K,N)$ and $\RCD^*(K,N)$ spaces, please see the ICM survey article by Ambrosio \cite{AmbrosioICM}.  The book by Gigli-Pasqualetto is a useful resource for students
\cite{GP19}.

Let us restate the 2013 splitting theorem proven by Gigli using modern terminology.   If $(X,d,m)$ is an
$\RCD(0,N)$ metric measure space and if there is a line $\gamma:{\mathbb R}\to X$,
then $(X,d, m)$ splits:
\begin{equation*}
X=X'\times {\mathbb R} \text{ and } d\big((x',t),(y',s)\big)=\sqrt{d'(x',y')^2+|t-s|^2}
\text{ and } m=m'\times {\mathcal L}^1.
\end{equation*}
where the quotient space, $(X',d',m')$, is an $\RCD(0,N-1)$ metric measure space
when $N\ge 2$ and $X'$ is a single point if $N\in [1,2)$.
Gigli also reformulated this statement as an almost splitting theorem in the style of Cheeger-Colding using Sturm's ${\mathbb D}$ distance.   The factorization not only of the distance, but also of the measure, as well as the fact that the quotient space is $\RCD(0,N-1)$,
were new even in the context of limit Ricci spaces and have been crucial tools in the subsequent development of
$\RCD(K,N)$ metric measure spaces.

One of the key techniques developed in Gigli's splitting article is the use of Sobolev calculus to obtain geometric rigidity for infinitesimally {H}ilbertian metric measure spaces with synthetic {R}icci curvature bounded below.     These tools are crucial in Gigli-Mondino-Rajala's proof that {infinitesimally Hilbertian} $\CD^*(0,N)$-spaces have weak tangent spaces almost everywhere that are isometric to
Euclidean space \cite{GMR15}.   They are applied in De Philippis and Gigli's extension of Cheeger-Colding's volume to metric cone rigidity theorem in \cite{DPG16}, and
Gigli-Rigoni's extension of Colding's torus rigidity theorem in \cite{GR17}.
In \cite{Mondino-Naber14}, Mondino and Naber built a
structure theory of metric measure spaces with lower Ricci
curvature bounds using these tools.
In \cite{ABS19}, Ambrosio, Bru\`{e}, and Semola
apply them to study sets of finite perimeter in {${\RCD}$} spaces.
Connell, Dai, N\'{u}\~{n}ez-Zimbr\'{o}n, Perales, Su\'{a}rez-Serrato, Pablo and Wei prove
maximal volume entropy rigidity for {${\RCD}^*(-(N-1),N)$} spaces in
\cite{CDPSW21}.
 Antonelli, Pasqualetto, and Pozzetta apply Gigli's Sobolev calculus to prove new results about
Isoperimetric sets even in smooth Riemannian manifolds with boundary
that have lower bounds on {R}icci curvature \cite{APP22}.   See also the works of Antonelli-Fogagnolo-Pozzetta \cite{AFP21} and Han \cite{han2023}.

The maximal diameter theorem for $\RCD$-spaces proven by Ketterer in \cite{Ketterer13} is crucially based on the Gigli splitting theorem (as creditted in his abstract).
There are many papers which apply this maximal diameter theorem to achieve geometric rigidity results
and naturally some cite only \cite{Ketterer13} rather than also including the Gigli splitting article.
Cavaletti-Mondino use the Ketterer's maximal diameter theorem for $\RCD$ spaces
to deduce their rigidity and almost rigidity results \cite[Theorem 1.4, Theorem 1.5, Corollary 1.6]{CavMon15} although their other theorems
are based on a set of tools different from the ones used by Gigli and their isoperimetric inequality is mostly independent of it and obtained on a class of spaces larger than the $\RCD$ one.
In \cite{K15}, Ketterer proves Obata's rigidity theorem for $\RCD$-spaces, using the maximal diameter theorem.   Erbar-Sturm apply it to prove their sphere rigidity theorem in \cite{ES21}.
It is applied by Deruelle, Schulze, and Simon to study the stability
of the Ricci flow for Ricci-pinched manifolds \cite{DSS22}.

Many of these applications of Gigli's splitting theorem produce new results about Riemannian manifolds and their metric measure limit spaces as well as RCD spaces.
See, for example, Sections 4.10 and 7.3 of Cheeger-Jiang-Naber in
\cite{ChJiNa21} where, in the so-called non collapsed scenario, sharp rectifiability results of the singular sets are obtained.

Every paper dealing with Geometric Measure Theory on $\RCD$ spaces applies
Gigli's splitting for $\RCD$ spaces as a prerequisite to perform the analysis done there.
Any time a Euclidean tangent appears or a dimension reduction argument comes into place by factoring out lines, the splitting is called into play as one of the key tools.
See, for example,
Bru\'e and Semola's proof that $\RCD(K,N)$ spaces have constant dimension \cite{BS18},
Antonelli, Bru\'e, and Semola 's study of
the singular sets of $\RCD(K,N)$ spaces \cite{AntBrSe19}, and
Bru\'e, Pasqualetto, and Semola work on the rectifiability of $\RCD(K,N)$ spaces
\cite{BPS21}.
See Kitabeppu and Lakzian's
characterization of low dimensional {${\RCD}^*(K,N)$} spaces in  \cite{Kit16}.
See the conjectures of Honda in \cite{Honda16} and Honda's characteriation of
non-collapsed $\RCD(K,N)$ spaces using a geometric flow \cite{H19}.
See Deng's proof of the H\"older continuity of tangent cones in {${\RCD}(K,N)$} spaces \cite{Deng20} .
See also work of Antonelli-Fogagnolo-Pozzetta \cite{AFP21}, Han \cite{Han21},
Caputo-Gigli-Pasqualetto \cite{CGP21},
Bru\'e-Naber-Semola \cite{BNS22}, Huang-Huang \cite{HH23}, and
Brena-Gigli-Honda-Zhu \cite{BGHZ23}.

Gigli's techniques are also applied in work of Gigli, Ketterer, Kuwada, and Ohta \cite{GKKO20}
and Ketterer-Kitabeppu-Lakzian \cite{KKL23} to study spaces achieving spectral gaps.
See also the bounds on the Laplacian proven by Mondino-Semola \cite{MS21}.
Gigli-Violo proved a monotonicity formula for harmonic functions and electrostatic potentials in  \cite{GV21}.
 Z Huang \cite{Huang23} applied Gigli's analysis techniques to prove new theorems
 about harmonic functions of polynomial growth on Riemannian manifolds.
Chan-Zhang-Zhu has applied them to one phase free boundary problems in \cite{CZZ22}
and Honda-Sire applied them to harmonic maps in \cite{HS23}.

Naturally Gigli's splitting theorem can be applied to study isometry groups, fundamental groups and
topological properties of $\RCD$ and $\RCD^*$ spaces.    Mondino and Wei proved the existence of universal covers for $\RCD^*(K,N)$ spaces \cite{MW19} and Wang proved they are
simply connected in  \cite{Wang23}.
Colding's torus rigidity and almost rigidity theorems are extended by
Gigli-Rigoni in \cite{GR17}, Mondello-Mondino-Perales in \cite{MMP22}, and Ye in \cite{Ye23}.
Kapovitch-Mondino have work on the topology and the boundary of {$N$}-dimensional {${\RCD}(K,N)$} spaces in \cite{KM19}.
Kapovitch proved a mixed curvature analogue of Gromov's almost flat manifolds in \cite{Kap21}.
Guijarro, Santos-Rodr{\'\i}guez, Zamora-Barrera, and Wang studied isometry groups
in \cite{GSR19} \cite{SR20} \cite{SRZB23}.
For results on spaces with linear volume
or small diameter growth see work of Huang \cite{Huang18}\cite{Huang20b}
and Qian \cite{Qian2022} extending results of Sormani.
Mondino-Navarro have results on moduli spaces of $\RCD(0,N)$ structures \cite{MN22},
Honda-Peng \cite{HP23} constructed homeomorphisms between $\RCD(K,N)$ spaces and
Riemannian manifolds that are close in the Gromov-Hausdorff sense.

The techniques in Gigli's paper have also been applied to study Alexandrov spaces.
In \cite{LS18}, Lytchak-Stadler proved Villani's conjecture that if a metric measure space, $(X,d, H^2)$,
whose measure is the 2-dimensional Hausdorff measure, $H^2$, on $X$. is an
$\RCD(K,2)$ space then $(X,d)$ is an Alexandrov space of curvature at least $K$.     While the backbone of the proof in Chen's work on maximal volume entropy in Alexandrov spaces
\cite{Chen22} is closer in spirit to the approach of Cheeger-Colding rather then to the one used by Gigli, several technical tools are gathered from Gigli's splitting paper.
For other applications to Alexandrov spaces
see work of Kapovitch-Ketterer \cite{KK18}, Deng-Galaz-Garc\'{\i}a-Guijarro-Munn \cite{DGGGM18},
Jiang \cite{Jiang19}, Chen \cite{Chen22b}, and Kapovitch-Zhu \cite{KZ23}.

It is impossible to provide a complete survey of all the work that was directly or indirectly
influenced by the ideas and theorems in Gigli's ``The splitting theorem in non-smooth context'' published for the first time in this volume of {\em Memoirs of the American Mathematical Society}.
We leave the readers to explore further in their libraries and hope they enjoy exploring the article itself within this volume.

\mainmatter

\chapter{Introduction} 
\section{Historical remarks and statement of the result}

The splitting theorem is a rigidity result in Riemannian geometry. It was proved at first for surfaces with non-negative curvature by Cohn-Vossen (\cite{CohnVossen36}), generalized by Toponogov (\cite{Toponogov59}) in arbitrary dimension for manifolds with non-negative sectional  curvature - see also the contribution of  Milka \cite{Milka67} for the case of Alexandrov spaces - and then by Cheeger and Gromoll in \cite{Cheeger-Gromoll-splitting} for manifolds with non-negative Ricci curvature, their statement being:
\begin{theorem}[Splitting]\label{thm:splitting}
Let $M$ be a Riemannian manifold with non-negative Ricci curvature and containing a line.

Then $M$ is isometric to the product $N\times\R$ for some Riemannian manifold $N$ of non-negative Ricci curvature.
\end{theorem}
Here and in the following by Riemannian manifold we intend a complete, connected and smooth Riemannian manifold without boundary.

A number of variants/generalizations have been obtained since then, in particular:
\begin{itemize}
\item A simplified proof provided by Eschenburg and Heintze in \cite{EschenburgHeintze84}.
\item A version for Lorentzian manifolds, conjectured by Yau, given by Newman in \cite{Newman90} (see also the  earlier contributions by Galloway \cite{Galloway84} and \cite{Galloway89}, Beem,  Ehrlich, Markvorsen, Steen, Galloway \cite{BEMSG84} and \cite{BEMSG85} and Eschenburg \cite{Eschenburg85}).
\item A version for orbifolds by Borzellino and Zhu in \cite{BorzellinoZhu94}.
\item A topological splitting for Alexandrov spaces with curvature bounded from below (possibly by a negative number) endowed with a measure satisfying the $MCP(0,N)$ condition, by Kuwae and Shioya in \cite{KuwaeShioya11}. Their result also provide an isometric splitting for some singular manifolds.
\item An isometric splitting for Alexandrov spaces with curvature bounded from below (possibly by a negative number) with non-negative Ricci curvature, by Zhang and Zhu in \cite{ZhangZhu10}. Here `non-negative Ricci curvature' is in a sense defined by the authors.
\item A diffeomorphic and measure preserving splitting theorem for Finsler manifolds which produces  a 1-parameter group of isometries on Berwald spaces, given by Ohta in \cite{Ohta12}
\end{itemize}
Among others, a  crucial generalization of the splitting theorem has been obtained by Cheeger and Colding in \cite{Cheeger-Colding96}: they  proved the following quantitative version of the rigidity property stated in Theorem \ref{thm:splitting}.
\begin{theorem}[Almost splitting]\label{thm:almostsplit}
There exists a non-negative function $(\delta,L,\eps,n,R)\mapsto\Psi(\delta,L,\eps|n,R)$ such that for given $n,R$ it holds $\lim_{\delta,\eps,L^{-1}\to0}\Psi(\delta,L,\eps|n,R)=0$ for which the following is true. 

Let $M$ be  an $n$-dimensional  Riemannian manifold  with ${\rm Ric}\geq -(n-1)\delta$ and $x,y_1,y_2\in M$ such that  $\min\{\sfd(x,y_0),\sfd(x,y_1)\}\geq L$ and $\sfd(x,y_0)+\sfd(x,y_1)-\sfd(y_0,y_1)\leq \eps$, $\sfd$ being the distance on $M$ induced by the metric tensor.

Then there exists a geodesic space $(X',\sfd')$ such that for some $x'\in X'$ the ball $B_R(x',0)\subset X'\times\R$ satisfies 
\[
\sfd_{\rm GH}\Big(B_R(x),B_R(x',0)\Big)\leq \Psi(\delta,L,\eps|n,R),
\]
where $\sfd_{\rm GH}$ is the Gromov-Hausdorff distance and the product space $X'\times\R$ is endowed with the distance $\sfd'\times\sfd_{\rm Eucl}$ defined by
\begin{equation}
\label{eq:intropit}
\sfd'\times\sfd_{\rm Eucl}\big((x',t),(y',s)\big):=\sqrt{\sfd'(x',y')^2+|t-s|^2}.
\end{equation}
\end{theorem}
A simple limiting argument shows that the almost splitting theorem yields the splitting for limit spaces:
\begin{theorem}[Splitting for limit spaces]\label{thm:splitlimit}
Let $(M_i)$ be a sequence of Riemannian manifolds with dimension uniformly bounded from above and such that ${\rm Ric}_{M_i}\geq -\delta_i$, where $\delta_i\downarrow0$ as $i\to\infty$. 

Let $\sfd_i$ be the Riemannian distance on $M_i$ and assume that for some sequence of points $x_i\in M_i$ the sequence $i\mapsto (M_i,\sfd_i,x_i)$ converges to a pointed metric space $(X,\sfd,x)$ in the pointed Gromov-Hausdorff sense. Assume also that $(X,\sfd)$ contains a line.

Then $(X,\sfd)$ is isometric to the product $X'\times\R$, where $(X',\sfd')$ is a length space and the product distance on $X'\times\R$ is defined as in \eqref{eq:intropit}.  
\end{theorem}

In \cite{Lott-Villani09} and \cite{Sturm06I},\cite{Sturm06II} Lott-Villani on one side and Sturm on the other independently proposed a definition of `having Ricci curvature bounded from below by $K$ and dimension bounded above by $N$' for metric measure spaces, these being called $CD(K,N)$ spaces (in \cite{Lott-Villani09} only the cases $K=0$ or $N=\infty$ were considered). Here $K$ is a real number and $N$ a real number $\geq 1$, the value $N=\infty$ being also allowed. 

The crucial properties of their definition are the compatibility with the smooth Riemannian case and the stability w.r.t. measured Gromov-Hausdorff convergence. Broadly speaking, a central question about the study of $CD(K,N)$ spaces is:  which of the properties valid for Riemannian manifolds with  ${\rm Ric}\geq K$ and ${\rm dim}\leq N$ are also true for $CD(K,N)$ spaces?

In particular, given the mentioned stability property and Theorem \ref{thm:splitlimit}, the following question arises naturally:
\begin{quote}
Let $(X,\sfd,\mm)$ be a $CD(0,N)$ space containing a line. Is it true that there exists another space $(X',\sfd',\mm')$ such that $(X,\sfd,\mm)$ is isomorphic to the product of $X'$ and $\R$, where we are endow $X'\times\R$ with the product measure $\mm'\times\mathcal L^1$ and the product distance $\sfd'\times\sfd_{\rm Eucl}$ is defined  as in \eqref{eq:intropit}?
\end{quote}
And also:
\begin{quote}
If the above is true, what can we say about the quotient space $(X',\sfd',\mm')$? In particular, is it a $CD(0,N-1)$ space?
\end{quote}
Soon after the definitions in \cite{Lott-Villani09} and \cite{Sturm06II}  have been proposed, it has been understood that the answer to the first question is `no'  in the class $CD(0,N)$. Indeed, as shown by Cordero-Erasquin, Sturm and Villani (see the last theorem in \cite{Villani09}), the metric measure space $(\R^d,\sfd_{\|\cdot\|},\mathcal L^d)$, where $\mathcal L^d$ is the Lebesgue measure and $\sfd_{\|\cdot\|}$ is the distance induced by the norm $\|\cdot\|$, is always a $CD(0,d)$ space, regardless of the choice of the norm (see also \cite{Ohta09} for the curved Finsler case). In particular, if we take $d=2$ and consider a norm not coming from a scalar product, then we see that the splitting cannot hold, because ``Pythagoras' theorem" stated in formula \eqref{eq:intropit} fails. 

\medskip

It is therefore natural to look for a stricter  synthetic notion of Ricci curvature bound which, while retaining the stability property w.r.t. mGH-convergence, still ensures a `Riemannian-like' behavior of the spaces, possibly enforcing geometric rigidity results. A proposal in this direction has been made in \cite{AmbrosioGigliSavare11-2} specifically for  the case $N=\infty$:  according to the slightly finer analysis done in \cite{AmbrosioGigliMondinoRajala12}, one says that $(X,\sfd,\mm)$ has \emph{Riemannian} Ricci curvature bounded from below by $K$ (is an $RCD(K,\infty)$ space, in short), provided it is a $CD(K,\infty)$ space and the Sobolev space $W^{1,2}(X,\sfd,\mm)$ of real valued functions defined on $X$ is Hilbert. 

Some comments about this definition are:
\begin{itemize}
\item  In abstract metric measure spaces $W^{1,2}$ is always a Banach space, and in the smooth situation a Finsler manifold is Riemannian if and only if $W^{1,2}$ is Hilbert. In this sense the additional requirement that such space is Hilbert can be seen as the  non-smooth analogous of `the norm comes from a scalar product' which distinguishes Riemannian manifolds among Finsler ones.
\item  Simple examples show that the condition `$W^{1,2}$ is Hilbert' is not stable w.r.t. mGH-convergence. This is certainly not surprising, because the former is a first order condition on the space, while the latter is a zeroth order convergence. However,  when coupling it with the curvature condition $CD(K,\infty)$, the resulting notion turns out to be stable, which was indeed one of the motivation for the introduction of $RCD(K,\infty)$ spaces. Heuristically, we can interpret this fact as the stability of a first order notion w.r.t. a zeroth order convergence under a uniform second order bound. In practice, the proof of the stability  is based on the following two properties of  the heat flow on $CD(K,\infty)$ spaces:
\begin{itemize}
\item  In accordance with the smooth case (\cite{JKO98}), the gradient flow in $L^2(X,\mm)$ of the - generically non-quadratic - natural Dirichlet energy on $(X,\sfd,\mm)$ and the gradient flow in $(\probt X,W_2)$ of the relative entropy functional coincide (\cite{Gigli10},  \cite{Gigli-Kuwada-Ohta10},  \cite{AmbrosioGigliSavare11}). 
\item If a sequence of $CD(K,\infty)$ metric measure spaces converges w.r.t. the mGH-convergence to a limit space, then the gradient flows of the relative entropies along the approximating sequence converge to that in the limit space (\cite{Gigli10}, \cite{AmbrosioGigliSavare11-2}, \cite{GMS15}).
\end{itemize}
Then one notice  that $W^{1,2}$ is Hilbert if and only if the gradient flow of the `Dirichlet energy' is linear and thus if and only if the gradient flow of the relative entropy is linear. Given that the latter is stable, the desired  stability follows.
\item In the original paper \cite{AmbrosioGigliSavare11-2}, the focus was on the properties of heat flow and connections with the theory of Dirichlet forms (which also provides a way to define curvature-dimension bounds   by means of $\Gamma_2$-calculus - see \cite{BakryEmery85}). The resulting theory works reasonably well in the infinite dimensional case $N=\infty$, but as of today it is not clear whether the heat flow can be successfully used to characterize curvature-dimension bounds also in the case $N<\infty$ (it is possible to provide a stable curvature dimension notion based on the property of the heat flow and $\Gamma_2$-calculus, but the relation with the $CD(K,N)$ condition is unclear, see \cite{AmbrosioGigliSavare12} and the end of the section), nor whether its role can be taken by the porous media flow (which in the smooth case is the gradient flow of the R\'enyi entropy used to define $CD(K,N)$ spaces with $N<\infty$ - see \cite{Otto01} and \cite{AmbrosioGigliSavare08}). 

Instead, the assumption `$W^{1,2}$ is Hilbert' makes sense regardless of curvature-dimension bounds, and we have seen on one side that in the smooth case it singles out Riemannian manifolds, and on the other, according to the above discussion and recalling the stability of the $CD(K,N)$ condition,  that the property  `being a $CD(K,N)$ space such that $W^{1,2}$ is Hilbert' is also stable. 

Therefore one can consider this latter class as substitute of the original $CD(K,N)$ one and try to understand which sort of consequences he gets from the further assumption  `$W^{1,2}$ is Hilbert'. This is the point of view adopted in this paper.
\item The requirement `$W^{1,2}$ is Hilbert' is analytic in nature, not geometric. This means that we cannot reasonably expect to derive immediate  geometric consequences out of it. The plan is instead to first develop appropriate Sobolev differential calculus - and in doing so the hypothesis `$W^{1,2}$ is Hilbert' should have evident effects - and then to use such calculus to deduce the expected geometric properties, by mimicking, whenever possible, the arguments valid in the smooth world. Some steps in this direction have already been done. In \cite{Gigli12} the duality relations between differentials and gradients of Sobolev functions has been investigated, both in general and in connection with the assumption  `$W^{1,2}$ is Hilbert'. Then in \cite{Gigli-Mosconi12}, closely following the original argument, it has been shown that the Abresch-Gromoll inequality holds on $CD(K,N)$ spaces with $W^{1,2}$ Hilbert in the same form as in the smooth case (recall that on general $CD(K,N)$ spaces this inequality may fail).
\end{itemize}
According to the terminology introduced in \cite{Gigli12}, a space $(X,\sfd,\mm)$ such that $W^{1,2}(X,\sfd,\mm)$ is Hilbert will be called infinitesimally Hilbertian.  Our main result is:
\begin{theorem}[Splitting in non-smooth context]\label{thm:main}
Let $(X,\sfd,\mm)$ be an infinitesimally \linebreak Hilbertian $CD(0,N)$ space, $N<\infty$, and assume that $\supp(\mm)$ contains a line. 

Then there exists a metric measure space $(X',\sfd',\mm')$ such that  $(X,\sfd,\mm)$  is isomorphic to the product $X'\times\R$, where  $X'\times\R$ is endowed with the product measure $\mm'\times\mathcal L^1$ and the product distance $\sfd'\times\sfd_{\rm Eucl}$ defined by 
\[
\sfd'\times\sfd_{\rm Eucl}\big((x',t),(y',s)\big):=\sqrt{\sfd'(x',y')^2+|t-s|^2}.
\]
Furthermore:
\begin{itemize} 
\item If $N\geq 2$,  then $(X',\sfd',\mm')$ is  an infinitesimally Hilbertian $CD(0,N-1)$ space.
\item if $N\in[1,2)$, then $X'$ is  just a point.
\end{itemize}
\end{theorem}
By `isomorphic' we mean  that there exists an isometry $\mau:(\supp(\mm')\times\R,\sfd'\times\sfd_{\rm Eucl})\to (\supp(\mm),\sfd)$ such that $\mau_\sharp(\mm'\times\mathcal L^1)=\mm$. In particular every space $(X,\sfd,\mm)$ is isomorphic to $(\supp(\mm),\sfd,\mm)$; this explains why the line  is required to take values in $\supp(\mm)$ rather than on $X$.

Notice that the assumption $N<\infty$ is necessary. Consider indeed $\R$ endowed with the Euclidean distance $\sfd_{\rm Eucl}$ and the standard Gaussian measure $\ggamma$. This is a $CD(1,\infty)$ space, in particular $CD(0,\infty)$, and contains a line. Yet, it does not split, the problem being at the level of measures. Inspecting the proof of the splitting theorem, we see that the step which fails is the proof that the Busemann function is harmonic (recall that the intrinsic Laplacian $\Delta$ in such space is given by $\Delta f=\partial_{xx}f-x\partial_xf$). This is not surprising, because as we will see in the next section once one knows that the Busemann function is harmonic the proof can be completed using only the infinite dimensional Bochner inequality, which holds on $(\R,\sfd_{\rm Eucl},\ggamma)$.

With a standard compactness argument based on the fact that $CD(K,N)$ spaces are uniformly doubling, Theorem \ref{thm:main} can be reformulated as follows:
\begin{theorem}[Splitting in non-smooth context - equivalent formulation]\label{thm:mainequiv}
\ \linebreak There exists a non negative function $(\delta,L,\eps,N,R)\mapsto\Psi(\delta,L,\eps|N,R)$ such that for given $N,R$ it holds $\lim_{\delta,\eps,L^{-1}\to0}\Psi(\delta,L,\eps|N,R)=0$ for which the following is true. 

Let $(X,\sfd,\mm)$ be an infinitesimally Hilbertian $CD(-\delta(N-1),N)$ space and $x,y_1,y_2\in \supp(\mm)$ such that  $\min\{\sfd(x,y_0),\sfd(x,y_1)\}\geq L$ and $\sfd(x,y_0)+\sfd(x,y_1)-\sfd(y_0,y_1)\leq \eps$.

Then there exists a metric measure space $(X',\sfd',\mm')$ and  $x'\in \supp(\mm')$ such that  the ball $B_R(x',0)\subset X'\times\R$ satisfies 
\[
\D\Big(\big(B_R(x),\sfd,\tilde\mm_{B_R(x)}\big),\big(B_R(x',0),\sfd'\times\sfd_{\rm Eucl},\nn_{B_R(x',0)}\big)\Big)\leq \Psi(\delta,L,\eps|N,R),
\]
where $\D$ is the Sturm(-Gromov-Hausdorff) distance, the distance $\sfd'\times\sfd_{\rm Eucl}$ is defined by \eqref{eq:intropit} and the  measures $\tilde\mm_{B_R(x)},\nn_{B_R(x',0)}$ are the normalized restrictions of $\mm$ and $\mm'\times\mathcal L^1$ to the balls $B_R(x),B_R(x',0)$.

Furthermore:
\begin{itemize}
\item If $N\geq 2$, then $(X',\sfd',\mm')$ can be chosen to be   an infinitesimally Hilbertian $CD(0,N-1)$ space.
\item If $N\in[1,2)$, then $X'$ can be chosen to be  just a point.
\end{itemize}
\end{theorem}
See \cite{Sturm06I} for the definition of the distance $\D$ and notice that in fact any distance metrizing a convergence of pointed metric-measure spaces for which `infinitesimal Hilbertianity + $CD(0,N)$' is closed can be used in the statement. We will not discuss these stability questions here and refer to \cite{Villani09}, \cite{Sturm06I}, \cite{GMS15} and references therein for possible variants. In particular, we will not explicitly prove Theorem \ref{thm:mainequiv} but just focus on Theorem \ref{thm:main}.

Comparing Theorems \ref{thm:main}, \ref{thm:mainequiv} with Theorems  \ref{thm:splitlimit},  \ref{thm:almostsplit} and beside the fact that the former are stated in the non-smooth context while the latter refer to the smooth one, we see the following relations:
\begin{itemize}
\item Theorems \ref{thm:main}, \ref{thm:mainequiv} give informations about the measures while Theorems  \ref{thm:splitlimit},  \ref{thm:almostsplit} do not. This is fact not a difference but rather a choice of exposition. As the proofs of Theorems  \ref{thm:splitlimit},  \ref{thm:almostsplit} show, similar informations about the reference measures can directly be obtained.
\item The proof of  Theorem \ref{thm:almostsplit} produces an explicit expression for the function $\Psi$, while Theorem \ref{thm:mainequiv} only grants its existence via a compactness argument, in line with the overall program about gathering information on the smooth world by investigating the non-smooth one. Although the explicit form of $\Psi$ seems not important in applications, strictly speaking Theorem \ref{thm:mainequiv} is not a generalization of Theorem \ref{thm:almostsplit}.
\item  Theorems \ref{thm:main}, \ref{thm:mainequiv} give a structural information about the  quotient space $X'$ which is not present in Theorems  \ref{thm:splitlimit},  \ref{thm:almostsplit}. This has little to do with the strategy and tools that we use here, but is rather due to the stability of `curvature-dimension bound plus infinitesimal Hilbertianity'. With these at disposal it is quite easy to see that the quotient space in Theorem \ref{thm:splitlimit} must be infinitesimal Hilbertian and $CD(0,N)$, see the proofs of Theorems \ref{thm:pitagora} and \ref{thm:dimm1} (the only non trivial tool is the existence of optimal transport maps)
\end{itemize}

Given the rigidity result expressed by the splitting theorem in non-smooth setting and following the terminology proposed in \cite{AmbrosioGigliSavare11-2}, it might be tempting to define the class  $RCD(K,N)$ of spaces with Riemannian Ricci curvature bounded from below by $K$ and dimension bounded above by $N$ as:
\begin{equation}
\label{eq:rcdkn}
\begin{split}
RCD(K,N)&:=CD(K,N)+\textrm{infinitesimal Hilbertianity}\\
&\phantom{:}=CD(K,N)\cap RCD(K,\infty).
\end{split}
\end{equation}
Yet, we believe it is too early to close this concept in a definition, because before doing so we should get a clearer picture of the interaction within the (Riemannian) curvature-dimension bounds and Bochner inequality. Recall indeed that the inequality
\begin{equation}
\label{eq:introboc}
\Delta\frac{|\nabla f|^2}2\geq\frac{(\Delta f)^2}{N}+\nabla\Delta f\cdot\nabla f+K|\nabla f|^2,\qquad\forall f\in C^\infty(M),
\end{equation}
characterizes Riemannian manifolds with ${\rm Ric}\geq K$ and ${\rm dim}\leq N$ and that via the means of the $\Gamma_2$-calculus introduced by Bakry-\'Emery in \cite{BakryEmery85}, it can be used to define a curvature-dimension condition on a non-smooth structure. One of the main results of \cite{AmbrosioGigliSavare12} is that requiring that \eqref{eq:introboc}, when properly written, holds on a metric measure space, yields a notion stable w.r.t. mGH-convergence. Thus a way alternative to \eqref{eq:rcdkn} to speak about `Riemannian' curvature dimension bounds is to ask for such proper formulation of \eqref{eq:introboc} to hold. The question is then: do the two approaches coincide? The results in \cite{AmbrosioGigliSavare11-2} and \cite{AmbrosioGigliSavare12} give a quite complete answer for the case $N=\infty$, the answer being yes. Nothing is known for the case $N<\infty$. We collect some informal comments about the relations between Bochner inequality and curvature-dimension bounds in Appendix \ref{app:bochner}.

\section{Splitting without the Hessian}\label{se:nohess}

The proof of the splitting theorem in the non-smooth context begins, as in the smooth case, with the study of the Busemann function $\b$ associated to the given line. One crucial  technical difference with the classical case is the fact that in the non-smooth world we currently don't have at disposal - nor we will build - an Hessian. In particular, this means that we don't have at disposal the Bochner identity which is typically used in conjunction with the Bochner inequality to deduce from $\Delta\b\equiv 0$ that $\nabla^2\b\equiv 0$ as well.  The only form of Bochner inequality that is currently available in the non-smooth world is the dimension-free one, which in the smooth case reads as:
\[
\Delta\frac{|\nabla f|^2}{2}\geq \nabla\Delta f\cdot\nabla f+K|\nabla f|^2,\qquad\forall f\in C^{\infty}(M),
\]
valid on manifolds with Ricci curvature $\geq K$. In the case ${\rm Ric}\geq 0$ - the one of interest for the splitting theorem - it reduces to
\begin{equation}
\label{eq:inftyBE}
\Delta\frac{|\nabla f|^2}{2}\geq \nabla\Delta f\cdot\nabla f,\qquad\forall f\in C^{\infty}(M).
\end{equation}
What we want to show here is that once one proves that  the Busemann function $\b$ is harmonic and smooth, inequality \eqref{eq:inftyBE} is sufficient to conclude. The proof in the non-smooth context will be based on these arguments.

Thus assume that $\Delta\b\equiv 0$, recall that  it holds $|\nabla\b|\equiv1$   and then write \eqref{eq:inftyBE} for the function $\b+\eps f$, where $f\in C^\infty(M)$ is arbitrary, to get
\[
\eps\,\Delta(\nabla\b\cdot\nabla f)+\frac{\eps^2}2\Delta|\nabla f|^2\geq \eps\,\nabla\Delta f\cdot\nabla \b+{\eps^2}\,\nabla\Delta f\cdot\nabla f.
\]
Divide by $\eps>0$, let $\eps\downarrow0$ and then substitute $f$ with $-f$ to obtain
\begin{equation}
\label{eq:eulerintro}
\Delta(\nabla\b\cdot\nabla f)=\nabla\Delta f\cdot\nabla \b,\qquad\forall f\in C^\infty(M).
\end{equation}
This equation, which can be thought as the Euler equation for $\b$, given that $\b$  is a pointwise minimizer for $f\mapsto\Delta\frac{|\nabla f|^2}{2}- \nabla\Delta f\cdot\nabla f$, encodes all the informations about the Busemann function one typically deduces from the Bochner identity. Indeed, expanding the left-hand side we see that \eqref{eq:eulerintro} is equivalent to
\begin{equation}
\label{eq:eulerintro2}
\la\nabla^2\b,\nabla^2 f\ra_{\rm HS}+{\rm Ric}(\nabla\b,\nabla f)\equiv0,\qquad\forall f\in C^\infty(M),
\end{equation}
which easily implies $\nabla^2\b\equiv 0$ and ${\rm Ric}(\nabla\b,\cdot)\equiv 0$ given that at any fixed point  $x\in M$ the gradient and the Hessian of $f$ can be chosen independently.

If we don't have at disposal the notion of Hessian we cannot write \eqref{eq:eulerintro2}, but we can still conclude arguing as follows. Let  $M\times\R\ni (x,t)\mapsto \X_t(x)\in M$ be the gradient flow of $\b$, so that $\X_0(x)\equiv x$ and $\frac\d{\d t}\X_t(x)=-\nabla\b(\X_t(x))$, pick $f\in C^\infty_c(M)$, put $f_t:=f\circ\X_t$ and notice that 
\[
\frac{\d}{\d t}\frac12\int|\nabla f_t|^2\,\d\vol=\int\nabla f_t\cdot\nabla(\frac\d{\d t}f_t)\,\d\vol =-\int\nabla f_t\cdot\nabla(\nabla f_t\cdot\nabla \b)\,\d\vol.
\]
Now observe that
\[
\begin{split}
-\int\nabla f_t\cdot\nabla(\nabla f_t\cdot\nabla \b)\,\d\vol&=\int f_t\Delta(\nabla f_t\cdot\nabla \b)\,\d\vol\stackrel{\eqref{eq:eulerintro}}=\int f_t\nabla\Delta f_t\cdot\nabla \b\,\d\vol\\
&=-\int\Delta f_t\nabla\cdot(f_t\nabla\b)\,\d\vol=-\int \Delta f_t\nabla f_t\cdot\nabla\b\,\d\vol\\
&=\int\nabla f_t\cdot\nabla(\nabla f_t\cdot\nabla \b)\,\d\vol,
\end{split}
\]
and therefore $\frac{\d}{\d t}\frac12\int|\nabla f_t|^2\,\d \vol=0$, which gives 
\begin{equation}
\label{eq:quasiiso}
\int|\nabla (f\circ\X_t)|^2\,\d\vol=\int |\nabla f|^2\,\d\vol,\qquad\forall t\in\R,\ f\in C^\infty_c(M).
\end{equation}
The interesting fact about \eqref{eq:quasiiso} is that it can be easily localized. Indeed, notice that by polarization from \eqref{eq:quasiiso} we get 
\begin{equation}
\label{eq:quasiiso2}
\int\nabla (f\circ\X_t)\cdot\nabla(g\circ \X_t)\,\d\vol=\int \nabla f\cdot\nabla g\,\d\vol,\qquad\forall t\in\R,\ f,g\in C^\infty_c(M),
\end{equation}
then fix  $f,g\in C^\infty_c(M)$, put $f_t:=f\circ\X_t$, $g_t:=g\circ\X_t$ and observe that
\begin{equation}
\label{eq:localintro}
\begin{split}
\int g_t|\nabla f_t|^2\,\d\vol&\stackrel{\phantom{\eqref{eq:quasiiso2}}}=\int\nabla (f_tg_t)\cdot\nabla f_t-\tfrac12\nabla g_t\cdot\nabla(f_t^2)\,\d\vol\\
&\stackrel{\eqref{eq:quasiiso2}}=\int\nabla (fg)\cdot\nabla f-\tfrac12\nabla g\cdot\nabla(f^2)\,\d\vol\ =\int g|\nabla f|^2\,\d\vol.
\end{split}
\end{equation}
Recalling now that $\Delta \b\equiv 0$ gives $(\X_t)_\sharp\vol=\vol$ for every $t\in \R$, we know that $\int g|\nabla f|^2\,\d\vol=\int g_t|\nabla f|^2\circ\X_t\,\d\vol$. Thus \eqref{eq:localintro} and the arbitrariness of $g$ gives 
\[
|\nabla (f\circ\X_t)|^2=|\nabla f|^2\circ\X_t,\qquad\forall t\in\R,\ f\in C^\infty_c(M),
\]
i.e.:  $\nabla\X_t(x)$ is an isometry from $T_xM$ into $T_{\X_t(x)}M$ for any $x\in M$ and $t\in\R$, which is the same as saying that $\X_t:(M,\sfd)\to (M,\sfd)$ is an isometry for every $t\in\R$, $\sfd$ being the distance coming from the metric tensor on $M$.

The proof is now almost done. To conclude we need only to   show that $N:=\b^{-1}(0)$ is a totally geodesic submanifold of $M$. To prove this is the same as proving that for $x,y\in N$ the function $t\mapsto \frac12\sfd^2(x,\X_t(y))$ attains a unique minimum at $t=0$. This can achieved without calling into play second order derivatives arguing as follows. Let $t_0$ be a minimum of $t\mapsto \frac12\sfd^2(x,\X_t(y))$ and observe that its Euler equation is 
\begin{equation}
\label{eq:eulminintro}
\la\nabla f,\nabla \b\ra (\X_{t_0}(y))=0,
\end{equation}
where $f(z):=\frac12\sfd^2(x,z)$ (strictly speaking this is not correct unless we prove in advance that $f$ is differentiable at $\X_{t_0}(y)$, but let's not focus on this point here). Now let $[0,1]\ni s\mapsto x_s$ be a constant speed minimizing geodesic from $x$ to $\X_{t_0}(y)$ and observe that a simple application of the triangle inequality and the fact that $\X_t$ is an isometry ensure that a minimum of $t\mapsto \frac12\sfd^2(x,\X_t(x_s))$ is attained at $t=0$, hence according to the above formula we deduce
\[
\la\nabla f,\nabla \b\ra(x_s)=0,\qquad\forall s\in[0,1].
\]
On the other hand, we have $x_s'=\frac1s\nabla f(x_s)$ and therefore 
\[
\b(x_1)-\b(x_0)=\int_0^1\frac{\d}{\d s}\b(x_s)\,\d s=\int_0^1\la x_s',\nabla\b(x_s)\ra\,\d s=\int_0^1\frac1s\la \nabla f(x_s),\nabla\b(x_s)\ra\,\d s=0,
\]
and the proof is complete.

\section{Structure of the paper and further technical comments}

The paper is organized in such a way that every chapter contains exactly one crucial step of the proof of the splitting theorem. At the beginning of each chapter we recall those preliminary results which are necessary to carry on the corresponding argument. This choice is certainly unusual w.r.t. the standard procedure of collecting all the preliminary notions at the beginning of the paper, but we  preferred to do so for the following two reasons. The first is that this work is heavily based on several recents papers whose results, to grant readability,  need to be recalled. Doing so at the beginning would postpone too much the moment in which the new content appears. The second and most important reason, is that proceeding in this way we can better clarify the structure of the proof and the role of the infinitesimal Hilbertianity assumption. In particular, we will see that the gradient flow of the Busemann function preserves the measure even if the space is not infinitesimally Hilbertian, in line with what is known in the Finsler case (\cite{Ohta12}). Infinitesimal Hilbertianity will instead be crucial in deriving all the metric properties.

We also remark that although our proof is synthetic in nature, given that we work on non-smooth spaces, there is quite a bit of (Sobolev) differential calculus involved. In this direction, most of the tools are taken from  \cite{Gigli12} together with some refinement that we do  along the paper, in particular in Chapters \ref{se:dist}, \ref{se:quot}.

We now turn to the a detailed description of the various chapters.

\bigskip

\noindent\underline{Multiples of $\b$ are  Kantorovich potentials}. Chapter \ref{se:metric} contains a simple, yet crucial, result which is valid in a general metric space: let $(X,\sfd)$ be a given proper geodesic  space containing a line $\bar\gamma$, \emph{assume} that
\begin{equation}
\label{eq:maxintro}
\lim_{t\to+\infty}2t-\sfd(x,\bar\gamma_t)-\sfd(x,\bar\gamma_{-t})=0,\qquad\forall x\in X,
\end{equation}
and put $\b(x):=\lim_{t\to+\infty}t-\sfd(x,\bar\gamma_t)$. Then for every real number $a\in\R$ the function $a\b$ is $c$-concave, its $c$-transform being given by $(a\b)^c=-a\b-\frac{a^2}2$. Notice that this is per se a quite strong rigidity result: on $\R^d$ the only functions with this property are affine ones.

The interest of this simple statement is that it provides a link between the gradient flow of $\b$ and optimal transport, which in turn allows to make use of the $CD(0,N)$ condition. Indeed, a simple consequence of the result stated is that if $X\times\R\ni (x,t)\mapsto\X_t(x)\in X$ is a (Borel) gradient flow of $\b$, then the curve $\R\ni t\mapsto (\X_t)_\sharp\mu\in \probt X$ is a constant speed $W_2$-geodesic for any $\mu\in\probt X$.

\medskip

\noindent\underline{The gradient flow of $\b$ preserves the measure}. In Chapter \ref{se:meas} we prove that the gradient flow of the Busemann function $\b$ preserves the measure on quite general $CD(0,N)$ spaces, provided we \emph{assume} that \eqref{eq:maxintro} holds. This assumption will not be needed once we assume infinitesimal Hilbertianity, but in more general situations, due to the lack of an appropriate strong maximum principle, it is not a priori clear whether \eqref{eq:maxintro} holds or not.

At the beginning of the chapter we recall   the definition of Sobolev space $W^{1,2}(X,\sfd,\mm)$ and the fact that for any $f\in W^{1,2}(X,\sfd,\mm)$ there exists a non-negative function $\weakgrad f\in L^2(X,\mm)$ playing the role of the modulus of the distributional differential of $f$. By construction $\weakgrad f$ is a convex function of $f$, and in particular for $f,g$ Sobolev the limits
\[
\lim_{\eps\downarrow0}\frac{\weakgrad{(g+\eps f)}^2-\weakgrad g^2}{2\eps},\qquad\qquad\lim_{\eps\uparrow0}\frac{\weakgrad{(g+\eps f)}^2-\weakgrad g^2}{2\eps},
\]
always exists, both in the $\mm$-a.e. sense and in $L^1(X,\mm)$. According to \cite{Gigli12}, we call `infinitesimally strictly convex' a space for which the two limits above coincide $\mm$-a.e. for any $f,g$ and denote the common value  by $Df(\nabla g)$. On $\R^d$ equipped with a norm and the Lebesgue measure,  such requirement is equivalent in asking that  the norm is strictly convex and in this case   $Df(\nabla g)$ is nothing but the value of the differential of $f$ applied to the gradient of $g$, whence the notation. It turns out that even in the non-smooth case the object $Df(\nabla g)$ can be interpreted as the `value of the differential of $f$ applied to the gradient of $g$' even if we do not provide a definition of what differentials and gradients are. By this we mean  that the standard first order calculus rules valid in a Finsler world are also valid on infinitesimally strictly convex spaces.

Having defined the value of $Df(\nabla g)$, we can integrate by parts and thus give a meaning to the notion of distributional Laplacian. More precisely, we can  give a rigorous meaning to the expression $\bd g=\mu$ for a Sobolev function  $g$ and a Radon measure  $\mu$ by requiring that
\[
-\int Df(\nabla g)\,\d\mm=\int f\,\d\mu,
\]
holds for any Lipschitz compactly supported $f$. Here and in the following we will keep the bold notation when dealing with such measure valued Laplacian. 

According to \cite{Gigli12}, on an infinitesimally strictly convex $CD(0,N)$ space such that\linebreak $W^{1,2}(\Omega,\sfd,\mm)$ is uniformly convex for any $\Omega\subset X$ open, for the Busemann function $\b$ associated to an half-line it holds $\bd\b=\mu$ for some $\mu\geq 0$. In particular,  for the Busemann function associated to a line satisfying \eqref{eq:maxintro} we have $\bd\b=0$.

\medskip

On a smooth setting and for smooth functions $f$ it is a triviality that $\Delta f=0$ implies that the gradient flow of $f$ preserves the measure. We don't know if the same holds in the non-smooth world, part of the problem being that it is not so clear what the gradient flow of $f$ is (for some class of functions containing the Lipschitz ones, the theory in \cite{AmbrosioGigliSavare08} well  covers existence, but the available uniqueness statements  require hypotheses which a priori are not fulfilled in the current setting). 

Yet, for the very special case of the Busemann function $\b$ associated to a line, we can use the fact that both $\b$ and $-\b$ are Kantorovich potentials, granted by  the previous chapter, to argue as follows. Recall that the functional $\u_N$ is defined on $\probt X$ as
\[
\u_N(\mu):=-\int\rho^{1-\frac1N}\,\d\mm,\qquad\mu=\rho\mm+\mu^s,\quad\mu^s\perp\mm,
\]
and that the $CD(0,N)$ assumption means that $\u_N$ is geodesically convex on $(\probt X,W_2)$.

Thus, formally, if $t\mapsto\X_t$ is a gradient flow of $\b$, by what we learned in the previous chapter we know that for any $\mu_0\in\probt X$ the curve $t\mapsto\mu_t:=(\X_t)_\sharp\mu$ is a $W_2$-geodesic. Hence the geodesic convexity of $\u_N$ yields
\[
\lims_{t\downarrow0}\frac{\u_N(\mu_t)-\u_N(\mu_0)}{t}\leq\u_N(\mu_1)-\u_N(\mu_0).
\]
Pretending we can work as in the smooth case, we know that $\mu_t=(\exp(-t\nabla\b))_\sharp\mu_0$ and therefore by direct computation we get
\begin{equation}
\label{eq:dagiustificare}
\lim_{t\downarrow0}\frac{\u_N(\mu_t)-\u_N(\mu_0)}{t}=-\frac1N\int D(\rho^{1-\frac1N})(\nabla\b)\,\d\mm=\frac1N\int\rho^{1-\frac1N}\,\d\bd\b=0,
\end{equation}
$\rho$ being the density of $\mu_0$, which together with the above estimate gives $\u_N(\mu_1)\geq \u_N(\mu_0)$. Reversing the time and noticing that $t\mapsto\X_{-t}$ is a gradient flow of $-\b$ we get the other inequality and thus that $\u_N(\mu)=\u_N((\X_1)_\sharp\mu)$ for every $\mu\in\probt X$. It is easy to see that this forces  $(\X_1)_\sharp\mm=\mm$ and by analogous arguments we get  $(\X_t)_\sharp\mm=\mm$ for every $t\in\R$.

It turns out that it is possible to make this sort of procedure rigorous   in the non-smooth case.  A key role is played by a first order differentiation formula (introduced in \cite{AmbrosioGigliSavare11-2} and generalized in \cite{Gigli12}) which allows to justify \eqref{eq:dagiustificare} on $CD(0,N)$ spaces. Assuming that \eqref{eq:maxintro} holds, that  the space is infinitesimally strictly convex and that $W^{1,2}(\Omega,\sfd,\mm)$ is uniformly convex for every $\Omega\subset X$ open, we are able to obtain that:
\begin{quote}
There exists a Borel gradient flow $\X:\supp(\mm)\times\R\to\supp(\mm)$ of $\b$ such that $(\X_t)_\sharp\mm=\mm$ holds for every $t\in\R$ and $\X_t\circ\X_s=\X_{t+s}$ holds $\mm$-a.e. for every $t,s\in\R$.

This gradient flow is unique in the class of Borel gradient flows $\tilde\X$ of $\b$ such that $(\tilde\X_t)_\sharp\mm\ll\mm$ for every $t\in\R$.
\end{quote}
See Theorem \ref{thm:gfpresmes} for the precise statement.

\medskip

\noindent\underline{The gradient flow of $\b$ preserves the distance}. In Chapter \ref{se:dist} we introduce the infinitesimal Hilbertianity assumption which we shall keep from that moment on. In particular, according to the results in \cite{Gigli12}, \cite{Bjorn-Bjorn07} and \cite{Gigli-Mondino12}, this means that $\bd(\b^++\b^-)\geq 0$, where $\b^+,\b^-$ are the Busmann functions associated to a line, and that the strong maximum principle holds. Therefore in accordance with the strategy used in the smooth setting, \eqref{eq:maxintro} can be proved so that we don't have to assume it a priori as in the previous chapters.

From the algebraic point of view, the crucial effect of infinitesimal Hilbertianity is that $Df(\nabla g)=Dg(\nabla f)$ $\mm$-a.e. for any Sobolev $f,g$. This identity can be interpreted as the abstract analogous of the fact that on Riemannian manifolds we can identify differentials and gradients via the Riesz theorem. To emphasize the symmetry of this object we will denote it by $\la\nabla f,\nabla g\ra$ and to mimic the standard notation used in the Riemannian context we will also write $|\nabla f|$ in place of $\weakgrad f$.

Following the same sort of computations done in the previous section, we will be able to prove that for $f\in W^{1,2}(X,\sfd,\mm)$ we have $f\circ\X_t\in W^{1,2}(X,\sfd,\mm)$ for any $t\in\R$ and
\begin{equation}
\label{eq:introloc}
|\nabla (f\circ\X_t)|=|\nabla f|\circ\X_t,\qquad\mm\ae,\qquad\forall f\in W^{1,2}(X,\sfd,\mm).
\end{equation}
To turn this Sobolev information into a metric one we use the following property of $CD(K,N)$ spaces, which links the Sobolev quantity $|\nabla f|$ to the metric one  $\Lip(f)$:
\begin{equation}
\label{eq:introsobtolip}
\textrm{Let $f\in W^{1,2}(X,\sfd,\mm)$ be with $|\nabla f|\leq 1$ $\mm$-a.e., then $f$ admits a 1-Lipschitz representative.}
\end{equation}
The fact that  \eqref{eq:introsobtolip} holds on $CD(K,N)$ spaces is a consequence of a result by Rajala \cite{Rajala12} concerning existence of $W_2$-geodesics with uniformly bounded densities. The same property also holds on $RCD(K,\infty)$ spaces \cite{AmbrosioGigliSavare11-2} as a consequence of  the regularizing properties of the heat flow, see also Section  \ref{se:sobtolip}.

The idea is then to consider a dense set $\{x_n\}_{n\in\N}\subset\supp(\mm)$ and the functions\linebreak $f_{k,n}:=\max\{0,\min\{\sfd(\cdot,x_n),k-\sfd(\cdot,x_n)\}\}$. These are 1-Lipschitz  with bounded support and thus belong to $W^{1,2}(X,\sfd,\mm)$ with $|\nabla f_{k,n}|\leq 1$ $\mm$-a.e.. Hence by \eqref{eq:introloc} we deduce that for any  $t\in\R$ we have $f_{k,n}\circ\X_t\in W^{1,2}(X,\sfd,\mm)$ with $|\nabla(f_{k,n}\circ\X_t)|\leq 1$ $\mm$-a.e.. Using the property \eqref{eq:introsobtolip} we get that outside a negligible set $\mathcal N_{k,n,t}$ the map $f_{k,n}\circ\X_t$ is 1-Lipschitz and therefore
\[
\sfd(x,y)\geq\sup_{k,n}\big|f_{k,n}(\X_t(x))-f_{k,n}(\X_t(y))\big|=\sfd(\X_t(x),\X_t(y)),\qquad\forall x,y\in X\setminus\cup_{k,n}\mathcal N_{n,t}.
\]
In other words, $\X_t:(\supp(\mm),\sfd)\to(\supp(\mm),\sfd)$ has a 1-Lipschitz representative. Reversing times we are then able to obtain that:
\begin{quote}
There exists a unique continuous map $\bar\X:\supp(\mm)\times\R\to\supp(\mm)$ which coincides $\mm\times\mathcal L^1$-a.e. with $\X$, and this map satisfies
\[
\begin{split}
\sfd(\bar\X_t(x),\bar\X_t(y))&=\sfd(x,y),\qquad\,\,\,\forall x,y\in\supp(\mm),\ t\in\R,\\
\bar\X_t(\bar\X_s(x))&=\bar\X_{t+s}(x),\qquad\forall x\in\supp(\mm), \ t,s\in\R.
\end{split}
\]
\end{quote}
It is worth to underline that the duality principle that allows to deduce from \eqref{eq:introloc} that $\X_t$ has a representative which is an isometry, has little to do with lower Ricci bounds and infinitesimal Hilbertianity: the same arguments can be carried out on spaces having the \emph{Sobolev-to-Lipschitz} property, these being defined as:
\begin{quote}
$(X,\sfd,\mm)$ has the Sobolev-to-Lipschitz property provided any $f\in W^{1,2}(X,\sfd,\mm)$ with $\weakgrad f\leq 1$ $\mm$-a.e. has a 1-Lipschitz representative,
\end{quote}
see Section \ref{se:sobtolip} for the precise definition. Then with the same arguments as before one can prove the following general dualism between metric measure theoretic structures and Sobolev norms:
\begin{quote}
Let $(X_1,\sfd_1,\mm_1)$ and $(X_2,\sfd_2,\mm_2)$ be two spaces with the Sobolev-to-Lipschitz property with $\mm_1,\mm_2$ giving finite mass to bounded sets and   $T:X_1\to X_2$ an invertible Borel map. 

Then $T$ is - up to modification on a negligible set - an isomorphism of metric measure spaces if and only if $\|f\circ T\|_{W^{1,2}(X_1,\sfd_1,\mm_1)}=\|f\|_{W^{1,2}(X_2,\sfd_2,\mm_2)}$ for every Borel function $f:X_2\to\R$.
\end{quote}
See Proposition \ref{prop:isom} (here we adopt the standard convention according to which the Sobolev norm is $+\infty$ if the function is not Sobolev). This duality result should be compared to the simple statement valid in metric spaces:
\begin{quote}
Let $(X_1,\sfd_1)$ and $(X_2,\sfd_2)$ be two metric spaces and $T:X_1\to X_2$ an invertible map.

Then $T$ is an isometry if and only if $\Lip_{X_1}(f\circ T)=\Lip_{X_2}(f)$ for any $f:X_2\to \R$.
\end{quote}
The latter is trivial to prove.  The metric-measure counterpart is slightly more delicate due to the fact that in general Sobolev functions do not carry any information about the geometry of the space: to see why just consider a totally disconnected space $(X,\sfd,\mm)$ and recall that in this case every function $f\in L^2(X,\mm)$ is Sobolev with $\weakgrad f\equiv 0$. Thus if we wish Sobolev functions to be the `dual' object of the metric-measure theoretic structure in the same way as Lipschitz functions are the `dual' object of a metric we need to impose some a priori condition: the   Sobolev-to-Lipschitz property does the job.

As a side remark, we point out that the this property is strongly reminiscent of  the construction of the intrinsic distance $\sfd_{\mathcal E}$ associated to a Dirichlet form $\mathcal E$:
\[
\sfd_{\mathcal E}(x,y):=\Big\{\sup |f(x)-f(y)|\ :\ f\in D(\mathcal E)\cap C(X),\ \textrm{such that}\ \Gamma(f,f)\leq 1,\ \mm\ae\Big\},
\]
the Sobolev-to-Lipschitz property being replaced by the assumption that the functions $f$ considered are continuous. 
We refer to \cite{AmbrosioGigliSavare11-2}, \cite{AmbrosioGigliSavare12} and \cite{Koskela-Zhou12} for some recent advances about the links between the theory of Dirichlet forms and the metric of the underlying space in connection with Ricci curvature lower bounds.

\medskip

\noindent\underline{The quotient space isometrically embeds into the original one}. Once we have a well defined gradient flow of $\b$ on the whole $\supp(\mm)$, we can define the quotient metric space $(X',\sfd')$ as: $X':=\supp(\mm)/\sim$ where $x\sim y$ provided $\bar\X_t(x)=y$ for some $t\in \R$ and
\[
\sfd'(\pi(x),\pi(y)):=\inf_{t\in\R}\sfd(x,\bar\X_t(y)),\qquad\forall x,y\in\supp(\mm),
\]
$\pi:\supp(\mm)\to X'$ being the natural projection. It is also easy to guess what is the correct measure $\mm'$ on $X'$: just put
\[
\mm'(E):=\mm\big(\pi^{-1}(E)\cap \b^{-1}([0,1])\big),\qquad\forall \textrm{ Borel } E\subset X'.
\]
Notice that the map $\pi$ has a natural right inverse $\iota:X'\to X$ defined by
\[
\iota(x'):=x,\qquad\textrm{ provided $\pi(x)=x'$ and $\b(x)=0$},
\]
and that a posteriori, i.e. once the splitting will be proved, we will know that $\iota$ is an isometric embedding. In practice, it seems technically preferable to prove such property in advance and only later deduce the splitting out of it. The reason for this is that it seems hard to get a direct proof of the splitting of the distance according to formula \eqref{eq:intropit}: it will instead be easier to compare the Sobolev spaces $W^{1,2}(X'\times\R,\sfd'\times\sfd_{\rm Eucl},\mm'\times\mathcal L^1)$ and $W^{1,2}(X,\sfd,\mm)$ and then use the same duality principle used in the previous chapter. Yet, in order to do this we need an a priori good knowledge of the Sobolev space $W^{1,2}(X',\sfd',\mm')$ and of its relation with $W^{1,2}(X,\sfd,\mm)$. Also, as a side advantage, knowing that $\iota$ is an isometry will quickly lead to the proof that $(X',\sfd',\mm')$ is an infinitesimally Hilbertian $CD(0,N)$ space, so that we will be allowed to use all the known results about these spaces in what will come next  (but to get that it is a $CD(0,N-1)$ space will not be trivial until we prove that the distance splits). 

\medskip

The geometric idea to get that $\iota$ is an isometry is the same as the one presented at the end of the previous section. The main problem in following that argument in the non-smooth world is in deriving/giving a meaning to the Euler equation \eqref{eq:eulminintro}: the issue is - clearly - that the map $t\mapsto \frac12\sfd^2(x,\bar\X_t(y))$ is not known to be $C^1$, but just Lipschitz. The idea to overcome this problem is to lift the analysis from points to probability measures with bounded density: this has the effect `averaging out' the unsmoothness of the space and leads to the desired $C^1$ regularity. This principle works  both in passing from the study of $s\mapsto \b(x_s)$ to that of $s\mapsto\int \b\,\d\mu_s$  and in passing from $t\mapsto\frac12\sfd^2(x,\bar\X_t(y))$ to $t\mapsto\frac12W_2^2(\mu,(\bar\X_t)_\sharp\nu)$, where $(x_s)$ and $(\mu_s)$ are geodesics in $X$ and $\probt X$ respectively.

Concretely, what we prove is:
\begin{quote}
Let $\mu_0,\mu_1\in\probt X$ with bounded support such that $\mu_0,\mu_1\leq C\mm$ for some $C>0$ and  $(\mu_t)$ the geodesic connecting them.

Then the map $t\mapsto\int\b\,\d\mu_t$ is $C^1$ and its derivative is given by
\[
\frac\d{\d t}\int\b\,\d\mu_t=\frac1t\int\la\nabla\b,\nabla\varphi_t\ra\,\d\mu_t,\qquad\forall t\in(0,1],
\]
where $\varphi_t$ is any Lipschitz Kantorovich potential from $\mu_t$ to $\mu_0$.
\end{quote}
See Proposition \ref{prop:intf}. And similarly:
\begin{quote} 
Let $\mu,\nu\in \probt X$ be with bounded support such that  $\nu\leq C\mm$ for some $C$. 

Then the map $t\mapsto \frac12 W_2^2\big(\mu,(\bar\X_t)_\sharp\nu\big)$ is $C^1$ and its derivative is given by 
\[
\frac\d{\d t}\frac12W_2^2(\mu,(\bar\X_t)_\sharp\nu)=\frac1t\int\la\nabla\phi_t,\nabla\varphi_t\ra\,\d\mu_t,\qquad\forall t\in(0,1],
\]
where $\varphi_t,\phi_t$ are any choice of Lipschitz Kantorovich potentials from $(\bar\X_t)_\sharp\nu$ to $\nu$ and from $(\bar\X_t)_\sharp\nu$ to $\mu$ respectively.
\end{quote}
See Proposition \ref{prop:intdist}. Notice the analogy with the differentiation formulas valid in the smooth setting (see e.g. Chapter 7 of \cite{AmbrosioGigliSavare08}). Notice also that in the first of the differentiation formulas above we mentioned `the' geodesic connecting $\mu_0$ to $\mu_1$ rather than `a' geodesic. Indeed, a recent result by Rajala and Sturm (\cite{RajalaSturm12}) grants that with the above assumptions there exists a unique optimal transport plan from $\mu_0$ to $\mu_1$, that this plan is induced by a map $T$ and that for $\mu_0$-a.e. $x$ the geodesic connecting $x$ to $T(x)$ is unique. This is a genuine  metric-measure-theoretic version of the celebrated Brenier-McCann theorem,  and yields in particular the uniqueness of the $W_2$-geodesic and useful regularity results for the interpolated densities.

With these formulas at disposal we can now proceed by approximation as follows.  Pick $x,y\in \supp(\mm)$, $\eps>0$, define $\mu:=\mm(B_\eps(x))^{-1}\mm\restr{B_\eps(x)}$, $\nu:=\mm(B_\eps(y))^{-1}\mm\restr{B_\eps(y)}$ and look for the minimum of $t\mapsto\frac12W_2^2(\nu,(\bar\X_t)_\sharp\mu)$. With the same computations done at the end of Section \ref{se:nohess} - which are now justified at the level of probability measures - we obtain that the minimum is achieved at the only $t_0$ such that $\int\b\,\d\nu=\int\b\,\d(\bar\X_{t_0})_\sharp\mu$. Since clearly we have $|\int\b\,\d\mu-\b(x)|\leq\eps$ and $|\int\b\,\d\nu-\b(y)|\leq\eps$, by letting $\eps\downarrow0$ we can conclude that the only minimum of $t\mapsto\frac12\sfd^2(y,\bar\X_t(x))$ is achieved for the $t_0$ such that $\b(y)=\b(\bar\X_{t_0}(x))$, which is equivalent to the fact that $\iota:(X',\sfd')\to(X,\sfd)$ is an isometry, as desired.

\medskip

\noindent\underline{``Pythagoras' theorem'' holds}. At this stage of the proof we know that
\begin{equation}
\label{eq:linee}
\begin{array}{rll}
\sfd(x,y)\!\!\!\!&=\sfd'(\pi(x),\pi(y)),\qquad &\textrm{ if $x,y\in\supp(\mm)$ are such that }\b(x)=\b(y),\\
\sfd(x,y)\!\!\!\!&=|\b(x)-\b(y)|,\qquad &\textrm{ if $x,y\in\supp(\mm)$ are such that }\pi(x)=\pi(y),
\end{array}
\end{equation}
but we still need to prove that $\sfd$ splits according to formula \eqref{eq:intropit}, i.e. that it holds
\[
\sfd(x,y)^2=\sfd'(\pi(x),\pi(y))^2+|\b(x)-\b(y)|^2,\qquad\forall x,y\in\supp(\mm).
\]
As said, a direct proof of this formula seems hardly achievable and we will instead proceed by duality with Sobolev functions, as in Chapter \ref{se:dist}.

We will therefore introduce the map  $\mad:\supp(\mm)\to X'\times\R$  by $\mad(x):=(\pi(x),\b(x))$ and prove that the right composition with $\mad$ provides an isometry of $W^{1,2}(X'\times\R)$ in $W^{1,2}(X)$. Given that known results (\cite{AmbrosioGigliSavare11-2}, \cite{AmbrosioGigliSavare12}) grant that $X'\times\R$ has the Sobolev-to-Lipschitz property, this will be sufficient to conclude.

By definition of $\mm'$ and the fact that $(\bar\X_t)_\sharp\mm=\mm$ for any $t\in\R$, it is obvious that $\mad_\sharp\mm=\mm'\times\mathcal L^1$, thus the problem reduces to prove that
\begin{equation}
\label{eq:introprod}
\begin{split}
&f\in W^{1,2}_{\rm loc}(X'\times\R)\quad\Leftrightarrow \quad f\circ\mad\in W^{1,2}_{\rm loc}(X) \\
&\textrm{ and in this case the identity }|\nabla f|_{X'\times\R}\circ\mad=|\nabla (f\circ\mad)|_X \textrm{ holds }\mm'\times\mathcal L^1\ae.
\end{split}
\end{equation}
From the identities \eqref{eq:linee} and the fact that $\mad_\sharp\mm=\mm'\times\mathcal L^1$ it easily follows that the above is true if $f$ depends only on one coordinate, i.e. if either  $f(x',t)=g(x')$ for some $g\in W^{1,2}_{\rm loc}(X')$ or $f(x',t)=h(t)$ for some $h\in W^{1,2}_{\rm loc}(\R)$.

Now we know by assumption that $(X,\sfd,\mm)$ is infinitesimally Hilbertian and from the structural characterization of product spaces given in \cite{AmbrosioGigliSavare11-2},  \cite{AmbrosioGigliSavare12} that $(X'\times\R,\sfd'\times\sfd_{\rm Eucl},\mm'\times\mathcal L^1)$ is infinitesimally Hilbertian as well. Using these informations we can deduce that for $g,h$ as above we have the natural orthogonality relations
\[
\begin{split}
\la \nabla(g\circ\pi_{X'}),\nabla(h\circ\pi_\R)\ra_{X'\times\R}&=0,\qquad\mm'\times\mathcal L^1\ae,\\
\la\nabla(g\circ\pi),\nabla(h\circ\b)\ra_X&=0,\qquad\mm\ae.
\end{split}
\]
From these and basic algebraic manipulation we can prove that \eqref{eq:introprod} holds if $f$ belongs to the algebra generated by functions depending on just one coordinate. The general case will then follow via an approximation argument.

\medskip

\noindent\underline{The quotient space has dimension $N-1$}. With the result of the previous chapter we know that $(X,\sfd,\mm)$ splits as the product of an infinitesimall Hilbertian $CD(0,N)$ space $(X',\sfd',\mm')$ and the Euclidean line $(\R,\sfd_{\rm Eucl},\mathcal L^1)$. What remains to prove is the dimension reduction property, namely that if $N\geq 2$ then $(X',\sfd',\mm')$ is a $CD(0,N-1)$ space and that if $N\in[1,2)$ then $X'$ is just a point.

The case $N\in [1,2)$ can be handled by simply looking at the Hausdorff dimension of $(\supp(\mm),\sfd)$ and recalling (\cite{Lott-Villani09},\cite{Sturm06II}) that it must be bounded from above by $N$.

For the case $N\geq 2$ the argument is very similar to those presented in \cite{Cavalletti-Sturm12} and \cite{Cavalletti12}, where also some sort of dimension reduction appeared. Actually, our situation is simpler than that of these papers because we have at disposal a product structure which is absent in \cite{Cavalletti-Sturm12}, \cite{Cavalletti12}. The idea is to fix a geodesic $t\mapsto\mu_t:=\rho_t\mm'$ on $\probt{X'}$,  choose arbitrary $\alpha,\beta>0$ and consider the geodesic $t\mapsto\nu_t:=\frac1{(1-t)\alpha+t\beta}\mathcal L^1\restr{[0,(1-t)\alpha+t\beta]}$ on $\probt \R$. Given that $X'$ is an infinitesimally Hilbertian $CD(0,N)$ space, using the aforementioned existence and uniqueness of optimal maps we know that there exists a unique optimal geodesic plan $\ppi\in\probt{C([0,1],X')}$ from $\mu_0$ to $\mu_1$.

Simple metric arguments ensure that $t\mapsto \mu_t\times\nu_t$ is a geodesic on $\probt{X'\times\R}$. Since we know that $X'\times\R$ is isomorphic to $X$ and that the latter is a $CD(0,N)$ space, using again the   existence and uniqueness of optimal maps to localize the $CD(0,N)$ condition we can deduce
\[
\left(\frac{\rho_t(\gamma_t)}{(1-t)\alpha+t\beta}\right)^{-\frac1N}\geq(1-t)\left(\frac{\rho_0(\gamma_0)}{\alpha}\right)^{-\frac1N}+t\left(\frac{\rho_1(\gamma_1)}{\beta}\right)^{-\frac1N},\qquad\ppi\ae \ \gamma.
\]
Then a simple  optimization in  $\alpha$ and $\beta$ gives the conclusion. 
%Given that $\alpha,\beta$ were arbitrary positive numbers, we further obtain that
%\[
%\left(\frac{\rho_t(\gamma_t)}{(1-t)\alpha+t\beta}\right)^{-\frac1N}\geq(1-t)\left(\frac{\rho_0(\gamma_0)}{\alpha}\right)^{-\frac1N}+t\left(\frac{\rho_1(\gamma_1)}{\beta}\right)^{-\frac1N},\qquad\forall \alpha,\beta\in\Q,\ \alpha,\beta>0,
%\]
%holds for $\ppi$-a.e. $\gamma$. It is now immediate to see that this implies
%\[
%\rho_{t}(\gamma_{t})^{-\frac1{N-1}}\geq (1-t)\rho_0(\gamma_0)^{-\frac1{N-1}}+t\rho_1(\gamma_1)^{-\frac1{N-1}},\qquad\ppi\ae \ \gamma,
%\]
%which integrated w.r.t. $\ppi$ yields 
%\[
%\u_{N-1}(\mu_t)\leq (1-t)\u_{N-1}(\mu_0)+t\u_{N-1}(\mu_1),
%\]
%as desired.

\bigskip

\noindent\emph{I wish to warmly thank  K.-T. Sturm for several stimulating conversations I had with him while working on this project}

\mainmatter
%    Include main chapters here.

\chapter{Multiples of $\b$ are Kantorovich potentials}\label{se:metric}
\section{Preliminary notions}

\subsection{Metric spaces}

The  metric spaces $(X,\sfd)$ that we shall consider will always be complete and separable. In most cases, we shall actually deal with proper spaces, i.e. spaces such that bounded closed sets are compact.

By support of a Borel measure $\mu$ on $X$ we intend the intersection of all closed sets where $\mu$ is concentrated, we will denote it by $\supp(\mu)$. Similarly for functions.

For $a,b\in\R$, $a<b$, a curve $\gamma:[a,b]\to X$ is said absolutely continuous provided there exists a function $f\in L^1(a,b)$ such that
\begin{equation}
\label{eq:ac}
\sfd(\gamma_t,\gamma_s)\leq\int_t^s f(r)\,\d r,\qquad\forall t<s\in[a,b].
\end{equation}
It turns out that if $\gamma$ is absolutely continuous then the limit
\begin{equation}
\label{eq:ms}
\lim_{h\to 0}\frac{\sfd(\gamma_{t+h},\gamma_t)}{|h|},
\end{equation}
exists for $\mathcal L^1$-a.e. $t\in[a,b]$, where here and in the following we denote by $\mathcal L^1$ the Lebesgue measure on $\R$. The limit in \eqref{eq:ms}  is called metric speed of the curve, denoted by $|\dot\gamma|$, it belongs to $L^1(a,b)$ and is the minimal $L^1$ function -  in the $\mathcal  L^1$-a.e. sense - that can be chosen as $f$ in \eqref{eq:ac}.

A curve $\gamma:[0,1]\to X$ is a minimizing constant speed geodesic, or simply geodesic, if 
\[
\sfd(\gamma_t,\gamma_s)\leq |s-t| \sfd(\gamma_0,\gamma_1),\qquad\forall t,s\in[0,1].
\]
The space $(X,\sfd)$ is said geodesic provided for any $x,y\in X$ there exists a geodesic connecting them.

A curve $\gamma:\R^+\to X$ is a \emph{half line} provided
\[
\sfd(\gamma_t,\gamma_s)=|t-s|,\qquad\forall t,s\geq 0.
\]
To an half line it is associate the Busemann function $\b:X\to\R$ defined by
\begin{equation}
\label{eq:halfbus}
\b(x):=\lim_{t\to+\infty}t-\sfd(x,\gamma_t),
\end{equation}
a simple application of the triangle inequality shows that the limit exists and is real valued for any $x\in X$, and thus $\b$ is a well defined 1-Lipschitz function. A curve $\gamma:\R\to X$ is a \emph{line}  provided 
\[
\sfd(\gamma_t,\gamma_s)=|t-s|,\qquad\forall t,s\in\R.
\]
To a line we can associate two Busemann functions $\b^+,\b^-$, one for each of the two naturally induced half-lines:
\begin{equation}
\label{eq:busemann}
\b^+(x):=\lim_{t\to+\infty}t-\sfd(x,\gamma_t),\qquad\qquad\b^-(x):=\lim_{t\to+\infty}t-\sfd(x,\gamma_{-t}).
\end{equation}
Notice that the triangle inequality ensures that
\begin{equation}
\label{eq:basebus}
\b^++\b^-\leq 0.
\end{equation}

We shall denote by $C([0,1],X)$ the space of continuous curves on $[0,1]$ with values in $X$ endowed with the $\sup$ norm. It is a complete and separable.  For $t\in[0,1]$ the evaluation map $\e_t:C([0,1],X)\to X$ is defined by
\[
\e_t(\gamma):=\gamma_t,\qquad\forall \gamma\in C([0,1],X).
\]
The space $\geo(X)\subset C([0,1],X)$ is the set of all geodesics; it is complete and separable.

Given a function $f:X\to\R$ we shall denote by $\lip(f):X\to[0,\infty]$ its local Lipschitz constant defined by
\[
\lip(f)(x):=\lims_{y\to x}\frac{|f(y)-f(x)|}{\sfd(x,y)},
\]
if $x$ is not isolated and $\lip(f)(x)=0$ otherwise. The global Lipschitz constant, or simply Lipschitz constant, $\Lip(f)$ is instead defined by
\[
\Lip(f):=\sup_{x,y\in X}\frac{|f(y)-f(x)|}{\sfd(x,y)}.
\]

Let $f:X\to\R$ be a Lipschitz function and $\gamma:I\to X$ an absolutely continuous curve, $I\subset \R$ being a non-trivial interval. Then the map $t\mapsto f(\gamma_t)$ is absolutely continuous and by direct computation one sees that
\[
\frac\d{\d t}f(\gamma_t)\leq |\dot\gamma_t|\,\lip(f)(\gamma_t),\qquad\mathcal L^1\ae \ t.
\]
In particular we deduce that for $t<s\in I$ it holds
\begin{equation}
\label{eq:tuttecurve}
f(\gamma_t)\leq f(\gamma_s)+\frac12\int_t^s|\dot\gamma_r|^2\,\d r+\frac12\int_t^s\lip(f)^2(\gamma_r)\,\d r.
\end{equation}
If the ambient space $X$ is $\R^d$, equality holds in \eqref{eq:tuttecurve} if and only if $\gamma_t'=-\nabla f(\gamma_t)$ for $\mathcal L^1$-a.e. $t$, i.e. if and only if $\gamma$ is a gradient flow trajectory for $f$. Asking for equality in \eqref{eq:tuttecurve} makes sense also in arbitrary metric spaces, thus we are lead to the following definition:
\begin{definition}[Gradient flows]\label{def:gf}
Let $(X,\sfd)$ be a metric space, $f:X\to\R$ a Lipschitz map, $I\subset\R$ a non trivial interval and $\gamma:I\to X$ a curve.

We say that $\gamma$ is a gradient flow trajectory of $f$ provided
\[
f(\gamma_t)= f(\gamma_s)+\frac12\int_t^s|\dot\gamma_r|^2\,\d r+\frac12\int_t^s\lip(f)^2(\gamma_r)\,\d r,
\]
holds for any $t<s$, $t,s\in I$.
\end{definition}
Notice that this is not really the appropriate  definition of gradient flow in a metric setting, as in general the \emph{descending slope} should be used in place of the local Lipschitz constant. Yet, given that we will need this concept only for the Busemann function and that for it the two approaches coincide, we preferred  to proceed with the more direct apporach. See \cite{AmbrosioGigliSavare08} for a general overview on the topic.

\subsection{Optimal transport}
Let $(X,\sfd)$ be complete and separable and denote by $\prob X$ the space of Borel probability measures on $X$.

The space $\probt X\subset \prob X$ is the set of those probability measures with finite second moment, i.e.
\[
\probt X:=\Big\{\mu\in\prob X\ :\ \int\sfd^2(\cdot,x_0)\,\d\mu<\infty,\textrm{ for some - and thus any - }x_0\in X\Big\}.
\]
For $\mu,\nu\in \probt X$ their quadratic transportation distance $W_2(\mu,\nu)$ is defined by
\begin{equation}
\label{eq:ot}
W_2^2(\mu,\nu):=\min_{\sggamma}\int\sfd^2(x,y)\,\d\ggamma(x,y),
\end{equation}
where the minimum is considered among all $\ggamma\in\prob{X^2}$ such that $\pi^1_\sharp\ggamma=\mu$ and $\pi^2_\sharp\ggamma=\nu$. It turns out that $W_2$ is actually a distance on $\probt X$ and that $(\probt X,W_2)$ is complete and separable provided $(X,\sfd)$ is so.

An important characterization of $W_2$ can be given in terms on the dual formulation of the optimal transport problem: one can see that it holds
\begin{equation}
\label{eq:dualot}
\frac12W_2^2(\mu,\nu)=\sup\int \varphi\,\d\mu+\int\varphi^c\,\d\nu,
\end{equation}
where the $\sup$ is taken among all the Borel functions $\varphi:X\to \R\cup\{-\infty\}$ such that $\varphi\in L^1(\mu)$ and the $c$-transform $\varphi^c$ of $\varphi$ is defined as
\[
\varphi^c(y):=\inf_{x\in X}\frac{\sfd^2(x,y)}2-\varphi(x).
\]
A function $\varphi:X\to \R$ such that $\varphi^{cc}=\varphi$ is said $c$-concave.

It turns out that the $\sup$ in \eqref{eq:dualot} is always attained and that optimizers $\varphi$ can be chosen to be $c$-concave: any such $\varphi$ is called Kantorovich potential from $\mu$ to $\nu$, or Kantorovich potential relative to $(\mu,\nu)$. It is also possible to see that Kantorovich potentials can be chosen to satisfy the following property, slight stronger than $c$-concavity:
\[
\varphi(x)=\inf_{y\in\supp(\mu)}\frac{\sfd^2(x,y)}{2}-\varphi^c(y).
\]
This  shows, in particular, that if $\nu$ has bounded support then $\varphi$ can be chosen to be Lipschitz on bounded sets. If $\mu$ has also bounded support, the $\varphi$ can be chosen to be bounded and globally Lipschitz.

For $\varphi:X\to\R$ $c$-concave, the $c$-superdifferential $\partial^c\varphi\subset X^2$ is defined as the set of $(x,y)$ such that
\[
\varphi(x)+\varphi^c(y)=\frac{\sfd^2(x,y)}{2},
\]
(notice that for arbitrary $(x,y)$ the inequality $\leq$ holds). For $x\in X$, the set $\partial^c\varphi(x)\subset X$ is then the set of those $y$'s such that $(x,y)\in\partial^c\varphi$.

A crucial property of Kantorovich potentials $\varphi$ is that 
\begin{equation}
\label{eq:optconc}
\textrm{any plan $\ggamma$ which minimizes \eqref{eq:ot} must be concentrated on $\partial^c\varphi$.}
\end{equation}

If we further assume that $(X,\sfd)$ is a geodesic space, then $W_2$ can be equivalently characterized in dynamic terms as:
\begin{equation}
\label{eq:dot}
W_2^2(\mu,\nu)=\min\iint_0^1|\dot\gamma_t|^2\,\d t\,\d\ppi(\gamma),
\end{equation}
where the minimum is taken among all $\ppi\in\prob{C([0,1],X)}$ such that $(\e_0)_\sharp\ppi=\mu$ and $(\e_1)_\sharp\ppi=\nu$. Here and in the following we adopt the  convention according to which $\int_0^1|\dot\gamma_t|^2\,\d t=+\infty$ if $\gamma$ is not absolutely continuous, thus  any $\ppi$ such that $\iint_0^1|\dot\gamma_t|^2\,\d t\,\d\ppi(\gamma)<\infty$ must be concentrated on absolutely continuous curves. Any plan that realizes the $\min$ in \eqref{eq:dot} is called optimal geodesic plan, or simply optimal plan. The set of optimal geodesic plans from $\mu$ to $\nu$ is denoted by $\gopt(\mu,\nu)$. Any plan in $\gopt(\mu,\nu)$ must be concentrated on $\geo(X)$.

 If $\ppi\in\gopt(\mu,\nu)$, then $\ggamma:=(\e_0,\e_1)_\sharp\ppi$ minimizes \eqref{eq:ot}, thus by  \eqref{eq:optconc} we get that 
any $\ppi\in\gopt(\mu,\nu)$ fulfills $\gamma_1\in\partial^c\varphi(\gamma_0)$ for $\ppi$-a.e. $\gamma$,
where $\varphi$ is any Kantorovich potential from $\mu$ to $\nu$. 

There is a strict link between geodesics in $(\probt X,W_2)$ and optimal geodesic plans: for any $(\mu_t)$ geodesic there exists $\ppi\in\gopt(\mu_0,\mu_1)$ such that
\[
(\e_t)_\sharp\ppi=\mu_t,\qquad\forall t\in[0,1],
\]
and conversely any $\ppi\in\gopt(\mu_0,\mu_1)$ produces a geodesic via the above formula. Any such $\ppi$ is said to induce, or to be a lifting of, the geodesic $(\mu_t)$ and any Kantorovich potential from $\mu_0$ to $\mu_1$ is also said to be a Kantorovich potential relative to $(\mu_t)$.

More generally, Lisini (\cite{Lisini07}) proved that every absolutely continuos curve $(\mu_t)$ admits a lifting in the following sense:
\begin{theorem}
Let $(X,\sfd)$ be a complete, separable metric space and $(\mu_t)\subset\probt X$ a $W_2$-absolutely continuous curve such that $\int_0^1|\dot\mu_t|^2\,\d t<\infty$. Then there exists a plan\linebreak $\ppi\in\prob{C([0,1],X)}$ such that
\[
\begin{split}
(\e_t)_\sharp\ppi&=\mu_t,\qquad\ \forall t\in[0,1],\\
\int|\dot\gamma_t|^2\,\d\ppi(\gamma)&=|\dot\mu_t|,\qquad\rm{a.e.} \ t\in[0,1].
\end{split}
\]
\end{theorem}

\section{Result}
In this section we shall assume that:
\begin{equation}
\label{eq:ass1}
\begin{split}
&(X,\sfd)\textrm{ is a proper geodesic space, }\bar\gamma:\R\to X\textrm{ is a line,}\\
& \b^+\textrm{ and $\b^-$ are the associated Busemann functions as in \eqref{eq:busemann}},\\
&\textrm{the identity $\b^++\b^-=0$ holds on all $X$. Put $\b:=\b^+$}
\end{split}
\end{equation}
Recall that according to \eqref{eq:basebus} in general only the inequality $\b^++\b^-\leq 0$ holds. Here we want to analyze which sort of metric rigidity properties can be inferred by the hypothesis $\b^++\b^-\equiv 0$:
\begin{theorem}[Multiples of $\b$ are Kantorovich potentials]\label{thm:basemetric}
Let $(X,\sfd)$ and $\bar\gamma$ be as in \eqref{eq:ass1}. Then the following are true.
\begin{itemize} 
\item[i)] For every $a\in\R$ the function $a\b$ is $c$-concave and fulfills
\begin{equation}
\label{eq:abcconc}
\begin{split}
(a\b)^c&=-a\b-\frac{a^2}2,\\
(-a\b)^c&=a\b-\frac{a^2}2.
\end{split}
\end{equation}
In particular, $(x,y)\in\partial^c(a\b)$ if and only if $(y,x)\in \partial^c(-a\b)$.
\item[ii)] $\lip(\b)\equiv 1$.
\item[iii)] For $a\in\R$ and $\gamma:[0,1]\to X$ the following are equivalent:
\begin{itemize}
\item[a)] $\gamma$ is a constant speed geodesic and $\gamma_1\in \partial^c(a\b)(\gamma_0)$
\item[b)] $\gamma$ is a gradient flow trajectory of $a\b$ 
\end{itemize}
In particular, if $\gamma:[0,1]\to X$ is a geodesic with $\gamma_1\in\partial^c(a\b)(\gamma_0)$, then we have
\[
\gamma_t\in\partial^c(ta\b)(\gamma_0),\qquad\forall t\in[0,1].
\]
\item[iv)] For $a\in\R$ holds
\begin{equation}
\label{eq:facile}
\begin{split}
y\in \partial^c(a\b)(x)\qquad&\Leftrightarrow\qquad\sfd(x,y)= |a|\quad\textrm{ and }\quad \b(x)-\b(y)=a\\
\end{split}
\end{equation}
\item[v)] For $a_1,a_2\in\R$ with $a_1a_2\geq 0$ it holds
\[
y\in\partial^c(a_1\b)(x),\quad z\in\partial^c(a_2\b)(y)\qquad\Rightarrow\qquad z\in\partial^c((a_1+a_2)\b)(x).
\]
\end{itemize}
\end{theorem}
\begin{proof}$\ $\\
\noindent{${\mathbf{ (i)}}$} Fix $a\in\R$ and notice that since $a\b$ is $|a|$-Lipschitz we have
\[
a\b(x)-a\b(y)\leq|a|\sfd(x,y)\leq \frac{\sfd^2(x,y)}2+\frac{a^2}2,\qquad\forall x,y\in X,
\]
which yields $\frac{\sfd^2(x,y)}{2}-a\b(x)\geq -a\b(y)-\frac{a^2}2$ for any $x,y\in X$, and thus
\[
(a\b)^c(y)\geq -a\b(y)-\frac{a^2}2,\qquad\forall y\in X.
\]
To prove the opposite inequality, fix $y\in X$ and assume for the moment $a\geq 0$. Let  $\gamma^{t,y}:[0,\sfd(y,\gamma_t)]\to X$ be a unit speed geodesic connecting $y$ to $\gamma_t$ and notice that since $(X,\sfd)$ is proper, for some sequence $t_n\uparrow +\infty$ the sequence $n\mapsto \gamma^{t_n,y}_a$ converges to some point $ y_a\in X$ which clearly has distance $a$ from $y$. 

Letting $n\to\infty$  in
\[
\begin{split}
t_n-\sfd(y_a,\bar\gamma_{t_n})\geq t_n-\sfd(\gamma^{t_n,y}_a,\bar\gamma_{t_n})-\sfd(y_a,\gamma^{t_n,y}_a)= t_n-\sfd(y,\bar\gamma_{t_n})+a-\sfd(y_a,\gamma^{t_n,y}_a),
\end{split}
\]
and recalling that $\b=\lim_{n\to\infty}t_n-\sfd(\cdot,\bar\gamma_{t_n})$ we deduce 
\begin{equation}
\label{eq:perdopo}
\b(y_a)\geq \b(y)+a.
\end{equation}
Choosing $ y_a$ as competitor in the definition of $(a\b)^c(y)$ we obtain
\[
(a\b)^c(y)=\inf_x\frac{\sfd^2(x,y)}{2}-a\b(x)\leq\frac{\sfd^2( y_a,y)}{2}-a\b( y_a)\leq -a\b(y)-\frac{a^2}2,
\]
as desired. The case $a\leq0$ is handled analogously by letting $y_a$ be any limit of $\gamma^{-t,y}_{|a|}$ as $t\to+\infty$ and using the fact that $\b=\lim_{t\to+\infty}\sfd(\cdot,\gamma_{-t})-t$.

This proves the first identity in \eqref{eq:abcconc}. The second follows from the first choosing $-a$ in place of $a$. Finally, the $c$-concavity of $a\b$ is obtained by direct algebraic manipulation:
\[
(a\b)^{cc}=\left(-a\b-\frac{a^2}2\right)^c=(-a\b)^c+\frac{a^2}2=a\b.
\]
The last assertion follows from the fact that $(x,y)\in\partial^c(a\b)$ if and only if $(y,x)\in\partial^c(a\b)^c$ and identities \eqref{eq:abcconc}.

\noindent{${\mathbf{ (ii)}}$}  We already know that $\b$ is 1-Lipschitz and thus $\lip(\b)(x)\leq 1$ for every $x\in X$. For $y\in X$ and $a>0$, we proved in point $(i)$, that there exists a point $y_a\in X$ such that $\sfd(y,y_a)=a$ and  \eqref{eq:perdopo} holds. Hence $\frac{\b(y_a)-\b(y)}{\sfd(y,y_a)}\geq 1$ and letting $a\downarrow 0$ we deduce $\lip(\b)(y)\geq 1$. 

\noindent{${\mathbf{ (iii)}}$}  The second part of the claim follows directly from the equivalence of $(a)$ and $(b)$.

\noindent\underline{(a) $\Rightarrow$ (b)} We know that
\[
a\b(\gamma_0)+\big(a\b\big)^c(\gamma_1)=\frac{\sfd^2(\gamma_0,\gamma_1)}{2}=\frac12\int_0^1|\dot\gamma_t|^2\,\d t,
\] 
where the second equality comes from the fact that $\gamma$ is a constant speed geodesic. Recalling the  first in \eqref{eq:abcconc} we
deduce
\[
a\b(\gamma_0)=a\b(\gamma_1)+\frac{a^2}2+\frac12\int_0^1|\dot\gamma_t|^2\,\d t=a\b(\gamma_1)+\frac12\int_0^1\lip(a\b)^2(\gamma_t)\,\d t+\frac12\int_0^1|\dot\gamma_t|^2\,\d t,
\]
having used point $(ii)$.

\noindent\underline{(b) $\Rightarrow$ (a)}  Recalling point $(ii)$ and the definition of gradient flow trajectory we have
\begin{equation}
\label{eq:ba}
\begin{split}
a\b(\gamma_0)&= a\b(\gamma_1)+\frac12\int_0^1\lip(a\b)^2(\gamma_t)\,\d t+\frac12\int_0^1|\dot\gamma_t|^2\,\d t\geq a\b(\gamma_1)+\frac{a^2}2+\frac12\left(\int_0^1|\dot\gamma_t|\,\d t\right)^2\\
&\geq a\b(\gamma_1)+\frac{a^2}2+\frac12\sfd^2(\gamma_0,\gamma_1).
\end{split}
\end{equation}
Thus from the first in \eqref{eq:abcconc} we obtain
\[
a\b(\gamma_0)\geq -\big(a\b\big)^c(\gamma_1)+\frac12\sfd^2(\gamma_0,\gamma_1),
\]
which shows that $\gamma_1\in\partial^c(a\b)(\gamma_0)$. Furthermore, this last inequality is in fact an equality, which forces the inequalities in \eqref{eq:ba} to be equalities, i.e. $\int_0^1|\dot\gamma_t|^2\,\d t=\sfd^2(\gamma_0,\gamma_1)$, which is true if and only if $\gamma$ is a constant speed geodesic.

\noindent{${\mathbf{ (iv)}}$}  $\Leftarrow$ is obvious. For $\Rightarrow$ pick  $x,y\in X$ such that $y\in \partial^c(a\b)(x)$ and $\gamma$ a constant speed geodesic connecting $x$ to $y$. Then by point $(iii)$ we know that $\gamma$ is a gradient flow trajectory for $a\b$, which in particular implies $|\dot\gamma_t|=\lip (a\b)(\gamma_t)=|a|$ for a.e. $t$ and thus $\sfd(\gamma_0,\gamma_1)=|a|$. The equality $\b(\gamma_0)-\b(\gamma_1)=a$ now follows from point $(i)$.

\noindent{${\mathbf{ (v)}}$} Direct consequence of point $(iv)$.
\end{proof}

\chapter{The gradient flow of $\b$ preserves the measure}\label{se:meas}
\section{Preliminary notions}
\subsection{Borel selection}

We shall occasionally make use of the following well known selection theorem, due to Kuratowski and Ryll-Nardzeweski (\cite{Kuratowski65}), which we will state as:
\begin{theorem}[Borel selection]\label{thm:borelsel}
Let $(X,\sfd)$ and $(X',\sfd')$ be complete and separable metric spaces, denote by $\mathcal F(X')$ the collection of closed subsets of $X'$ and let $S:X\to\mathcal F(X')$ be a map. Assume that:
\begin{itemize}
\item[i)] $S(x)\neq\emptyset$ for every $x\in X$,
\item[ii)] for every open set $\Omega\subset X'$ the set $\{x\in X:S(x)\cap \Omega\neq\emptyset\}$ is Borel.
\end{itemize}
Then there exists a Borel map $T:X\to X'$ such that $T(x)\in S(x)$ for every $x\in X$.
\end{theorem}

Given a complete separable space $(X,\sfd)$ and a Borel set $E\subset X$, we say that a map $\mathbf T:E\to\prob X$ is Borel provided for every Borel $E'\subset X$ the map $x\mapsto \mathbf T(x)(E')\in[0,1]$ is Borel. Theorem \ref{thm:borelsel} above implies the following.
\begin{corollary}[Variant of Borel selection]\label{cor:borelsel}
Let $\mathbf T:X\to\prob X$ be a Borel map. Then there exists a Borel map $T:X\to X$ such that $T(x)\in\supp(\mathbf T(x))$ for every $x\in X$.
\end{corollary}
For the proof, just notice that $\supp(\mathbf T(x))$ is a closed non-empty set for any $x\in X$ and that for $\Omega\subset X$ open it holds
\[
\Big\{x\in X\ :\ \supp(\mathbf T(x))\cap \Omega\neq\emptyset\Big\}\quad=\quad \Big\{x\in X\ :\ \mathbf T(x)(\Omega)>0\Big\}.
\]

\subsection{Metric measure spaces and  $CD(0,N)$ condition}

In this paper all metric measure spaces $(X,\sfd,\mm)$ considered are such that $(X,\sfd)$ is complete and separable and $\mm$ is a non-negative Radon measure on $X$. These assumption will always be taken implicitly. 

\begin{definition}[Isomorphisms between metric measure spaces]
Two metric measure spaces $(X_1,\sfd_1,\mm_1)$, $(X_2,\sfd_2,\mm_2)$ are said \emph{isomorphic} provided there exists an isometry $T:(\supp(\mm_1),\sfd_1)\to(\supp(\mm_2),\sfd_2)$ such that $T_\sharp\mm_1=\mm_2$. Any such $T$ is called isomorphism.
\end{definition}
Notice that in the definition of isomorphism only the portion of the space where the measure is concentrated matters, so that $(X,\sfd,\mm)$ is always isomorphic to $(\supp(\mm),\sfd,\mm)$. For this reason and in order to simplify the terminology, in the following when saying that a certain subset $A$ of $X$ is bounded (resp. compact) we will always implicitly mean that $A\cap\supp(\mm)$ is bounded (resp. compact). 

\medskip

We turn to the definition of metric measure space with ${\rm Ric}\geq 0$ and ${\rm dim}\leq N$ according to Lott-Sturm-Villani (\cite{Lott-Villani09}, \cite{Sturm06II}). For $N\in(1,\infty)$ define  $u_N:[0,\infty]\to\R$ by
\[
u_N(z):=-z^{1-\frac1N},
\]
and define the functional $\u_N:\prob X\to\R\cup\{-\infty\}$ by
\[
\u_N(\mu):=\int u_N(\rho)\,\d\mm,\qquad\mu=\rho\mm+\mu^s,\ \mu^s\perp\mm.
\]
In the limiting case $N=1$ we put
\[
\u_1(\mu):=-\mm(\{\rho>0\}),\qquad\mu=\rho\mm+\mu^s,\ \mu^s\perp\mm.
\]

\begin{definition}[$CD(0,N)$ spaces]
Let $N\in[1,\infty)$ and $(X,\sfd,\mm)$ a metric measure space. We say that $(X,\sfd,\mm)$ is a $CD(0,N)$ space provided for any $\mu_0,\mu_1\in\prob X$ with supports contained in $\supp(\mm)$ there exists $\ppi\in\gopt(\mu_0,\mu_1)$ such that
\begin{equation}
\label{eq:defcd}
\u_{N'}\big((\e_t)_\sharp\ppi\big)\leq(1-t)\u_{N'}(\mu_0)+t\u_{N'}(\mu_1),\qquad\forall t\in[0,1],
\end{equation}
for every $N'\geq N$.
\end{definition}
In the following proposition we recall the Bishop-Gromov volume estimates on $CD(0,N)$ spaces proved in \cite{Lott-Villani09} and \cite{Sturm06II}. For $x\in \supp(\mm)$ we shall put ${\sf v}_x(r):=\mm(B_r(x))$ and ${\sf s}_x(r):=\lims_{\eps\downarrow0}\frac{{\sf v}_x(r+\eps)-{\sf v}_x(r)}{\eps}$.
\begin{proposition}[Generalized Bishop-Gromov estimates]\label{prop:BG}
Let $(X,\sfd,\mm)$ be a $CD(0,N)$ space and $x\in \supp(\mm)$. Then  ${\sf v}_x(r),{\sf s}_x(r)<\infty$ for every $r\geq 0$ and 
\begin{equation}
\label{eq:bgarea}
\R^+\ni r\qquad\mapsto\qquad\frac{{\sf s}_x(r)}{r^{N-1}}\textrm{\qquad is non-increasing,}
\end{equation}
$r\mapsto {\sf v}_x(r)$ is locally Lipschitz and
\begin{equation}
\label{eq:bgvol}
\R^+\ni r\qquad\mapsto\qquad\frac{{\sf v}_x(r)}{r^{N}}\textrm{\qquad is non-increasing.}
\end{equation}
\end{proposition}

\subsection{Sobolev functions}\label{se:sob}
Here we recall the definition of Sobolev function on a metric measure space $(X,\sfd,\mm)$.  The definition is taken from \cite{AmbrosioGigliSavare11} (along the presentation given in \cite{Gigli12}), where also the proof of the equivalence with the notions introduced in \cite{Cheeger00} and \cite{Shanmugalingam00} is given. See also \cite{AmbrosioGigliSavare11-3}.
\begin{definition}[Test Plans]
Let $\ppi\in\prob{C([0,1],X)}$. We say that $\ppi$ is a test plan provided 
\[
(\e_t)_\sharp\ppi\leq C\mm,\qquad\forall t\in[0,1],
\]
for some constant $C>0$, and 
\[
\iint_0^1|\dot\gamma_t|^2\,\d t\ppi(\gamma)<\infty.
\]
\end{definition}
Notice that in particular according to the convention $\int_0^1|\dot\gamma_t|^2\,\d t=+\infty$ if $\gamma$ is not absolutely continuous, any test plan must be concentrated on absolutely continuous curves.
\begin{definition}[The Sobolev class $\s^2(X,\sfd,\mm)$]
The Sobolev class $\s^2(X,\sfd,\mm)$ (resp.\linebreak $\s^2_{\rm loc}(X,\sfd,\mm)$) is the space of all Borel functions $f:X\to\R$ such that there exists a non-negative $G\in L^2(X,\mm)$ (resp. $G\in L^2_{\rm loc}(X,\sfd,\mm)$) for which it holds
\begin{equation}
\label{eq:defsob}
\int|f(\gamma_1)-f(\gamma_0)|\,\d\ppi(\gamma)\leq \iint_0^1G(\gamma_t)|\dot\gamma_t|\,\d t\,\d\ppi(\gamma),\qquad\forall \ppi\textrm{ test plan}.
\end{equation}
\end{definition}
Here and in the following by $L^2_{\rm loc}(X,\mm)$ we intend the set of Borel functions $f$ such that for every $x\in X$ there exists an open set $\Omega_x\ni x$ such that $f\in L^2(\Omega_x,\mm)$.

It turns out that for $f\in \s^2(X,\sfd,\mm)$ there exists a minimal $G$ in the $\mm$-a.e. sense for which \eqref{eq:defsob} holds: we will denote it by $\weakgrad f$ and call it minimal weak upper gradient (this terminology is the standard one in the setting of analysis in metric measure spaces, yet,  being this object defined in duality with speed of curves, it is closer to the norm of a cotangent vector rather to a tangent one, whence the notation used).

The minimal weak upper gradient $\weakgrad f$ is a local object, in the sense that for $f\in\s^2_{\rm loc}(X,\sfd,\mm)$ we have
\begin{equation}
\label{eq:nullgrad}
\weakgrad f=0,\qquad\textrm{ on }f^{-1}(N),\qquad\forall N\subset \R,\ \textrm{ with }\mathcal L^1(N)=0,
\end{equation}
and
\begin{equation}
\label{eq:localgrad}
\weakgrad f=\weakgrad g,\qquad\mm\textrm{-a.e.\ on }\{f=g\},\ \forall f,g\in\s^2_{\rm loc}(X,\sfd,\mm).
\end{equation}
Also, for any $f\in \s^2(X,\sfd,\mm)$, $\ppi$ test plan and $t<s\in[0,1]$ it holds
\begin{equation}
\label{eq:localplan}
|f(\gamma_s)-f(\gamma_t)|\leq \int_t^s\weakgrad f(\gamma_r)|\dot\gamma_r|\,\d r,\qquad\ppi\ae\ \gamma.
\end{equation}

In particular, the definition of Sobolev class can be directly localized to produce the notion of Sobolev function defined on an open set $\Omega\subset X$:
\begin{definition}\label{def:parigi}
Let $\Omega\subset X$ be an open set. A Borel function $f:\Omega\to X$ belongs to $\s^2_{\rm loc}(\Omega,\sfd,\mm)$ provided for any Lipschitz function $\nchi:X\to \R$ with $\supp(\nchi)\subset \Omega$ it holds $f\nchi\in \s^2_{\rm loc}(X,\sfd,\mm)$. In this case, the function $\weakgrad f:\Omega\to [0,\infty]$ is $\mm$-a.e. defined by
\[
\weakgrad{f}:=\weakgrad{(\nchi f)},\qquad\mm\ae \ {\rm on}\ \nchi\equiv1,
\] 
for any $\nchi$ as above. Notice that thanks to \eqref{eq:localgrad} this is a good definition. The space $\s^2(\Omega)\subset\s^2_{\rm loc}(\Omega)$ is the set of $f$'s such that $\weakgrad f\in L^2(\Omega,\mm)$.
\end{definition}

The basic calculus properties of Sobolev functions are collected below.  $\Omega\subset X$ is open and all the (in)equalities are intended $\mm$-a.e. on $\Omega$.

\noindent\underline{Lower semicontinuity of minimal weak upper gradients}. Let $(f_n)\subset \s^2(\Omega,\sfd,\mm)$ and $f:\Omega\to\R$ be such that $f_n(x)\to f(x)$ as $n\to\infty$ for $\mm$-a.e. $x\in\Omega$. Assume that $(\weakgrad {f_n})$ converges to some $G\in L^2(\Omega,\mm)$ weakly in $L^2(\Omega,\mm)$.

Then $f\in \s^2(\Omega)$ and $\weakgrad f\leq  G$ $\mm$-a.e..

\noindent\underline{Weak gradients and local Lipschitz constants}. For any $f:X\to\R$ Lipschitz it holds
\begin{equation}
\label{eq:lipweak}
\weakgrad f\leq \lip(f).
\end{equation}
\noindent\underline{Vector space structure}. $\s^2_{\rm loc}(\Omega,\sfd,\mm)$ is a vector space and it holds
\[
\weakgrad{(\alpha f+\beta g)}\leq |\alpha|\weakgrad f+|\beta|\weakgrad g,\qquad\textrm{for any $f,g\in\s^2_{\rm loc}(\Omega,\sfd,\mm)$, $\alpha,\beta\in\R$,}
\]
similarly for  $\s^2(\Omega,\sfd,\mm)$.

\noindent\underline{Algebra structure}.  $\s^2_{\rm loc}(\Omega,\sfd,\mm)\cap L^\infty_{\rm loc}(\Omega,\mm)$. is an algebra and it holds
\begin{equation}
\label{eq:leibbase}
\weakgrad{(fg)}\leq |f|\weakgrad g+g\weakgrad f,\qquad\textrm{for any $f,g\in \s^2_{\rm loc}(\Omega,\sfd,\mm)\cap L^\infty_{\rm loc}(\Omega,\mm)$,}
\end{equation}
and analogously  for the space  $\s^2(\Omega,\sfd,\mm)\cap L^\infty(\Omega,\mm)$. Similarly, if $f\in \s^2_{\rm loc}(\Omega,\sfd,\mm)$ and $g:\Omega\to\R$ is locally Lipschitz, then $fg\in\s^2_{\rm loc}(\Omega,\sfd,\mm)$ and the bound \eqref{eq:leibbase} holds.

\noindent\underline{Chain rule}. Let  $f\in \s^2_{\rm loc}(\Omega,\sfd,\mm)$ and $\varphi:\R\to \R$  with the following property: for any $x\in \Omega$ there exists a neighborhood $U_x\subset\Omega$ of $x$ and an interval $I_x\subset \R$ such that $\mm(f^{-1}(\R\setminus I_x)\cap U_x)=0$ and the restriction of $\varphi$ to $I_x$ is Lipschitz. Then $\varphi\circ f\in \s^2_{\rm loc}(\Omega,\sfd,\mm)$ and
\begin{equation}
\label{eq:chainbase}
\weakgrad{(\varphi\circ f)}=|\varphi'|\circ f\weakgrad f,
\end{equation}
where $|\varphi'|\circ f$ is defined arbitrarily at points where $\varphi$ is not differentiable (observe that the identity \eqref{eq:nullgrad} ensures that on $f^{-1}(\mathcal N)$ both $\weakgrad{(\varphi\circ f)}$ and $\weakgrad f$ are 0 $\mm$-a.e., $\mathcal N$ being the negligible set of points of non-differentiability of $\varphi$). In particular, if $f\in \s^2(\Omega,\sfd,\mm)$ and $\varphi$ is Lipschitz, then $\varphi\circ f\in \s^2(\Omega,\sfd,\mm)$ as well.

\medskip

Finally, we remark that from the definition of Sobolev class it is easy to produce the definition of Sobolev space $W^{1,2}(X,\sfd,\mm)$: it is sufficient to put 
\begin{equation}
\label{eq:w12}
W^{1,2}(X,\sfd,\mm):=L^2(X,\mm)\cap\s^2(X,\sfd,\mm)
\end{equation}
the corresponding $W^{1,2}$-norm being given by
\[
\|f\|_{W^{1,2}}^2:=\|f\|^2_{L^2}+\|\weakgrad f\|^2_{L^2}.
\]
An analogous definition works for the Sobolev space $W^{1,2}(\Omega,\sfd,\mm)$ for $\Omega\subset X$ open. As a consequence of the lower semicontinuity of minimal weak upper gradients one easily gets that  $W^{1,2}(\Omega,\sfd,\mm)$ is a Banach space for every $\Omega\subset X$ open, see for instance the argument in Theorem 2.7 in \cite{Cheeger00}.

\medskip

\emph{To simplify the notation, in the following we will often write $W^{1,2}(X)$, $\s^2_{\rm loc}(X)$, $\s^2_{\rm loc}(\Omega)$ ecc. in place of  $W^{1,2}(X,\sfd,\mm)$, $\s^2_{\rm loc}(X,\sfd,\mm)$, $\s^2_{\rm loc}(\Omega,\sfd,\mm)$. Similarly, we will write $L^p(X)$, $L^p(\Omega)$, $L^p_{\rm loc}(\Omega)$ in place of $L^p(X,\mm)$, $L^p(\Omega,\mm)$, $L^p_{\rm loc}(\Omega,\mm)$.}

\medskip

 In \cite{AmbrosioGigliSavare11} the following approximation result has been proved:
\begin{theorem}[Density in energy of Lipschitz functions in $W^{1,2}(X,\sfd,\mm)$]\label{thm:energylip}\ \linebreak
Let $(X,\sfd,\mm)$ be a metric measure space.

Then Lipschitz functions are dense in energy in $W^{1,2}(X)$, i.e. for any $f\in W^{1,2}(X)$ there exists a sequence $(f_n)\subset W^{1,2}(X)$ of Lipschitz functions such that $f_n\to f$, $\weakgrad {f_n}\to\weakgrad f$ and $\lip(f_n)\to\weakgrad f$ as $n\to\infty$ in $L^2(X)$.

Furthermore:
\begin{itemize}
\item if $\mm$ gives finite mass to bounded sets, then the $f_n$'s can be chosen with bounded support for every $n\in\N$,
\item if $(\supp(\mm),\sfd)$ is proper, then the $f_n$'s can be chosen with compact support for every $n\in\N$.
\end{itemize}
\end{theorem}
See \cite{AmbrosioGigliSavare11-3}  for a stronger version of this theorem.

\subsection{Differentials and  gradients}
Here we recall the definition of duality relation between differentials and gradients of Sobolev functions as given in \cite{Gigli12}. Notice that we are not going to define what the differential and the gradient of a Sobolev function are, but only the `value of the differential of $f$ applied to the gradient of $g$'. More precisely, we will define two functions $D^\pm f(\nabla g)$ which in the case of normed spaces are the maximal/minimal values of the differential of $f$ applied to all the possible gradients of $g$. Recall indeed that on $(\R^d,\|\cdot\|)$ the gradient of a smooth function might be not uniquely defined if the norm $\|\cdot\|$ is not strictly convex (see e.g. the introduction of \cite{Gigli12} for a quick glance at this sort of problematic).

\begin{definition}[The objects $D^\pm f(\nabla g)$]\label{def:dpmfg}
Let $\Omega\subset X$ be open and $f,g\in \s^2_{\rm loc}(\Omega)$. The functions $D^\pm f(\nabla g):\Omega\to \R$ are defined $\mm$-a.e. by 
\[
\begin{split}
D^+f(\nabla g)&=\inf_{\eps>0}\frac{\weakgrad{(g+\eps f)}^2-\weakgrad g^2}{2\eps},\\
D^-f(\nabla g)&=\sup_{\eps<0}\frac{\weakgrad{(g+\eps f)}^2-\weakgrad g^2}{2\eps},\\
\end{split}
\]
\end{definition}
Notice that the $\mm$-a.e. convexity of $\eps\mapsto\frac12\weakgrad{(g+\eps f)}^2$  grants the the $\inf_{\eps>0}$ and $\sup_{\eps<0}$ in this definition can be substituted with $\lim_{\eps\downarrow 0}$ and $\lim_{\eps\uparrow 0}$ respectively.

The following calculus rules hold, $\Omega\subset X$ being any given open set.

\noindent\underline{Basic properties}. For any $f,g\in\s^2(\Omega)$ we have
\begin{equation}
\label{eq:basefg}
\begin{split}
D^-f(\nabla g)&\leq D^+f(\nabla g),\\
|D^\pm f(\nabla g)|&\leq\weakgrad f\weakgrad g,\\
D^+(-f)(\nabla g)&=-D^-f(\nabla g)=D^+f(\nabla (-g)),\\
D^\pm f(\nabla f)&=\weakgrad f^2,
\end{split}
\end{equation}
$\mm$-a.e. on $\Omega$.

\noindent\underline{1-Lipschitz continuity of differentials}.  For any $f,\tilde f,g\in \s^2_{\rm loc}(\Omega)$ it holds
\[
|D^\pm f(\nabla g)-D^\pm \tilde f(\nabla g)|\leq \weakgrad{(f-\tilde f)}\,\weakgrad g,\qquad\mm\ae\ \textrm{on }\Omega.
\]

\noindent\underline{Locality}.  For any $f,\tilde f,g,\tilde g\in \s^2_{\rm loc}(\Omega)$ it holds
\begin{equation}
\label{eq:localgrad21}
D^\pm f(\nabla g)=D^\pm \tilde f(\nabla \tilde g),\qquad\mm\ae \ \textrm{on}\ \{f=\tilde f\}\cap\{g=\tilde g\}\cap\Omega.
\end{equation}

\noindent\underline{Leibniz rule for differentials}. For any $\Omega\subset X$ open, $f_0,f_1\in \s^2_{\rm loc}(\Omega)\cap L^\infty_{\rm loc}(\Omega)$ and $g\in \s^2_{\rm loc}(\Omega)$ it holds
\begin{equation}
\label{eq:leibineq}
\begin{split}
D^+(f_0f_1)(\nabla g)&\leq f_0\,D^{s_1}f_1(\nabla g)+f_1\,D^{s_2}f_0(\nabla g),\\
D^-(f_0f_1)(\nabla g)&\geq f_0\,D^{-s_1}f_1(\nabla g)+f_1\,D^{-s_2}f_0(\nabla g),
\end{split}
\end{equation}
$\mm$-a.e. on $\Omega$, where $s_i={\rm sign}f_i$, $i=1,2$.

\noindent\underline{Chain rules}. For $f,g\in\s^2(\Omega)$ and $\varphi:\R\to\R$ Lipschitz it holds
\begin{equation}
\label{eq:chainsbase}
\begin{split}
D^\pm (\varphi\circ f)(\nabla g)&=\varphi'\circ f\,D^{\pm{\rm sign}(\varphi'\circ f)}f(\nabla g),\\
D^\pm f(\nabla(\varphi\circ g))&=\varphi'\circ g\,D^{\pm{\rm sign}(\varphi'\circ g)}f(\nabla g),
\end{split}
\end{equation}
$\mm$-a.e. on $\Omega$, where $\varphi'$ is defined arbitrarily at points where $\varphi$ is not differentiable. The Lipschitz continuity of $\varphi$ can be relaxed as in the chain rule \eqref{eq:chainbase}.

We will also use the equality
\begin{equation}
\label{eq:buonofg}
D^-f(\nabla(g+\eps f))=D^+f(\nabla(g+\eps f)),\qquad\mm\ae\textrm{ on }\Omega,\ \textrm{ for every $\eps$ except a countable number},
\end{equation}
which follows from the the first inequality in \eqref{eq:basefg} and the fact that $\eps\mapsto\frac12\int_\Omega\weakgrad{(g+\eps f)}^2\,\d\mm$ is convex, thus its left and right derivatives are equal for every $\eps$ except a countable number.

The definition of $D^\pm f(\nabla g)$ also allows to state and prove a crucial first order differentiation formula. Notice at first that for $g\in \s^2(\Omega)$ and $\ppi$ test plan such that  $\supp((\e_t)_\sharp\ppi)\subset \Omega$ for every $t\in[0,1]$, an easy consequence of the definitions is that
\begin{equation}
\label{eq:perplanrepr}
\lims_{t\downarrow0}\int\frac{g(\gamma_t)-g(\gamma_0)}{t}\,\d\ppi(\gamma)\leq \frac12\int\weakgrad g^2(\gamma_0)\,\d\ppi(\gamma)+\frac12\lims_{t\downarrow0}\frac1{2t}\iint_0^t|\dot\gamma_s|^2\,\d s\,\d\ppi(\gamma).
\end{equation}
In a smooth setting, the opposite inequality holds if and only if $\gamma_0'=\nabla g(\gamma_0)$ for $\ppi$-a.e. $\gamma$. This remark and the fact that in inequality \eqref{eq:perplanrepr} the behavior of curves in the support of $\ppi$ is relevant only for $t$ close to 0,  justify the following definition:
\begin{definition}[Plans representing gradients]\label{def:planrepgrad}
Let $\Omega\subset X$ be an open set, $g\in \s^2(\Omega)$ and $\ppi\in\prob{C([0,1],X)}$ a plan. 

We say that $\ppi$ represents $\nabla g$ in $\Omega$ provided  for some $T\in (0,1]$ it holds $\supp((\e_t)_\sharp\ppi)\subset \Omega$ and $(\e_t)_\sharp\ppi\leq C\mm$ for every $t\in[0,T]$ and some $C>0$, $\iint_0^T|\dot\gamma_t|^2\,\d t\,\d\ppi(\gamma)<\infty$ and furthermore
\begin{equation}
\label{eq:defrepr}
\limi_{t\downarrow0}\int\frac{g(\gamma_t)-g(\gamma_0)}{t}\,\d\ppi(\gamma)\geq \frac12\int\weakgrad g^2(\gamma_0)\,\d\ppi(\gamma)+\frac12\lims_{t\downarrow0}\frac1{2t}\iint_0^t|\dot\gamma_s|^2\,\d s\,\d\ppi(\gamma).
\end{equation}
\end{definition}
See Section 3.2 in \cite{Gigli12} for a general existence theorem of plans representing gradients. With this notion at disposal, we can study the limit as $t\downarrow0$ of the incremental ratios
\[
\int \frac{f(\gamma_t)-f(\gamma_0)}{t}\,\d\ppi(\gamma),
\]
where $f$ is Sobolev and $\ppi$ represents the gradient of some other Sobolev function $g$. It is certainly expected that such `horizontal' limit is in relation with the `vertical' limit used in the definition of $D^\pm f(\nabla g)$. This is indeed the case:
\begin{proposition}[First order differentiation formula]\label{prop:firstdiff}
Let $\Omega\subset X$ be an open set, $f,g\in \s^2(\Omega)$ and $\ppi\in\prob{C([0,1],X)}$  be representing $\nabla g$ in $\Omega$.

Then
\begin{equation}
\label{eq:firstdiff}
\begin{split}
\int D^-f(\nabla g)(\gamma_0)\,\d\ppi(\gamma)&\leq \limi_{t\downarrow0}\int\frac{f(\gamma_t)-f(\gamma_0)}{t}\,\d\ppi(\gamma)\\
&\leq\lims_{t\downarrow0}\int\frac{f(\gamma_t)-f(\gamma_0)}{t}\,\d\ppi(\gamma)\leq \int D^+f(\nabla g)(\gamma_0)\,\d\ppi(\gamma).
\end{split}
\end{equation}
\end{proposition}
\begin{proof}
Write inequality \eqref{eq:perplanrepr} for the function $g+\eps f$ and subtract inequality \eqref{eq:defrepr} to get
\[
\lims_{t\downarrow0}\eps\int\frac{f(\gamma_t)-f(\gamma_0)}{t}\,\d\ppi(\gamma)\leq\frac12\int \weakgrad{(g+\eps f)}^2(\gamma_0)-\weakgrad g^2(\gamma_0)\,\d\ppi(\gamma).
\]
Divide by $\eps>0$ (resp. $\eps<0$) and let $\eps\downarrow 0$ (resp. $\eps\uparrow0$) to conclude.
\end{proof}

\subsection{Infinitesimally strictly convex spaces}
Following \cite{Gigli12} we introduce a class of metric measure spaces resembling Finsler manifolds such that for a.e. $x$ the norm in the tangent space at $x$ is strictly convex. The fact that a norm is strictly convex if and only if its dual norm is differentiable and that in a smooth framework the object $\weakgrad f$ is the  (dual) norm of the distributional differential of $f$ is at the basis of the following definition:
\begin{definition}[Infinitesimally strictly convex spaces]\label{def:infstrconv}
We say that $(X,\sfd,\mm)$ is infinitesimally strictly convex provided
\[
\int_XD^+f(\nabla g)=\int_XD^-f(\nabla g),\qquad\textrm{for any $f,g\in \s^2(X,\sfd,\mm)$.}
\]
\end{definition}
The first inequality in \eqref{eq:basefg} and the locality property \eqref{eq:localgrad21} ensures that   if $(X,\sfd,\mm)$ is infinitesimally strictly convex and  $\Omega\subset X$ is open, then for every  $f,g\in \s^2(\Omega)$, it holds $D^+f(\nabla g)=D^-f(\nabla g)$ $\mm$-a.e. on $\Omega$. Such common value will be denoted by $Df(\nabla g)$.

On these spaces, the general calculus rules presented before simplify, for $\Omega\subset X$ open we have:

\noindent\underline{1-Lipschitz continuity of differentials}.  For any $f,\tilde f,g\in \s^2_{\rm loc}(\Omega)$ it holds
\[
|D f(\nabla g)-D \tilde f(\nabla g)|\leq \weakgrad{(f-\tilde f)}\,\weakgrad g,\qquad\mm\ae\ \textrm{on }\Omega.
\]

\noindent\underline{Locality}.  For $f,\tilde f,g,\tilde g\in \s^2_{\rm loc}(\Omega)$ it holds
\begin{equation}
\label{eq:localgrad22}
Df(\nabla g)=D\tilde f(\nabla \tilde g),\qquad\mm\ae\textrm{ on } \{f=\tilde f\}\cap\{g=\tilde g\}.
\end{equation}
\noindent\underline{Linearity of the differential}. For  $f_0,f_1,g\in \s^2(\Omega)$ and $\alpha_0,\alpha_1\in \R$ it holds
\[
D(\alpha_0 f_0+\alpha_1  f_1)(\nabla g)=\alpha_0 Df_0(\nabla g)+\alpha_1  Df_1(\nabla g),\qquad\mm\ae\ \textrm{on }\Omega.
\]

\noindent\underline{Leibniz rule for differentials}. For $f_0,f_1\in \s^2(\Omega)\cap L^\infty(\Omega)$ and $g\in \s^2(\Omega)$ it holds
\[
D(f_0f_1)(\nabla g)=f_0Df_1(\nabla g)+f_1Df_0(\nabla g),\qquad\mm\ae\ \textrm{on }\Omega.
\]

\noindent\underline{Chain rules}. For $f,g\in\s^2(\Omega)$ and $\varphi:\R\to\R$ Lipschitz it holds
\[
\begin{split}
D(\varphi\circ f)(\nabla g)&=\varphi'\circ fDf(\nabla g),\\
Df(\nabla(\varphi\circ g))&=\varphi'\circ gDf(\nabla g),
\end{split}
\]
$\mm$-a.e. on $\Omega$, where $\varphi'$ is defined arbitrarily at points where $\varphi$ is not differentiable. The Lipschitz continuity of $\varphi$ can be relaxed as in the chain rule \eqref{eq:chainbase}.

\noindent\underline{First order differentiation formula}. For  $f,g\in\s^2(\Omega)$ and $\ppi$ representing the gradient of $g$ in $\Omega$ it holds
\begin{equation}
\label{eq:firststrictconv}
\lim_{t\downarrow0}\int\frac{f(\gamma_t)-f(\gamma_0)}{t}\,\d\ppi(\gamma)=\int Df(\nabla g)(\gamma_0)\,\d\ppi(\gamma),
\end{equation}
in particular, the limit at the left-hand side exists.

\subsection{Measure valued Laplacian}

The definition of the object $Df(\nabla g)$ allows to integrate by parts. In particular we can give a meaning to the equation $\bd g=\mu$ for a Sobolev function $g$ and a measure $\mu$.

\begin{definition}[Measure valued Laplacian]\label{def:measlap}
Let $(X,\sfd,\mm)$ be an infinitesimally strictly convex space, and $g\in \s^2_{\rm loc}(X,\sfd,\mm)$. We say that $g$ is in the domain of the Laplacian, and write $g\in D(\bd)$, provided there exists a Radon measure $\mu$ on $X$ such that
\begin{equation}
\label{eq:deflap}
-\int Df(\nabla g)\,\d\mm=\int f\,\d\mu,
\end{equation}
for any Lipschitz function $f:X\to\R$ in $L^1(X,|\mu|)$ such that $\supp(f)$ is bounded and of finite $\mm$ measure.
\end{definition}
It is clear that if $g\in D(\bd)$ the measure $\mu$ satisfying \eqref{eq:deflap} is unique: we will denote it by $\bd g$, the bold notation standing to remember that we deal with a measure valued Laplacian, possibly absolutely continuous w.r.t. $\mm$. This definition of Laplacian directly generalizes the one available in the smooth Finsler setting (see e.g. \cite{Shen98}). Notice that in general the Laplacian is not a linear operator and similarly $D(\bd)$ might be not a vector space. These properties will be granted if we further assume that $(X,\sfd,\mm)$ is infinitesimally Hilbertian, an assumption that we will do from the next chapter on.

Yet, 1-homogeneity of $\bd$ is granted, i.e. for $g\in D(\bd)$ and $\alpha\in \R$ we have $\alpha g\in D(\bd)$ and $\bd(\alpha g)=\alpha\bd g$.

One of the main results in \cite{Gigli12} is that on an infinitesimally strictly convex $CD(K,N)$ space, for the distance function it holds the same Laplacian comparison that holds on Riemannian manifold with Ricci curvature bounded from below by $K$ and dimension bounded above by $N$. The basic idea to get the result is to combine the first order differentiation formula
\begin{equation}
\label{eq:descrder}
\lim_{t\downarrow0}\frac{\u_N(\mu_t)-\u_N(\mu_0)}t=-\frac1N\int D(\rho^{1-\frac1N})(\nabla \varphi)\,\d\mm,
\end{equation}
valid in a smooth Finsler world, where $(\mu_t)$ is a geodesic induced by the Kantorovich potential $\varphi$ and $\mu_0=\rho\mm$, with the curvature-dimension condition. For example, for $\mu_1=\delta_{\bar x}$ we can take $\varphi:=\frac{\sfd^2(\cdot,\bar x)}2$ independently on the chosen $\mu_0$ and  in the $CD(0,N)$ case we get
\[
\frac{\u_N(\mu_t)-\u_N(\mu_0)}t\leq\u_N(\mu_1)-\u_N(\mu_0)=\int \rho^{1-\frac1N}\,\d\mm, \qquad\forall t\in(0,1],
\]
so that the sharp Laplacian comparison estimate for $\varphi$ follows this inequality, \eqref{eq:descrder} and  the arbitrariness of $\rho$.

Part of the job carried out in \cite{Gigli12} was to prove formula \eqref{eq:descrder} in the non-smooth world  with a $\geq$ in place of the $=$. We mention in particular one issue, given that later on we will face a similar problem. The convexity of $u_N$ yields
\[
\frac{\u_N(\mu_t)-\u_N(\mu_0)}t\geq-\left(1-\frac1N\right)\int\rho^{-\frac1N}\,\d\frac{\mu_t-\mu_0}t=-\left(1-\frac1N\right)\int\frac{\rho^{-\frac1N}\circ\e_t-\rho^{-\frac1N}\circ\e_0}t\,\d\ppi,
\]
where $\ppi$ is any lifting of  $(\mu_t)$. It turns out that $\ppi$ represents the gradient of $-\varphi$ in quite high generality (by the metric Brenier theorem proved in \cite{AmbrosioGigliSavare11}, see also Section \ref{se:metrbren}), thus formally  letting $t\downarrow0$ and using the first order differentiation formula we obtain exactly the right hand side of \eqref{eq:descrder}. However, in practice this might not work  if for $t>0$ the support of $\mu_t$ is not contained in the one of $\mu_0$ (which is often the case): the problem is that the term $\rho^{-\frac1N}\circ\e_t$ would be equal to $+\infty$ on a set of positive $\ppi$-measure, thus destroying all the informations.

To get around this issue, in \cite{Gigli12} two propositions have been used: one to prove that whenever $\supp(\mu_t)\subset\supp(\mu_0)$ holds for any $t\in[0,1]$ indeed the computation can be carried out, and another one to show that one can always reduce to such well behaved case. They are recalled below in the formulation that we will need later on.
\begin{proposition}[Bound from below on the derivative of the internal energy]\label{prop:boundbasso}\ \linebreak
Let $(X,\sfd,\mm)$ be an infinitesimally strictly convex $CD(0,N)$ space and $\mu_0\in \prob X$ a measure with bounded support such that  $\mu_0\ll\mm$, say $\mu_0=\rho\mm$. Assume that $\supp(\mu_0)=\overline\Omega$, with $\Omega$ bounded, open and such that $\mm(\partial\Omega)=0$. Assume also that the restriction of  $\rho$ to $\overline\Omega$ is Lipschitz and bounded from below by a positive constant.  Also, let $\mu_1\in\prob X$ and $\ppi\in\gopt(\mu_0,\mu_1)$. Assume that  $\supp((\e_t)_\sharp\ppi)\subset \overline\Omega$ for every $t\in[0,1]$ and that  $\u_N((\e_t)_\sharp\ppi)\to \u_N(\mu_0)$, as $t\downarrow0$. 

Then:
\begin{equation}
\label{eq:bounddalbasso}
\limi_{t\downarrow0}\frac{\u_N((\e_t)_\sharp\ppi)-\u_N((\e_0)_\sharp\ppi)}t\geq-\frac1N\int_\Omega D(\rho^{1-\frac1N})(\nabla\varphi)\,\d\mm,
\end{equation}
where $\varphi$ is any Kantorovich potential from $\mu_0$ to $\mu_1$ which is Lipschitz on bounded sets.
\end{proposition}
Notice that the integrand in the right-hand side of \eqref{eq:bounddalbasso} is well defined because $\rho^{1-\frac1N},\varphi\in \s^2(\Omega)$.
\begin{proposition}\label{le:cheppalle}
Let $(X,\sfd,\mm)$ be a proper  geodesic metric measure space, $\varphi:X\to\R$ a locally Lipschitz $c$-concave function and $B\subset X$ a compact set. Then there exists another locally Lipschitz $c$-concave function $\tilde\varphi:X\to\R$ and a bounded open set $\Omega \supset B$ such that the following are true. 
\begin{itemize}
\item[i)] $\tilde\varphi=\varphi$ on $B$.
\item[ii)] For any $x\in X$, the set $\partial^c\tilde\varphi(x)$ is non empty.
\item[iii)] For any $x\in\Omega$, $y\in\partial^c\tilde\varphi(x)$ and $\gamma\in\geo(X)$ connecting $x$ to $y$ it holds $\gamma_t\in\Omega$   for any $t\in[0,1]$.
\item[iv)] $\mm(\partial\Omega)=0$.
\end{itemize}
\end{proposition}
Finally, we recall that with a limiting argument similar to the one used in the smooth context gives the following result, see the last proposition in \cite{Gigli12} for the proof.
\begin{proposition}[Laplacian comparison for the Busemann function]\label{prop:lapbus}
Let $(X,\sfd,\mm)$ be an infinitesimally strictly convex $CD(0,N)$ space, $N<\infty$, such that $W^{1,2}(\Omega,\sfd,\mm)$ is uniformly convex for any $\Omega\subset X$ open. Assume that there is an half line $\bar\gamma:\R^+\to\supp(\mm)$ and let $\b$ be the Busemann function associated to it as in formula \eqref{eq:halfbus}.

Then $\b\in D(\bd)$ and $\bd\b\geq 0$.
\end{proposition}

\section{Result}
Throughout this section we will assume that
\begin{equation}
\label{eq:ass2}
\begin{split}
&(X,\sfd,\mm)\textrm{ is an infinitesimally strictly convex $CD(0,N)$ space},\\
&W^{1,2}(\Omega,\sfd,\mm)\textrm{ is uniformly convex for any $\Omega\subset X$ open},\\
&\bar\gamma:\R\to \supp(\mm)\textrm{ is a line,  $\b^+,\b^-$ are the associated Busemann functions},\\
&\textrm{the identity $\b^++\b^-=0$ holds on $\supp(\mm)$. Put $\b:=\b^+$}
\end{split}
\end{equation}
It is unclear to us whether the identity  $\b^++\b^-=0$ on $\supp(\mm)$ can be deduced or not from the other hypothesis. It is so in the smooth Finsler case, but in the non-smooth one it is currently not available   a version of the strong maximum principle sufficient to get such result, see the discussion at the end of \cite{Gigli-Mondino12}.  In \cite{Gigli-Mondino12}, it has been proved that indeed $\b^++\b^-\equiv 0$ holds on infinitesimally Hilbertian spaces, see also Section \ref{se:maxprin}.

Notice that  the assumption $\b^++\b^-\equiv 0$ on $\supp(\mm)$, the inequalities $\bd(\b^\pm)\geq 0$ and the 1-homogeneity of the distributional Laplacian trivially yield
\[
\forall a\in \R\textrm{ it holds }\qquad a\b\in D(\bd)\quad\textrm{ and }\quad\bd(a\b)=0.
\]

We further remark  that the results of the previous chapter can be used, provided we replace $(X,\sfd)$ with $(\supp(\mm),\sfd)$ in all the instances. In this sense, notice that multiples of  $\b$ are known to be $c$-concave only in the space $(\supp(\mm),\sfd)$, the cost $c$ being the restriction of $\frac{\sfd^2}2$ to $[\supp(\mm)]^2$  (being $c$-concavity a global notion, it could be destroyed by how the distance behaves on $X\setminus\supp(\mm)$). Thus here and throughout all the paper, when referring to the $c$-concavity of $a\b$ we will always implicitly speak about the $c$-concavity of its restriction to $\supp(\mm)$, also, the set $\partial^c(a\b)$ will be thought as a subset of $[\supp(\mm)]^2$.

\bigskip

We now turn to the proof of existence, uniqueness (in an appropriate class) and preservation of the measure of the gradient flow of $\b$. It will be technically convenient to work for a while with `multivalued' gradient flows, more precisely, with maps from $X$ to the space of probability measures concentrated on curves in such a way that for $\mm$-a.e. $x$ the measure associated to $x$ is concentrated on gradient flow trajectories of $\b$ passing through $x$ at time 0. According to point $(iii)$ of Theorem \ref{thm:basemetric}, the study of such maps can be reduced to the study of maps $\mathbf T:X\to\prob X$ such that for some $a\in\R$ we have  $\supp(\mathbf T_a(x))\subset\partial^c(a\b)(x)$ for $\mm$-a.e. $x\in X$.

Given a Borel map $\mathbf T:X\to\prob X$ and a non-negative Borel measure $\nn$ on $X$ the non-negative Borel measure $\mathbf T_\sharp\nn$ on $X$  is defined by
\[
\mathbf T_\sharp\nn(E):=\int \mathbf T(x)(E)\,\d\nn(x),\qquad\forall E\subset X \textrm{ Borel.}
\]
Notice that if $\mathbf T(x)=\delta_{T(x)}$ for every $x\in X$ and some Borel map $T:X\to X$, then $\mathbf T_\sharp\nn=T_\sharp\nn$.
\begin{proposition}[Existence]\label{prop:esiste}
Assume \eqref{eq:ass2}  and let $a\in\R$. Then there exists a Borel map $\mathbf T_a:X\to\prob X$ such that
\begin{itemize}
\item $\supp(\mathbf T_a(x))\subset\partial^c(a\b)(x)$ for $\mm$-a.e. $x\in X$,
\item $(\mathbf T_a)_\sharp\mm\ll\mm$.
\end{itemize} 
\end{proposition}
\begin{proof} It is not restrictive to assume $N>1$ ($N$ being the dimension bound in the $CD(0,N)$ condition).

With a patching argument, to conclude it is sufficient to show that for every bounded Borel set $E\subset X$ with $\mm(E)>0$ there exists a  Borel map $\mathbf T_a:E\to\prob X$ satisfying  $\supp(\mathbf T_a(x))\subset\partial^c(a\b)(x)$ for $\mm$-a.e. $x\in E$ and $(\mathbf T_a)_\sharp(\mm\restr E)\ll\mm$. Fix such $E$ and let $E'$ be the $|a|$-neighborhood of $E$:
\[
E':=\{x\in X\ :\ \sfd(x,E)\leq |a|\}.
\]
It is immediate to verify that the closed valued map $x\mapsto  \partial^c(2a\b)(x)$ satisfies the assumptions of Theorem \ref{thm:borelsel}, thus there exists a Borel  map $T:X\to X$ such that $T(x)\in \partial^c(2a\b)(x)$ for any $x\in X$. 

Define $\mu:=\mm(E)^{-1}\mm\restr E$, $\nu:= T_\sharp\mu$ and let $\ppi\in\gopt(\mu,\nu)$ be such that \eqref{eq:defcd} is satisfied (it exists because $\mu$ - and thus by \eqref{eq:facile} also $\nu$ - has bounded support). Let $\{\ppi_x\}$ be the disintegration of $\ppi$ w.r.t. $\e_0$ and notice that since by construction $2a\b$ is a Kantorovich potential from $\mu$ to $\nu$, taking into account that $\partial^c(2a\b)(x)$ is closed for every $x\in \supp(\mm)$ we get that for $\mu$-a.e. $x$ it holds $\supp((\e_1)_\sharp\ppi_x)\subset \partial^c(2a\b)(x)$. Put $\sigma:=(\e_{1/2})_\sharp\ppi$ and write $\sigma=\rho\mm+\sigma^s$ with $\sigma^s\perp\mm$. The inequality \eqref{eq:defcd}, the fact that $\u_N(\mu)=-\mm(E)^{1/N}$ and the trivial bound $\u_N(\nu)\leq 0$ give $\u_N(\sigma)\leq -\frac{\mm(E)^{1/N}}2$. By construction, point $(iii)$ of Theorem \ref{thm:basemetric} and \eqref{eq:facile} $\sigma$ is concentrated on $E'$, thus from
\[
\frac{\mm(E)^{1/N}}2\leq-\u_N(\sigma)=\int_X\rho^{1-\frac1N}\,\d\mm=\int_{E'}\rho^{1-\frac1N}\,\d\mm\leq \left(\int_{E'}\rho\,\d\mm\right)^{\frac{N-1}N}\mm(E')^{\frac1N}
\]
we deduce
\begin{equation}
\label{eq:massarho}
\int_X\rho\,\d\mm\geq \left(\frac{\mm(E)}{\mm(E')}\right)^{\frac{1}{N-1}}2^{-\frac N{N-1}}.
\end{equation}
Let $A\subset X$ be a Borel set where $\rho\mm$ is concentrated and such that  $\sigma^s(A)=0$, so that $(\e_{1/2})_\sharp(\ppi\restr{\e_{1/2}^{-1}(A)})=\rho\mm$ and notice that $(\e_0)_\sharp(\ppi\restr{\e_{1/2}^{-1}(A)})\leq \mu$ and in particular $(\e_0)_\sharp(\ppi\restr{\e_{1/2}^{-1}(A)})\ll\mm$, say  $(\e_0)_\sharp(\ppi\restr{\e_{1/2}^{-1}(A)})=\rho_0\mm$. Put $E_0:=\{\rho_0>0\}$ and notice that since $\rho_0\leq \mm(E)^{-1}$ we have $\frac{\mm(E_0)}{\mm(E)}\geq  \int_X\rho_0\,\d\mm=\ppi\restr{\e_{1/2}^{-1}(A)}(\geo(X))=\int_X\rho\,\d\mm$ and thus \eqref{eq:massarho} yields
\[
\mm(E_0)\geq  \frac{\mm(E)^{\frac{N}{N-1}}}{\mm(E')^{\frac{1}{N-1}}}2^{-\frac N{N-1}}.
\]
Define $\mm$-a.e. on $E_0$ the map $\mathbf T_a:E_0\to\prob X$ by putting $\mathbf T_a(x):=((\e_0,\e_{1/2})_\sharp\ppi\restr{\e_{1/2}^{-1}(A)})_x$, where $\{((\e_0,\e_{1/2})_\sharp\ppi\restr{\e_{1/2}^{-1}(A)})_x\}$ is the disintegration of $(\e_0,\e_{1/2})_\sharp\ppi\restr{\e_{1/2}^{-1}(A)}\in\prob{X^2}$ w.r.t. the projection on the first marginal. Notice that $\mathbf T_a$ is Borel.

Since $\supp((\e_0,\e_1)_\sharp\ppi)\subset \partial^c(2a\b)$, by point $(iii)$ of Theorem \ref{thm:basemetric} we have $\supp((\e_0,\e_{1/2})_\sharp\ppi)\subset \partial^c(a\b)$ and thus $\supp(\mathbf T_a(x))\subset\partial^c(a\b)(x)$ for $\mm$-a.e. $x\in E_0$. By construction it is also clear that $(\mathbf T_a)_\sharp(\mm\restr{E_0})\ll\mm$.

Now we repeat the construction with $E$ replaced by $E\setminus E_0$ to get a Borel set $E_1\subset E\setminus E_0$ with
\[
\mm(E_1)\geq   \frac{\mm(E\setminus E_0)^{\frac{N}{N-1}}}{\mm(E')^{\frac{1}{N-1}}}2^{-\frac N{N-1}}
\]
and an extension of $\mathbf T_a$ to $E_0\cup E_1$ satisfying $\supp(\mathbf T_a(x))\subset\partial^c(a\b)(x)$ for $\mm$-a.e. $x\in E_0\cup E_1$ and $(\mathbf T_a)_\sharp(\mm\restr{E_0\cup E_1})\ll\mm$.

Iterating the construction a countable number of times we produce a family $\{E_i\}_{i\in\N}$ of Borel sets which covers  $E$ up to $\mm$-negligible sets and a Borel map $\mathbf T_a: E\to\prob X$ defined $\mm$-a.e. satisfying $\supp(\mathbf T_a(x))\subset\partial^c(a\b)(x)$ for $\mm$-a.e. $x\in E$ and $(\mathbf T_a)_\sharp(\mm\restr{E})\ll\mm$, which is the thesis.
\end{proof}
The uniqueness and measure preservation  results are based one the following inequality, valid for couples of measures such that $a\b$ is a Kantorovich potential for them. 
\begin{proposition}[Energy inequality]\label{prop:der0} Assume \eqref{eq:ass2} and let $\mu,\nu\in\probt X$ be two measures with bounded support such that for some $a\in\R$ the function $a\b$ is a Kantorovich potential for the couple $(\mu,\nu)$. Assume also that $\mu\ll\mm$.

Then 
\[
\u_N(\nu)\geq \u_N(\mu).
\]
\end{proposition}
\begin{proof} It is not restrictive to assume $N>1$. Let $\Omega_0$ be a bounded open set containing $\supp(\mu)$ and  $\tilde\varphi$, $\Omega$  be given by Proposition \ref{le:cheppalle} with $a\b$ in place of $\varphi$ and $\overline\Omega_0$ in place of $B$. Let $\ggamma\in\probt{X^2}$ be an optimal plan from $\mu$ to $\nu$ and denote by $\{\ggamma_x\}$ its disintegration w.r.t. the projection onto the first marginal. Observe that $Y:=\cup_{x\in\Omega}\partial^c\tilde\varphi(x)$ is bounded and thus the space $\prob{\overline Y}$ endowed with the weak$^*$ topology is a Polish space. It is then easy to check that we can apply the Borel selection Theorem \ref{thm:borelsel} with $\Omega$ in place of $X$ and $\prob{\overline Y}$ in place of $X'$ to  find a Borel map ${\mathbf T}:\Omega\to\prob X$ such that for any $x\in \Omega$ it holds $\supp(\mathbf T(x))\subset\partial^c\tilde\varphi(x)$ and for $\mu$-a.e. $x$ it holds $\mathbf T(x)=\ggamma_x$. Notice that $\mathbf T_\sharp\mu=\nu$.

From the definition it directly follows that
\begin{equation}
\label{eq:tv}
\|{\mathbf T}_\sharp\sigma_1-{\mathbf T}_\sharp\sigma_2\|_{\rm TV}\leq\|\sigma_1-\sigma_2\|_{\rm TV},\qquad\forall \sigma_1,\sigma_2\in\prob X,\ \supp(\sigma_1),\supp(\sigma_2)\subset \Omega,
\end{equation}
and from  the fact that $u_N(z)=-z^{1-\frac 1N}$ has sublinear growth it is easy to deduce that 
\begin{equation}
\label{eq:limuN}
\supp(\sigma_n)\subset\Omega\ \forall n\in\N,\ \lim_{n\to\infty} \|\sigma_n-\sigma\|_{\rm TV}=0\qquad\Rightarrow\qquad
\lim_{n\to\infty}\u_N(\sigma_n)=\u_N(\sigma).
\end{equation}
Let $\rho$ be the density of $\mu$ and find a sequence $(\rho_n)$ of probability densities such that $\supp(\rho_n)\subset\Omega_0$ and $\rho_n^{1-\frac 1N}$ is Lipschitz for every $n\in\N$ and satisfying $\|\rho-\rho_n\|_{L^1(X)}\to0 $. Put $\mu_n:=\rho_n\mm$, $\nu_n:=\mathbf T_\sharp\mu_n$  and notice that \eqref{eq:tv} and \eqref{eq:limuN} give
\begin{equation}
\label{eq:limn1}
\textrm{ $\u_N(\mu_{n})\to\u_N(\mu)$ \quad and\quad $\u_N(\nu_{n})\to\u_N(\nu)$ \quad as\quad $n\to\infty$.}
\end{equation}
For $n\in\N$ and  $\eps>0$ define $\rho_{n,\eps}:=c_{n,\eps}(\rho_n^{1-\frac1N}+\eps)^{\frac N{N-1}}$ on $\Omega$ and $\rho_{n,\eps}:=0$ on $X\setminus\Omega$, $c_{n,\eps}$ being the normalization constant. Put $\mu_{n,\eps}:=\rho_{n,\eps}\mm$ and $\nu_{n,\eps}:=\mathbf T_\sharp\mu_{n,\eps}$ so that again from  \eqref{eq:tv} and \eqref{eq:limuN} we obtain
\begin{equation}
\label{eq:limneps}
\textrm{ $\u_N(\mu_{n,\eps})\to\u_N(\mu_n)$\quad and\quad $\u_N(\nu_{n,\eps})\to\u_N(\nu_n)$\quad as\quad $\eps\downarrow0$\quad for every $n\in\N$.}
\end{equation}
By construction, $\mu_{n,\eps}$ and $\nu_{n,\eps}$ have support bounded and contained in $\supp(\mm)$ (for the latter replace if necessary $X$ by $\supp(\mm)$ in all the instances), thus the $CD(0,N)$ condition ensures that there exists $\ppi_{n,\eps}\in\gopt(\mu_{n,\eps},\nu_{n,\eps})$ for which \eqref{eq:defcd} holds and notice that \eqref{eq:defcd} and the weak lower semicontinuity of $\u_N$ on sequences with uniformly bounded support ensures that $\u_N((\e_t)_\sharp\ppi_{n,\eps})\to\u_N(\mu_{n,\eps})$ as $t\downarrow0$. Also,  the construction ensures that  $\tilde\varphi$ is a Kantorovich potential for $(\mu_{n,\eps},\nu_{n,\eps})$ and therefore point $(iii)$ of Proposition \ref{le:cheppalle} yields that $(\e_t)_\sharp\ppi_{n,\eps}$ is concentrated on $\Omega$ for every $n,\eps,t$. Given that the restriction of $\rho_{n,\eps}$ to $\Omega$ is Lipschitz and bounded from below by a positive constant, all the assumptions of Proposition \ref{prop:boundbasso} are fulfilled and for every $n\in\N$, $\eps>0$ we get
\[
\limi_{t\downarrow0}\frac{\u_N((\e_t)_\sharp\ppi_{n,\eps})-\u_N(\mu_{n,\eps})}{t}\geq -\frac1N\int_{\Omega} D(\rho_{n,\eps}^{1-\frac 1N})(\nabla\tilde\varphi)\,\d\mm=-\frac{c_{n,\eps}^{1-\frac1N}}{N}\int_{\Omega} D(\rho_{n}^{1-\frac 1N})(\nabla\tilde\varphi)\,\d\mm.
\]
By \eqref{eq:defcd} we have $\u_N(\nu_{n,\eps})-\u_N(\mu_{n,\eps})\geq \frac{\u_N((\e_t)_\sharp\sppi_{n,\eps})-\u_N(\mu_{n,\eps})}{t}$ for any $t\in(0,1]$ and $\eps>0$ and thus recalling \eqref{eq:limneps} and the obvious limit relation $c_{n,\eps}\uparrow 1$ as $\eps\downarrow0$  we deduce
\[
\u_N(\nu_{n})-\u_N(\mu_{n})\geq -\frac{1}{N}\int_{\Omega} D(\rho_{n}^{1-\frac 1N})(\nabla\varphi)\,\d\mm.
\]
Now recall that $\supp(\mu_n)\subset\Omega_0$ for every $n\in\N$ and that $\tilde\varphi=a\b$ on $\Omega_0$, so that by the locality property  \eqref{eq:localgrad22} and the fact that $\rho_{n}^{1-\frac 1N}$ is Lipschitz with compact support we get
\[
-\frac{1}{N}\int_{\Omega} D(\rho_{n}^{1-\frac 1N})(\nabla\varphi)\,\d\mm=-\frac{1}{N}\int_{\Omega_0} D(\rho_{n}^{1-\frac 1N})(\nabla(a\b))\,\d\mm=\frac{a}{N}\int_{\Omega_0} \rho_{n}^{1-\frac 1N}\,\d\bd\b=0,
\]
and therefore $\u_N(\nu_{n})-\u_N(\mu_{n})\geq 0$. Recalling the limiting relations \eqref{eq:limn1} we get the thesis.
\end{proof}
We can now prove the uniqueness result.
\begin{proposition}[Uniqueness, measure preservation and  single value property]\label{prop:unica}\ \linebreak  Assume \eqref{eq:ass2} and let $a\in \R$. 

Then there exists a unique (up to $\mm$-a.e. equality) Borel map  $\mathbf T_a:X\to\prob X$ such that
\begin{itemize}
\item[i)] $\supp(\mathbf T_a(x))\subset\partial^c(a\b)(x)$ for $\mm$-a.e.   $x\in X$,
\item[ii)] $(\mathbf T_a)_\sharp\mm\ll\mm $.
\end{itemize} 
Furthermore, $\mathbf T_a(x)$ is of the form $\mathbf T_a(x)=\delta_{T_a(x)}$ for $\mm$-a.e. $x\in X$ for some Borel map $T_a:X\to X$ which is $\mm$-a.e. invertible, its inverse being $T_{-a}$ (i.e.  $T_a(T_{-a}(x))=x=T_{-a}(T_a(x))$ for $\mm$-a.e. $x\in X$) and satisfies
\[
(T_a)_\sharp\mm=\mm.
\]
\end{proposition}
\begin{proof} 

\noindent{\bf Single value property}
Let   $\mathbf T:X\to\prob X$ be such that $(i),(ii)$ of the assumptions are fulfilled and  $\mathcal B$ be the set of equivalence classes of  bounded Borel subsets of $X$ with positive $\mm$-measure where we identify two sets whose symmetric difference has 0 $\mm$-measure.

We define a map $R:\mathcal B\to\mathcal B$ as follows. For $E\in\mathcal B$ let $\mu:=\mm(E)^{-1}\mm\restr E$ ($\mu$ depends only on the equivalence class of $E$), consider $\nu:=(\mathbf T_a)_\sharp\mu$ and notice that by $(ii)$ we have $\nu\ll\mm$, say $\nu=\rho\mm$. Then put $R(E):=\{\rho>0\}$.

We claim that 
\begin{equation}
\label{eq:mespres}
\mm(R(E))=\mm(E),\qquad\forall E\in\mathcal B.
\end{equation}
Indeed, given $E\in\mathcal B$ and $\mu,\nu$ as above, the assumption $(i)$ grants that $a\b$ is a Kantorovich potential from $\mu$ to $\nu$ and thus  by point $(i)$ of Theorem \ref{thm:basemetric} we get that $-a\b$ is Kantorovich potential from $\nu$ to $\mu$. By construction, we have $\nu\ll\mm$ and thus Proposition   \ref{prop:der0} yields $\u_N(\mu)\geq\u_N(\nu)$. Therefore
\[
-\mm(E)^{\frac1N}=\u_N(\mu)\geq \u_N(\nu)\geq-\mm(R(E))^{-\frac1N},
\]
where the second inequality follows from Jensen's inequality and the fact that $\nu$ in concentrated on $R(E)$. 

Now define $\ggamma\in\prob{X^2}$ by $\ggamma(A\times B):=\int_A \mathbf T_a(x)(B)\,\d\mu(x)$ for every $A,B\subset X$ Borel so that $\pi^1_\sharp\ggamma=\mu$ and $\pi^2_\sharp\ggamma=\nu$. Let $\tilde\ggamma\in\prob{X^2}$ be given by $\d\tilde\ggamma(x,y):=\frac{1}{\mm(R(E))\rho(y)}\d\ggamma(x,y)$, $\rho$ being the density of $\nu$, and put $\tilde\mu:=\pi^1_\sharp\tilde\ggamma$, $\tilde\nu:=\pi^2_\sharp\tilde\ggamma$. By construction it holds $\tilde\nu=\mm(R(E))^{-1}\mm\restr{R(E)}$ and $\tilde\mu$ is concentrated on $E$, thus arguing as before we get
\[
-\mm(R(E))^{\frac1N}=\u_N(\tilde\nu)=\u_N(\tilde\mu)\geq-\mm(E)^{-\frac1N},
\]
and our claim \eqref{eq:mespres} is proved.

Now assume by contradiction that $\mathbf T_a(x)$ is not a Dirac delta for a set of $x$ of positive $\mm$-measure. Then using the Borel selection result stated in Corollary \ref{cor:borelsel} it is easy to see that there exists $r>0$, a bounded Borel set $E$  with $\mm(E)>0$ and two Borel maps $T^i:E\to X$, $i=1,2$  such that $T^i(x)\in\supp(\mathbf T_a(x))$ for $\mm$-a.e. $x\in E$, $i=1,2$, $\sfd(T^1(x),T^2(x))>r$ for $\mm$-a.e. $x\in E$ and ${\rm diam}(T^1(E))\leq r/3$. Define two Borel maps $\mathbf T^i_a:X\to\prob X$, $i=1,2$, by
\[
\mathbf T^i_a(x):= \left\{\begin{array}{ll}
\mathbf T_a(x)&\qquad\textrm{ if $x\notin E$},\\
&\\
c_i(x)\mathbf T_a(x)\restr{B_{r/3}(T^i(x))},&\qquad\textrm{ if $x\in E$},
\end{array}\right.
\]
where $c_i(x):=\big(\mathbf T_a(x)\big(B_{r/3}(T^i(x))\big)\big)^{-1}$, $i=1,2$, is the normalization constant. By construction the maps $\mathbf T^i_a$, $i=1,2$, satisfy the properties $(i),(ii)$ in the assumption, hence defining the associated maps $R^i:\mathcal B\to\mathcal B$ as we previously did for $\mathbf T_a$, we get $\mm(E)=\mm(R^1(E))=\mm(R^2(E))$. The contradiction comes from the fact that by construction it holds $R^1(E)\cap R^2(E)=\emptyset$ and $R(E)\supset R^1(E)\cup R^2(E)$, thus
\[
\mm(E)=\mm(R(E))\geq\mm(R^1(E)\cup R^2(E))=\mm(R^1(E))+\mm(R^2(E))=2\mm(E).
\]
We therefore deduce that for some Borel map $T_a:X\to X$ it holds $\mathbf T_a(x)=\delta_{T_a(x)}$ for $\mm$-a.e. $x$.

\noindent{\bf Uniqueness} Assume that $\mathbf {\tilde T}_a:X\to \prob X$ also fulfills the assumptions. Then  also the map $x\mapsto \frac12(\mathbf T_a(x)+\mathbf{\tilde T}_a(x))$ would do so, and if $\mathbf {\tilde T}_a(x)\neq \mathbf T_a(x)$ for a set of positive $\mm$-measure, then $\frac12(\mathbf T_a(x)+\mathbf{\tilde T}_a(x))$ would not be a Dirac delta for a set of positive $\mm$-measure, contradicting what we just proved.

\noindent{\bf Invertibility} Apply  Proposition \ref{prop:esiste} with $-a$ in place of $a$ to find a Borel map $\mathbf T_{-a}$ fulfilling $(i),(ii)$ of the thesis. Hence it satisfies the assumptions of the current proposition and by what we just proved we know that $\mathbf T_{-a}(x)=\delta_{T_{-a}(x)}$ for $\mm$-a.e. $x$ for some Borel map $T_{-a}:X\to X$.

Let $\eeta:=(\Id,T_{-a})_\sharp\mm$ and $\{\eeta_x\}_{x\in X}$ its disintegration w.r.t. the projection onto the second coordinate. Notice that since $\pi^2_\sharp\eeta=(T_{-a})_\sharp\mm\ll\mm$, $\eeta_x$ is  well defined for $\mm$-a.e. $x\in\{\rho>0\}$, where $\rho$ is  the density of $(T_{-a})_\sharp\mm$ w.r.t. $\mm$. Define $\mm$-a.e. the Borel map $\tilde{\mathbf T}_a:X\to\prob X$ by
\[
\tilde{\mathbf T}_a(x):=\left\{
\begin{array}{ll}
\eeta_x,&\qquad\textrm{ if }\rho(x)>0,\\
\delta_{T_a(x)},&\qquad\textrm{ if }\rho(x)=0.
\end{array}
\right.
\]
It is clear by the construction that $\tilde{\mathbf T}_a$ fulfills  the assumptions  $(i),(ii)$ of the proposition. Thus by the uniqueness statement that we proved it follows that $\tilde{\mathbf T}_a(x)=\delta_{T_a(x)}$ for $\mm$-a.e. $x$, which easily implies $T_a(T_{-a}(x))=x$ for $\mm$-a.e. $x$. Inverting the roles of $T_a$ and $T_{-a}$ in this argument we also get $T_{-a}(T_a(x))=x$ for $\mm$-a.e. $x$.

\noindent{\bf Measure preservation}  We shall prove that $(T_{-a})_\sharp\mm=\mm$, the proof for $T_a$ being similar. Observe that the identity $T_a(T_{-a}(x))=x$ valid for $\mm$-a.e. $x$ gives
\begin{equation}
\label{eq:ultima}
\mm\Big(\big(T^{-1}_a(T^{-1}_{-a}(E))\setminus E\big)\cup \big(E\setminus T^{-1}_a(T^{-1}_{-a}(E))\big)\Big)=0,\qquad\forall \textrm{ Borel } E\subset  X.
\end{equation}
Also, by definition we have $(T_a)_\sharp\mm\ll\mm$ and $(T_{-a})_\sharp\mm\ll\mm$, therefore $\mm=(T_{-a})_\sharp(T_a)_\sharp\mm\ll(T_a)_\sharp\mm\ll\mm$ i.e.
\begin{equation}
\label{eq:asscont}
\mm(E)=0\quad\Leftrightarrow\quad \mm(T^{-1}_a(E))=0,\qquad\forall \textrm{ Borel } E\subset  X.
\end{equation}
Let $R:\mathcal B\to\mathcal B$ be defined as in the first step of the proof and for $E\in\mathcal B$ let $\mu:=\mm(E)^{-1}\mm\restr E$ and $\nu:=(T_a)_\sharp\mu$, as before. Then for every Borel $A\subset X$  we have
\[
\begin{split}
\mm(R(E)\cap A)>0&\quad\Leftrightarrow \quad\nu(A)>0 \quad\Leftrightarrow\quad \mu(T_a^{-1}(A))>0\quad\Leftrightarrow\quad\mm(E\cap T^{-1}_a(A))>0\\
\textrm{by \eqref{eq:ultima}}\qquad&\quad\Leftrightarrow \quad\mm\big(T^{-1}_a(T^{-1}_{-a}(E))\cap T^{-1}_a(A)\big)>0\\
&\quad\Leftrightarrow \quad \mm\big(T^{-1}_a\big(T^{-1}_{-a}(E)\cap A\big)\big)>0\\
\textrm{by \eqref{eq:asscont}}\qquad&\quad\Leftrightarrow \quad \mm(T^{-1}_{-a}(E)\cap A)>0.
\end{split}
\]
which shows that the symmetric difference between $R(E)$ and $T^{-1}_{-a}(E)$ has 0 $\mm$-measure. The conclusion follows from the arbitrariness of $E\in\mathcal B$ and \eqref{eq:mespres}.
\end{proof}

Collecting together all these results, we can now prove the main result of this chapter.
\begin{theorem}[The gradient flow of $\b$ preserves the measure]\label{thm:gfpresmes}  Assume \eqref{eq:ass2}.

Then there exists a Borel map $\X: X\times\R\to X$ such that the following are true.
\begin{itemize}
\item[i)] For $\mm$-a.e. $x\in X$ the curve $\R\ni t\mapsto\X_t(x)$ is a gradient flow trajectory of $\b$ passing from $x$ at time $t=0$. In particular it holds
\begin{equation}
\label{eq:distfl}
\sfd(\X_t(x),\X_s(x))=|t-s|,\qquad\forall t,s\in\R,\ \mm\ae\ x.
\end{equation}
\item[ii)] For any $t\in \R$ it holds
\begin{equation}
\label{eq:presmeas}
(\X_t)_\sharp\mm=\mm.
\end{equation}
\item[iii)] For any $t,s\in\R$ it holds
\begin{equation}
\label{eq:group}
\X_t\circ\X_s=\X_{t+s},\qquad\mm\ae.
\end{equation} 
\end{itemize}
The map $\X$ is unique in the following sense. Let $t_0\leq0\leq t_1$ and $\mathbf F:X\to\prob{C([t_0,t_1],X)}$ be a Borel map satisfying:
\begin{itemize}
\item[i')] For $\mm$-a.e. $x\in X$ and $\mathbf F(x)$-a.e. $\gamma$, $\gamma$ is a gradient flow trajectory of $\b$ passing from $x$ at time $t=0$.
\item[ii')] For any $t\in[t_0,t_1]$ it holds $(\mathbf F_t)_\sharp\mm\ll\mm$, where $\mathbf F_t:X\to\prob X$ is defined by $\mathbf F_t(x):=(\e_t)_\sharp(\mathbf F(x))$.
\end{itemize}
Then $\mathbf F(x)$ is concentrated on the curve $[t_0,t_1]\ni t\mapsto \X_t(x)$ for $\mm$-a.e. $x\in X$. 
\end{theorem}
Note: the uniqueness part gives in particular that if a Borel map $\tilde\X:X\times\R\to X$ fulfills $(i)$ and $(\tilde\X_t)_\sharp\mm\ll\mm$, for any $t\in \R$, then for any $t\in\R$ it holds $\tilde \X_t(x)=\X_t(x)$ for $\mm$-a.e. $x\in X$.
\begin{proof}

\noindent{\bf Existence}
Let $t\in\Q$ and combine Propositions \ref{prop:esiste} and \ref{prop:unica} to get the existence of a Borel map $\X_t:X\to X$ such that $\X_t(x)\in\partial^c(t\b)(x)$ for $\mm$-a.e. $x\in X$ and $(\X_t)_\sharp\mm=\mm$. From point $(v)$ of Theorem \ref{thm:basemetric} and the uniqueness part of Proposition \ref{prop:unica} we deduce that
\begin{equation}
\label{eq:groupq}
\X_t\circ \X_s=\X_{t+s},\qquad\mm\ae,\qquad\forall t,s\in \Q.
\end{equation}
(use the $\mm$-a.e. invertibility of $\X_t$ to deal with the case where $t$ and $s$ have different signs). Let $\mathcal N\subset X$ be the set of $x$'s such that either $\X_t(x)\notin\partial^c(t\b)(x)$ for some $t\in \Q$ or $\X_t(\X_s(x))\neq \X_{t+s}(x)$ for some $t,s\in \Q$. Then $\mathcal N$ is Borel and negligible.

From point $(iv)$ of Theorem \ref{thm:basemetric} and the group property \eqref{eq:groupq} we deduce
\begin{equation}
\label{eq:distq}
\sfd(\X_t(x),\X_s(x))=|t-s|,\qquad\forall \ x\in X\setminus\mathcal N,
\end{equation}
for any $t,s\in\Q$. Hence for $x\in X\setminus\mathcal N$ and any sequence $(t_n)\subset \Q$ converging to $t\in\R$, the limit of $(\X_{t_n}(x))$ as $n\to\infty$ exists and coincides with $\X_t(x)$ for $t\in \Q$. Call $\X_t(x)$ such limit.

By construction, \eqref{eq:distq} holds for any $t,s\in\R$, which means that for $x\in X\setminus\mathcal N$ the curve $t\mapsto \X_t(x)$ is a line. Clearly, such line passes through $x$ at time 0. Moreover, from point $(iii)$ of Theorem \ref{thm:basemetric} and the group property \eqref{eq:groupq} we deduce that such line is a gradient flow trajectory of $\b$. This proves point $(i)$ of the thesis.

By \eqref{eq:distq} it follows that the curve   $t\mapsto (\X_t)_\sharp\mm$ is weakly continuous in duality with $C_c(X)$. Thus, since \eqref{eq:presmeas} holds for $t\in\Q$ by construction, it holds for every $t\in\R$.

Finally, the group property \eqref{eq:group} follows directly from \eqref{eq:groupq} and \eqref{eq:distq}.

\noindent{\bf Uniqueness} By the uniqueness part of Proposition \ref{prop:unica} we know that there exists a Borel negligible set $\mathcal N'\subset X$ such that for any $t\in \Q\cap [t_0,t_1]$ it holds $\mathbf F_t(x)=\delta_{\X_t(x)}$ for every $x\in X\setminus \mathcal N'$, and by $(i')$ we know that for some Borel negligible set $\mathcal N''\subset X$ the measure $\mathbf F(x)$ is concentrated on 1-Lipschitz curves for any $x\in X\setminus\mathcal N''$. Thus for  $x\in X\setminus(\mathcal N\cup\mathcal N'\cup\mathcal N'')$ (with $\mathcal N$ defined as before) we know that $\mathbf F_t(x)=\delta_{\X_t(x)}$ for $t\in\Q\cap [t_0,t_1]$, that $t\mapsto \X_t(x)$ is continuous and that each curve in $\supp(\mathbf F(x))$ is continuous as well. This is enough to conclude.
\end{proof}
\begin{remark}{\rm
Notice that our uniqueness result does not tell that for $\mm$-a.e. $x\in X$ there is a unique gradient flow trajectory for $\b$ passing through $x$ at time $0$: at the moment we don't know whether such uniqueness holds (the technical problem is that  it seems not possible to drop the assumption $\mu\ll\mm$ in Proposition \ref{prop:der0}). What we know is that uniqueness is granted, also at the level of measure valued flows $\mathbf F$, provided the pointwise information `the measure is concentrated on gradient flows' is coupled with the local requirement  $(\mathbf F_t)_\sharp\mm\ll\mm$.

This is an unexpected link with Ambrosio's notion of regular Lagrangian flow \cite{Ambrosio04} in connection with DiPerna-Lions theory \cite{DiPerna-Lions89}, where a similar requirement is imposed to get uniqueness of the flow associated to vector fields on $\R^d$ with Sobolev/BV regularity. 
}\fr\end{remark}

The measure preservation property just proved   has the following important consequence about the behavior of Sobolev functions along the flow:
\begin{proposition}\label{prop:nizza} Assume \eqref{eq:ass2} and let $\X$ be the gradient flow of $\b$ given by Theorem \ref{thm:gfpresmes}. Then for every $f\in \s^2_{\rm loc}(X)$ and $t\geq 0$ it holds
\begin{equation}
\label{eq:lungogf}
|f(\X_t(x))-f(x)|\leq\int_0^t\weakgrad f(\X_s(x))\,\d s,\qquad\mm\ae\ x\in X,
\end{equation}
similarly for $t\leq 0$ replacing the integral from 0 to $t$ with the integral from $t$ to 0. In particular, for $f\in\s^2(X)$ we have
\begin{equation}
\label{eq:lungogf2}
\int|f(\X_t(x))-f(x)|^2\,\d\mm(x)\leq t^2\int\weakgrad f^2(x)\,\d\mm(x),\qquad\forall t\in\R.
\end{equation}
\end{proposition}
\begin{proof} We start with \eqref{eq:lungogf}. An iteration argument based on the measure preservation property shows that it is sufficient to prove the thesis for $t\in[0,1]$.  Let $T:X\to C([0,1],X)$ be $\mm$-a.e. defined by $(T(x))_t:=\X_t(x)$, $\tilde\mm\in\prob X$ such that $\tilde\mm\leq \mm$ and $\mm\ll\tilde\mm$ and put   $\ppi:=T_\sharp\tilde\mm\in \prob{C([0,1],X)}$. Then $\ppi$ is concentrated on 1-Lipschitz curves and $(\e_t)_\sharp\ppi=(\X_t)_\sharp\tilde\mm\leq (\X_t)_\sharp\mm=\mm$ for every $t\in[0,1]$. Thus $\ppi$ is a test plan. If $f\in\s^2(X)$ inequality \eqref{eq:localplan} yields
\[
|f(\gamma_t)-f(\gamma_0)|\leq \int_0^t\weakgrad f(\gamma_s)|\dot\gamma_s|\,\d s=\int_0^t\weakgrad f(\gamma_s)\,\d s,\qquad\ppi\ae \ \gamma,
\]
which, by definition of $\ppi$ is equivalent to the thesis. If $f$ just belongs to $\s^2_{\rm loc}(X)$ use a cut-off argument to reduce to the previous case and the locality of minimal weak upper gradients to conclude. The case $t\leq 0$ is handled analogously.

For \eqref{eq:lungogf2}, square and integrate \eqref{eq:lungogf} to get
\[
\begin{split}
\int | f(\X_t(x))-f(x)|^2\,\d\mm(x)&\leq \int\left(\int_0^t\weakgrad f(\X_s(x))\,\d s\right)^2\,\d\mm(x)\\
&\leq t\iint_0^t\weakgrad f^2(\X_s(x))\,\d s\,\d\mm(x)\\
&=t \int\weakgrad f^2(x)\,\d\left(\int_0^t (\X_s)_\sharp\mm\,\d s\right)(x)\\
&=t^2\int\weakgrad f^2(x)\,\d \mm(x).
\end{split}
\]
\end{proof}
A direct consequence of this proposition and its proof is that  $\weakgrad \b=1$ $\mm$-a.e., a fact that we shall frequently use later on without explicit reference. Indeed, being $\b$ 1-Lipschitz it holds $\weakgrad\b\leq 1$ $\mm$-a.e., and with $\ppi$ as in the proof we have
\[
\begin{split}
\int |\b(\gamma_1)-\b(\gamma_0)|\,\d\ppi(\gamma)&\leq \iint_0^1\weakgrad \b(\gamma_t)|\dot\gamma_t|\,\d t\,\d\ppi(\gamma)\leq \iint_0^1\weakgrad \b(\gamma_t)\,\d t\,\d\ppi(\gamma)=\int \weakgrad \b\,\d\nn,
\end{split}
\]
where $\nn:=\int_0^1 (\X_t)_\sharp\tilde\mm$. Given that by construction of $\ppi$ we know that $|\b(\gamma_1)-\b(\gamma_0)|=1$ for $\ppi$-a.e. $\gamma$, the left-hand side of the above inequality is 1 and the fact that $\mm\ll\nn$ yields the claim (notice that in fact the more sophisticated arguments of Cheeger \cite{Cheeger00}  ensure that $\weakgrad f=\lip(f)$ $\mm$-a.e. for any $f$ locally Lipschitz, given that $CD(0,N)$ spaces are doubling and support a weak-local 1-1 Poincar\'e inequality - \cite{Lott-Villani07}, \cite{Rajala12}). 

\medskip

We know that for $\mm$-a.e. $x\in X$ the curve $t\mapsto \X_t(x)$ is a gradient flow trajectory of $\b$. The next Corollary reformulates this property in a different way w.r.t. that used to define gradient flows in Definition \ref{def:gf} and closer in spirit to the first order differentiation formula in Proposition \ref{prop:firstdiff}.
\begin{corollary}\label{cor:derfiga}  Assume \eqref{eq:ass2}, let $\X$ be given by Theorem \ref{thm:gfpresmes} and pick  $f\in\s^2(X)$. 

Then
\begin{equation}
\label{eq:derfiga}
\lim_{t\to 0}\frac{f\circ\X_t-f}{t}=-Df(\nabla\b),\qquad\textrm{ weakly in }L^2(X).
\end{equation}
\end{corollary}
\begin{proof} We shall only prove the limiting property as $t\downarrow0$, the proof of the case $t\uparrow0$ being similar. Inequality \eqref{eq:lungogf2} grants that the $L^2(X)$ norms of the functions in the left hand side of \eqref{eq:derfiga} are uniformly bounded. To conclude the proof it is therefore sufficient to show that for any $g\in L^2(X)$ it holds
\[
\lim_{t\downarrow0}\int \frac{f\circ\X_t-f}{t}g\,\d\mm=-\int Df(\nabla \b)g\,\d\mm.
\]
A simple approximation argument shows that it is sufficient to prove it just  for $g\in L^1\cap L^\infty(X)$ non-negative and with bounded support. Pick such $g$, assume $g$ is not identically 0 (otherwise there is nothing to prove) and define $\mu:=(\int g\,\d\mm)^{-1}g\mm\in\prob X$ and $\ppi:=T_\sharp\mu\in \prob{(C([0,1],X))}$, where $T:X\to C([0,1],X)$ is  defined by $(T(x))_t:=\X_t(x)$ as in the proof of Proposition \ref{prop:nizza}. Arguing as in the proof of Proposition \ref{prop:nizza}  and using the fact that $g$ is bounded we get that $\ppi$ is a test plan. By construction, for some bounded open set $\Omega$ it holds $\supp((\e_t)_\sharp\ppi)\subset\Omega$ for any $t\in[0,1]$. Notice that $\b\in\s^2(\Omega)$ and that 
\[
\lim_{t\downarrow0}\int\frac{\b(\gamma_0)-\b(\gamma_t)}{t}\,\d\ppi(\gamma)=1=\frac12\int\weakgrad\b^2(\gamma_0)\,\d\ppi(\gamma)+\frac12\lims_{t\downarrow0}\frac1t\iint_0^t|\dot\gamma_s|^2\,\d s\,\d\ppi(\gamma),
\]
i.e. $\ppi$ represents the gradient of $-\b$ on $\Omega$ (Definition \ref{def:planrepgrad}). By the first order differentiation formula \ref{eq:firststrictconv} we deduce
\[
\lim_{t\downarrow0}\int \frac{f\circ\X_t-f}{t}g\,\d\mm=\lim_{t\downarrow 0}\int\frac{f(\gamma_t)-f(\gamma_0)}{t}\,\d\ppi(\gamma)=-\int Df(\nabla \b)g\,\d\mm,
\]
which is the thesis.
\end{proof}

\chapter{The gradient flow of $\b$ preserves the distance}\label{se:dist}
\section{Preliminary notions}
\subsection{Infinitesimally Hilbertian spaces}\label{se:infhil}
We shall now introduce the crucial property ensuring a `Riemannian-like' behavior of our $CD(0,N)$ space. Recall that a smooth Finsler manifold $F$ is Riemannian if and only if the Sobolev space $W^{1,2}$ built over it is Hilbert. 

Motivated by the results in \cite{AmbrosioGigliSavare11-2} (see also \cite{AmbrosioGigliMondinoRajala12}), in \cite{Gigli12} the following definition has been proposed:
\begin{definition}[Infinitesimally Hilbertian spaces]
Let $(X,\sfd,\mm)$ be a metric measure space. We say that $(X,\sfd,\mm)$ is infinitesimally Hilbertian provided $W^{1,2}(X,\sfd,\mm)$ is an Hilbert space.
\end{definition}

It is immediate to verify that infinitesimal Hilbertianity implies infinitesimal strict convexity. Moreover, much like in Riemannian manifolds, on infinitesimally Hilbertian spaces we can `identify differential and gradients' in the sense made precise by the following proposition, see \cite{AmbrosioGigliSavare11-2} and \cite{Gigli12} for a proof (see also \cite{Gigli-Mosconi12} for a quick overview on the subject).
\begin{proposition}
Let $(X,\sfd,\mm)$ be an infinitesimally Hilbertian space and $\Omega\subset X$ an open set. Then for every $f,g\in \s^2_{\rm loc}(\Omega)$ it holds
\[
Df(\nabla g)=Dg(\nabla f),\qquad\mm\ae\ \rm{ on }\ \Omega.
\]
\end{proposition}
To highlight the symmetry of this object, we will use the notation $\la\nabla f,\nabla g\ra$ in place of $Df(\nabla g)$.  Also, to mimic the notation used in a Riemannian context, we shall write $|\nabla f|$ in place of $\weakgrad f$ to denote the minimal weak upper gradient of the Sobolev function $f$. Still, we remark again that we are not really defining what the gradient of a Sobolev function is, but just what is the value of the `scalar product between two gradients'. The notation is justified by the following calculus rules,  where $\Omega\subset X$ is open and all the state equalities must be intended  $\mm$-a.e. on $\Omega$.

\noindent\underline{Squared norm as scalar product}. For any $f\in\s^2_{\rm loc}(\Omega)$ it holds
\begin{equation}
\label{eq:norm}
\la\nabla f,\nabla f\ra=|\nabla f|^2.
\end{equation}
\noindent\underline{Symmetry}. For any $f,g\in\s^2_{\rm loc}(\Omega)$ it holds 
\begin{equation}
\label{eq:simm}
\la\nabla f,\nabla g\ra=\la\nabla g,\nabla f\ra.
\end{equation}
\noindent\underline{Linearity in $f,g$}. For any $f_0,f_1,g\in \s^2_{\rm loc}(\Omega)$ and any $\alpha_0,\alpha_1\in \R$ it holds 
\begin{equation}
\label{eq:lin}
\la\alpha_0\nabla f_0+\alpha_1\nabla f_1,\nabla g\ra=\alpha_0\la\nabla f_0,\nabla g\ra+\alpha_1\la\nabla f_1,\nabla g\ra,
\end{equation}
and similarly, due to \eqref{eq:simm}, for linearity in $g$.

\noindent\underline{1-Lipschitz continuity  in $f,g$}. For any $f_0,f_1,g\in \s^2_{\rm loc}(\Omega)$  it holds 
\begin{equation}
\label{eq:1lipfg}
|\la \nabla f_0,\nabla g\ra-\la\nabla f_1,\nabla g\ra|\leq |\nabla (f_0-f_1)||\nabla g|,
\end{equation}
and similarly, due to \eqref{eq:simm}, for 1-Lipschitz continuity in $g$.

\noindent\underline{Leibniz rule in $f,g$}. For any $f_0,f_1\in \s^2_{\rm loc}(\Omega)\cap L^\infty_{\rm loc}(\Omega)$ and $g\in \s^2_{\rm loc}(\Omega)$ it holds
\begin{equation}
\label{eq:leibniz}
\la\nabla(f_1f_2),\nabla g\ra=f_1\la\nabla f_2,\nabla g\ra+f_2\la \nabla f_1,\nabla g\ra,
\end{equation}
and similarly, due to \eqref{eq:simm}, for the Leibniz rule in $g$.

\noindent\underline{Chain rule in $f,g$}. For any $f,g\in \s^2_{\rm loc}(\Omega)$ and $\varphi:\R\to\R$ Lipschitz it holds
\begin{equation}
\label{eq:chainhil}
\la\nabla(\varphi\circ f),\nabla g\ra=\varphi'\circ f\la \nabla f,\nabla g\ra,
\end{equation}
where $\varphi'$ is defined arbitrarily at points where $\varphi$ is not differentiable. Similarly, due to \eqref{eq:simm}, at the level of  $g$. The Lipschitz continuity of $\varphi$ can be relaxed as in the chain rule \eqref{eq:chainbase}.
\noindent\underline{First order differentiation formula} For $\Omega\subset X$ open,  $f,g\in\s^2(\Omega)$ and $\ppi$ representing the gradient of $g$ in $\Omega$ it holds
\begin{equation}
\label{eq:firsthil}
\lim_{t\downarrow0}\int\frac{f(\gamma_t)-f(\gamma_0)}{t}\,\d\ppi(\gamma)=\int \la\nabla f,\nabla g\ra(\gamma_0)\,\d\ppi(\gamma),
\end{equation}
in particular, the limit at the left-hand side exists.

\medskip

Notice that in particular we also have that  $W^{1,2}(\Omega)$ is Hilbert for every $\Omega\subset X$ open, that $D(\bd)$ is a vector space and that $\bd:D(\bd)\to\textrm{ \{Radon measures on $X$\}}$ is a linear operator.

\bigskip

It is worth to underline that although these calculus rules are the same as the ones available in a Riemannian framework, they are produced following a different path. So, for instance, while on a Riemannian manifold the given data is the scalar product and \eqref{eq:norm} is the definition of the norm of $\nabla f$, in our context the given (actually: built) object is $|\nabla f|$ and the starting from it and with the assumption of infinitesimal Hilbertianity we produce a sort of scalar product between gradients. In this sense, the identity \eqref{eq:norm} is a theorem and not a definition in the current setting. Other basic first order calculus rules follows from those presented, for instance from \eqref{eq:norm} and \eqref{eq:lin} we get
\[
|\nabla(f+g)|^2=|\nabla f|^2+2\la\nabla f,\nabla g\ra+|\nabla g|^2,
\]
which we shall occasionally use later on without explicit reference. In summary,  the first order Sobolev calculus works as in the smooth Riemannian case, and this fact - not surprisingly - plays a crucial role in the derivation of the desired geometric properties of infinitesimally Hilbertian $CD(K,N)$ spaces.

These calculus rules are strongly reminiscent of the $\Gamma$-calculus available in the context of Dirichlet forms,  the carr\'e du champ $\Gamma(f,g)$ taking the place of what we are calling $\la\nabla f,\nabla g\ra$. There are indeed strong connections between the two points of view, in particular in relation with lower bounds on the Ricci curvature, see \cite{AmbrosioGigliSavare11-2}, \cite{AmbrosioGigliSavare12} and \cite{Koskela-Zhou12} for recent progresses in this direction. Yet, beside the fact that a priori the two are defined on different structures (topological spaces with a measure and a Dirichlet form for the former, and metric measure spaces for the latter), there are two differences between the two approaches. The first is that Dirichlet forms are quadratic forms by definition, in this sense they always produce, by nature, an Hilbertian-like calculus. We have seen instead that the duality relation between differentials and gradients of Sobolev functions covers also the case of Finsler structures and that if one wants an Hilbertian-like structure something must be imposed: what we are calling infinitesimal Hilbertianity. The second, and most important, is the possibility of stating and proving the first order differentiation formula \ref{eq:firstdiff} in the context of metric measure spaces, which reduces to \eqref{eq:firsthil} in the current setting: this sort of `horizontal' derivation seems unavailable by direct means in the `vertical'   $L^2$ world of Dirichlet forms (but if the form is good enough it is certainly possible to first produce the intrinsic metric from the form and then prove \eqref{eq:firsthil} in the resulting metric measure space, see e.g. \cite{AmbrosioGigliSavare12}).

\bigskip

Notice that the uniform convexity of $W^{1,2}(X,\sfd,\mm)$ and Theorem \ref{thm:energylip} give for free the following stronger density result for Lipschitz functions:
\begin{theorem}[Strong density of Lipschitz functions in $W^{1,2}(X,\sfd,\mm)$]\label{thm:stronglip}
Let $(X,\sfd,\mm)$ be an infinitesimally Hilbertian space.

Then Lipschitz functions are dense in  $W^{1,2}(X)$, i.e. for any $f\in W^{1,2}(X)$ there exists a sequence $(f_n)\subset W^{1,2}(X)$ of Lipschitz functions such that $f_n\to f$, $\weakgrad {(f_n-f)}\to 0$ and $\lip(f_n)\to\weakgrad f$ as $n\to\infty$ in $L^2(X)$.

Furthermore:
\begin{itemize}
\item if $\mm$ gives finite mass to bounded sets, then the $f_n$'s can be chosen with bounded support for every $n\in\N$,
\item if $(\supp(\mm),\sfd)$ is proper, then the $f_n$'s can be chosen with compact support for every $n\in\N$.
\end{itemize}
\end{theorem}

\begin{remark}{\rm
From the definition proposed, it seems that the Sobolev exponent $p=2$ plays a special role in the definition of infinitesimal Hilbertianity. This is in contrast with the smooth world, where recognizing Riemannian manifolds among Finsler ones requires no choice of any Sobolev exponent. 

The point is the following. On a general metric measure space one can define the Sobolev class $\s^p(X,\sfd,\mm)$ for any $p\in(1,\infty)$ (the borderline cases $p=1,\infty$ can also be dealt with, but not surprisingly they are harder to handle). For $f\in\s^p(X,\sfd,\mm)$ there is a well defined  $p$-minimal weak upper gradient $|Df|_p$. However, in general for $f\in \s^p\cap\s^{p'}(X,\sfd,\mm)$, $p<p'$, one only has $|Df|_p\leq |Df|_{p'}$ $\mm$-a.e., the other inequality being unknown. Therefore also the duality between differentials and gradients can a priori  be affected by the choice of the Sobolev exponent one is working with.  Yet, if $\mm$ is doubling and the space supports a weak-local 1-1 Poincar\'e inequality (which is always the case for $CD(K,N)$ spaces with $N<\infty$ - see \cite{Lott-Villani07} and \cite{Rajala12}), then the results of Cheeger in \cite{Cheeger00} ensure that $|Df|_p=|Df|_{p'}$ $\mm$-a.e. for $f$ as above and using this fact and Lusin's type approximation with Lipschitz functions (see e.g. Theorem 5.1 in \cite{Bjorn-Bjorn11}) one sees that infinitesimal Hilbertianity can be defined asking for the map $f\mapsto |D f|_p^2$ to be a quadratic form for some (and thus any) $p\in(1,\infty)$.
}\fr\end{remark}
\subsection{Heat flow}

On an infinitesimally Hilbertian space $(X,\sfd,\mm)$, the map 
\[
L^2(X,\mm)\ni f\qquad\mapsto\qquad\left\{\begin{array}{ll}
\displaystyle{\frac12\int|\nabla f|^2\,\d\mm},&\qquad\textrm{ if }f\in W^{1,2}(X,\sfd,\mm),\\
+\infty,&\qquad\textrm{ otherwise},
\end{array}\right.
\]
is a closed Dirichlet form. We shall denote by $\h_t$ the associated semigroup and call it  heat flow. A lower bound on the Ricci curvature and an upper bound on the dimension grants several regularizing effects for the heat flow, in the following theorem we collect  those that  we shall use later on.
\begin{theorem}[Basic properties of the heat flow]
Let $(X,\sfd,\mm)$ be an infinitesimally\linebreak Hilbertian $CD(0,N)$ space. Then the following holds.
\begin{itemize}
\item\underline{$L^2\to W^{1,2}$ regularization}. It holds
\begin{equation}
\label{eq:l2contr}
\|\h_t(f)\|_{L^2(X)}\leq \|f\|_{L^2(X)},\qquad\forall f\in L^2(X),\ t\geq 0,
\end{equation}
and $\h_t(f)\in \s^2(X)$ for any $f\in L^2(X)$, $t>0$ with
\begin{equation}
\label{eq:l2s2}
\||\nabla \h_t(f)|\|_{L^2(X)}\leq \frac1{\sqrt{2t}}\|f\|_{L^2(X)},\qquad \forall f\in L^2(X),\ t>0.
\end{equation}
\item\underline{Heat kernel}. There exists a Borel map $(0,\infty)\times [\supp(\mm)]^2\ni (t,x,y)\mapsto \rho(t,x,y)\in \R^+$ satisfying $\rho(t,x,y)=\rho(t,y,x)$ for any $(t,x,y)\in(0,\infty)\times [\supp(\mm)]^2$, 
\begin{equation}
\label{eq:mass1}
\int\rho(t,x,y)\,\d\mm(y)=1,\qquad\forall t\in(0,\infty),\ x\in \supp(\mm),
\end{equation}
and 
\begin{equation}
\label{eq:heatbdd}
\sup_{y\in Y}\rho(t,x,y)<\infty,\qquad\forall t\in(0,\infty),\ x\in \supp(\mm),
\end{equation}
such that for any $f\in L^2(X)$ we have the representation formula
\begin{equation}
\label{eq:reprker}
\h_t(f)(x)=\int f(y)\rho(t,x,y)\,\d\mm(y),\qquad\forall t\in(0,\infty).
\end{equation}
\item\underline{Gaussian estimates}. There exists a constant $\mathcal C_1(N)$  such that the Gaussian bound
\begin{equation}
\label{eq:gaussest}
\rho(t,x,y)\leq\frac{ \mathcal C_1}{\mm(B_{\sqrt t}(x))}\,e^{-\dfrac{\sfd^2(x,y)}{5t}},
\end{equation}
holds for any $t,x,y\in (0,\infty)\times X^2$. 
\item\underline{Bakry-\'Emery condition}. For any $f\in W^{1,2}(X)$ and any $t\in (0,\infty)$ it holds
\begin{equation}
\label{eq:BE}
|\nabla(\h_t(f))|^2\leq \h_t(|\nabla f|^2),\qquad\mm\ae,
\end{equation}
the right-hand side being defined by formula \eqref{eq:reprker} which, thanks to \eqref{eq:heatbdd} and the fact that  $|\nabla f|^2\in L^1(X)$, makes sense.
\item\underline{Lipschitz regularization}. There exists  constants $\mathcal C_2(t,x,N)$ such that for every $t>0$ and $x\in\supp(\mm)$ the map  $\supp(\mm)\ni y\mapsto\rho(t,x,y),$ is $\mathcal C_2(t,x,N)$-Lipschitz. 
\item\underline{Fisher information estimate}. It holds $\rho(t,x,y)>0$ for any $(t,x,y)\in (0,\infty)\times X^2$ and denoting by $\rho_t[x]:X\to\R^+$ the map $\rho_t[x](y):=\rho(t,x,y)$ we have $\rho_t[x]\in \s^2_{\rm loc}(X)$  for every $t\in(0,\infty)$, $x\in \supp(\mm)$ and the bound
\begin{equation}
\label{eq:fish}
t^2\int\frac{|\nabla\rho_t[x]|^2}{\rho_t[x]}\,\d\mm\leq\mathcal C_3(N)(1+t)
\end{equation}
holds for any $x\in\supp(\mm)$, $t\in(0,\infty)$ for some constant $\mathcal C_3(N)$.
\end{itemize}
\end{theorem}
\begin{proof}
The $L^2\mapsto W^{1,2}$ regularization estimates are classical. For the existence of the heat kernel see for instance \cite{AmbrosioGigliSavare11-2} and \cite{AmbrosioGigliMondinoRajala12}. The gaussian estimates are a consequence of \cite{Sturm96III} in conjunction with the results in \cite{AmbrosioGigliSavare11-2}. The Bakry-\'Emery condition has been proved in \cite{AmbrosioGigliSavare11-2} and \cite{AmbrosioGigliMondinoRajala12} (see also the original argument in \cite{Gigli-Kuwada-Ohta10}). For the Lipschitz regularity notice that from the Gaussian estimates and the doubling property we get that $\rho(t,x,\cdot)\in L^2(X)$ for every $t>0$ and $x\in\supp(\mm)$. Then use the  Bakry-\'Emery condition in conjunction with the argument given in \cite{Gigli-Kuwada-Ohta10} (see also the presentation given in  \cite{AmbrosioGigliSavare11-2}). Finally, the Fisher information estimates are a consequence of general bounds on the slope along a gradient flow satisfying the so-called EVI condition, the Gaussian estimates and the polynomial volume growth, see for instance the arguments in \cite{GMS15}.
\end{proof}
It is a standard construction within the theory of Dirichlet form to build the diffusion operator associated to the form itself. In our case the definition reads as:
\begin{definition}[Laplacian in $L^2$] Let $(X,\sfd,\mm)$ be an infinitesimally Hilbertian space. Then the  space $D(\Delta)\subset W^{1,2}(X,\sfd,\mm)$ is the space of all $f$ such that for some $h\in L^2(X,\mm)$ it holds
\[
\int gh\,\d\mm=-\int\la\nabla f,\nabla g\ra\,\d\mm,\qquad\forall g\in W^{1,2}(X,\sfd,\mm).
\]
In this case the function $h$, which is clearly uniquely determined, will by denoted by $\Delta f$ and called Laplacian of $f$.
\end{definition}
On infinitesimally Hilbertian spaces such that $\mm$ gives finite mass to bounded sets (the latter being true on any $CD(K,\infty)$ space - see \cite{Sturm06I}),  this definition is nothing but a particular case of the one of measure valued Laplacian given in Definition \ref{def:measlap}. Indeed, it is immediate to verify that 
\[
f\in D(\Delta)
\]
is equivalent to 
\[
\textrm{$f\in W^{1,2}(X,\sfd,\mm)\cap D(\bd)$ and $\bd f=h\mm$ for some $h\in L^2(X,\mm)$,}
\]
and that if these holds we also have $h=\Delta f$: one implication is obvious, and the other one follows from the approximation result in Theorem \ref{thm:stronglip}. 

It is anyway useful to  single-out the definition of $\Delta$ as it is more manageable and the standard one used in the theory of linear semigroups. In particular it allows us to use the following well known results which will be useful later on.
\begin{proposition}\label{prop:semi} Let $(X,\sfd,\mm)$ be an infinitesimally Hilbertian space. Then:
\begin{itemize}
\item[i)] For any $f\in L^2(X)$ and $t>0$ it holds $\h_t(f)\in D(\Delta)$ and
\[
\lim_{h\to 0}\frac{\h_{t+h}(f)-\h_t(f)}{h}=\Delta\h_t(f),\qquad\textrm{ in }W^{1,2}(X).
\]
If $f\in D(\Delta)$ with $\Delta f\in W^{1,2}(X)$ we can take $t=0$.
\item[ii)] For any $f\in L^2(X)$ it holds $\h_t(\Delta f)=\Delta\h_t(f)$, $\forall t>0$.
\item[iii)] Denote by $\Delta^{(n)}$ the application of $n$-times $\Delta$ and define the space $D(\Delta^{(n)})$ inductively as the space of $f$'s in $D(\Delta^{(n-1)})$ such that $\Delta^{(n-1)}f\in D(\Delta)$. Then for every $t>0$,  $n\in\N$ and $f\in L^2(X)$ we have $\h_t(f)\in D(\Delta^{(n)})$ and 
\[
\textrm{the map\qquad $L^2(X,\mm)\ni f\qquad\mapsto\qquad \Delta^{(n)}\h_t(f)\in W^{1,2}(X)$ \qquad is continuous.}
\] 
\item[iv)] For any $\eps>0$ and $f\in L^2(X)$
\[
\textrm{the map\qquad $[\eps,+\infty)\ni t\qquad\mapsto\qquad \Delta^{(n)}\h_t(f)\in W^{1,2}(X)$\qquad is Lipschitz.}
\] 
If $f\in D(\Delta^{(n)})$ with $\Delta^{(n)}f\in W^{1,2}(X)$ then we can take $\eps=0$.
\item[v)] For any $f\in L^2(X)$ it holds
\[
\lim_{t\downarrow0}\int|\nabla\h_t(f)|^2\,\d\mm=\int|\nabla f|^2\,\d\mm,
\]
where the right-hand side is intended to be $+\infty$ if $f\notin W^{1,2}(X)$.
\end{itemize}
\end{proposition}

\begin{remark}[The  Sobolev space $W^{2,2}$]{\rm On a smooth Riemannian ma\-nifold $M$ with Ricci curvature bounded from below by $K$ the Bochner identity implies
\[
\Delta\frac{|\nabla f|^2}{2}\geq \|\nabla ^2f\|^2_{\rm HS}+\la\nabla f,\nabla\Delta f\ra+K|\nabla f|^2,\qquad\forall f\in C^\infty_c(M).
\]
Hence by integration we get
\begin{equation}
\label{eq:hessian}
\int \|\nabla ^2f\|^2_{\rm HS}\,\d\mm\leq \int (\Delta f)^2-K|\nabla f|^2\,\d\mm,\qquad\forall f\in C^\infty_c(M),
\end{equation}
which shows that in this case functions in $W^{2,2}$ can be characterized as functions in $W^{1,2}$ whose Laplacian is in $L^2$.

Given that \eqref{eq:hessian} is a dimension-free inequality, its seems natural to expect that on infinitesimally Hilbertian $CD(K,\infty)$  spaces (=$RCD(K,\infty)$ spaces) the same inequality holds. It is unclear to us if this can really be done, part of the problem being  to define what the Hessian is, but we point out that Honda in \cite{Honda11} proved that some sort of second order differential structure  exists on spaces which are limits of Riemannian manifolds with Ricci curvature uniformly bounded from below, and in the recent paper \cite{Savare13} Savar\'e, generalizing some inequalities due to  Bakry,  proved promising estimates in this direction directly in the abstract case.
}\fr\end{remark}

The Gaussian estimates and the volume growth  also allow to extend the domain of the definition of the heat flow far beyond the space $L^2(X,\mm)$. We will be satisfied in considering as Domain of the Heat flow the (non maximal) space $\dom(X)=\dom(X,\sfd,\mm,\bar x)$ defined by
\[
\dom(X):=\Big\{f:X\to\R\ \textrm{ Borel : }\int |f|(x)e^{-\sfd(x,\bar x)}\,\d\mm(x)<\infty\Big\},
\]
where $\bar x\in \supp(\mm)$ is a point that we shall consider as fixed from now on. Clearly, the choice of $\bar x$ does not affect the set $\dom(X)$, but to keep $\bar x$ fixed allows to introduce the norm
\[
\|f\|_\dom:=\int |f|(x)e^{-\sfd(x,\bar x)}\,\d\mm(x),
\]
so that $(\dom(X),\|\cdot\|_\dom)=L^1(X,e^{-\sfd(\cdot,\bar x)}\mm)$ is a Banach space. We claim that  for $f\in \dom(X)$ and $t>0$ it holds $f\rho_t[x]\in L^1(X,\mm)$ for any $x\in \supp(\mm)$ and that defining $\h_t(f)$ by   the formula
\begin{equation}
\label{eq:defext}
\h_t(f)(x):=\int f(y)\rho_t[x](y)\,\d\mm(y),
\end{equation}
for some constant $\mathcal C(t,N)$  it holds
\begin{equation}
\label{eq:normheat}
\|\h_t(f)\|_\dom\leq \mathcal C(t,N)\|f\|_\dom,
\end{equation}
so that $\h_t$ maps $\dom(X)$ into $\dom(X)$.

To see this, let $f\in\dom(X)$ and notice that from the simple inequalities
\[
\begin{split}
\int |\h_t(f)|(x)e^{-\sfd(x,\bar x)}\,\d\mm(x)&\leq \iint |f|(y)\rho(t,x,y) e^{-\sfd(x,\bar x)}\,\d\mm(y)\,\d\mm(x)\\
&\leq \int |f|(y)e^{-\sfd(y,\bar x)}\int\rho(t,x,y) e^{\sfd(y, x)}\,\d\mm(x)\,\d\mm(y),
\end{split}
\]
we deduce that to prove \eqref{eq:normheat} it is sufficient to prove that
\begin{equation}
\label{eq:pernormheat}
\int \rho(t,x,y) e^{\sfd(y, x)}\,\d\mm(x)\leq \mathcal C(t,N),\qquad\forall y\in\supp(\mm),\ t>0.
\end{equation}
We have
\begin{equation}
\label{eq:dentro}
\int_{B_{\sqrt t}(y)} \rho(t,x,y) e^{\sfd(y, x)}\,\d\mm(x)\leq \mathcal C_1(N)\sup_{r\leq \sqrt t}e^{-\frac{r^2}{5t}+r}
\end{equation}
Now recall that ${\sf s}_y(r):=\lims_{\eps\downarrow0}\frac{\mm(B_{r+\eps}(y)\setminus B_r(y))}{\eps}$ and that (Proposition \ref{prop:BG}) $r\mapsto\mm(B_r(y))$ is  locally Lipschitz, thus  differentiable a.e. with derivative given by ${\sf s}_y(r)$. As a consequence we have
\begin{equation}
\label{eq:volint}
\mm(B_r(y))=\int_0^r{\sf s}_y(s)\,\d s,\qquad\forall r\geq 0,\ y\in\supp(\mm),
\end{equation}
which can be equivalently written as $\d(\sfd_y)_\sharp\mm(r)={\sf s}_y(r)\d r$, where $\sfd_y(x):= \sfd(x,y)$. By   \eqref{eq:volint} and the monotonicity property \eqref{eq:bgarea}  we have
\begin{equation}
\label{eq:pallasfera}
\mm(B_{r}(y))= \int_0^{r}{\sf s}_y(s)\,\d s=\int_0^{r}\frac{{\sf s}_y(s)}{s^{N-1}}s^{N-1}\,\d s\geq \frac{{\sf s}_y(r)}{r^{N-1}}\int_0^{r}s^{N-1}\,\d s=\frac1N{r}\,{\sf s}_y({r}),
\end{equation}
and therefore the Gaussian estimates \eqref{eq:gaussest} give
\[
\begin{split}
\int_{X\setminus B_{\sqrt t}(y)} \rho(t,x,y) e^{\sfd(y, x)}\,\d\mm(x)&\leq\frac{\mathcal C_1(N)}{\mm(B_{\sqrt t}(y))}\int_{X\setminus B_{\sqrt t}(y)}  e^{-\frac{\sfd^2(x,y)}{5t}+\sfd(x,y)}\,\d\mm(y)\\
&=\frac{\mathcal C_1(N)}{\mm(B_{\sqrt t}(y))}\int _{\sqrt t}^\infty e^{r-\frac{r^2}{5t}}{\sf s}_y(r)\,\d r\\
&=\frac{\mathcal C_1(N)}{\mm(B_{\sqrt t}(y))}\int _{\sqrt t}^\infty e^{r-\frac{r^2}{5t}}\frac{{\sf s}_y(r)}{r^{N-1}}r^{N-1}\,\d r\\
\textrm{by \eqref{eq:bgarea}}\qquad&\leq \frac{\mathcal C_1(N)}{\mm(B_{\sqrt t}(y))} \frac{{\sf s}_y(\sqrt t)}{\sqrt t^{N-1}}\int _{\sqrt t}^\infty e^{r-\frac{r^2}{5t}}r^{N-1}\,\d r\\
\textrm{by \eqref{eq:pallasfera}}\qquad&\leq \frac{N\mathcal C_1(N)}{\sqrt t^N}\int _0^\infty e^{r-\frac{r^2}{5t}}r^{N-1}\,\d r.
\end{split}
\]
Coupling this bound with \eqref{eq:dentro} we get \eqref{eq:pernormheat} and thus \eqref{eq:normheat}.

With similar means, we can obtain the bound
\begin{equation}
\label{eq:momenti}
\int \sfd^n(y,x)\rho_t[x](y)\,\d\mm(y)\leq \mathcal C(n,N) t^{n/2},\qquad\forall  x\in\supp(\mm),\ t>0.
\end{equation}
Indeed from the Gaussian bounds \eqref{eq:gaussest} we have
\[
\begin{split}
\int_{ B_{\sqrt t}(x)} \sfd^n(y,x)\rho_t[x](y)\,\d\mm(y)&\leq\frac{\mathcal C_1(N)}{\mm(B_{\sqrt t}(x))} \int_{B_{\sqrt t}(x)} \sfd^n(y,x)e^{-\frac{\sfd^2(y,x)}{5t}}\,\d\mm(y)\\
&\leq \mathcal C_1(N)\sup_{r>0}r^ne^{-\frac{r^2}{5t}}= \mathcal C_1(N)(5t)^{n/2}\sup_{r>0}r^ne^{-r^2},
\end{split}
\]
and
\[
\begin{split}
\int_{ X\setminus B_{\sqrt t}(x)} \sfd^n(y,x)\rho_t[x](y)\,\d\mm(y)&\leq\frac{\mathcal C_1(N)}{\mm(B_{\sqrt t}(x))} \int_{\sqrt t}^\infty r^ne^{-\frac{r^2}{5t}}{\sf s}_x(r)\,\d r\\
\textrm{by \eqref{eq:bgarea}}\qquad&\leq \frac{\mathcal C_1(N)}{\mm(B_{\sqrt t}(x))}\frac{{\sf s}_x(\sqrt t)}{t^{\frac{N-1}2}} \int_{\sqrt t}^\infty r^{n+N-1}e^{-\frac{r^2}{5t}}\,\d r\\
\textrm{by \eqref{eq:pallasfera}}\qquad&\leq N\mathcal C_1(N) (5t)^{\frac{n}2}\int_{r>0}r^{n+N-1}e^{-r^2}\,\d r.
\end{split}
\]
\subsection{From Sobolev to Lipschitz}\label{se:sobtolip}

On arbitrary metric measure spaces a Sobolev information on a functions may yield little to none information about its metric regularity, the standard example being a space where there are no non-constant Lipschitz curves: in this case every $L^2$ function is Sobolev with 0 minimal weak upper gradient. We single out in the following definition a basic property which allows to pass from a Sobolev information to a metric one:
\begin{definition}[Sobolev-to-Lipschitz property]\label{def:sobtolip}
Let $(X,\sfd,\mm)$ be a metric measure space.

We say that $(X,\sfd,\mm)$ has the Sobolev-to-Lipschitz property provided any $f\in W^{1,2}(X,\sfd,\mm)$ with $\weakgrad f\leq 1$ $\mm$-a.e. admits a 1-Lipschitz representative, i.e. a 1-Lipschitz map $g:X\to\R$ such that $f=g$ $\mm$-a.e..
\end{definition}

As we shall see in Proposition \ref{prop:isom}, on spaces with the Sobolev-to-Lipschitz property isometries can be recognized by means of Sobolev calculus.

Two important class of spaces have such property:
\begin{itemize}
\item $CD(K,N)$ spaces. Indeed, Rajala proved in \cite{Rajala12} the following result:
\begin{quote}
Let $(X,\sfd,\mm)$ be a $CD(K,N)$ space and  $\mu_0,\mu_1\in\probt X$ with compact support and such that $\mu_0,\mu_1\leq C\mm$ for some $C>0$. Then there exists a $W_2$-geodesic $(\mu_t)$ form $\mu_0$ to $\mu1$ such that $\mu_t\leq C'\mm$ for every $t\in[0,1]$ for some $C'>0$.
\end{quote}
With this statement at disposal, the proof follows easily. Indeed, pick  $f\in \s^2_{loc}(X)$, notice that with a truncation argument we can assume $f\in L^\infty(X)$, let $x,y\in\supp(\mm)$ and for $r>0$ define $\mu^r_0:=\mm(B_r(x))^{-1}\mm\restr{B_r(x)}$ and $\mu^r_1:=\mm(B_r(x))^{-1}\mm\restr{B_r(x)}$. Letting $\ppi^r$ being any lifting of the geodesic provided by Rajala's construction, we know that $\ppi^r$ is an optimal geodesic plan $\ppi^r$ such that $(\e_t)_\sharp\ppi^r\leq C'\mm$ for every $t\in[0,1]$ and some $C'>0$ depending on $r$. Thus $\ppi^r$ is a test plan and from $\weakgrad f\leq 1$ $\mm$-a.e. we get
\[
\begin{split}
\left|\int f\,\d\mu^r_1-\int f\,\d\mu^r_0\right|\leq \int|f(\gamma_1)-f(\gamma_0)|\,\d\ppi^r(\gamma)\leq \iint_0^1|\dot\gamma_t|\,\d t\,\d\ppi^r(\gamma)\leq W_2(\mu_0^r,\mu_1^r).
\end{split}
\]
The conclusion the follows by picking $x,y$ to be Lebesgue points for $f$ and letting $r\downarrow0$.
\item $RCD(K,\infty)$ spaces. This has been proved in  \cite{AmbrosioGigliSavare11-2} as a consequence of  the Bakry-\'Emery contraction estimates. 
\end{itemize}
Any of these two implies:
\begin{theorem}\label{thm:link2} 
Infinitesimally Hilbertian $CD(0,N)$ spaces have the Sobolev-to-Lipschitz property.
\end{theorem} 

\subsection{Strong maximum principle for the Busemann function}\label{se:maxprin}

In \cite{Bjorn-Bjorn07} it has been proved that on metric measure spaces with a doubling measure and supporting a weak-local 1-2 Poincar\'e inequality, the strong maximum principle holds for local subminimizers of the energy $\int_\Omega \weakgrad f^2\,\d\mm$. In \cite{Gigli-Mondino12} it has been shown that local subminimizers $f$ of the energy can be characterized by the inequality $\bd f\geq  0$. 

Thus taking into account  the Laplacian comparison estimate for the Busemann function recalled in Proposition \ref{prop:lapbus} (whose hypotheses are fulfilled in the infinitesimally Hilbertian case), the linearity of $\bd$ ensured by the infinitesimal Hilbertianity assumption  and in accordance with the strategy used in the smooth setting, we get:
\begin{theorem}\label{thm:busharm}
Let $(X,\sfd,\mm)$ be an infinitesimally Hilbertian $CD(0,N)$ space, $\bar\gamma:\R\to\supp(\mm)$ a line and $\b^\pm$ the Busemann functions associated to it as in \eqref{eq:busemann}.

Then $\b^++\b^-\equiv0$ on $\supp(\mm)$ and $\bd\b^+=\bd\b^-=0$.
\end{theorem}
See \cite{Gigli-Mondino12} for the details. We stress that in order to get this result it is crucial to have at disposal the measure valued Laplacian, because this is the a priori regularity of  the Laplacians of $\b^+,\b^-$.

\section{Result}
From this section on we shall always assume the following:
\begin{equation}
\label{eq:assfin}
\begin{split}
&(X,\sfd,\mm)\textrm{ is an infinitesimally Hilbertian $CD(0,N)$ space, $\bar\gamma:\R\to \supp(\mm)$ is a line}\\
&\textrm{and  $\b:=\b^+$ is the corresponding Busemann function, $\b^+$ being defined as in \eqref{eq:busemann}}.
\end{split}
\end{equation}
By Theorem \ref{thm:busharm} and the fact that $W^{1,2}(\Omega)$ is Hilbert for every $\Omega\subset X$ open, we know that the assumptions \eqref{eq:ass2} at the basis of the previous chapter are fulfilled. Thus Theorem \ref{thm:gfpresmes} holds, and $\bd\b=0$. 

\medskip

The polynomial volume growth \eqref{eq:bgvol} easily gives that $\b\in\dom(X)$, thus from $\bd\b=0$ we expect  $\h_t(\b)=\b$ to hold for every $t\geq 0$.  The next simple proposition shows that this is actually the case, the proof being based on the estimates on the heat kernel we previously recalled. Notice that in stating the result we are using the fact that formula \eqref{eq:defext} defines the value of $\h_t(\b)$ for every $x\in\supp(\mm)$, and not just $\mm$-a.e..
\begin{proposition}[Invariance of $\b$ under the heat flow]\label{prop:bstable}
Assume \eqref{eq:assfin}. Then for any $t> 0$ it holds
\[
\h_t(\b)(x)=\b(x),\qquad\forall x\in\supp(\mm).
\]
\end{proposition}
\begin{proof} The mass preservation formula \eqref{eq:mass1} yields  $\b(x)=\int\b(x)\rho_t[x](y)\,\d\mm(y)$ and thus from
\[
\begin{split}
\left|\int \b(x)-\b(y)\rho(t,x,y)\,\d\mm(y)\right|\leq \int \sfd(x,y)\rho(t,x,y)\,\d\mm(y),
\end{split}
\]
and the bound \eqref{eq:momenti}  we deduce that  $\lim_{t\to 0}\h_t(\b)(x)=\b(x)$. Hence to conclude it is sufficient to show that for any $t_1>t_0>0$ and $x_0\in \supp(\mm)$ it holds $\h_{t_1}(\b)(x_0)=\h_{t_0}(\b)(x_0).$

Fix $x_0\in\supp(\mm)$, let $R>|\b(x_0)|$ and  $\nchi_R:\supp(\mm)\to[0,1]$  a 1-Lipschitz function identically 1 on $B_R(x_0)$ and identically 0 on $B_{R+1}(x_0)$. From
\[
\begin{split}
\left|\h_t(\b\nchi_R)(x_0)-\h_t(\b)(x_0)\right|&\leq\int_X|\b|(y)(1-\nchi_R(y))\rho(t,x_0,y)\,\d\mm(y)\\
&\leq\int_{X\setminus B_R(x_0)}\big(\sfd(y,x_0)+|\b(x_0|\big)\rho(t,x_0,y)\,\d\mm(y)\\
&\leq\frac2R\int_{X}\sfd(y,x_0)^2\rho(t,x_0,y)\,\d\mm(y),
\end{split}
\]
and the moment estimates \eqref{eq:momenti} we get $\lim_{R\to+\infty}\h_t(\b\nchi_R)(x_0)=\h_t(\b)(x_0)$, $\forall t\geq 0$.

It is trivial that $\rho_{t_0/2}[x_0]\in L^2(X)$ and thus according to point $(iv)$ in Proposition \ref{prop:semi} the map $t\mapsto \h_{t}(\rho_{t_0}[x_0])\in L^2(X)$ is Lipschitz. Since $\nchi_R\b\in L^2(X)$ as well, the map $t\mapsto\h_t(\b\nchi_R)(x_0)=\int \nchi_R\b\rho_t[x_0]\,\d\mm$ is Lipschitz on $[t_0,t_1]$.  Its derivative is given by
\[
\begin{split}
\frac{\d}{\d t}\h_t(\b\nchi_R)(x_0)&=\int \b\nchi_R\,\frac{\d}{\d t}\rho_t[x_0]\,\d\mm\\
&=\int  \b\nchi_R\Delta (\rho_t[x_0])\,\d\mm\\
&=-\int\la\nabla(\b\nchi_R),\nabla(\rho_t[x_0])\ra\d\mm\\
&=\int-\la\nabla\b,\nabla(\nchi_R\rho_t[x_0])\ra+ \la\nabla \b,\nabla\nchi_R\ra\rho_t[x_0]-\b\la\nabla\nchi_R,\nabla(\rho_t[x_0])\ra\d\mm,
\end{split}
\]
having used the Leibniz rule \eqref{eq:leibniz} in the last step. Given that  $\bd \b=0$ and that $\nchi_R\rho_t[x_0]$ is Lipschitz with compact support we have $\int\la\nabla\b,\nabla(\nchi_R\rho_t[x_0])\ra\,\d\mm=0$.

The  fact that $|\nabla\nchi_R|\equiv 0$ on $B_R(x_0)$ gives
\[
\left|\int_X \la\nabla \b,\nabla\nchi_R\ra\rho_t[x_0]\,\d\mm\right|\leq \int_{X\setminus B_R(x_0)}\rho_t[x_0]\,\d\mm\leq\frac1R\int_X\sfd(\cdot,x_0)\rho_t[x_0]\,\d\mm
\]
and thus the moment estimate \eqref{eq:momenti} yields
\[
\lim_{R\to\infty}\left|\int \la\nabla \b,\nabla\nchi_R\ra\rho_t[x_0]\,\d\mm\right|\to 0,\qquad\textrm{  uniformly on }t\in[t_0,t_1].
\]
Since $R> |\b(x_0)|$ and  $|\nabla\nchi_R|\equiv 0$ on $B_R(x_0)$  we also have
\[
\begin{split}
\left|\int_X \b\la\nabla\nchi_R,\nabla(\rho_t[x_0])\ra\,\d\mm\right|&\leq\int_{X\setminus B_R(x_0)}|\b|\,|\nabla(\rho_t[x_0])|\,\d\mm\\
&\leq \frac2R\int_{X\setminus B_R(x_0)}\sfd^2(\cdot,x_0)|\nabla(\rho_t[x_0])|\,\d\mm\\
&\leq \frac2R \sqrt{\int_X\sfd^4(y,\bar x)\rho_t[x](y)\,\d\mm(y)}\,\sqrt{\int_X\frac{|\nabla\rho_t[x]|^2(y)}{\rho_t[x](y)}\,\d\mm(y)},
\end{split}
\]
thus the moment estimate \eqref{eq:momenti} and the bound \eqref{eq:fish} on the Fisher information give
\[
\lim_{R\to\infty}\left|\int \b\la\nabla\nchi_R,\nabla(\rho_t[x_0])\ra\,\d\mm\right|=0,\qquad\textrm{  uniformly on }t\in[t_0,t_1].
\]
Collecting together all these informations we obtain
\[
\begin{split}
|\h_{t_1}\b(x_0)-\h_{t_0}\b(x_0)|&=\lim_{R\to\infty}|\h_{t_1}(\nchi_R\b)(x_0)-\h_{t_0}(\nchi_R\b)(x_0)|\\
&\leq \lim_{R\to\infty}\int_{t_0}^{t_1}\left|\frac\d{\d t}\h_t(\nchi_R\b)(x_0)\right|\,\d t=0,
\end{split}
\]
and the proof is completed.
\end{proof}
The next simple Lemma extends the domain of validity of the Bakry-\'Emery condition.
\begin{proposition}[Bakry-\'Emery condition on $\dom(X)$]\label{le:BEext} Assume \eqref{eq:assfin} and let\linebreak $f\in\dom(X)\cap\s^2_{\rm loc}(X)$ be such that $f^2,|\nabla f|^2\in \dom(X)$. 

Then $\h_t(f)\in\s^2_{\rm loc}(X)$ and
\[
|\nabla (\h_t(f))|^2\leq \h_t(|\nabla f|^2),\qquad\mm\ae, \qquad\forall t\geq 0.
\]
\end{proposition}
\begin{proof}
Let $(B_n)$ be an increasing sequence of bounded sets such that $\supp(\mm)=\cup_nB_n$ and for every $n\in\N$ let $\nchi_n:\supp(\mm)\to[0,1]$ be a 1-Lipschitz function with compact support identically 1 on $B_n$.

Clearly, $f\nchi_n\in L^2(X)$ so that  inequality \eqref{eq:leibbase} and the assumption $|\nabla f|^2\in \dom(X)$ also grant  $f\nchi_n\in \s^2(X)$. Thus $f\nchi_n\in W^{1,2}(X)$ and \eqref{eq:BE} yields
\begin{equation}
\label{eq:beperlim}
|\nabla (\h_t(f\nchi_n))|^2\leq \h_t(|\nabla (f\nchi_n)|^2),\qquad\mm\ae,
\end{equation}
for any $t\geq 0$.

Again from  \eqref{eq:leibbase} we obtain $|\nabla (f\nchi_n)|^2\leq 2|\nabla f|^2+2|f|^2\in\dom(X)$, and given that trivially $|\nabla(f\nchi_n)|\to|\nabla f|$ $\mm$-a.e. as $n\to\infty$, by the dominate convergence theorem we deduce $\||\nabla (f\nchi_n)|^2-|\nabla f|^2\|_\dom\to 0$ as $n\to\infty$. Inequality \eqref{eq:normheat} then ensures that 
\begin{equation}
\label{eq:sonno}
\h_t(|\nabla(f\nchi_n)|^2)\to \h_t(|\nabla f|^2)\quad\textrm{ in $\dom(X)$ as $n\to\infty$ for any $t\geq 0$}.
\end{equation}

By construction we have $\|f\nchi_n-f\|_\dom\to 0$ as $n\to\infty$, so that \eqref{eq:normheat} yields $\|\h_t(f\nchi_n)-\h_t(f)\|_\dom\to0$ as $n\to\infty$. Thus up to pass to a subsequence, not relabeled, we can assume that $\h_t(f\nchi_n)\to\h_t(f)$ $\mm$-a.e. as $n\to\infty$. This fact, the lower semicontinuity of minimal weak upper gradients stated after Definition \ref{def:parigi}, \eqref{eq:sonno} and \eqref{eq:beperlim} give the conclusion.
\end{proof}
The last two proposition allow to write down the Euler equation for the Busemann function $\b$.
\begin{corollary}[Euler's equation for $\b$] Assume \eqref{eq:assfin}.
Then for any $f\in W^{1,2}(X)$ it holds
\begin{equation}
\label{eq:euler}
\h_t(\la\nabla\b,\nabla f\ra)=\la\nabla\b,\nabla \h_t(f)\ra,\qquad\mm\ae.
\end{equation}
Furthermore,  for  $f\in W^{1,2}(X)\cap D(\Delta)$ with $\Delta f\in W^{1,2}(X)$ and $g\in D(\Delta)$ we have
\begin{equation}
\label{eq:key}
\int \Delta g\la\nabla\b,\nabla f\ra\d\mm=\int g\la\nabla\b,\nabla\Delta f\ra \d\mm.
\end{equation}
\end{corollary}
\begin{proof} It is obvious that $\b,\b^2,|\nabla\b|^2\in \dom(X)$. By definition we have $L^1(X)\subset\dom(X)$ and the inequality 
\[
\int |f|e^{-\sfd(\cdot,\bar x)}\,\d\mm\leq \sqrt{\int f^2\,\d\mm}\sqrt{\int e^{-2\sfd(\cdot,\bar x)}\,\d\mm},
\]
and the polynomial growth rate \eqref{eq:bgvol} grant that $L^2(X)\subset \dom(X)$. Thus for $f\in W^{1,2}(X)$ we get $f,f^2,|\nabla f|^2\in \dom(X)$ as well. It is then immediate to see that for any $\eps\in\R$ we also have $(\b+\eps f),(\b+\eps f)^2,|\nabla(\b+\eps f)|^2\in\dom(X)$.

Hence we can apply  Proposition \ref{le:BEext} to the function  $\b+\eps f$ and obtain 
\begin{equation}
\label{eq:letto}
|\nabla(\h_t(\b+\eps f))|^2\leq \h_t(|\nabla(\b+\eps f)|^2).
\end{equation}
The linearity of $\h_t$, Proposition \ref{prop:bstable} and the identity $|\nabla\b|= 1$ $\mm$-a.e.  give
\[
\begin{split}
|\nabla(\h_t(\b+\eps f))|^2&=1+2\eps\la\nabla \b,\nabla \h_t(f)\ra+\eps^2|\nabla \h_t(f)|^2,\\
 \h_t(|\nabla(\b+\eps f)|^2)&=1+2\eps \h_t(\la\nabla\b,\nabla f\ra)+\eps^2\h_t(|\nabla f|^2),
\end{split}
\]
$\mm$-a.e.. Using these equalities in \eqref{eq:letto} we obtain the Euler equation written as in \eqref{eq:euler}.

To get \eqref{eq:key}, start noticing  that from $|\nabla \b|= 1$ $\mm$-a.e. we deduce that both sides of \eqref{eq:euler} are in $L^2(X)$.

Now assume that $f\in W^{1,2}(X)\cap D(\Delta)$ with $\Delta f\in W^{1,2}(X)$ and let $g\in D(\Delta)$. From  \eqref{eq:euler} we get
\begin{equation}
\label{eq:trenino2}
\int\frac{\h_t(g)-g}{t}\la\nabla\b,\nabla f\ra\,\d\mm=\int_Xg\la\nabla\b,\nabla\Big(\frac{\h_t(f)-f}t\Big)\ra\,\d\mm,\qquad\forall t>0.
\end{equation}
The assumption $g\in D(\Delta)$  grants that $\frac{\h_t(g)-g}{t}\to \Delta g$ in $L^2(X)$ as $t\downarrow0$, thus the left-hand side of \eqref{eq:trenino2} converges to the one of \eqref{eq:key} as $t\downarrow0$.

The assumptions on $f$ and point $(i)$ of Proposition \ref{prop:semi} ensure that $\frac{\h_t(f)-f}{t}$ converges to $\Delta f$ as $t\downarrow0$ in $W^{1,2}(X)$. Thus we have
\[
\begin{split}
&\left|\int_Xg\la\nabla\b,\nabla\Big(\frac{\h_t(f)-f}t\Big)\ra\,\d\mm-\int g\la\nabla \b,\nabla\Delta f\ra\,\d\mm\right|\\
&\leq \int |g|\left|\la\nabla\b,\nabla\Big(\frac{\h_t(f)-f}t\Big)\ra-\la\nabla \b,\nabla\Delta f\ra\right|\,\d\mm\leq\int |g|\left|\nabla \Big(\frac{\h_t(f)-f}t-\Delta f\Big)\right| \,\d\mm\to 0,
\end{split}
\]
and the conclusion follows.
\end{proof}
\begin{remark}[Hessian of $\b$]{\rm
\emph{Formally} we can rewrite the Euler equation \eqref{eq:key} as
\[
\Delta(\la\nabla\b,\nabla f\ra)=\la\nabla\Delta f,\nabla\b\ra,
\]
for any `smooth' $f$. This is formal because we don't really know if $\la\nabla\b,\nabla f\ra\in D(\Delta)$. Replacing $f$ with $\frac{f^2}2$ and after little algebraic manipulation based on the calculus rules recalled in Section \ref{se:infhil} we get
\begin{equation}
\label{eq:eulerfake}
\la\nabla(\la\nabla\b,\nabla f\ra),\nabla f\ra=\la\nabla \b,\nabla\tfrac{f^2}2\ra
\end{equation}
Recalling that on a smooth Riemannian manifold the formula
\[
{\rm Hess}(g)(\nabla f,\nabla f)=\la\nabla(\la\nabla g,\nabla f\ra),\nabla f\ra-\la\nabla g,\nabla\tfrac{f^2}2\ra,
\]
holds, we can interpret the formal equation \eqref{eq:eulerfake} as the fact that the Hessian of $\b$ is 0. In a smooth world, this easily implies that $\b$ is affine along geodesics. Let us show a \emph{formal} 
argument which yields the same conclusion in the non-smooth one. According to Corollary \ref{cor:pertcont} that we shall see later on, for any  $\mu_0,\mu_1\in\probt X$ with bounded support and such that $\mu_0,\mu_1\leq C\mm$ for some $C>0$, there exists a unique geodesic $(\mu_t)$ connecting them which further satisfies $\mu_t\leq C\mm$ for every $t\in[0,1]$. With an  approximation argument we see that the claim is equivalent to prove that $t\mapsto \int \b\,\d\mu_t$ 
is affine for any such geodesic. According to Proposition \ref{prop:intf} that we shall see later, the map $t\mapsto\int\b\,\d\mu_t$ is $C^1$ and its derivative is given by
\begin{equation}
\label{eq:firstb}
\frac{\d}{\d t}\int\b\,\d\mu_t=\int\la\nabla\b,\nabla Q_t(-\varphi)\ra\,\d\mu_t,\qquad\forall t\in(0,1),
\end{equation}
where $\varphi$ is any Kantorovich potential from $\mu_0$ to $\mu_1$ and $Q_t$ is the Hopf-Lax semigroup (see Definition \ref{def:hl}). In \cite{AmbrosioGigliSavare11} it has been proved that for any $f$ bounded from below the map $t\mapsto Q_t(f)$ is locally semiconcave and that the formula
\begin{equation}
\label{eq:hlhj}
\frac{\d}{\d t}Q_tf(x)+\frac{\lip(Q_tf)^2(x)}{2}=0,\qquad \textrm{ for any $t>0$ except a countable number},
\end{equation}
is valid for any $x\in X$, in line with the fact that the Hopf-Lax formula produces solutions of the Hamilton-Jacobi equation on the Euclidean space. By the well known results of Cheeger (\cite{Cheeger00}) we know that $\mm$-a.e. it holds $\lip(Q_tf)=|\nabla f|$, thus by formally applying the first order differentiation formula \eqref{eq:firstb} to the function $\la \nabla\b,\nabla Q_t(-\varphi)\ra$ and taking into account \eqref{eq:hlhj} we get
\[
\frac{\d^2}{\d t^2}\int\b\,\d\mu_t=\int\la\nabla(\la\nabla\b,\nabla (Q_t(-\varphi))\ra),\nabla Q_t(-\varphi) \ra-\la\nabla \b,\nabla\tfrac{|\nabla Q_t(-\varphi)|^2}2\ra\,\d\mu_t,
\]
and we see from  \eqref{eq:eulerfake} taking $f:=Q_t(-\varphi)$ that  the right hand side is 0.

This cannot be rigorously justified with the current technology. Yet, this is not a crucial issue, because even pretending that we know that $\b$ is affine along geodesics, the proof of the splitting would still be quite far. Indeed, one should proceed by first proving that its gradient flow preserves the distance (which, following the ideas used in the next chapter, is possible), then by proving that the quotient space is an infinitesimally Hilbertian $CD(0,N-1)$ space and that the distance splits according to `Pythagora's theorem'. In other words, one should still repeat all the arguments contained in the next chapters.

Instead, as mentioned in the introduction, we will never really need the fact that $\b$ is affine along geodesics, and use the Euler equation \eqref{eq:euler} and \eqref{eq:key} to prove that the right composition with the gradient flow of $\b$ produces isometries of $W^{1,2}(X)$ into itself.  By the general duality principle expressed in Proposition \ref{prop:isom} below, this will be enough to prove that the gradient flow is a one-parameter group of isometries.
}\fr\end{remark}
Notice that Corollary \ref{cor:derfiga} reads, in the current notation, as
\begin{equation}
\label{eq:horver}
\lim_{h\to 0}\frac{f\circ \X_{h}-f}{h}=-\la\nabla f,\nabla\b\ra,\qquad\textrm{weakly in }L^2(X,\mm)\textrm{ for any }f\in\s^2(X).
\end{equation}
We shall also need the identity
\begin{equation}
\label{eq:divb}
\int \la \nabla g,\nabla\b\ra f\,\d\mm=-\int \la \nabla f,\nabla\b\ra g\,\d\mm,\qquad\forall f,g\in W^{1,2}(X),
\end{equation}
which can be proved by first choosing sequences $(f_n),(g_n)$ of Lipschitz functions with compact support converging to $f,g$ respectively in $W^{1,2}(X)$ (Theorem \ref{thm:stronglip}), then noticing that $\bd \b=0$ yields $\int\la \nabla (f_ng_n),\nabla\b\ra\,\d\mm=0$ and thus
\begin{equation}
\label{eq:divlip}
\int \la \nabla g_n,\nabla\b\ra f_n\,\d\mm=-\int \la \nabla f_n,\nabla\b\ra g_n\,\d\mm,\qquad\forall n\in\N,
\end{equation}
then observing that $\la\nabla f_n,\nabla\b\ra\to\la\nabla f,\nabla\b\ra $ in $L^2(X)$ as $n\to\infty$ (and similarly $\la\nabla g_n,\nabla\b\ra\to\la\nabla g,\nabla\b\ra $ in $L^2(X)$) and finally passing to the limit in \eqref{eq:divlip}.

\medskip

The next proposition  provides the crucial ingredient which allows to deduce that the gradient flow of $\b$ is a one parameter family of isometries: it shows that  the right composition with $\X_t$ is a bijective isometry of $W^{1,2}(X)$ into itself. Together with the duality argument which we present in Proposition \ref{prop:isom} this will be sufficient to conclude. The proof will follow the same ideas presented in Section \ref{se:nohess}, the major difference being that we don't have any a priory Sobolev regularity for $f\circ\X_t$ and we will thus need an intermediate regularization via the heat flow.
\begin{proposition}[Right compositions with $\X_t$ give isometries of $W^{1,2}$ into itself]\label{prop:crucial} Assume \eqref{eq:assfin} and recall that the gradient flow $\X$ of $\b$ has been defined in Theorem \ref{thm:gfpresmes}. 

Then for any $t\in\R$ the map $f\mapsto f\circ\X_t$ is an isometry of $W^{1,2}(X)$ into itself, i.e. $f\in W^{1,2}(X)$ if and only if $f\circ \X_t\in W^{1,2}(X)$ and in this case $\|f\|_{W^{1,2}}=\|f\circ\X_t\|_{W^{1,2}}$.
\end{proposition}
\begin{proof} We already know that $\X_t$ is measure preserving and thus $\|f\circ\X_t\|_{L^2}=\|f\|_{L^2}$. We claim that  for $f\in W^{1,2}(X)$ it holds $f\circ\X_t\in \s^2(X)$ with $\||\nabla (f\circ\X_t)|\|_{L^2}=\||\nabla f|\|_{L^2}$ for any $t\in\R$. This will be sufficient to conclude by applying this statement also to $\X_{-t}$ and recalling the group property \eqref{eq:group}.

Fix such $f$ and recall inequality \eqref{eq:lungogf} to get
\[
\int |f\circ\X_s-f\circ \X_t|^2\,\d\mm=\int |f\circ\X_{s-t}-f|^2\,\d\mm\leq |s-t|^2\int |\nabla f|^2\,\d\mm,
\]
which shows that the map $\R\ni t\mapsto f_t:=f\circ\X_t\in L^2(X)$ is Lipschitz with Lipschitz constant bounded by $\||\nabla f|\|_{L^2}$.

Fix $\eps>0$ and notice that from inequalities \eqref{eq:l2contr} and \eqref{eq:l2s2} we deduce
\[
\|\h_\eps(f_s)-\h_\eps(f_t)\|_{W^{1,2}}\leq C(\eps)\|f_s-f_t\|_{L^2},
\]
and thus 
\begin{equation}
\label{eq:elip}
\textrm{the map \quad$\R\ni t\quad\mapsto\quad \h_\eps(f_t)\in W^{1,2}(X)$\quad is Lipschitz for every $\eps>0$,}
\end{equation}
its Lipschitz constant being bounded by $C(\eps)\||\nabla f|\|_{L^2}$. 

In particular, the map $t\mapsto \frac1{2}\int |\nabla \h_\eps(f_t)|^2\,\d\mm$ is Lipschitz; our aim is  to show that it is constant. Start from
\[
\begin{split}
\int |\nabla \h_\eps(f_{t+h})|^2-|\h_\eps(\nabla f_t)|^2\,\d\mm&= \int 2\la\nabla \h_\eps(f_t),\nabla {\h_\eps(f_{t+h}-f_t)}{}\ra+|\nabla(\h_\eps(f_{t+h}-f_t))|^2\,\d\mm,
\end{split}
\]
and notice that the simple bound  $\int |\nabla(\h_\eps(f_{t+h}-f_t))|^2\,\d\mm\leq \big(C(\eps)\||\nabla f|\|_{L^2}\big)^2|h|^2$ yields that  for any $t\in\R$ it holds
\[
\lim_{h\to 0}\int\frac{ |\nabla \h_\eps(f_{t+h})|^2-|\h_\eps(\nabla f_t)|^2}{2h}\,\d\mm=\lim_{h\to 0}\int\la \nabla \h_\eps(f_t),\nabla\frac{\h_\eps(f_{t+h})-\h_\eps(f_t)}h\ra\,\d\mm.
\]
We compute the limit in the right-hand-side of this expression.
\begin{equation}
\label{eq:cuneo2}
\begin{split}
\lim_{h\to 0}\int \la\nabla \h_\eps(f_t),\nabla\frac{\h_\eps(f_{t+h})-\h_\eps(f_t)}h\ra\,\d\mm&=-\lim_{h\to 0}\int \Delta \h_\eps( f_t) \frac{\h_\eps\big(f_{t+h}-f_t\big)}h\,\d\mm\\
&=-\lim_{h\to 0}\int \Delta \h_{2\eps}(f_t) \frac{ f_t\circ\X_{h}-f_t}h\,\d\mm\\
&=-\lim_{h\to 0}\int \frac{\big(\Delta \h_{2\eps}(f_t)\big)\circ\X_{-h}-\Delta \h_{2\eps}(f_t)}h   f_t \,\d\mm\\
&=-\int\la\nabla\big(\Delta \h_{2\eps}(f_t)\big),\nabla\b\ra\, f_t\,\d\mm,
\end{split}
\end{equation}
having used \eqref{eq:horver} in the last step. We claim that
\begin{equation}
\label{eq:claim0}
\int\la\nabla\big(\Delta \h_{2\eps}(f)\big),\nabla\b\ra\, f\,\d\mm=0,\qquad\forall f\in L^2(X).
\end{equation}
Recall that by point $(iii)$ of Proposition \ref{prop:semi} the map $L^2(X)\ni f\mapsto \Delta \h_{2\eps}(f)\in W^{1,2}(X)$ is continuous, thus from the fact that $\b$ is Lipschitz we get
\[
L^2(X)\ni f\quad\mapsto\quad \la\nabla\big(\Delta \h_{2\eps}(f)\big),\nabla\b\ra\in L^2(X)\qquad\textrm{ is continuous}.
\]
Hence it is sufficient to check \eqref{eq:claim0}  for $f\in D(\Delta)$ such that $\Delta f\in W^{1,2}(X)$, because - by regularization with the heat flow - the set of such $f$'s is dense in $L^2(X)$. With this choice of $f$, recalling \eqref{eq:divb} and the Euler equation \eqref{eq:key} we have
\begin{equation}
\label{eq:pezzoclaim1}
\begin{split}
\int\la\nabla\big(\Delta \h_{2\eps}(f)\big),\nabla\b\ra\, f\,\d\mm=-\int \Delta \h_{2\eps}(f)\la\nabla\b,\nabla f\ra\,\d\mm=-\int \h_{2\eps}(f)\la\nabla\Delta f,\nabla \b\ra\,\d\mm.
\end{split}
\end{equation}
On the other hand, the Euler equation \eqref{eq:euler} applied with $\Delta f$ in place of $f$ yields
\[
\int\la\nabla\big(\Delta \h_{2\eps}(f)\big),\nabla\b\ra\, f\,\d\mm=\int\la\nabla\big( \h_{2\eps}(\Delta f)\big),\nabla\b\ra\, f\,\d\mm=\int \h_{2\eps}(\la\nabla\Delta f,\nabla\b\ra) f\,\d\mm,
\]
which together with \eqref{eq:pezzoclaim1} yields \eqref{eq:claim0}. According to \eqref{eq:elip} and \eqref{eq:cuneo2} we thus obtained that
\[
\R\ni t\qquad\mapsto\qquad\frac12\int|\nabla\h_\eps(f_t)|^2\,\d\mm,\qquad\textrm{is constant for every }\eps>0.
\]
Letting $\eps\downarrow0 $, recalling point $(v)$ of Proposition \ref{prop:semi} and the fact that $f_0=f\in W^{1,2}(X)$ we deduce that $f_t\in W^{1,2}(X)$ for any $t\in\R$ and that
\[
\R\ni t\qquad\mapsto\qquad\frac12\int|\nabla f_t|^2\,\d\mm,\qquad\textrm{is constant,}
\]
which by the initial discussion gives  the conclusion.
\end{proof}
We now want to show that the information given by Proposition \ref{prop:crucial} is sufficient to deduce that - up to a redefinition on a negligible set - the gradient flow $\X$ of $\b$ is a one-parameter group of isometries. As discussed in the introduction, this fact has nothing to do with infinitesimal Hilbertianity and lower Ricci curvature bounds, and is rather based on the Sobolev-to-Lipschitz property expressed in Definition \ref{def:sobtolip}.

\begin{lemma}[Localization]\label{le:puntuale}
Let $(X_1,\sfd_1,\mm_1)$ and $(X_2,\sfd_2,\mm_2)$ be metric measure spaces  and $T:X_1\to X_2$ an invertible map such that $T_\sharp\mm_1=\mm_2$ and for which  $f\in W^{1,2}(X_2,\sfd_2,\mm_2)$ iimplies  $f\circ T\in W^{1,2}(X_1,\sfd_1,\mm_1)$ and in this case 
\begin{equation}
\label{eq:glob}
\int_{X_1}\weakgrad{  (f\circ T)}^2\,\d\mm_1=\int_{X_2}\weakgrad f^2\,\d\mm_2.
\end{equation}
Then for every  $f\in  W^{1,2}(X_2,\sfd_2,\mm_2)$ it holds   $f\circ T\in W^{1,2}(X_1,\sfd_1,\mm_1)$ and in this case 
\begin{equation}
\label{eq:puntu}
\weakgrad{(f\circ T)}=\weakgrad f \circ T,\qquad\mm_1\ae.
\end{equation}
\end{lemma}
\begin{proof}
By definition of $D^\pm f(\nabla g)$ (Definition \ref{def:dpmfg}), the fact that $T_\sharp\mm_1=\mm_2$ and assumption \eqref{eq:glob} yield
\begin{equation}
\label{eq:polar}
\int_{X_2}D^\pm f(\nabla g)\,\d\mm_2=\int_{X_1} D^\pm (f\circ T)(\nabla (g\circ T))\d\mm_1,\qquad\forall f,g\in W^{1,2}(X_2).
\end{equation}
Pick $f\in W^{1,2}(X_2)$ non-negative, assume for the moment that $f\in L^\infty(X_2)$ as well and let $g:X_2\to[0,\infty)$ be bounded and Lipschitz  with $\mm_2(\supp(g))<\infty$ and for $\eps>0$ put $f_\eps:=f+\eps g$. Notice that $f_\eps^2,f_\eps g\in W^{1,2}(X_2)$. Using the last in \eqref{eq:basefg}, the first Leibniz rule in \eqref{eq:leibineq} and then the second chain rule in \eqref{eq:chainsbase} we get
\[
\begin{split}
\int_{X_2}g\weakgrad{f_\eps}^2\,\d\mm_2&\geq\int_{X_2} D^+(f_\eps g) (\nabla f_\eps)-fD^+g(\nabla f_\eps)\,\d\mm_2\\
&=\int_{X_2} D^+(f_\eps g) (\nabla f_\eps)- D^+g(\nabla (\tfrac{f_\eps^2}2))\,\d\mm_2.
\end{split}
\]
Put for brevity $\tilde f_\eps:=f_\eps\circ T$ and $\tilde g:=g\circ T$ and notice that by \eqref{eq:polar} we get 
\[
\int_{X_2} D^+(f_\eps g) (\nabla f_\eps)- D^+g(\nabla (\tfrac{f_\eps^2}2))\,\d\mm_2=\int_{X_1} D^+(\tilde f_\eps \tilde g) (\nabla \tilde f_\eps)- D^+\tilde g(\nabla (\tfrac{\tilde f_\eps^2}2))\,\d\mm_1.
\]
Then continue using the first inequality in \eqref{eq:basefg}, the second Leibniz rule in  \eqref{eq:leibineq} and then again the second chain rule in \eqref{eq:chainsbase} to obtain
\[
\begin{split}
\int_{X_1} D^+(\tilde f_\eps \tilde g) (\nabla \tilde f_\eps)- D^+\tilde g(\nabla (\tfrac{\tilde f_\eps^2}2))\,\d\mm_1&\geq\int_{X_1} D^-(\tilde f_\eps \tilde g) (\nabla \tilde f_\eps)- \tilde f_\eps D^+\tilde g(\nabla \tilde f_\eps)\,\d\mm_1 \\
&\geq \int_{X_1} \tilde g D^-\tilde f_\eps  (\nabla \tilde f_\eps)+ \tilde f_\eps D^- g (\nabla \tilde f_\eps)- \tilde f_\eps D^+\tilde g(\nabla \tilde f_\eps)\,\d\mm_1  \\
&=\int_{X_1} \tilde g \weakgrad{\tilde f_\eps}^2+\tilde f_\eps \Big(D^-\tilde g (\nabla \tilde f_\eps)- D^+\tilde g(\nabla \tilde f_\eps)\Big)\,\d\mm_1,
\end{split}
\]
having used the last identity in \eqref{eq:basefg} in the last step. In summary, we proved that
\begin{equation}
\label{eq:sumlim}
\int_{X_2}g\weakgrad{f_\eps}^2\,\d\mm_2\geq \int_{X_1} \tilde g \weakgrad{\tilde f_\eps}^2+\tilde f_\eps D^-\tilde g (\nabla \tilde f_\eps)- \tilde f_\eps D^+\tilde g(\nabla \tilde f_\eps)\,\d\mm_1,\qquad\forall\eps> 0.
\end{equation}
It is immediate to check that $\weakgrad{(f-f_\eps)}\to 0$ in $L^2(X_2)$ as $\eps\downarrow 0$ which also gives, by our assumptions on $T$, that  $\weakgrad{(\tilde f-\tilde f_\eps)}\to 0$ in $L^2(X_1)$ as $\eps\downarrow 0$, where $\tilde f:=f\circ T$. In particular $\int_{X_2}g\weakgrad{f_\eps}^2\,\d\mm_2\to \int_{X_2}g\weakgrad{f }^2\,\d\mm_2$  and $\int_{X_1}\tilde g\weakgrad{\tilde f_\eps}^2\,\d\mm_1\to \int_{X_1}\tilde g\weakgrad{\tilde f}^2\,\d\mm_1$  as $\eps\downarrow 0$.

Now recall that by \eqref{eq:buonofg} we know that there exists a countable set $\mathcal N\subset\R$ such that for $\eps\in\R\setminus\mathcal N$ it holds $D^-\tilde g (\nabla \tilde f_\eps)=D^+\tilde g(\nabla \tilde f_\eps)$ $\mm_1$-a.e., hence it is sufficient to pass to the limit in \eqref{eq:sumlim} as $\eps\downarrow 0$ in $\R\setminus\mathcal N$ to obtain
\[
\int_{X_2}g\weakgrad{f}^2\,\d\mm_2\geq \int_{X_1} \tilde g \weakgrad{\tilde f}^2\,\d\mm_1.
\]
Arguing similarly starting from the inequality $g\weakgrad{f_\eps}^2\leq D^-(f_\eps g)(\nabla f_\eps)-f_\eps D^-g(\nabla f_\eps)$ we get $\int_{X_2}g\weakgrad{f}^2\,\d\mm_2\leq \int_{X_1} \tilde g \weakgrad{\tilde f}^2\,\d\mm_1$ and thus
\[
\int_{X_2}g\weakgrad{f}^2\,\d\mm_2=\int_{X_1} \tilde g \weakgrad{\tilde f}^2\,\d\mm_1.
\]
Observing that $\int_{X_2}g\weakgrad{f}^2\,\d\mm_2=\int_{X_1}\tilde g\weakgrad{f}^2\circ T\,\d\mm_1$, from the arbitrariness of $g$ we deduce that \eqref{eq:puntu} holds for non-negative $f\in W^{1,2}\cap L^\infty(X_2)$. It is now obvious that the restriction to non-negative functions can be dropped. The general case then follows via a truncation argument using the local nature of the thesis.
\end{proof}

\begin{remark}\label{re:chissa}{\rm
It is natural to ask whether substituting the assumption \eqref{eq:glob} with the weaker
\[
\int_{X_1}\weakgrad{(f\circ T)}^2\,\d\mm_1\leq \int_{X_2}\weakgrad{f}^2\,\d\mm_2,
\]
one could deduce
\[
\weakgrad{(f\circ T)}\leq \weakgrad f\circ T,\qquad\mm_1\ae.
\]
This is indeed the case on the smooth Riemannian/Finslerian case. It is unclear to us if the same holds in the abstract setting.
}\fr\end{remark}
\begin{lemma}[Contractions by local duality]\label{le:contrdual}
Let $(X_1,\sfd_1,\mm_1)$, and  $(X_2,\sfd_2,\mm_2)$ be two metric measure spaces with the Sobolev-to-Lipschitz property (Definition \ref{def:sobtolip})  and   $T:X_1\to X_2$ a  Borel map such that $T_\sharp\mm_1\leq C\mm_2$ for some $C>0$. Assume also that $\mm_2$ gives finite mass to bounded sets. Then the following are equivalent
\begin{itemize}
\item[i)] $T$ is $\mm_1$-a.e. equivalent to a 1-Lipschitz map  from $(\supp(\mm_1),\sfd_1)$ to $(\supp(\mm_2),\sfd_2)$, i.e. there exists a 1-Lipschitz map $\tilde T$ from $(\supp(\mm_1),\sfd_1)$ to $(\supp(\mm_2),\sfd_2)$ such that $\tilde T=T$ $\mm_1$-a.e..
\item[ii)] For any $f\in W^{1,2}(X_2,\sfd_2,\mm_2)$ it holds $f\circ T\in  W^{1,2}(X_1,\sfd_1,\mm_1)$ with 
\[
\weakgrad{(f\circ T)}\leq \weakgrad{ f}\circ T,\qquad \mm_1\ae.
\]
\end{itemize} 
\end{lemma}
\begin{proof}$\ $\\
\noindent{${\mathbf{ (i)\Rightarrow (ii)}}$} Obvious.

\noindent{${\mathbf{ (ii)\Rightarrow (i)}}$} Let   $\{y_n\}_{n\in\N}\subset X_2$ be a countable dense set and for $k,n\in\N$ define $f_{k,n}:X_2\to\R$ by 
\[
f_{k,n}:=\max\{0,\min\{\sfd_2(\cdot,y_n),k-\sfd_2(\cdot,y_n)\}\}. 
\]
Since  $f_{n}:X_2\to\R$ is 1-Lipschitz with bounded support, from the assumption that $\mm_2$ gives finite mass to bounded sets we deduce $f_{k,n}\in W^{1,2}(X_2)$. It is also clear that, being 1-Lipschitz, we also have $\weakgrad {f_{k,n}}\leq 1$ $\mm_2$-a.e..

Using our  assumption  we deduce that $f_{k,n}\circ T$ is in $W^{1,2}(X_1)$ and $\weakgrad{ (f_{k,n}\circ T)}\leq 1$ $\mm_1$-a.e.. Now we use the Sobolev-to-Lipschitz property  to deduce that there exists an $\mm_1$-negligible Borel set $\mathcal N_{k,n}$ such that the restriction of  $f_{k,n}\circ T$ to $X_1\setminus\mathcal N_{n}$ is 1-Lipschitz.

Let $\mathcal N:=\cup_{k,n}\mathcal N_{k,n}$ so that $\mathcal N$ is Borel and  $\mm_1$-negligible and observe that the inequality
\[
\sfd_1(x,y)\geq \sup_{k,n\in\N}|f_{k,n}(T(x))-f_{k,n}(T(y))|=\sfd_{2}(T(x),T(y)),\qquad\forall x,y\in X_1\setminus\mathcal N,
\] X
grants that the restriction of $T$ to $X_1\setminus \mathcal N$ is 1-Lipschitz to get the conclusion.
\end{proof}
\begin{proposition}[Isomorphisms via duality with  Sobolev norms]\label{prop:isom}
Let $(X_1,\sfd_1,\mm_1)$\linebreak and $(X_2,\sfd_2,\mm_2)$ be two metric measure spaces with the Sobolev-to-Lipschitz property and let $T:X_1\to X_2$ be a Borel map. Assume that both $\mm_1$ and $\mm_2$ give finite mass to bounded sets. Then the following are equivalent.
\begin{itemize}
\item[i)] Up to a modification on a $\mm_1$-negligible set, $T$ is an isomorphism of the metric measure spaces, i.e.  $T_\sharp\mm_1=\mm_2$ and $\sfd_2(T(x),T(y))=\sfd_1(x,y)$ for any $x,y\in \supp(\mm_1)$.
\item[ii)] The following two are true.
\begin{itemize}
\item[ii-a)] There exist a Borel $\mm_1$-negligible set $\mathcal N\subset X_1$ and a Borel map $S:X_2\to X_1$ such that $S(T(x))=x$, $\forall x\in X_1\setminus \mathcal N$.
\item[ii-b)]The right composition with $T$ produces a bijective isometry of $W^{1,2}(X_2,\sfd_2,\mm_2)$ in $W^{1,2}(X_1,\sfd_1,\mm_1)$, i.e. $f\in W^{1,2}(X_2,\sfd_2,\mm_2)$ if and only if $f\circ T\in W^{1,2}(X_1,\sfd_1,\mm_1)$ and in this case $\|f\|_{W^{1,2}(X_2)}=\|f\circ T\|_{W^{1,2}(X_1)}$.
\end{itemize}
\end{itemize}
\end{proposition}
\begin{proof}

\noindent{${\mathbf{ (i)\Rightarrow (ii)}}$} Obvious.

\noindent{${\mathbf{ (ii)\Rightarrow (i)}}$} Pick $\bar x\in\supp(\mm_2)$ and for each  $r>0$ consider the function $\nchi_r:X_2\to[0,1]$ defined by $\nchi_r(x):=\max\{0,\min\{1,2-\sfd(x,\bar x)/r\}\}$, so that $\nchi_r\in W^{1,2}(X_2,\sfd_2,\mm_2)$ and $\weakgrad{\nchi_r}=0$ $\mm_2$-a.e. on $B_r(\bar x)$. 

For $f\in W^{1,2}(X_2)$ with $\supp(f)\subset B_r(\bar x)$ the locality property \eqref{eq:localgrad} grants that\linebreak $\weakgrad{(\nchi_r-f)}=\weakgrad f$ $\mm_2$-a.e. on $B_r(\bar x)$ and $\weakgrad{(\nchi_r-f)}=\weakgrad {\nchi_r}$ $\mm_2$-a.e. on $X_2\setminus B_r(\bar x)$. Hence by direct computation we get
\[
\|f\|_{W^{1,2}(X_2)}^2-\|\nchi_r-f\|_{W^{1,2}(X_2)}^2-\|\nchi_r\|_{W^{1,2}(X_2)}^2=2\int_{X_2}f\nchi_r\,\d\mm_2=2\int_{X_2}f\,\d\mm_2
\]
Taking into account our assumption $(ii$-$b)$ we deduce
\[
\int_{X_2}f\,\d\mm_2=\int_{X_1}f\circ T\,\d\mm_1,\qquad\forall f\in W^{1,2}(X_2)\textrm{ such that  } \supp(f)\subset B_r(\bar x),
\]
so that taking into account the  arbitrariness of $r>0$  we deduce  $T_\sharp\mm_1=\mm_2$. 

In particular, the right composition with $T$ provides an isometry of $L^2(X_2)$ into $L^2(X_1)$ and thus by the assumption $(ii$-$b)$ we get that $f\in W^{1,2}(X_2)$ if and only if $f\circ T\in W^{1,2}(X_1)$ and in this case it holds $\int_{X_1}\weakgrad{(f\circ T)}^2\,\d\mm_1=\int_{X_2}\weakgrad{f}^2\,\d\mm_2$.  Applying first Lemma \ref{le:puntuale} and then Lemma \ref{le:contrdual} we get the existence of a Borel $\mm_1$-negligible set $\mathcal N_1\subset X_1$ such that the restriction of $T$ to $X_1\setminus \mathcal N_1$ is 1-Lipschitz.

Now we claim that $S_\sharp\mm_2=\mm_1$. To see this, observe that for $x\in X_1\setminus\mathcal N$ the identity $S(T(x))=x$ yields $\{x\}\subset T^{-1}(S^{-1}(x))$ and therefore $E\subset T^{-1}(S^{-1}(E))$ for every Borel $E\subset X_1$ with $E\cap \mathcal N=\emptyset$, which gives
\[
\mm_1(E)=\mm_1(E\setminus\mathcal N)\leq \mm_1(T^{-1}(S^{-1}(E\setminus \mathcal N)))=\mm_2(S^{-1}(E\setminus \mathcal N))\leq \mm_2(S^{-1}(E)),
\]
for every Borel set $E\subset X_1$. Similarly, if $\tilde x\in T^{-1}(S^{-1}(x))\setminus\mathcal N$ we have both $S(T(\tilde x))=x$ (because $\tilde x\in T^{-1}(S^{-1}(x))$) and $S(T(\tilde x))=\tilde x$ (because $\tilde x\notin\mathcal N$). Hence $\{x\}\supset T^{-1}(S^{-1}(x))\setminus \mathcal N$ and thus $E \supset T^{-1}(S^{-1}(x))\setminus \mathcal N$ for every Borel set $E\subset X_1$, which yields
\[
\mm_1(E)\geq \mm_1( T^{-1}(S^{-1}(E))\setminus \mathcal N)=\mm_1( T^{-1}(S^{-1}(E)))=\mm_2(S^{-1}(E)),
\]
for every Borel set $E\subset X_1$. Thus $S_\sharp\mm_2=\mm_1$ as claimed. In particular, the Borel set $F\subset X_2$\linebreak of $x$'s such that $T(S(x))\neq x$ is $\mm_2$-negligible (because if $x\in T^{-1}(F)$ it holds $T(S(T(x)))\neq T(x)$ and thus $S(T(x))\neq x$, so that $T^{-1}(F)\subset \mathcal N$). Therefore for any Borel function $g:X_2\to\R$ it holds $g\circ T\circ S=g$ $\mm_2$-a.e. and assumption $(ii$-$b)$ yields that the right composition with $S$ produces an isometry of $W^{1,2}(X_1)$ with $W^{1,2}(X_2)$. Arguing as before we therefore deduce that there exists some Borel $\mm_2$-negligible set $\mathcal N_2$ such that the restriction of $S$ to $X_2\setminus\mathcal N_2$ is 1-Lipschitz.

Conclude observing that $\mathcal N_1\cup T^{-1}(\mathcal N_2)\subset X_1$ is Borel and $\mm_1$-negligible and that for $x,y\in X_1\setminus (\mathcal N_1\cup T^{-1}(\mathcal N_2))$ it holds $\sfd_1(x,y)=\sfd_2(T(x),T(y))$.
\end{proof}
\begin{remark}{\rm
The assumption about the left invertibility of $T$ cannot be dropped, as shown by the example where $(X_1,\sfd_1)$ is made by two distant copies of $(X_2,\sfd_2)$ each one carrying the measure $\frac12\mm_2$.
}\fr\end{remark}
\begin{remark}{\rm
It is unclear to us if the assumption `the measures give finite mass to bounded sets' can be dropped or not.
}\fr\end{remark}

\begin{remark}{\rm Proposition \ref{prop:isom} can be interpreted in terms of category theory by saying that: in the category of metric measure spaces with the Sobolev-to-Lipschitz property and with measures giving finite mass to bounded sets, a reasonable choice of morphisms is given by
\[
T:X\to Y\textrm{ is a morphism if it is Borel and }\|f\circ T\|_{W^{1,2}(X)}\leq \|f\|_{W^{1,2}(Y)},\ \forall f:Y\to\R\ \textrm{ Borel},
\]
the role of Proposition \ref{prop:isom} being to tell that   $X$ and $Y$ are isomorphic if and only if there are morphisms $T:X\to Y$ and $S:Y\to X$ with $S\circ T=\Id_X$ and $T\circ S=\Id_Y$. 

This is in analogy with the fact that in the category {\bf Met} of metric spaces the natural choice of morphisms is given by 1-Lipschitz maps (see also the notion of enriched category  \cite{Lawvere74}), or equivalently
\[
T:X\to Y\textrm{ is a morphism if } \Lip_X(f\circ T)\leq \Lip_Y(f),\quad\forall f:Y\to \R.
\]
}\fr\end{remark}
\begin{remark}{\rm
There is nothing special about the Sobolev exponent 2 in Proposition \ref{prop:isom}. Similar results hold for maps preserving the $W^{1,p}$-norm, provided the appropriate reformulation of the Sobolev-to-Lipschitz property is considered.
}\fr\end{remark}

We are now ready to prove the main result of the chapter:
\begin{theorem}[The gradient flow of $\b$ preserves the distance]\label{thm:gfpresdist} Assume \eqref{eq:assfin} and recall that the map $\X$ is defined in Theorem \ref{thm:gfpresmes}. Then:
\begin{itemize}
\item[i)] There exists a unique continuous map $\bar\X:\supp(\mm)\times\R\to\supp(\mm)$ such that $\bar\X=\X$ $\mm\times\mathcal L^1$-a.e..
\item[ii)] For every $t\in\R$ the map $\bar\X_t:(\supp(\mm),\sfd)\to(\supp(\mm),\sfd)$ is an isometry.
\item[iii)] It holds $\bar\X_t(\bar \X_s(x))=\bar\X_{t+s}(x)$, for any $x\in\supp(\mm)$ and $t,s\in\R.$
\end{itemize}
\end{theorem}
\begin{proof}$\ $\\
\noindent{${\mathbf{(i),(ii)}}$} Uniqueness is obvious. By Proposition \ref{prop:crucial} we know that for every $t\in\R$ the right composition with $\X_t$ produces an isometry of $W^{1,2}(X)$ into itself. Apply  Proposition \ref{prop:isom} to get the existence of an isometry $\bar\X_t$ of $(\supp(\mm),\sfd)$ into itself $\mm$-a.e. coinciding with $\X_t$. Identity \eqref{eq:distfl} then yields $\sfd(\bar\X_t(x),\bar\X_s(x))=|t-s|$ for every $x\in\supp(\mm)$ and $t,s\in\R$, which gives the continuity of $\bar\X$ jointly in $t,x$.

\noindent{${\mathbf{(iii)}}$} Direct consequence of the group property \eqref{eq:group}, the measure preservation property \eqref{eq:presmeas} and what we just proved.
\end{proof}

\chapter{The quotient space isometrically embeds into the original one}\label{se:quot}
\section{Preliminary notions}
\subsection{Evolution of Kantorovich potentials along geodesics}
We briefly recall how Kantorovich potentials evolve along a $W_2$-geodesic.
\begin{definition}[Hopf-Lax formula]\label{def:hl}
Let $(X,\sfd)$ be a metric space and $f:X\to\R\cup\{\pm\infty\}$ a function. For $t>0$ define the function $Q_tf:X\to\R\cup\{\pm\infty\}$ as
\[
Q_tf(x):=\inf_{y\in X}f(y)+\frac{\sfd^2(x,y)}{2t}.
\]
Put $Q_0f:=f$.
\end{definition}
We recall the following simple continuity result (see e.g. \cite{AmbrosioGigliSavare11} for a proof)
\begin{equation}
\label{eq:basehl2}
 f\in C(X)\qquad\Rightarrow\qquad Q_{s}f(x)\to Q_tf(x)  \textrm{ as $s\to t$ for every  }x\in X,\ t\geq 0,
\end{equation}
and the Lipschitz continuity estimate
\begin{equation}
\label{eq:hllip}
\Lip(Q_tf)\leq 2\sqrt{\frac{\sup f-\inf f}{t}},
\end{equation}
which directly follows from the definition.

The next proposition shows in what sense the evolution of Kantorovich potentials is driven by the Hopf-Lax formula, see e.g. Theorem 7.36 in \cite{Villani09} or Theorem 2.18 in \cite{AmbrosioGigli11} for a proof.
\begin{proposition}\label{prop:kantev}
Let $(X,\sfd)$ be a metric space, $(\mu_t)\subset \probt X$ a geodesic and $\varphi: X\to\R\cup\{-\infty\}$ a Kantorovich potential relative to it.

Then for every $t\in [0,1]$:
\begin{itemize} 
\item the function $tQ_t(-\varphi)$ is a Kantorovich potential from $\mu_t$ to $\mu_0$,
\item  the function $(1-t)Q_{1-t}(-\varphi^c)$ is a Kantorovich potential from $\mu_t$ to $\mu_1$.
\end{itemize}
Furthermore, for every $t\in[0,1]$ it holds
\begin{equation}
\label{eq:intpot}
\begin{split}
Q_t(-\varphi)+Q_{1-t}(-\varphi^c)&\geq 0,\qquad\textrm{ everywhere},\\
Q_t(-\varphi)+Q_{1-t}(-\varphi^c)&=0,\qquad\textrm{on } \supp(\mu_t),
\end{split}
\end{equation}
and for $t\in(0,1)$ the functions $Q_t(-\varphi)$ and $Q_{1-t}(-\varphi^c)$ are Lipschitz on bounded sets.
\end{proposition}
The aforementioned references do not mention the Lipschitz continuity of $Q_t(-\varphi),Q_{1-t}(-\varphi^c)$ but this can be easily deduced observing that:
\begin{itemize}
\item[-] they are both bounded from above on bounded sets by definition,
\item[-] the first inequality in \eqref{eq:intpot} ensures that they are also bounded from below on bounded sets,
\item[-] a $c$-concave function which is bounded on some open set $\Omega$ is also Lipschitz on any set $C\subset \Omega$ with $\sfd(C,X\setminus\Omega)>0$ (see for instance the argument in \cite{FigalliGigli11}).
\end{itemize}

\subsection{Metric Brenier's theorem}\label{se:metrbren}
In \cite{AmbrosioGigliSavare11} a  metric-measure theoretic version of Brenier's theorem has been proved which links the minimal weak upper gradient of Kantorovich potentials to the $W_2$-distance. The version we give is weaker than the one proved in \cite{AmbrosioGigliSavare11}, but sufficient for our purposes.
\begin{theorem}\label{thm:brenmetr}
Let $(X,\sfd,\mm)$ be a metric measure space, $(\mu_t)\subset\probt X$ a geodesic such that  for some $C,T>0$ it holds $\mu_t\leq C\mm$ for every $t\in[0,T]$ and $\varphi$ a Kantorovich potential inducing it. Assume that $\varphi$ is locally Lipschitz.

Then 
\[
\int \weakgrad \varphi^2\,\d\mu_0=W_2^2(\mu_0,\mu_1).
\]
Assume furthermore that  for some open set $\Omega$ it holds $\varphi\in\s^2(\Omega)$ and $\supp(\mu_t)\subset \Omega$ for every $t\in[0,T]$. Then every lifting $\ppi\in\prob{C([0,1],X)}$ of $(\mu_t)$ represents the gradient of $-\varphi$ in $\Omega$ in the sense of Definition \ref{def:planrepgrad}.
\end{theorem}

\subsection{Optimal maps}

Given a metric measure space $(X,\sfd,\mm)$, the relative entropy functional ${\rm Ent}_\mm:\prob X\to\R\cup\{+\infty\}$ is defined as
\[
{\rm Ent}_\mm(\mu):=\left\{\begin{array}{ll}
\displaystyle{\int\rho\log\rho\,\d\mm},&\qquad\textrm{ if $\mu=\rho\mm$ and $(\rho\log\rho)^-\in L^1(X,\mm)$},\\
+\infty,&\qquad\textrm{ otherwise}.
\end{array}
\right.
\]
We shall denote by $D({\rm Ent}_\mm)\subset\probt X$ the set of those $\mu\in\probt X$ such that ${\rm Ent}_\mm(\mu)<\infty$.

We recall the definition of $CD(K,\infty)$ and $RCD(K,\infty)$ spaces (\cite{Lott-Villani09}, \cite{Sturm06I}, \cite{AmbrosioGigliSavare11-2}, \cite{AmbrosioGigliMondinoRajala12}):
\begin{definition}[$CD(K,\infty)$ and $RCD(K,\infty)$]\label{def:cdinfty} A metric measure space $(X,\sfd,\mm)$ is a\linebreak $CD(K,\infty)$ space provided for any $\mu,\nu\in D({\rm Ent}_\mm)$ there exists $\ppi\in\gopt(\mu,\nu)$ such that
\[
{\rm Ent}_\mm((\e_t)_\sharp\ppi)\leq(1-t){\rm Ent}_\mm(\mu)+t{\rm Ent}_\mm(\nu)-\frac K2W_2^2(\mu,\nu),\qquad\forall t\in[0,1].
\]
A $CD(K,\infty)$ space which is also infinitesimally Hilbertian is called $RCD(K,\infty)$ space.
\end{definition}

In the recent paper \cite{RajalaSturm12} (see also \cite{Gigli12a}), the following result has been proved:
\begin{theorem}[Optimal maps in $RCD(K,\infty)$ spaces]\label{thm:optmap}
Let $(X,\sfd,\mm)$ be an  $RCD(K,\infty)$ space and $\mu,\nu\in\probt X$ two measures absolutely continuous w.r.t. $\mm$. 

Then there exists a unique $\ppi\in\gopt(\mu,\nu)$ and this plan is induced by a map, i.e. there exists a Borel map $T:X\to\geo(X)$ such that $\ppi=T_\sharp\mu$.
\end{theorem}
Notice that the uniqueness result is expressed at the level of geodesics, which in particular means that for $\ppi$-a.e. $\gamma$ the geodesic connecting $\gamma_0$ and $\gamma_1$ is unique. In this sense, also in this abstract setting we recover the fact that `optimal maps almost never hits the cut locus', a well known property of optimal transport   in the smooth framework of Riemannian manifolds.

This result can certainly be seen as a generalization of the well-known Brenier-McCann theorem about optimal maps on Riemannian manifolds. However, the strategy of the proof is very different from the classical one: in order to prove Theorem \ref{thm:optmap}, neither the dual formulation of the transport problem, nor Kantorovich potentials are used, not even implicitly. In particular, this result has no relation with the metric Brenier theorem recalled before.

The full proof of Theorem \ref{thm:optmap} is spread around various recent papers, we recall which are the key steps leading to the result:
\begin{itemize}
\item[i)] In \cite{AmbrosioGigliSavare11-2} (see also \cite{AmbrosioGigliMondinoRajala12}) it is proved that on $RCD(K,\infty)$ spaces the relative entropy admits gradient flows in a sense stronger than the one given in Definition \ref{def:gf}: the so-called $K$-Evolution-Variational-Inequality formulation of gradient flows  (see also Appendix \ref{app:infhil}). 
\item[ii)] A general result in  \cite{DaneriSavare08} tells  that if a functional has gradient flows in the $K$-EVI sense, then it is $K$-convex along \emph{all} geodesics. Thus in particular this is the case for the relative entropy on $RCD(K,\infty)$ spaces.
\item[iii)] in \cite{RajalaSturm12} it is proved that if the relative entropy is $K$-convex along any $W_2$-geodesic, then given absolutely continuous measures $\mu,\nu$ every optimal geodesic plan $\ppi\in\gopt(\mu,\nu)$  must be concentrated on a set of non-branching geodesics
\item[iv)] In \cite{Gigli12a} it has been shown  that on non-branching $CD(K,\infty)$ spaces the same conclusions of Theorem \ref{thm:optmap} hold provided one assumes that both $\mu$ and $\nu$ have finite entropy. In \cite{RajalaSturm12} the authors observed that their result mentioned in $(iii)$ above is sufficient for the  argument in \cite{Gigli12a} to work and  - via a localization procedure  - that the hypothesis  about finiteness of the entropy can be weakened into absolute continuity of the measures, thus leading to Theorem \ref{thm:optmap}.
\end{itemize}
One of the effects of the uniqueness part of Theorem \ref{thm:optmap} is that optimal geodesic plans between absolutely continuous measures are concentrated on a set of non-branching geodesics (as said, this is in fact the ingredient of the proof produced in \cite{RajalaSturm12}). This property allows for a localization of the $CD(K,\infty)$ condition similar to the one available on non-branching metric spaces which we present in the following corollary. Once one has at disposal Theorem \ref{thm:optmap}, the proof follows standard means (see for instance the proof of Theorem 30.32 in \cite{Villani09}) and therefore we omit it.

\begin{corollary}\label{cor:cdpunt}
Let $(X,\sfd,\mm)$ be an $RCD(K,\infty)$ space and $(\mu_t)\subset\probt X$ a geodesic such that $\mu_0,\mu_1\ll\mm$. Then $\mu_t\ll\mm$ for every $t\in[0,1]$, say $\mu_t=\rho_t\mm$, and for any $0\leq t\leq r\leq s\leq 1$ it holds
\begin{equation}
\label{eq:pointcd}
\log(\rho_{r}(\gamma_{r}))\leq \frac{s-r}{s-t}\log(\rho_t(\gamma_t))+\frac{r-t}{s-t}\log(\rho_s(\gamma_s))-\frac K2\frac{(r-t)(s-r)}{(s-t)^2}\sfd^2(\gamma_t,\gamma_s),\qquad\ppi\ae \ \gamma,
\end{equation}
where $\ppi\in\gopt(\mu_0,\mu_1)$ is the unique optimal plan given by Theorem \ref{thm:optmap}. 
\end{corollary}

A useful consequence of inequality \eqref{eq:pointcd} is the following $\mm$-a.e. convergence result for densities of measures along a geodesic:
\begin{corollary}\label{cor:pertcont}
Let $(X,\sfd,\mm)$ be an $RCD(K,\infty)$ space and $(\mu_t)\subset\probt X$ a geodesic such that $\mu_0,\mu_1$ are absolutely continuous w.r.t. $\mm$ with bounded density and bounded support.

Then for some constant $C$ it holds $\mu_t\leq C\mm$ for any $t\in[0,1]$ and denoting by $\rho_t$ the density of $\mu_t$ the following holds: for any $t\in[0,1]$ and any sequence $(t_n)\subset [0,1]$ converging to $t$ there exists a subsequence $(t_{n_k})$ such that
\[
\rho_{t_{n_k}}\to\rho_t,\qquad \mm\ae\ {\rm{ as }}\ k\to\infty.
\]
\end{corollary}
\begin{proof}
The fact that $\mu_t\leq C\mm$ for some $C$ follows directly from \eqref{eq:pointcd} and the fact that $\rho_0,\rho_1$ are bounded and with bounded support.

For the second part of the statement it is sufficient to prove that $\rho_s\to\rho_t$ strongly in $L^p(X,\mm)$ for every $p\in[1,\infty)$, which follows by general arguments involving Young's measures. Indeed, the fact that the $\rho_t$'s are uniformly bounded and the weak convergence of $\mu_s$ to $\mu_t$ as $s\to t$ in duality with continuous and bounded functions (because of $W_2$-convergence) yield that $\rho_s\weakto\rho_t$ in $L^p(X,\mm)$ as $s\to t$. Hence to conclude it is sufficient to prove convergence of the $L^p$-norms.

Integrating \eqref{eq:pointcd} w.r.t. $\ppi$ we deduce that the map $t\mapsto{\rm Ent}_\mm(\mu_t)$ is $K$-convex, and given that it is finite at $t=0,1$ it is continuous, i.e.
\begin{equation}
\label{eq:content}
{\rm Ent}_\mm(\mu_{s})\to{\rm Ent}_\mm(\mu_{t}),\qquad\textrm{ as $s\to t$}.
\end{equation}
Let $K$ be a bounded set containing $\supp(\rho_t)$ for every $t\in[0,1]$, so that $\mm(K)<\infty$ and define $\nu_t:=(\Id,\rho_t)_\sharp(\mm\restr K)$, so that $\{\nu_t\}$ is a family of measures in $X\times[0,C]$ with uniformly bounded mass. The tightness of $\{\mu_t\}_{t\in[0,1]}$ easily yields the one of $\{\nu_t\}_{t\in[0,1]}$, hence for any sequence $(s_n)$ converging to some $t\in[0,1]$ there is a subsequence, not relabeled, such that $(\nu_{s_n})$ converges to some measure $\nu$ in duality with $C_b(X\times[0,C])$. It is obvious that $\pi^X_\sharp\nu=\mm$ and choosing test functions in $C_b(X\times\R)$ of the form $\varphi(x)z$ for  $\varphi\in C_b(X)$ we see that $\int z\,\d\nu_x(z)=\rho_t(x)$ for $\mm$-a.e. $x$, where $\{\nu_x\}$ is the disintegration of $\nu$ w.r.t. the projection on $X$. Consider the function $\Psi:X\times[0,C]\to\R$ given by $\Psi(x,z):=z\log z$. Clearly $\Psi\in C_b(X\times[0,C])$ and therefore
\[
\begin{split}
{\rm Ent}_\mm(\mu_{s_n})=\int \Psi\,\d\nu_{s_n}\to\int\Psi\,\d\nu&=\iint z\log z\,\d\nu_x(z)\,\d\mm(x)\\
&\geq \int \left(\int z\,\d\nu_x(z)\right)\log\left(\int z\,\d\nu_x(z)\right)\,\d\mm(x),
\end{split}
\]
having used Jansen's inequality. Recalling \eqref{eq:content}, that $\int z\,\d\nu_x(z)=\rho_t(x)$ for $\mm$-a.e. $x$ and that $z\log z$ is strictly convex, from the equality case of Jensen's inequality we deduce that $\nu_x(z)=\delta_{\rho_t(x)}$ for $\mm$-a.e. $x$. Given that the result does not depend on the particular subsequence chosen, we proved that $(\Id,\rho_s)_\sharp(\mm\restr K)\to (\Id,\rho_t)_\sharp(\mm\restr K)$ as $s\to t$ in duality with $C_b(X\times[0,C])$. Considering now test functions of the form $(x,z)\mapsto |z|^p$ we get the desired continuity of the $L^p$-norms and the conclusion.
\end{proof}

\section{Result}
It is obvious that on a smooth Riemannian manifold $M$, given a geodesic  $(\mu_t)\subset\probt M$ the Kantorovich potential $\varphi$ inducing it is not unique in general: one can add an arbitrary constant to it and, more generally, slightly different constants on different connected components of $\mu_0$ provided an appropriate rearrangement which does not destroy $c$-concavity exists. Yet, Brenier-McCann's theorem ensures that if $\mu_0$ is absolutely continuous w.r.t. the volume measure, then the gradient $\nabla\varphi$ of the Kantorovich potential $\varphi$ is uniquely determined $\mu_0$-a.e., because the only optimal transport map from $\mu_0$ to $\mu_t$ is given by the formula $x\mapsto\exp_x(-t\nabla\varphi(x))$.

The next lemma is an analogous of this uniqueness result valid on general infinitesimally Hilbertian spaces.

\begin{lemma}\label{le:gradkp} Let $(\tilde X,\tilde\sfd,\tilde\mm)$ be an infinitesimally Hilbertian space, $(\mu_t)\subset \probt X$ a geodesic and $\varphi_1,\varphi_2$ two Kantorovich potentials inducing it. Assume that $\varphi_1,\varphi_2$ are locally Lipschitz and that for some $T>0$ it holds $\mu_t\leq C\mm$ for every $t\in[0,T]$ and some $C>0$.

Then
\[
|\nabla(\varphi_1-\varphi_2)|=0,\qquad\mu_0\ae.
\]
\end{lemma}
\begin{proof}
Notice that $|\nabla(\varphi_1-\varphi_2)|$ is well defined $\mm$-a.e., thus since $\mu_0\ll\mm$ the statement makes sense. Assume for a moment that there exists an open set $\Omega\subset X$ such that $\varphi_1,\varphi_2\in\s^2(\Omega)$ and $\supp(\mu_t)\subset\Omega$ for every $t\in[0,T']$ for some $T'>0$.

By the first part of the metric Brenier theorem \ref{thm:brenmetr} we know 
\begin{equation}
\label{eq:lettino0}
\int|\nabla\varphi_1|^2\,\d\mu_0=\int|\nabla\varphi_2|^2\,\d\mu_0=W_2^2(\mu_0,\mu_1).
\end{equation}
Now let $\ppi\in\prob{C([0,1],X)}$ be a lifting of $(\mu_t)$ and use the second part of  the metric Brenier theorem  \ref{thm:brenmetr} to get that $\ppi$ represents $\nabla(-\varphi_1)$. Thus the first order differentiation formula \eqref{eq:firsthil} yields
\begin{equation}
\label{eq:lettino1}
\lim_{t\downarrow0}\int\frac{\varphi_2(\gamma_t)-\varphi_2(\gamma_0)}{t}\,\d\ppi(\gamma)=\int \la\nabla\varphi_2,\nabla(-\varphi_1)\ra\,\d\mu_0=-\int \la\nabla\varphi_2,\nabla\varphi_1\ra\,\d\mu_0.
\end{equation}
On the other hand, since $\varphi_2$ is also a Kantorovich potential, for any $\gamma\in\supp(\ppi)$ it holds $\gamma_1\in\partial^c\varphi_2(\gamma_0)$ and thus
\[
\varphi_2(\gamma_0)-\varphi_2(\gamma_t)\geq\frac{\sfd^2(\gamma_0,\gamma_1)}{2}-\frac{\sfd^2(\gamma_t,\gamma_1)}{2}=\sfd^2(\gamma_0,\gamma_1)(t-t^2/2).
\]
Dividing by $-t$, integrating w.r.t. $\ppi$ and letting $t\downarrow0$ we get
\begin{equation}
\label{eq:lettino2}
\lim_{t\downarrow0}\int\frac{\varphi_2(\gamma_t)-\varphi_2(\gamma_0)}{t}\,\d\ppi(\gamma)\leq-\int\sfd^2(\gamma_0,\gamma_1)\,\d\ppi(\gamma)=-W_2^2(\mu_0,\mu_1).
\end{equation}
Coupling \eqref{eq:lettino1} and \eqref{eq:lettino2} we get $-\int \la\nabla\varphi_2,\nabla\varphi_1\ra\,\d\mu_0\leq -W_2^2(\mu_0,\mu_1)$ which together with  \eqref{eq:lettino0} gives
\[
\int|\nabla(\varphi_1-\varphi_2)|^2\,\d\mu_0=\int|\nabla\varphi_1|^2\,\d\mu_0+\int|\nabla\varphi_2|^2\,\d\mu_0-2\int\la\nabla\varphi_1,\nabla\varphi_2\ra\,\d\mu_0\leq 0,
\]
so that in this case the thesis is proved.

To reduce to the case where an $\Omega$ with the stated properties exists, we use the local nature of the thesis, the Lindelof property of $(\tilde X,\tilde \sfd)$ and the fact that for any Borel $\Gamma\subset C([0,1],X)$ with $\ppi(\Gamma)>0$ the Kantorovich potentials $\varphi_1,\varphi_2$ induce the geodesic $t\mapsto(\e_t)_\sharp(c\ppi\restr\Gamma)$, where $c:=\ppi(\Gamma)^{-1}$ is the normalizing constant. The thesis follows.
\end{proof}
It is a classical fact in optimal transport theory that on a Riemannian manifold $M$, a geodesic $(\mu_t)\subset\probt M$ solves the continuity equation $\frac\d{\d t}\mu_t+\frac1t\nabla\cdot(\nabla\varphi_t\mu_t)=0$ in the sense of distributions, where    $\varphi_t$ is any Kantorovich potential from $\mu_t$ to $\mu_0$ differentiable on $\supp(\mu_t)$ (there exists at least one of these). Reversing the time we also know that $(\mu_t)$ solves $\frac\d{\d t}\mu_t-\frac1{1-t}\nabla\cdot(\nabla\psi_t\mu_t)=0$,  where  $\psi_t$ is any Kantorovich potential from $\mu_t$ to $\mu_1$ differentiable on $\supp(\mu_t)$. Thus we expect that $\frac1t\nabla\varphi_t+\frac1{1-t}\nabla\psi=0$ holds on $\supp(\mu_t)$ for every $t$. This can indeed be rigorously proved, and the next lemma provides an analogous of this statement on infinitesimally Hilbertian spaces. The proof is based on the relations \eqref{eq:intpot}.

We shall make use of the restriction (and rescaling) maps ${\rm Restr}_t^s:C([0,1],X)\to C([0,1],X)$ defined for any $t,s\in[0,1]$ by
\[
({\rm Restr}_t^s(\gamma))_r:=\gamma_{(1-r)t+rs}.
\]
\begin{lemma}\label{le:bilatero}
Let $(\tilde X,\tilde\sfd,\tilde\mm)$ be an infinitesimally Hilbertian space and $(\mu_t)\subset\probt {\tilde X}$ a geodesic such that $\mu_t\leq C\mm$ for every $t\in[0,1]$ and some $C>0$. 

Then for every $t\in(0,1)$ it holds
\begin{equation}
\label{eq:oppvel}
|\nabla(\tfrac1t\varphi_t+\tfrac1{1-t}\psi_t)|=0,\qquad\mu_t\ae,
\end{equation}
for any choice of locally Lipschitz Kantorovich potentials $\varphi_t,\psi_t$  relative to the couples $(\mu_t,\mu_0)$ and $(\mu_t,\mu_1)$ respectively.
\end{lemma}
\begin{proof}
By Lemma \ref{le:gradkp} and the simple 1-Lipschitz estimate \eqref{eq:1lipfg} it is sufficient to prove \eqref{eq:oppvel} for some specific choice of locally Lipschitz Kantorovich potentials.  Hence by Proposition \ref{prop:kantev} we can choose an arbitrary Kantorovich potential $\varphi$ from $\mu_0$ to $\mu_1$, put  for simplicity $\tilde\varphi_t:=Q_t(-\varphi)$ and $\tilde\psi_t:=Q_{1-t}(-\varphi^c)$  and  reduce to prove that
\[
|\nabla \big(\tilde\varphi_t+\tilde\psi_t\big)|=0,\qquad\mu_t\ae.
\]
Assume for a moment that for some bounded open set $\Omega$ we have $\supp(\mu_t)\subset\Omega$ for every $t\in[0,1]$. Fix $t\in(0,1)$, let $\ppi\in\gopt(\mu_0,\mu_1)$ be inducing the geodesic $(\mu_t)$ and define the plans $\ppi^\pm\in\prob{\geo(X)}$ by 
\[
\begin{split}
\ppi^+&:=({\rm Restr}_t^1)_\sharp\ppi,\qquad\qquad\qquad\ppi^-:=({\rm Restr}_t^0)_\sharp\ppi.
\end{split}
\]
By construction, $\ppi^+$ and $\ppi^-$ induce the geodesics $s\mapsto \mu_{t+s(1-t)}$ and $s\mapsto \mu_{t(1-s)}$ respectively and the metric Brenier theorem \ref{thm:brenmetr} together with Proposition \ref{prop:kantev} give
\[
\int|\nabla (t\tilde\varphi_t)|^2\,\d\mu_t=W_2^2(\mu_t,\mu_0),\qquad\textrm{and}\qquad\int|\nabla ((1-t)\tilde\psi_t)|^2\,\d\mu_t=W_2^2(\mu_t,\mu_1),
\]
which implies
\begin{equation}
\label{eq:permeno}
\int|\nabla \tilde\varphi_t|^2\,\d\mu_t=\int|\nabla \tilde\psi_t|^2\,\d\mu_t=W_2^2(\mu_0,\mu_1).
\end{equation}
The metric Brenier theorem also ensures that $\ppi^+$ represents the gradient of $-\tilde\psi_t $ in $\Omega$ so that the first order differentiation formula \eqref{eq:firsthil} gives
\begin{equation}
\label{eq:stanchino2}
\lim_{h\downarrow0}\int \frac{\tilde\varphi_t(\gamma_h)-\tilde\varphi_t(\gamma_0)}{h}\,\d\ppi^+(\gamma)=-\int\la\nabla \tilde\varphi_t,\nabla\tilde\psi_t\ra\,\d\mu_t.
\end{equation}
Now notice that the identity $(\e_0)_\sharp\ppi^+=\mu_t$ grants that for every $\gamma\in\supp(\ppi^+)$ it holds $\gamma_0\in\supp(\mu_t)$, thus from the relations \eqref{eq:intpot} we get
\begin{equation}
\label{eq:stanchino3}
\int \frac{\tilde\varphi_t (\gamma_h)-\tilde\varphi_t (\gamma_0)}{h}\,\d\ppi^+(\gamma)\geq -\int \frac{\tilde\psi_t (\gamma_h)-\tilde\psi_t (\gamma_0)}{h}\,\d\ppi^+(\gamma),\qquad\forall h\in(0,1].
\end{equation}
Letting $h\downarrow0$ and using again the fact that $\ppi^+$ represents the gradient of $-\tilde\psi_t $ we have
\[
\lim_{h\downarrow0}-\int \frac{\tilde\psi_t (\gamma_h)-\tilde\psi_t (\gamma_0)}{h}\,\d\ppi^+(\gamma)=\int|\nabla \tilde\psi_t |^2\,\d\mu_t.
\]
This fact together with \eqref{eq:stanchino2} and \eqref{eq:stanchino3} imply
\[
\int\la\nabla \tilde\varphi_t ,\nabla\tilde\psi_t \ra\,\d\mu_t\leq- \int|\nabla \tilde\psi_t |^2\,\d\mu_t,
\]
and thus using \eqref{eq:permeno} we obtain
\[
\begin{split}
\int |\nabla \big(\tilde\varphi_t +\tilde\psi_t \big)|^2\,\d\mu_t=\int |\nabla\tilde\varphi_t  |^2+|\nabla\tilde\psi_t |^2+2\la\nabla\tilde\varphi_t ,\nabla\tilde\psi_t \ra\,\d\mu_t\leq 0,
\end{split}
\]
which is the thesis. To remove the assumption on the existence of $\Omega$ we use the same localization procedure used in the proof of Lemma \ref{le:gradkp} above, we omit the details. 
\end{proof}
\begin{remark}[Infinitesimally smooth spaces]\label{re:infsmooth}{\rm
In the two lemmas \ref{le:gradkp}, \ref{le:bilatero} above, the assumption of infinitesimal Hilbertianity is more than what actually needed to conclude. Notice indeed that it has been used only to deduce that for
\begin{equation}
\label{eq:infsmooth1}
\begin{split}
&f,g\in\s^2(\Omega),\ \textrm{and } \ppi\in\prob{C([0,1],X)}\textrm{ representing $\nabla(- g)$}\\
& \textrm{such that} \int_\Omega\weakgrad f^2(\gamma_0)\,\d\ppi(\gamma)=\int_\Omega\weakgrad g^2(\gamma_0)\,\d\ppi(\gamma)\leq \limi_{t\downarrow0}\int\frac{f(\gamma_0)-f(\gamma_t)}{t}\,\d\ppi(\gamma), \\
\end{split}
\end{equation}
it must hold
\begin{equation}
\label{eq:infsmooth2}
\weakgrad{(f-g)}=0,\qquad(\e_0)_\sharp\ppi\ae.
\end{equation}
This conclusion can be derived in \emph{infinitesimally smooth spaces}, defined as
\begin{quote}
$(X,\sfd,\mm)$ is infinitesimally smooth provided for any $\mu\in\probt X$ with $\mu\leq C\mm$ for some $C>0$, the seminorm $\|f\|_\mu:=\sqrt{\int\weakgrad f^2\,\d\mu}$ on $\s^2(X,\sfd,\mm)$ is strictly convex in the sense that: $\|f\|_\mu=\|g\|_\mu=\frac12\|f+g\|_\mu$ implies $\weakgrad{(f-g)}=0$ $\mu$-a.e..
\end{quote}
To see that on an infinitesimally smooth space one can deduce \eqref{eq:infsmooth2} from \eqref{eq:infsmooth1}, just notice that from \eqref{eq:perplanrepr} and putting $h:=\frac{f+g}{2}$ we get
\[
\begin{split}
\frac12\int\weakgrad h^2\,\d(\e_0)_\sharp\ppi(\gamma)&\geq \lims_{t\downarrow0}\int\frac{h(\gamma_0)-h(\gamma_t)}{t}\,\d\ppi(\gamma)-\frac12\lims_{t\downarrow}\frac1t\iint_0^t|\dot\gamma_s|^2\,\d s\,\d\ppi(\gamma)\\
&=\frac12\int\weakgrad g^2\,\d(\e_0)_\sharp\ppi(\gamma),
\end{split}
\]
having used the assumptions and the identity $\lims_{t\downarrow0}\frac1t\iint_0^t|\dot\gamma_s|^2\,\d s\,\d\ppi(\gamma)=\int\weakgrad g^2\,\d(\e_0)_\sharp\ppi(\gamma)$, which directly follows from the fact that $\ppi$ represents the gradient of $-g$.

Notice that the hypothesis of being infinitesimally smooth is, in a sense, dual of the one of being infinitesimally strictly convex. On $\R^d$ equipped with the Lebesgue measure and a norm it is equivalent to the fact that the squared norm  is $C^1$, whence the terminology. We won't discuss this topic any further.
}\fr\end{remark}
The next lemma can be seen as a variant of the basic result in Hilbert spaces granting that `weak convergence of $(v_n)$ to $v$ plus strong convergence of $(w_n)$ to $w$ implies convergence of $\la v_n,w_n\ra$ to $\la v,w\ra$'.
\begin{lemma}[`Weak-strong convergence']\label{le:weakcont} Let $(\tilde X,\tilde\sfd,\tilde\mm)$ be an infinitesimally Hilbertian space and $\Omega\subset \tilde X$ an open set. Also:
\begin{itemize}
\item[i)] Let $(\mu_n)\subset \probt X$ be such that $\mu_n\leq C\mm$ for some $C>0$ and $\supp(\mu_n)\subset \Omega$ for any $n\in\N$.  Let $\rho_n$ be the density of $\mu_n$ and assume that $\rho_n\to\rho$ $\mm$-a.e. for some probability density $\rho$ with $\supp(\rho)\subset \Omega$. Put $\mu:=\rho\mm$.
\item[ii)] Let  $(f_n)\subset \s^2(\Omega)$ be such that 
\begin{equation}
\label{eq:boundedweak}
\sup_{n\in\N}\int_{\Omega}|\nabla f_n|^2\,\d\mm<\infty,
\end{equation}
and assume that $f_n\to f$ $\mm$-a.e. on $\Omega$ as $n\to\infty$ for some Borel function  $f:\Omega\to\R$.
\item[iii)]  Let $(g_n)\subset \s^2(\Omega)$ and $g\in \s^2(\Omega)$ be such that $g_n\to g$ $\mm$-a.e. as $n\to\infty$ and
\[
\sup_{n\in\N}\int_{\Omega}|\nabla g_n|^2\,\d\mm<\infty,\qquad\textrm{ and }\qquad \lim_{n\to\infty}\int |\nabla g_n|^2\,\d\mu_n=\int|\nabla g|^2\,\d\mu.
\]
\end{itemize}
Then
\begin{equation}
\label{eq:tesiweak}
\lim_{n\to\infty}\int\la\nabla f_n,\nabla g_n\ra\d\mu_n=\int\la\nabla f,\nabla g\ra\d\mu.
\end{equation}
\end{lemma}
\begin{proof} The assumption \eqref{eq:boundedweak} grants the weak relative compactness of $(|\nabla f_n|)$ in $L^2(\Omega)$, thus the lower semicontinuity of minimal weak upper gradients stated after Definition \ref{def:parigi} yields $f\in \s^2(\Omega)$. Hence  taking into account that $\supp(\mu)\subset\Omega$ and $\mu\leq C\mm$ by assumption, the right hand side of \eqref{eq:tesiweak} is well defined and the statement makes sense.

Put $C':=\sup_{n\in\N}\int_\Omega|\nabla f_n|^2\,\d\mm<\infty$, fix $\eps\in\R$ and notice that
\begin{equation}
\label{eq:bien0}
\begin{split}
2\eps\int_\Omega \la\nabla f_n,\nabla g_n\ra\,\d\mu_n=\int|\nabla(\eps  f_n+  g_n)|^2\rho_n\,\d\mm-\eps^2\int|\nabla f_n|^2\rho_n\,\d\mm-\int|\nabla g_n|^2\rho_n\,\d\mm.
\end{split}
\end{equation}
From the assumption \eqref{eq:boundedweak} and the boundedness of the $\rho_n$'s again we obtain
\begin{equation}
\label{eq:boundf}
\sup_{n\in\N}\int |\nabla f_n|^2\rho_n\,\d\mm\leq C\,C'.
\end{equation}
For every $n\in\N$ the function  $G_n:=|\nabla(\eps f_n+ g_n)|$ is $\mm$-a.e. well defined on $\Omega$ and the sequence $(G_n)$ is  bounded in $L^2(\Omega)$, thus up to pass to a subsequence - not relabeled - it weakly converges in $L^2(\Omega)$ to some  $G\in L^2(\Omega)$. From the fact that the $\rho_n$'s are uniformly bounded and $\mm$-a.e. converge to $\rho$ we easily deduce that  $(G_n\sqrt{\rho_n})$ weakly converges in $L^2(\Omega)$ to $G\sqrt\rho$. Hence
\begin{equation}
\label{eq:bien}
\limi_{n\to\infty}\int G_n^2\rho_n\,\d\mm\geq \int G^2\rho\,\d\mm.
\end{equation}
The definition of $G$ and the lower semicontinuity of minimal weak upper gradients  again, ensure that $\mm$-a.e. on $\Omega$ it holds  $|\nabla(\eps f+g)|\leq G$  and therefore \eqref{eq:bien} gives
\[
\limi_{n\to\infty}\int |\nabla(\eps f_n+ g_n)|^2\rho_n\,\d\mm\geq \int|\nabla(\eps f+g)|^2\rho\,\d\mm.
\]
Plugging this inequality in \eqref{eq:bien0} and using the assumption on $(g_n),g$ and \eqref{eq:boundf} we obtain
\[
\begin{split}
\limi_{n\to\infty}2\eps\int \la\nabla f_n,\nabla g_n\ra\,\d\mu_n&\geq \int|\nabla(\eps f+g)|^2\rho\,\d\mm-\eps^2C\,C'-\int|\nabla g|^2\rho\,\d\mm\\
&=2\eps\int \la \nabla f,\nabla g\ra\,\d\mu-\eps^2\left(C\,C'-\int |\nabla f|^2\,\d\mu\right).
\end{split}
\]
Dividing by $\eps>0$ (resp. $\eps<0$), letting $\eps\downarrow0$ (resp. $\eps\uparrow 0$) and noticing that the result does not depend on the particular subsequence chosen we conclude.
\end{proof}

\begin{remark}[Infinitesimally uniformly convex spaces]\label{re:infunifconv}{\rm In Lemma \ref{le:weakcont} above, the assumption about infinitesimal Hilbertianity can be weakened into `the space is \emph{infinitesimally uniformly convex}', this class being defined as:
\begin{quote}
$(X,\sfd,\mm)$ is infinitesimally uniformly convex provided there exists a bounded Borel function $X\times[0,\infty]\ni (x,\tau)\mapsto\omega_x(\tau)$ such that for $\mm$-a.e. $x$ the map $\tau\mapsto\omega_x(\tau)$ is continuous,  non-decreasing and satisfies $\omega_x(0)=0$ and for which  the inequality
\[
\weakgrad g\frac{\weakgrad{(g+\eps f)}+\weakgrad{(g-\eps f)}-2\weakgrad g}{2\eps}\leq \weakgrad f^2\,\omega\left(\eps\frac{\weakgrad f}{\weakgrad g}\right),\qquad\mm\ae,
\]
holds  for any $f,g\in\s^2(X,\sfd,\mm)$ (on the set $\{\weakgrad f=\weakgrad g=0\}$ the left-hand side is 0, thus it has no importance how to define the value of the ratio $\frac{\weakgrad f}{\weakgrad g}$ on this set).
\end{quote}
The terminology comes from the following fact: on a Banach space $(B,\|\cdot\|)$ the inequality
\[
\|v\|\frac{\|v+\eps w\|+\|v-\eps w\|-2\|v\|}{2\eps}\leq\|w\|^2\omega\Big(\eps\frac{\|w\|}{\|v\|}\Big),\qquad\forall v,w\in B,
\]
is equivalent to the inequality
\[
\frac{\|v+\eps w\|+\|v-\eps w\|}{2}-1\leq\eps \omega(\eps),\qquad\forall v,w\in B,\ \|v\|=\|w\|=1,
\]
as shown by some simple algebraic manipulation. The validity of the latter for some continuous function $\omega$ with $\omega(0)=0$ is the defining property of uniformly smooth Banach spaces, and since a space is uniformly smooth if and only if its dual is uniformly convex (see e.g. \cite{Diestel75}), we deduce that a Finsler manifold $F$ is infinitesimally uniformly convex in the terminology above if and only if for a.e. $x\in F$ the norm in the tangent space at $x$ is uniformly convex (not necessarily uniformly on $x$).

\medskip

Coming back to the case of metric measure spaces and recalling Definition \ref{def:dpmfg}, it is not hard to show that
\begin{equation}
\label{eq:dpmvar}
D^+f(\nabla g)=\inf_{\eps>0}\weakgrad g\frac{\weakgrad{(g+\eps f)}-\weakgrad g}{\eps},\qquad D^-f(\nabla g)=\sup_{\eps<0}\weakgrad g\frac{\weakgrad{(g+\eps f)}-\weakgrad g}{\eps},
\end{equation}
and in particular we see that infinitesimally uniformly convex spaces are infinitesimally strictly convex. In this remark we want to prove that under the same assumptions of the Lemma \ref{le:weakcont}, weakening  `infinitesimal Hilbertianity' into `infinitesimal uniform convexity' and further assuming that
\begin{equation}
\label{eq:domination}
\mu_n\leq h\mm,\qquad\forall n\in\N,\quad\textrm{ for some $h\in L^1(\Omega,\mm)$},
\end{equation}
one can still conclude that
\[
\lim_{n\to\infty}\int Df_n(\nabla g_n)\,\d\mu_n=\int Df(\nabla g)\,\d\mu.
\]

Indeed, put
\[
{\rm Err}(\eps,f,g):=\weakgrad g\frac{\weakgrad{(g+\eps f)}+\weakgrad{(g-\eps f)}-2\weakgrad g}{2\eps}
\]
and notice that by assumption and the triangle inequality for minimal weak upper gradients we have
\begin{equation}
\label{eq:err}
{\rm Err}(\eps,f,g)\leq\min\left\{ \weakgrad f\weakgrad g,\weakgrad f^2\,\omega\left(\eps\frac{\weakgrad f}{\weakgrad g}\right)\right\},
\end{equation}
then pick $\eps>0$ and use \eqref{eq:dpmvar} to get
\[
\begin{split}
\int Df_n(\nabla g_n)\,\d\mu_n&\leq \int \weakgrad{g_n}\frac{\weakgrad{(g_n+\eps f_n)}-\weakgrad {g_n}}{\eps}\,\d\mu_n\\
&=\int \weakgrad{g_n}\frac{\weakgrad {g_n}-\weakgrad{(g_n-\eps f_n)}}{\eps}+2{\rm Err}(\eps,f_n,g_n)\,\d\mu_n.
\end{split}
\]
Arguing as in the proof of Lemma \ref{le:weakcont} we obtain
\[
\begin{split}
\lims_{n\to\infty}\int\weakgrad{g_n}\frac{\weakgrad {g_n}-\weakgrad{(g_n-\eps f_n)}}{\eps}\,\d\mu_n&\leq\int\weakgrad g\frac{\weakgrad {g}-\weakgrad{(g-\eps f)}}{\eps}\,\d\mu\\
&=\int\weakgrad g\frac{\weakgrad{(g-\eps f)}-\weakgrad {g}}{-\eps}\,\d\mu\leq \int Df(\nabla g)\,\d\mu.
\end{split}
\]
Notice that by assumption the sequence $(\weakgrad{g_n})$ is bounded in $L^2(\Omega)$, thus up to pass to a subsequence we can assume that it weakly converges to some $G\in L^2(\Omega)$. The lower semicontinuity of minimal weak upper gradients grants $G\geq\weakgrad g$ $\mm$-a.e. on $\Omega$. The $\mm$-a.e. convergence of $\rho_n$ to $\rho$ and the fact that these densities are uniformly bounded also give that $(\weakgrad{g_n}\sqrt{\rho_n})$ weakly converges to $G\sqrt\rho$ and thus taking into account the assumption $(iii)$ on the $g_n$'s we get
\[
\int \weakgrad g^2\rho\,\d\mm\leq\int G^2\rho\,\d\mm\leq \limi_{n\to\infty}\int \weakgrad{g_n}^2\rho_n\,\d\mm\leq \lims_{n\to\infty}\int \weakgrad{g_n}^2\rho_n\,\d\mm=\int \weakgrad g^2\rho\,\d\mm,
\]
which forces strong $L^2$-convergence of $(\weakgrad{g_n}\sqrt{\rho_n})$ to $\weakgrad g\sqrt\rho$. Thus up to pass to a non-relabeled subsequence, we can assume that
\begin{equation}
\label{eq:dom2}
\weakgrad{g_n}\sqrt{\rho_n}\leq H,\qquad\forall n\in\N,\textrm{ for some }H\in L^2(\Omega).
\end{equation}
Now  fix $a,b>0$ and notice that
\begin{equation}
\label{eq:finalm}
\begin{split}
\int_\Omega{\rm Err}(\eps,f_n,g_n)\,\d\mu_n&= \int_{\big(\{\weakgrad {f_n}\geq a\}\big)\cup\big(\{\weakgrad {g_n}\leq b\}\big)\cup\big(\{\weakgrad{f_n}\leq a\}\cap\{\weakgrad {g_n}\geq b\}\big)}{\rm Err}(\eps,f_n,g_n)\,\d\mu_n\\
\textrm{by \eqref{eq:err}}\qquad&\leq\int_{\{\weakgrad{f_n}\geq a\}}\weakgrad {f_n}\weakgrad {g_n}\,\d\mu_n+b\int_{\Omega}\weakgrad {f_n}\,\d\mu_n+a^2\int_\Omega \omega\big(\eps\frac{a}{b}\big)\,\d\mu_n\\
\textrm{by \eqref{eq:domination},\eqref{eq:dom2}}\qquad&\leq C\sqrt{\int_{\{\weakgrad{f_n}\geq a\}}H^2\,\d\mm}+bC+a^2\int_\Omega h\,\omega\big(\eps\frac{1}{a}\big)\,\d\mm,
\end{split}
\end{equation}
where $C:=\sup_n\sqrt{\int\weakgrad{f_n}^2\,\d\mu_n}<\infty$. The dominate convergence theorem and the assumption on $\omega$ grants that $\lim_{\eps\downarrow0}\int_\Omega h\,\omega\big(\eps\frac{1}{a}\big)\,\d\mm=0$, while the Chebyshev inequality, the uniform bound on $\|\weakgrad{f_n}\|_{L^2(\Omega)}$ and the absolute continuity of the integral ensure that\linebreak $\lim_{a\uparrow\infty}\int_{\{\weakgrad{f_n}\geq a\}}H^2\,\d\mm=0$. Thus letting first $\eps\downarrow0$ and then $a,b^{-1}\uparrow\infty$ in \eqref{eq:finalm} we deduce
\[
\lim_{\eps\downarrow0}\sup_{n\in\N}\int_\Omega{\rm Err}(\eps,f_n,g_n)\,\d\mu_n=0.
\]
Taking all together we thus proved that
\[
\lims_{n\to\infty}\int Df_n(\nabla g_n)\,\d\mu_n\leq \int Df(\nabla g)\,\d\mu,
\]
and exchanging $f_n$ with $-f_n$ we get the other inequality and the conclusion.
}\fr\end{remark}
\begin{remark}[Strong-strong convergence on inf. strictly convex spaces]\label{re:strstrstrconf}{\rm
If assumptions $(ii),(iii)$ in Lemma \ref{le:weakcont} are strengthened into 
\begin{itemize}
\item[ii')]  $(f_n)\subset \s^2(\Omega)$ and $f\in \s^2(\Omega)$ are such that $f_n\to f$ $\mm$-a.e. as $n\to\infty$ and
\[
\weakgrad{(f_n-f)}\to0,\quad\mm\ae,\qquad\quad\weakgrad{f_n}\leq F,\quad\forall n\in\N\ \textrm{ for some $F\in L^2(\Omega)$},
\]
\item[iii')] $(g_n)\subset \s^2(\Omega)$ and $g\in \s^2(\Omega)$ are such that $g_n\to g$ $\mm$-a.e. as $n\to\infty$ and
\[
\weakgrad{(g_n-g)}\to0,\quad\mm\ae,\qquad\quad\weakgrad{g_n}\leq G,\quad\forall n\in\N\ \textrm{ for some $G\in L^2(\Omega)$},
\]
\end{itemize}
then the conclusion
\[
\lim_{n\to\infty}\int Df_n(\nabla g_n)\,\d\mu_n=\int Df(\nabla g)\,\d\mu,
\]
holds on infinitesimally strictly convex spaces. Indeed, it is easy to see that the hypothesis yield that for $\delta\in\R\setminus\{0\}$ the sequence of functions $\frac{\weakgrad{(g_n+\delta f_n)}^2-\weakgrad{g_n}^2}{2\delta}$
is dominated and $\mm$-a.e. converge to $\frac{\weakgrad{(g+\delta f)}^2-\weakgrad{g}^2}{2\delta}$. Thus fix $\eps>0$, pick $\delta<0$ such that
\[
\int Df(\nabla g)\,\d\mu\leq \eps+\int\frac{\weakgrad{(g+\delta f)}^2-\weakgrad{g}^2}{2\delta}\,\d\mu,
\]
and notice that
\[
\int\frac{\weakgrad{(g+\delta f)}^2-\weakgrad{g}^2}{2\delta}\,\d\mu=\lim_{n\to\infty}\int\frac{\weakgrad{(g_n+\delta f_n)}^2-\weakgrad{g_n}^2}{2\delta}\,\d\mu\leq \limi_{n\to\infty}\int Df_n(\nabla g_n)\,\d\mu_n.
\]
Thus $\int Df(\nabla g)\,\d\mu\leq\limi_{n\to\infty}\int Df_n(\nabla g_n)\,\d\mu_n$ and exchanging $f_n$ with $-f_n$ we get the other inequality and the conclusion.
}\fr\end{remark}

We will now use the calculus tools just developed to produce quite general first order regularity results. In order to make use of Lemma \ref{le:weakcont}, we  will work with $W_2$-geodesics as follows:
\begin{definition}[Curves with time-continuous densities]\label{def:timecont}
Let $(X,\sfd,\mm)$ be a metric\linebreak measure space and $(\mu_t)\subset\probt X$ an absolutely continuous curve. We say that $(\mu_t)$ has time-continuous density provided for some $C>0$ it holds $\mu_t\leq C\mm$ for every $t\in[0,1]$ and, denoting by $\rho_t$ the density of $\mu_t$, the following holds:  for any $t\in[0,1]$ and any sequence $(t_n)\subset [0,1]$ converging to $t$ there exists a subsequence $(t_{n_k})$ such that
\[
\rho_{t_{n_k}}\to\rho_t,\qquad \mm\ae \ {\rm{ as }}\ k\to\infty.
\]
\end{definition}
In our applications, Corollary \ref{cor:pertcont} will grant the existence of many geodesics of this kind on $RCD(K,\infty)$ spaces.
\begin{proposition}[From $t\mapsto f(x_t)$ to $t\mapsto \int f\,\d\mu_t$]\label{prop:intf}
Let $(X,\sfd,\mm)$ be an infinitesimally\linebreak Hilbertian metric measure space, $(\mu_t)\subset \probt X$ a geodesic with time-continuous density, $\Omega\subset X$ a bounded open set such that $\supp(\mu_t)\subset \Omega$ for every $t\in[0,1]$  and $f\in \s^2(\Omega)\cap L^1(\Omega)$.

Then the map $t\mapsto \int f\,\d\mu_t$ is $C^1$ and it holds
\begin{equation}
\label{eq:pezzi}
\begin{split}
\frac{\d}{\d t}\int f\,\d\mu_t&=\frac1t\int \la\nabla f,\nabla \varphi_t\ra\d\mu_t\qquad\qquad\ \ \forall t\in(0,1]\\
&=-\frac1{1-t}\int \la\nabla f,\nabla \psi_t\ra\d\mu_t,\qquad\forall t\in[0,1),
\end{split}
\end{equation}
where for any $t\in[0,1]$ the functions $\varphi_t,\psi_t$ are locally Lipschitz Kantorovich potentials from $\mu_t$ to $\mu_0$ and from $\mu_t$ to $\mu_1$ respectively. In particular, for any $t\in(0,1)$ it holds
\begin{equation}
\label{eq:intero}
\frac{\d}{\d t}\int f\,\d\mu_t=\int\la\nabla f,\nabla(Q_t(-\varphi))\ra\d\mu_t=-\int\la\nabla f,Q_{1-t}(-\varphi^c)\ra\d\mu_t,
\end{equation}
where $\varphi$ is any Kantorovich potential from $\mu_0$ to $\mu_1$.
\end{proposition}
\begin{proof} A direct consequence of the assumptions is that $\int |f|\,\d\mu_t<\infty$ for any $t\in[0,1]$ so that the statement makes sense. Also, according to Lemma \ref{le:gradkp} the values of the right-hand sides of expressions \eqref{eq:pezzi} and \eqref{eq:intero} do not depend on the particular choice of Kantorovich potentials. Fix a bounded Kantorovich potential $\varphi$ from $\mu_0$ to $\mu_1$ so  that in particular - as it is easily seen - both $\varphi$ and $\varphi^c$  are Lipschitz.

Fix $t_0\in[0,1)$, put $\ppi^+_{t_0}:=({\rm Restr}_{t_0}^1)_\sharp\ppi$ so that $\ppi^+_{t_0}$ is a lifting of the geodesic $t\mapsto\mu_{t_0+t(1-t_0)}$ and recall that by Proposition \ref{prop:kantev} $Q_{1-t_0}(-\varphi^c)$ induces such geodesic. The metric Brenier theorem \ref{thm:brenmetr} and the first order differentiation formula \eqref{eq:firsthil} give
\[
\begin{split}
\lim_{h\downarrow0}\int f\,\d\frac{\mu_{t_0+h}-\mu_{t_0}}{h}&=\lim_{h\downarrow0}\int\frac{f(\gamma_{t_0+h})-f(\gamma_{t_0})}h\,\d\ppi(\gamma)\\
&=\lim_{h\downarrow0}\int\frac{f(\gamma_{h/(1-t_0)})-f(\gamma_{0})}h\,\d\ppi^+_{t_0}(\gamma)\\
&=\frac{1}{1-t_0}\lim_{h\downarrow0}\int\frac{f(\gamma_{h})-f(\gamma_{0})}h\,\d\ppi^+_{t_0}(\gamma)\\
&=-\int \la\nabla f,\nabla (Q_{1-t_0}(-\varphi^c))\ra\,\d(\e_0)_\sharp\ppi^+_{t_0}\\
&=-\int \la\nabla f,\nabla (Q_{1-t_0}(-\varphi^c))\ra\,\d\mu_{t_0}.
\end{split}
\]
We claim that  $t\mapsto \int \la\nabla f,\nabla (Q_{1-t}(-\varphi^c))\ra\,\d\mu_{t}$ is continuous on $[0,1)$ and to this aim we shall apply Lemma \ref{le:weakcont}. Let $(t_n)\subset[0,1)$ be converging to $t\in[0,1)$ and use the assumption that $(\mu_t)$ has time-continuous density to extract a subsequence, not relabeled, such that the density of $\mu_{t_n}$ converges $\mm$-a.e. to the density of $\mu_t$   as $n\to\infty$. Put $\mu_n:=\mu_{t_n}$, $\mu:=\mu_t$, $g_n:=f$, $g:=f$, $f_n:=-Q_{1-t_n}(-\varphi^c)$ and $f:=-Q_{1-t}(-\varphi^c)$ so that assumptions $(i)$ and $(iii)$ of Lemma \ref{le:weakcont} are fulfilled. By the continuity property \eqref{eq:basehl2} and the fact that the functions $-Q_{1-s}(-\varphi^c)$ are uniformly Lipschitz on $s\in[0,1-\eps]$ (recall \eqref{eq:hllip}) we also get that assumption $(ii)$ is satisfied and we get $\int \la\nabla f,\nabla (Q_{1-t_n}(-\varphi^c))\ra\,\d\mu_{t_n}\to \int \la\nabla f,\nabla (Q_{1-t}(-\varphi^c))\ra\,\d\mu_{t}$ as desired. Given that the result does not depend on the particular subsequence chosen, the claimed continuity follows.

Now repeat the argument for the time reversed geodesic $t\mapsto\mu_{1-t}$ to deduce that the left derivative of $\int f\,\d\mu_t$ exists, is continuous on $(0,1]$ and equal to  $\int \la\nabla f,\nabla (Q_t(-\varphi))\ra\d\mu_t$. 

Using Lemma \ref{le:bilatero} we get that left and right derivatives coincide on $(0,1)$ and Proposition \ref{prop:kantev} allows to switch between \eqref{eq:pezzi} and \eqref{eq:intero}.
\end{proof}
\begin{proposition}[From $t\mapsto\frac12\sfd^2(x_t,y)$ to $t\mapsto\frac12W_2^2(\mu_t,\nu)$]\label{prop:intdist} Let $(\tilde X,\tilde\sfd,\tilde\mm)$ be an infinitesimally Hilbertian space,  $(\mu_t)\subset \probt X$ a geodesic with time-continuous density and $\nu\in\probt X$ be such that   for some bounded open set $\Omega$ it holds $\supp(\mu_t)\subset\Omega$ for every $t\in[0,1]$ and $\supp(\nu)\subset\Omega$.

Then the function $t\mapsto \frac12W_2^2(\mu_t,\nu)$ is $C^1$ and its derivative is given by
\begin{equation}
\label{eq:derdistpezzi}
\begin{split}
\frac{\d}{\d t}\frac12W_2^2(\mu_t,\nu)&=\frac1t\int \la\nabla \phi_t,\nabla\varphi_t\ra\d\mu_t\qquad\qquad\ \ \forall t\in(0,1]\\
&=-\frac1{1-t}\int \la\nabla \phi_t,\nabla\psi_t\ra\d\mu_t,\qquad\forall t\in[0,1),
\end{split}
\end{equation}
where $\varphi_t,\psi_t,\phi_t$ are locally Lipschitz Kantorovich potentials from $\mu_t$ to $\mu_0$, from $\mu_t$ to $\mu_1$ and from $\mu_t$ to $\nu$ respectively. In particular, for any $t\in(0,1)$ it holds
\begin{equation}
\label{eq:derdist}
\frac{\d}{\d t}\frac12W_2^2(\mu_t,\nu)=\int \la\nabla \phi_t,\nabla (Q_t(-\varphi))\ra\d\mu_t=-\int \la\nabla \phi_t,\nabla (-Q_{1-t}(\varphi^c))\ra\d\mu_t,
\end{equation}
where $\varphi$ is any   Kantorovich potential from $\mu_0$ to $\mu_1$ and $\phi_t$ is as before.
\end{proposition}
\begin{proof} By Lemma \ref{le:gradkp} we know that the right-hand sides of \eqref{eq:derdistpezzi} and \eqref{eq:derdist} do not depend on the particular choice of Kantorovich potentials. Thus fix a family $\{\phi_t\}$ of Kantorovich potentials from $\mu_t$ to $\nu$ uniformly Lipschitz and uniformly bounded on $\Omega$ and a Kantorovich potential $\varphi$ from $\mu_0$ to $\mu_1$ such that both $\varphi$ and $\varphi^c$ are Lipschitz on $\Omega$  (these exist  because the supports of the measures are all contained in the same bounded set).

It is obvious that  $t\mapsto\frac12W_2^2(\mu_t,\nu)$ is Lipschitz and in particular differentiable for a.e. $t\in[0,1]$. Let $\mathcal D\subset(0,1)$ be the set of points of differentiability, pick $t\in\mathcal D$  and notice that
\[
\begin{split}
\frac12W_2^2(\mu_t,\nu)&=\int\phi_t\,\d\mu_t+\int\phi_t^c\,\d\nu,\\
\frac12W_2^2(\mu_{t+h},\nu)&\geq \int\phi_t\,\d\mu_{t+h}+\int\phi_t^c\,\d\nu,\qquad\forall h\in\R\textrm{ s.t. }t+h\in[0,1].
\end{split}
\]
Hence 
\begin{equation}
\label{eq:perderivare}
\begin{split}
\lim_{h\downarrow0}&\frac{W_2^2(\mu_{t+h},\nu)-W_2^2(\mu_t,\nu)}{2h}\geq \lims_{h\downarrow0}\int\phi_t\,\d\frac{\mu_{t+h}-\mu_t}{h}=\lims_{h\downarrow0}\int\frac{\phi_t\circ\e_{t+h}-\phi_t\circ\e_t}h\,\d\ppi,\\
\lim_{h\uparrow0}&\frac{W_2^2(\mu_{t+h},\nu)-W_2^2(\mu_t,\nu)}{2h}\leq \limi_{h\uparrow0}\int\phi_t\,\d\frac{\mu_{t+h}-\mu_t}{h}=\limi_{h\uparrow0}\int\frac{\phi_t\circ\e_{t+h}-\phi_t\circ\e_t}h\,\d\ppi,\\
\end{split}
\end{equation}
where $\ppi\in\prob{\geo(X)}$ is any lifting of $(\mu_t)$. With the same restriction and rescaling arguments used in the proof of Proposition \ref{prop:intf} above we see that the $\lims$ and $\limi$ at the right-hand sides in  the above expressions are actually limits and that for any $t\in\mathcal D$ it holds
\[
\begin{split}
\lim_{h\downarrow0}\int\frac{\phi_t\circ\e_{t+h}-\phi_t\circ\e_t}h\,\d\ppi&=-\int \la\nabla \phi_t,\nabla (Q_{1-t}(-\varphi^c))\ra\d\mu_t,\\
\lim_{h\uparrow0}\int\frac{\phi_t\circ\e_{t+h}-\phi_t\circ\e_t}h\,\d\ppi&=\int \la\nabla \phi_t,\nabla (Q_{t}(-\varphi))\ra\d\mu_t.
\end{split}
\] 
By Lemma \ref{le:bilatero} we know that the right-hand sides of these expressions coincide, therefore  by \eqref{eq:perderivare} we deduce that formula \eqref{eq:derdist} holds for $t\in\mathcal D$.

Now we claim that $t\mapsto \int\la\nabla\phi_t,\nabla Q_{1-t}(-\varphi^c)\ra\,\d\mu_t$ is continuous on $[0,1)$. Using the same time-reversion argument used in Proposition \ref{prop:intf} above and Proposiiton \ref{prop:kantev} to switch between \eqref{eq:derdist} and \eqref{eq:derdistpezzi}, this will be sufficient to conclude.

We shall apply Lemma \ref{le:weakcont}. Let $(t_n)\subset[0,1)$ be converging to $t\in[0,1)$ and use the assumption that $(\mu_t)$ has time-continuous density to extract a subsequence, not relabeled, such that the density of $\mu_{t_n}$ converges $\mm$-a.e. to the density of $\mu_t$   as $n\to\infty$. Put $\mu_n:=\mu_{t_n}$, $\mu:=\mu_t$, $g_n:=-Q_{1-t_n}(-\varphi^c)$, $g:=-Q_{1-t}(-\varphi^c)$ and $f_n:=\phi_{t_n}$. It is clear that assumptions $(i)$ and $(iii)$ of Lemma \ref{le:weakcont} are fulfilled (use the fact that $s\mapsto\int|\nabla Q_{1-s}(-\varphi^c)|^2\,\d\mu_s$ is constantly equal to $W_2^2(\mu_0,\mu_1)$ - see the proof of Lemma \ref{le:bilatero} - to get that $\int|\nabla g_n|^2\,\d\mu_n\to\int|\nabla g|^2\,\d\mu$). Now recall that the functions $f_n$ are uniformly Lipschitz and uniformly bounded on $\Omega$. Thus up to passing to a subsequence, not relabeled, we can assume that for some Lipschitz function $f$ it holds $f_n(x)\to f(x)$ as $n\to\infty$ for every $x\in \Omega$. Thus also the assumption $(ii)$ of Lemma \ref{le:weakcont} is fulfilled and noticing that  $f$ must be a Kantorovich potential from $\mu_t$ to $\nu$ and that  the result does not depend on the particular subsequence chosen, we get the thesis.
\end{proof}
\begin{remark}[Averaging out the unsmoothness]\label{rem:averaging}{\rm It is a banality that one of the main problems in making analysis on non-smooth spaces is that they are, indeed, non-smooth. This means in particular that we cannot realistically hope to have a well defined and well behaved tangent space at every point, and hence there is little chance to define $C^1$ functions. However, if one believes that a tangent space exists at $\mm$-a.e. point (which as we have seen is in some sense always the case from the point of view of Sobolev calculus), then analysis should becomes easier when lifting computations from  points to absolutely continuous measures with bounded densities. Very roughly and heuristically said, this has the effect of `averaging out the  unsmoothness' because the $\mm$-a.e. defined tangent space becomes a well behaved `integrated' tangent space at the considered measures. This is the idea behind Propositions \ref{prop:intf} and \ref{prop:intdist}, where we see that after such lift we can carry out a first order calculus similar to that available in the smooth setting.

As a curiosity, we point out that Propositions \ref{prop:intf} and \ref{prop:intdist} give improved regularity even if the underlying space is a smooth Riemannian manifold. Indeed, the squared distance is certainly not a $C^1$ function in general, but when lifted to probability measures with bounded densities such $C^1$ regularity along minimal $W_2$-geodesics is always granted.
}\fr\end{remark}
\begin{remark}[Weakly $C^1$ curves]{\rm
One of the ingredients needed in Propositions \ref{prop:intf}, \ref{prop:intdist}  to get $C^1$ regularity is that, shortly and roughly said, the velocity vector fields of the considered curves are continuous in time: the same conclusions certainly cannot be derived under the only hypothesis that the curve $(\mu_t)$ is absolutely continuous w.r.t. $W_2$ and with  time-continuous densities as in Definition \ref{def:timecont}.

At least for Proposition \ref{prop:intf}, the same sort of $C^1$ regularity can be obtained if $(\mu_t)$ is a  \emph{weakly $C^1$ curve}, these being defined as:
\begin{quote}
$(\mu_t)$ is a weakly $C^1$ curve provided it has time-continuous densities and there exists a family $\{f_t\}_{t\in[0,1]}\subset\s^2(\Omega)$, $\Omega$ being an open set such that $\supp(\mu_t)\subset\Omega$ for every $t\in[0,1]$, such that the following holds: 
\begin{itemize}
\item[i)] $\sup_{t\in[0,1]}\int_\Omega\weakgrad {f_t}^2\,\d\mm<\infty$ and for every sequence $(t_n)$ converging to $t$ there exists a subsequence $(t_{n_k})$ such that $f_{t_{n_k}}\to f_t$ $\mm$-a.e. as $k\to \infty$,
\item[ii)] for some (and thus any) lifting $\ppi$ of $(\mu_t)$ and for any $t\in[0,1]$ the plans $({\rm Restr}_t^0)_\sharp\ppi$ and $({\rm Restr}_t^1)_\sharp\ppi$ represent the gradients of $-tf_t$ and $(1-t)f_t$ respectively.
\end{itemize}
\end{quote}
Then if the space $(X,\sfd,\mm)$ is infinitesimally smooth and infinitesimally uniformly convex (Remarks \ref{re:infsmooth} and \ref{re:infunifconv})  Proposition \ref{prop:intf} holds if instead of geodesics one considers weakly  $C^1$ curves. 

Proposition \ref{prop:intdist} is a bit more delicate to handle because we have to check the continuity of $\int D\phi_t(\nabla(Q_t(-\varphi)))\,\d\mu_t$ rather than of $\int Df(\nabla(Q_t(-\varphi)))\,\d\mu_t$: the additional time-dependence of the $\phi_t$'s requires some care. We won't enter the discussion about the natural `infinitesimal' requirement needed to get such $C^1$ regularity.

An example of a weakly $C^1$ curve on a general space $(X,\sfd,\mm)$ with $\mm\in\prob X$ is given by $t\mapsto\mu_t:=\rho_t\mm$, where $(\rho_t)$ is the gradient flow trajectory of $\frac12\int\weakgrad\rho^2\,\d\mm$ starting from some probability density $\rho_0\in \s^2(X)$ with $\rho,\rho^{-1}\in L^\infty(X)$, in this case the functions $f_t$ are given by $f_t:=\log(\rho_t)$. See \cite{GigliHan13} for further comments in this direction.
}\fr\end{remark}

We will now use the  $C^1$ regularity just proved to get an important geometric information about the behavior of the squared distance along the gradient flow of $\b$ on our infinitesimally Hilbertian $CD(0,N)$ space $(X,\sfd,\mm)$. Notice that while the above calculus rules are valid even in non-infinitesimally Hilbertian spaces (Remarks \ref{re:infsmooth}, \ref{re:infunifconv} and \ref{re:strstrstrconf}), the following geometric consequence strongly relies on such assumption, see Remark \ref{rem:whyinfhil}.

We shall use    formulas \eqref{eq:abcconc}, which can be equivalently  rewritten as
\begin{equation}
\label{eq:hlb}
\begin{split}
Q_t\b&=\b-\frac t2,\qquad\qquad\qquad Q_t(-\b)=-\b-\frac t2,
\end{split}
\end{equation}
for any $t\geq 0$.

\begin{corollary}\label{cor:min}
With the same notation and assumptions as in \eqref{eq:assfin} and recalling that the flow $\bar\X$ is defined in Theorem \ref{thm:gfpresdist}, the following holds.

Let $\mu,\nu\in\probt X$ be two measures with bounded support such that $\mu,\nu\leq C\mm$ for some $C>0$ and put $\mu_t:=(\bar\X_t)_\sharp\mu$.

Then the map $t\mapsto W_2^2(\mu_t,\nu)$ has a unique minimum and such minimum is the only $t\in\R$ for which $\int\b\,\d\mu_{t}=\int\b\,\d\nu$.
\end{corollary}
\begin{proof} Since $\int\b\,\d\mu_t=\int\b\circ\bar\X_t\,\d\mu=-t+\int\b\,\d\mu$, there is exactly one $t\in\R$ such that $\int\b\,\d\mu_{t}=\int\b\,\d\nu$. Also, the map $t\mapsto W_2^2(\mu_t,\nu)$ is continuous and converges to $+\infty$ as $|t|\to+\infty$, hence it has a minimum, say $t_0$. Thus to conclude we need just to show that $\int\b\,\d\mu_{t_0}=\int\b\,\d\nu$.

Let $\ppi\in \gopt(\nu,\mu_{t_0})$ and $s\mapsto \nu_s:=(\e_s)_\sharp\ppi$ be the corresponding geodesic from $\nu$ to $\mu_{t_0}$. We claim that for each $s\in[0,1]$ the map $t\mapsto W_2^2(\nu,(\bar\X_t)_\sharp\nu_s)$ has a minimum for $t=0$. Indeed, if by absurdum for some $t\in\R$ it holds $W_2(\nu,(\bar\X_t)_\sharp\nu_s)<W_2(\nu,\nu_s)$, the fact that $\bar\X_t:\supp(\mm)\to \supp(\mm)$ is an isometry would give
\[
\begin{split}
W_2(\nu,(\bar\X_t)_\sharp\mu_{t_0})&\leq W_2(\nu,(\bar\X_t)_\sharp\nu_s)+W_2((\bar\X_t)_\sharp\nu_s,(\bar\X_t)_\sharp\mu_{t_0})\\
&<W_2(\nu,\nu_s)+W_2(\nu_s,\mu_{t_0})=W_2(\nu,\mu_{t_0}),
\end{split}
\]
thus contradicting the minimality of $t_0$.

Let $\varphi$ be a  Kantorovich potential from $\nu$ to $\mu_{t_0}$ and recall (Proposition \ref{prop:kantev}) that for $\varphi_s:=Q_s(-\varphi)$ the function $s\varphi_s$ is a Kantorovich potential from $\nu_s$ to $\nu$ which is Lipschitz on bounded sets. Also, by Theorem \ref{thm:basemetric}  we know that $\b$ is a Kantorovich potential relative to the geodesic $[0,1]\ni t\mapsto(\bar\X_{t-\frac12})_\sharp\nu_s$.

From the measure preservation property $(\bar\X_t)_\sharp\mm=\mm$ it is immediate to verify that $t\mapsto(\bar\X_{t-\frac12})_\sharp\nu_s$ has time-continuous density, thus  by Proposition \ref{prop:intdist}, the minimality of $\nu_s$ and recalling the formulas \eqref{eq:hlb} we have
\[
0=\frac{\d}{\d t}\frac12W_2^2(\nu,(\X_t)_\sharp\nu_s)\restr{t=0}=s\int \la\nabla\varphi_s,\nabla\b \ra\,\d\nu_s,\qquad\forall s\in[0,1].
\]
By Corollary \ref{cor:pertcont} we know that also $(\nu_s)$ has time-continuous density, hence applying Proposition \ref{prop:intf}  we get
\[
\frac{\d}{\d t}\int\b\,\d\nu_s=\int \la\nabla\varphi_s,\nabla\b \ra\,\d\nu_s=0,\qquad\forall s\in(0,1]
\]
i.e., $s\mapsto \int\b\,\d\nu_s$ is constant and the proof is completed.
\end{proof}

We are now ready to introduce the quotient metric space:
\begin{definition}[The quotient metric space]\label{def:xprimo}
With the same notation and assumptions as in \eqref{eq:assfin} and recalling that the flow $\bar\X$ is built in Theorem \ref{thm:gfpresdist}, we put  $X':=\supp(\mm)/\sim$ where $x\sim y$ if $x=\X_t(y)$ for some $t\in \R$. 

We also let $\pi:\supp(\mm)\to X'$ be the natural projection and define  $\sfd':[X']^2\to\R^+$ by
\begin{equation}
\label{eq:defquot}
\sfd'(\pi(x),\pi(y)):=\inf_{t\in\R}\sfd(x,\bar\X_t(y)),\qquad\forall x,y\in \supp(\mm).
\end{equation}
\end{definition}
From the fact that $(\bar\X_t)$ is a one-parameter group of isometries it is immediate to see that the definition of $\sfd'$  is well posed, i.e. that the right hand side of \eqref{eq:defquot} depends only on $\pi(x),\pi(y)$. 
\begin{proposition}\label{prop:basexp} With the same notation and assumptions as in Definition \ref{def:xprimo} the following holds.
$(X',\sfd')$ is a complete, separable and geodesic metric space, and the topology induced by $\sfd'$ is the quotient topology.
\end{proposition}
\begin{proof}
Symmetry, triangle inequality and the fact that $\sfd'(x',x')=0$ for any $x'\in  X'$ are obvious. Let $x',y'\in X'$ and notice that the map $t\mapsto \sfd(x,\X_t(y))$ is continuous and goes to $+\infty$ as $|t|\to+\infty$. Hence it has a minimum and from this fact and the definition \eqref{eq:defquot} it easily follows that for any couple $x',y'\in X'$ and $x\in\pi^{-1}(x')$ we can find $y\in\pi^{-1}(y')$  such that $\sfd(x,y)=\sfd'(x',y')$. In particular, $\sfd'(x',y')=0$ implies  $x'=y'$ so that  $(X',\sfd')$ is a metric space.

The separability follows from the one of $(X,\sfd)$. To prove completeness let $(x_n')\subset X'$ be such that $\sum_i\sfd'(x_i',x_{i+1}')<\infty$ and find $(x_n)\subset \supp(\mm)$ such that $\pi(x_n)=x'_n$ and $\sfd(x_n,x_{n+1})=\sfd'(x_n',x_{n+1}')$ for any $n\in\N$ (the above discussion ensures that such $x_n$'s can be found). Then $(x_n)$ is a Cauchy sequence in $\supp(\mm)$ and thus it converges to a limit $x\in  \supp(\mm)$. Given that $\sfd'(\pi(x),\pi(x_n))\leq \sfd(x,x_n)\to 0$, completeness is proved.

To prove that it is geodesic, let $x',y'\in X'$ and  find $x\in\pi^{-1}(x'),y\in \pi^{-1}(y')$ with   $\sfd(x,y)=\sfd'(x',y')$. Let $\gamma$ be a geodesic connecting $x$ to $y$. Then the fact that $\pi:(X,\sfd)\to (X',\sfd')$ is 1-Lipschitz yields
\[
\sfd'(\pi(\gamma_t),\pi(\gamma_s))\leq \sfd(\gamma_t,\gamma_s)=|s-t|\sfd(x,y)=|s-t|\sfd'(x',y'),\qquad\forall t,s\in[0,1],
\]
i.e., the curve $t\mapsto \pi(\gamma_t)$ is a geodesic connecting $x'$ to $y'$.

For the final statement, let $\tau$ be the topology on $\supp(\mm)$ induced by $\sfd$, $\tau_{\sfd'}$ be the topology on $X'$ induced by $\sfd'$ and $\tau_{\pi}$ the quotient topology on $X'$. Since $\pi:(\supp(\mm),\sfd)\to (X',\sfd')$ is 1-Lipschitz, it is continuous from $(\supp(\mm),\tau)$ to $(X',\tau_{\sfd'})$ and thus $\tau_{\sfd'}\subset\tau_{\pi}$. 

We claim that $\pi:(\supp(\mm),\tau)\to(X,\tau_{\sfd'})$ is open. Indeed, let $U\in\tau$, pick $x'\in\pi(U)$ and let $(x_n')\subset X'$ be such that $\sfd'(x',x_n')\to 0$ and $\sum_i\sfd'(x_{i}',x_{i+1}')<\infty$: to get the claim it is sufficient to prove that eventually $x_n'\in\pi(U)$. Let, as before,  $(x_n)\subset  \supp(\mm)$ be such that  $\pi(x_n)=x_n'$ and $\sfd(x_n,x_{n+1})=\sfd'(x_n',x_{n+1}')$ for any $n\in\N$. The completeness of $(\supp(\mm),\sfd)$ grants that there exists a limit point $x\in \supp(\mm)$ of $(x_n)$ and the continuity of $\pi:(\supp(\mm),\tau)\to (X,\tau_{\sfd'})$ that $\pi(x)=x'$. The assumption $x'\in \pi(U)$ gives that $\pi^{-1}(x)\cap U\neq\emptyset$, thus the fact that $\pi(x)=x'$ implies the existence of $t\in\R$ such that $\X_t(x)\in U$. Since $U$ is open and $\X_t:(\supp(\mm),\sfd)\to(\supp(\mm),\sfd)$ an isometry, for $n\in\N$ large enough it holds $\X_t(x_n)\in U$. Thus for $n\in\N$ large enough $x_n'\in \pi(U)$, as desired.

To conclude the proof pick now $V\in\tau_{\pi}$, so that $\pi^{-1}(V)$ is open in $\supp(\mm)$ and observe that by what we just proved $V=\pi(\pi^{-1}(V))$ belongs to $\tau_{\sfd'}$. Hence $\tau_{\pi}\subset \tau_{\sfd'}$ and the proof is completed.
\end{proof}
The projection map $\pi:X\to X'$ has a natural right inverse:
\begin{definition}[The embedding of $X'$ into $X$]\label{def:iota} With the same notation and assumptions as in Definition \ref{def:xprimo},  the map $\iota:X'\to X$ is defined by
\[
\iota(x')=x,\textrm{ provided $\pi(x)=x'$ and $\b(x)=0$},
\]
\end{definition}
Corollary \ref{cor:min} allows us to prove the main result of this chapter:
\begin{theorem}[The quotient space isometrically embeds into the original one]\label{thm:embed}\ \linebreak With the same notation and assumptions as in Definition \ref{def:xprimo},
$\iota$ is an isometric embedding of $(X',\sfd')$ into $(X,\sfd)$.
\end{theorem}
\begin{proof}
Let $x',y'\in X'$ and $x:=\iota(x')$, $y:=\iota(y')$. By definition of $\sfd'$ and $\iota$ it certainly holds $\sfd'(x',y')\leq \sfd(x,y)$. To prove the converse inequality amounts to prove that the minimum of the function $f(t):=\sfd(x,\X_t(y))$ is attained at $t=0$. Notice that by definition we have $x,y\in\supp(\mm)$ and for $\eps>0$ let $\mu_\eps,\nu_\eps\in\probt X$ be given by $\mu_\eps:=\mm(B_\eps(x))^{-1}\mm\restr{B_\eps(x)}$, $\nu_\eps:=\mm(B_\eps(y))^{-1}\mm\restr{B_\eps(y)}$.

Define $f_\eps(t):=W_2(\mu_\eps,(\X_t)_\sharp\nu_\eps)$, notice that $f_\eps$ is 1-Lipschitz and that the inequality
\[
\begin{split}
\big|\sfd^2(x_1,y_1)-\sfd^2(x,\X_t(y))\big|&\leq \big|\sfd(x_1,y_1)-\sfd(x,\X_t(y))\big|\big|\sfd(x_1,y_1)+\sfd(x,\X_t(y))\big|\\
&\leq 2\eps(2\eps+2t+2\sfd(x,y)),\qquad\forall t\in\R,\ x_1\in B_\eps(x),\ y_1\in B_\eps(\X_t(y)),
\end{split}
\]
yields
\[
\begin{split}
\left|\sqrt{\int\sfd^2(x_1,y_1)\,\d\ggamma(x_1,y_1)}-\sfd(x,\bar\X_t(y))\right|&\leq \sqrt{\int\left|\sfd^2(x_1,y_1)-\sfd^2(x,\bar\X_t(y))\right|\d\ggamma(x_1,y_1)}\\
&\leq\sqrt{ 2\eps(2\eps+2t+2\sfd(x,y))},
\end{split}
\]
where $\ggamma$ is any transport plan from $\mu_\eps$ to $(\X_t)_\sharp\nu_\eps$. Hence
\[
|f_\eps(t)-f(t)|\leq \sqrt{2\eps(2\eps+2t+2\sfd(x,y))},\qquad\forall t\in\R.
\]
By definition, it holds $|\int \b\,\d\mu_\eps|,|\int\b\,\d\nu_\eps|\leq \eps$, thus letting $t_\eps$ be the minimum of $f_\eps$, Corollary \ref{cor:min} and the fact that $\int\b\,\d(\X_t)_\sharp\nu_\eps=\int\b\,\d\nu_\eps+t$ for any $t\in\R$  yield $|t_\eps|\leq 2\eps$.

Thus for any $t\in\R$ we have
\[
\begin{split}
f(0)&\leq \sqrt{2\eps(2\eps+2\sfd(x,y))}+f_\eps(0)\\
&\leq \sqrt{2\eps(2\eps+2\sfd(x,y))}+f_\eps(t_\eps)+|t_\eps|\\
&\leq \sqrt{2\eps(2\eps+2\sfd(x,y))}+f_\eps(t)+2\eps\\
&\leq \sqrt{2\eps(2\eps+2\sfd(x,y))}+f(t)+\sqrt{2\eps(2\eps+2t+2\sfd(x,y))}+2\eps,
\end{split}
\] 
so that letting $\eps\downarrow0$ we conclude $f(0)\leq f(t)$ for any $t\in\R$, as desired.
\end{proof}
\begin{remark}[The role of infinitesimal Hilbertianity]\label{rem:whyinfhil}{\rm As the proof shows, Theorem \ref{thm:embed} relies only Corollary \ref{cor:min}, which in turn heavily depends on the infinitesimal Hilbertianity assumption. To see why, let's perform the same kind of computation on a smooth Finsler manifold $F$. Assume that there is a smooth function $\b$ on $F$ whose gradient flow $\bar\X$ is a one-parameter group of isometries from $F$ to itself. Fix points $x,y\in X$, let $f(z):=\frac12\sfd^2(z,x)$ and assume that  the minimum of $\R\ni t\mapsto f(\bar\X_t(y))$ is attained at $t=0$. Let $(x_s)$ be a geodesic from $x$ to $y$ and notice that with the same arguments as above we get that the minimum of $\R\ni t\mapsto f(\bar\X_t(x_s))$ is attained at $t=0$ for every $s\in[0,1]$.

Pretending that for each $s\in[0,1]$  the map  $t\mapsto f(\bar\X_t(x_s))$ is smooth near 0, the minimality of $x_s$ gives
\[
0=\frac{\d}{\d t}f(\bar\X_t(x_s))\restr{t=0}= Df(\bar\X_0'(x_s))=-Df(\nabla\b)(x_s).
\]
On the other hand and again neglecting smoothness issues, we have the trivial identity  $x_s'=\frac1s \nabla f(x_s)$, thus for the derivative of $\b$ along $x_s$ we have
\[
\frac{\d}{\d s}\b(x_s)=\frac1s D\b(\nabla f)(x_s).
\]
The problem is now evident: from the fact that $Df(\nabla\b)(x_s)=0$ we cannot deduce\linebreak $D\b(\nabla f)(x_s)=0$ on a general Finsler manifold. Indeed, the identity $Df(\nabla g)=Dg(\nabla f)$ is true for arbitrary smooth $f,g$ if and only if the manifold is actually Riemannian. 

The identification of differentials and gradients (i.e. the symmetry relation \eqref{eq:simm}) is precisely what makes the argument of the proof of Corollary \ref{cor:min} work. 
}\fr\end{remark}
Theorem \ref{thm:embed} has a number of simple consequences about the structure of $X'$. We start defining the natural maps from $X'\times\R$ to $X$ and viceversa.
\begin{definition}[From $X'\times\R$ to $X$ and viceversa]\label{def:maumad}  With the same notation and assumptions as in Definitions \ref{def:xprimo} and \ref{def:iota},
the maps $\mau:X'\times\R\to \supp(\mm)$ and $\mad:\supp(\mm)\to X'\times\R$ are defined by
\[
\begin{split}
\mau(x',t)&:=\bar\X_t(\iota(x')),\\
\mad(x)&:=(\pi(x),\b(x)).
\end{split}
\]
\end{definition}

\begin{proposition}[$\mau$ and $\mad$ are homeomorphisms]\label{prop:homeo}
With the same notation and assumptions as in Definition \ref{def:xprimo}, the maps $\mau,\mad$ are homeomorphisms each one inverse of the other which satisfy
\begin{equation}
\label{eq:bilip}
\begin{split}
\frac1{\sqrt 2}\sqrt{\sfd'(x_1',x_2')^2+|t_1-t_2|^2}\leq \sfd\big(\mau(x'_1,t_1),\mau(x'_2,t_2)\big)&\leq\sqrt 2\sqrt{\sfd'(x_1',x_2')^2+|t_1-t_2|^2},\\
\end{split}
\end{equation}
for any $x_1',x_2'\in X'$, $t_1,t_2\in\R$.
\end{proposition}
\begin{proof} It is clear that $\mau\circ\mad={\rm Id}_{\supp(\mm)}$ and $\mad\circ\mau={\rm Id}_{X'\times\R}$, thus we only need to prove \eqref{eq:bilip}. 

For the first inequality notice that since both  $\pi:(\supp(\mm),\sfd)\to (X',\sfd')$ and $\b:(\supp(\mm),\sfd)\to (\R,\sfd_{\rm Eucl})$ are 1-Lipschitz, it holds
\[
\begin{split}
\sfd\big(\mau(x'_1,t_1),\mau(x'_2,t_2)\big)^2\geq\max\{\sfd'(x_1',x_2')^2,|t_1-t_2|^2\}\geq\frac12\big(\sfd'(x_1',x_2')^2+|t_1-t_2|^2\big).
\end{split}
\]
The second follows from:
\[
\begin{split}
\sfd\big(\mau(x'_1,t_1),\mau(x'_2,t_2)\big)&=\sfd\big(\X_{t_1}(\iota(x'_1)),\X_{t_2}(\iota(x'_2))\big)\\
&= \sfd\big(\X_{t_1-t_2}(\iota(x'_1)),\iota(x'_2)\big)\\
&\leq  \sfd\big(\X_{t_1-t_2}(\iota(x'_1)),\iota(x'_1)\big)+ \sfd\big(\iota(x'_1),\iota(x'_2)\big)\\
&=|t_1-t_2|+\sfd'(x'_1,x'_2)\\
&\leq \sqrt 2\sqrt{\sfd'(x_1',x_2')^2+|t_1-t_2|^2}.
\end{split}
\]
\end{proof}
We can now introduce the natural measure on $X'$ as follows:
\begin{definition}[The measure $\mm'$]\label{def:mmp} With the same notation and assumptions as in Definition \ref{def:xprimo}, 
 we define the measure $\mm'$ on $(X',\sfd')$ as:
\[
\mm'(E):=\mm\big(\pi^{-1}(E)\cap \b^{-1}([0,1])\big),\qquad\forall E\subset X'\ \textrm{Borel}.
\]
\end{definition}
Notice that the definition is well posed because from Proposition \ref{prop:homeo} we know that for $E\subset X'$ Borel the set $\pi^{-1}(E)\subset X$ is also Borel.

The fact that $\bar\X_t$ preserves $\mm$ easily grants that $\mau,\mad$ are measure preserving:
\begin{proposition}[$\mau$ and $\mad$ are measure preserving]\label{prop:prodmeas}
With the same notation and assumptions as in Definitions \ref{def:xprimo}, \ref{def:mmp}, 
we have $\mau_\sharp(\mm'\times\mathcal L^1)=\mm$ and  $\mad_\sharp\mm=\mm'\times\mathcal L^1$.
\end{proposition}
\begin{proof}
It is sufficient to   prove that $\mad_\sharp\mm=\mm'\times\mathcal L^1$. Given that both $\mm'\times\mathcal L^1$ and $\mad_\sharp\mm$ are Borel measures defined on the product space $X'\times\R$, to prove that they coincide it is sufficient (see e.g. Corollary 1.6.3 in \cite{Cohn80}) to prove that for any Borel set $E\subset X'$ and any interval $I\subset \R$ it holds
\begin{equation}
\label{eq:rettangoli}
\mad_\sharp\mm(E\times I)=\mm'(E)\mathcal L^1(I).
\end{equation}
By definition of $\mm'$, this is true if $I=[0,1)$ and the identity 
\[
\bar\X_a^{-1}(\mau(E\times[a,a+1)))=\mau (E\times[0,1)),\qquad\forall a\in\R,
\]
together with the fact that $(\X_a)_\sharp\mm=\mm$ shows that  \eqref{eq:rettangoli} also holds for $I$ of the kind $[a,a+1)$ for any $a\in\R$. Then using the fact that $(\X_{1/2})_\sharp\mm=\mm$ and the trivial identities
\[
E\times [0,1)=(E\times [0,1/2))\cup(E\times [1/2,1)),\qquad (E\times [0,1/2))\cap(E\times [1/2,1))=\emptyset, 
\]
we deduce that \eqref{eq:rettangoli} holds for $I=[0,1/2)$, and then again using  $(\X_a)_\sharp\mm=\mm$, that it holds for all intervals of the kind $[a,a+1/2)$ , $a\in\R$.

Continuing this way by bisections, we deduce the validity of \eqref{eq:rettangoli} for $I=[a,a+k/2^n)$ for $a\in\R$ and $k,n\in \N$. Then a simple approximation argument gives \eqref{eq:rettangoli} for any interval $I\subset \R$, and thus the conclusion.
\end{proof}
The metric information given by Theorem \ref{thm:embed}  and the measure theoretic one given by Proposition \ref{prop:prodmeas} grant natural relations   between Sobolev functions on $X$ and   $X'$. To emphasize the fact that the minimal weak upper gradients depend on the space and to help keeping track of spaces themselves, we write $|\nabla f|_X$ (resp. $|\nabla f|_{X'}$) for functions $f\in \s^2_{\rm loc}(X)$ (resp. in $\s^2_{\rm loc}(X')$).  Notice that we use the notation $|\nabla f|_{X'}$ even if for the moment we  don't know that $(X',\sfd',\mm')$ is infinitesimally Hilbertian:  this will be soon evident once the following proposition is proved, see Corollary \ref{cor:xp}.
\begin{proposition}\label{prop:sezioni1} With the same notation and assumptions as in Definitions \ref{def:xprimo}, \ref{def:mmp}, the following holds.
\begin{itemize}
\item[i)]  Let $f\in \s^2_{\rm loc}(X)$ and for $t\in\R$ let $f^{(t)}:X'\to\R$ be given by $f^{(t)}(x'):=f(\mau(x',t))$. Then for $\mathcal L^1$-a.e. $t$ it holds $f^{(t)}\in \s^2_{\rm loc}(X')$ and
\[
|\nabla f^{(t)}|_{X'}(x')\leq |\nabla f|_X(\mau(x',t)),\qquad \mm'\times\mathcal L^1\ae\ (x',t)\in X'\times\R.
\]
\item[ii)] Let $g\in \s^2_{\rm loc}(X')$ and define $f:X\to\R$ by $f(x):=g\circ\pi$. Then $f\in \s^2_{\rm loc}(X)$ and
\begin{equation}
\label{eq:nablaaperto2}
|\nabla f|_{X}(x)=|\nabla g|_{X'}(\pi(x)),\qquad\mm\ae\ x\in X.
\end{equation}
\end{itemize}
\end{proposition}
\begin{proof} For $f:X\to\R$  denote by $\lip_X(f):X\to[0,+\infty]$ its local Lipschitz constant in the space $(X,\sfd)$ and similarly for $g:X'\to\R$,  $\lip_{X'}(g):X'\to [0,+\infty]$ is its local Lipschitz constant of in the space $(X',\sfd')$.

\noindent{$\mathbf{(i)}$} We have the simple inequality
\begin{equation}
\label{eq:lipaperte}
\begin{split}
\lip_X(f)(x)&=\lims_{y\to x}\frac{|f(x)-f(y)|}{\sfd(x,y)}\geq \lims_{y\to x\atop \b(y)=\b(x)}\frac{|f(x)-f(y)|}{\sfd(x,y)}\\
&=\lims_{y\to x\atop \b(y)=\b(x)}\frac{|f^{(\b(x))}(\pi(x))-f^{(\b(x))}(\pi(y))|}{\sfd'(\pi(x),\pi(y))}=\lip_{X'}(f^{(\b(x))})(\pi(x)).
\end{split}
\end{equation}
Now notice that with a truncation and cut-off argument and thanks to the local nature of the claim it is not restrictive to assume that $f\in W^{1,2}(X)$. According to Theorem \ref{thm:stronglip} there exists a sequence $(f_n)\subset L^2(X)$ of Lipschitz functions such that $f_n\to f$  and $\lip_X(f_n)\to|\nabla f|_X$ in $L^2(X)$ as $n\to\infty$. Up to pass to a subsequence - not relabeled - we can further assume that $\sum_n\|f_n-f_{n+1}\|_{L^2}<\infty$ and $\sum_n\|\lip_X(f_n)-|\nabla f|_X\|_{L^2}<\infty$ which, taking into account Proposition \ref{prop:prodmeas}, easily implies that for $\mathcal L^1$-a.e. $t$ we have $f_n(\mau(t,\cdot))\to f(\mau(t,\cdot))$ and $\lip_X(f_n)(\mau(t,\cdot))\to|\nabla f|_X(\mau(t,\cdot))$ in $L^2(X')$ as $n\to\infty$.

Fix such $t$ and apply inequality \eqref{eq:lipaperte} to the function $f_n$ on $\b^{-1}(t)$, then let $n\to\infty$ and recall the inequality $|\nabla g|_{X'}\leq \lip_{X'}(g)$ valid for any Lipschitz function $g$ (inequality \eqref{eq:lipweak}) and the lower semicontinuity of minimal weak upper gradients to conclude.

\noindent{$\mathbf{(ii)}$} Let $R>0$ and $\nchi:\R\to[0,1]$ be a Lipschitz cut-off function  with compact support and identically 1 on $[-R,R]$. Thanks to the local nature of the thesis we can assume that $g\in W^{1,2}(X')$ and prove that $f\nchi\circ\b\in W^{1,2}(X)$ with the identity  \eqref{eq:nablaaperto2} being true for $\mm$-a.e. $x\in\b^{-1}([-R,R])$.

The argument is similar to the one we already used: let $(g_n)$ be a sequence of Lipschitz functions on $X'$ such that  $g_n\to g$  and  $\lip_{X'}(g_n)\to|\nabla g|_{X'}$ in $L^2(X')$ and notice that by Proposition \ref{prop:prodmeas} the functions $f_n:=g_n\circ \pi\nchi\circ \b$ converge to $f\nchi\circ\b$ in $L^2(X)$. Now observe that for $x\in \b^{-1}\big((-R,R)\big)$ and $n\in\N$ it holds
\begin{equation}
\label{eq:lipaperte2}
\begin{split}
\lip_X( f_n)(x)=\lims_{y\to x}\frac{|f_n(y)-f_n(x)|}{\sfd(x,y)}\leq \lims_{y\to x}\frac{|g_n(\pi(y))-g_n(\pi(x))|}{\sfd'(\pi(x),\pi(y))}=
\lip_{X'}( g_n)(\pi(x)),
\end{split}
\end{equation}
and that the inequality \eqref{eq:leibbase} and the construction ensure that $(\lip_X(f_n))$ is bounded in $L^2(X)$. Thus up to pass to a subsequence - not relabeled - we can assume that $\lip_X(f_n)\to G$ weakly in $L^2(X)$ for some Borel function $G$. By the lower semicontinuity of minimal weak upper gradients and the $L^2(X)$-convergence of $f_n$ to $f\nchi\circ\b$ we deduce that $|\nabla(f\nchi\circ\b)|_X\leq G$ $\mm$-a.e. and by the locality property \eqref{eq:localgrad} that $|\nabla f|_X=|\nabla (f\nchi\circ\b)|_X$ $\mm$-a.e. on $\b^{-1}([-R,R])$. Passing to the limit in \eqref{eq:lipaperte2} we deduce that inequality $\leq$ holds in \eqref{eq:nablaaperto2}. The other inequality is a consequence of point $(i)$ of the statement, hence the proof is complete.
\end{proof}
It is now easy to prove that the quotient space is an infinitesimally Hilbertian $CD(0,N)$ space, the proof of the curvature-dimension bound being similar to the one given in \cite{Lott-Villani09} for the case of a compact group action on non-branching spaces. Notice that to get  the dimension reduction we  will need  a further argument (which will be easy once the product structure will be clear), described  in the last chapter.
\begin{corollary}[$(X',\sfd',\mm')$ is an infinitesimally Hilbertian $CD(0,N)$ space]\label{cor:xp}
With the same notation and assumptions as in Definitions \ref{def:xprimo} and \ref{def:mmp}, $(X',\sfd',\mm')$ is an infinitesimally Hilbertian $CD(0,N)$ space.
\end{corollary}
\begin{proof}$\ $\\
\noindent{\bf Infinitesimal Hilbertianity} Let $f',g'\in \s^2_{\rm loc}(X')$ and define $f,g$ as $f(x):=f'(\pi(x))$, $g(x):=g'(\pi(x))$. By Proposition \ref{prop:sezioni1} above we know that $f,g\in \s^2_{\rm loc}(X)$, hence, since $(X,\sfd,\mm)$ is infinitesimally Hilbertian, it holds
\[
|\nabla(f+g)|^2_X+|\nabla (f-g)|^2_X=2\big(|\nabla f|_X^2+|\nabla g|_X^2\big),\qquad\mm\ae.
\]
Then noticing that $(f\pm g)(x)=(f'\pm g')(\pi(x))$, using Proposition \ref{prop:prodmeas} and Fubini's theorem we deduce
\[
|\nabla(f'+g')|^2_{X'}+|\nabla (f'-g')|^2_{X'}=2\big(|\nabla f'|_{X'}^2+|\nabla g'|_{X'}^2\big),\qquad\mm'\ae,
\]
which, by the arbitrariness of $f',g'\in \s^2_{\rm loc}(X)$, yields the claim.

\noindent{\bf Curvature Dimension condition} A simple approximation argument shows that we can prove the curvature-dimension inequality \eqref{eq:defcd} only for given $\mu',\nu'\in\probt{X'}$ with bounded support and absolutely continuous w.r.t. $\mm'$. Fix such  $\mu',\nu'$, say $\mu'=\rho'\mm$ and $\nu'=\eta'\mm'$,  and define $\mu,\nu\in\probt X$ as 
\[
\begin{split}
\mu(x)&=\mau_\sharp(\mu'\times\mathcal L^1\restr{[0,1]}),\qquad\qquad\qquad\nu(x)=\mau_\sharp(\nu'\times\mathcal L^1\restr{[0,1]}),
\end{split}
\]
and notice that $\mu,\nu$ have bounded support and, by Proposition \ref{prop:prodmeas}, they are absolutely continuous w.r.t. $\mm$ with density  $\rho:=(\rho'\circ\pi)\restr{\b^{-1}([0,1])}$ and   $\eta:=(\eta'\circ\pi)\restr{\b^{-1}([0,1])}$ respectively. 

Choose $\ppi'\in\gopt(\mu',\nu')$. We claim that $\ppi:=\mathcal I_\sharp(\ppi'\times\mathcal L^1\restr{[0,1]})$ belongs to $\gopt(\mu,\nu)$, where $\mathcal I:C([0,1],X')\times[0,1]\to C([0,1],X)$ is defined by 
\[
(\mathcal I(\gamma',t_0))_t:=\mau(\gamma'_t,t_0),\qquad\forall \gamma'\in C([0,1],X'),\ t_0,t\in[0,1].
\] 
We shall prove this by direct comparison:  let $\aalpha\in\gopt(\mu,\nu)$ and put $\aalpha':=\mathcal P_\sharp\aalpha$, where $\mathcal P:C([0,1],X)\to C([0,1],X')$ is given by
\[
\mathcal P(\gamma)_t:=\pi(\gamma_t),\qquad\forall \gamma\in C([0,1],X),\ t\in[0,1].
\]
Since $\pi:\supp(\mm)\to X'$ is 1-Lipschitz we have $|\dot\gamma_t|\geq |\dot{\mathcal P(\gamma)}_t|$ for a.e. $t\in[0,1]$ and any $\gamma\in AC([0,1],X)$, and by construction it also holds $(\e_0)_\sharp\aalpha'=\mu'$, $(\e_1)_\sharp\aalpha'=\nu'$. Thus we have
\[
\iint_0^1|\dot\gamma_t|^2\,\d t\,\d\aalpha(\gamma)\geq \iint_0^1|\dot\gamma'_t|^2\,\d t\,\d\aalpha'(\gamma')\geq W_2^2(\mu',\nu').
\]
On the other hand,  the definition of $\mathcal I$ and Theorems \ref{thm:gfpresdist} and \ref{thm:embed} we have that $\ppi$-a.e. $\gamma$ is a geodesic which satisfies $\sfd(\gamma_0,\gamma_1)=\sfd'(\pi(\gamma_0),\pi(\gamma_1))$. Hence
\[
\iint_0^1|\dot\gamma_t|^2\,\d t\,\d\ppi(\gamma)=\int \sfd^2(\gamma_0,\gamma_1)\,\d\ppi(\gamma)=\int \sfd'^2(\gamma'_0,\gamma'_1)\,\d\ppi'(\gamma')= W_2^2(\mu',\nu').
\] 
Thus our claim is proved, that is: $\ppi\in\gopt(\mu,\nu)$.  Denoting by $\u_N(\cdot|\mm)$ and $\u_N(\cdot|\mm')$ the R\'enyi entropy functional on $\prob X$, $\prob{X'}$ respectively,  by the uniqueness part of Theorem \ref{thm:optmap} and the $CD(0,N)$ property of $(X,\sfd,\mm)$ we deduce
\begin{equation}
\label{eq:preproj}
\mathcal U_N\big((\e_t)_\sharp\ppi|\mm\big)\leq (1-t)\mathcal U_N\big((\e_0)_\sharp\ppi|\mm\big)+t\,\mathcal U_N\big((\e_1)_\sharp\ppi|\mm\big).
\end{equation}
Putting $\mu_t':=(\e_t)_\sharp\ppi'$, by construction we have
\begin{equation}
\label{eq:inv}
\mu_t:=(\e_t)_\sharp\ppi=\mau_\sharp(\mu_t'\times\mathcal L^1\restr{[0,1]}),\qquad\forall t\in[0,1].
\end{equation}
Again by Theorem \ref{thm:optmap}, we have that $\mu_t\ll\mm$, say $\mu_t=\rho_t\mm$, for any $t\in[0,1]$.  Then the identity \eqref{eq:inv} yields that $\mu'_t\ll\mm'$, say $\mu_t'=\rho_t'\mm$, $\rho_t(x)=\rho_t'(\pi(x))$ on $\b^{-1}([0,1])$ and $\rho_t(x)=0$ on $\b^{-1}(\R\setminus[0,1])$, for any $t\in[0,1]$. Hence for any $t\in[0,1]$ we have
\[
\begin{split}
\mathcal U_N(\mu_t|\mm)&=-\int_X\rho_t^{1-\frac1N}\,\d\mm=-\int_{\b^{-1}([0,1])}\rho_t^{1-\frac1N}\,\d\mm\\
&=-\int_{X'\times [0,1]} \rho_t^{1-\frac1N}\circ \mau\,\d(\mm'\times\mathcal L^1\restr{[0,1]})=-\int_{X'}(\rho_t')^{1-\frac1N}\,\d\mm'=\mathcal U_N(\mu_t'),
\end{split}
\]
for any $t\in[0,1]$. Thus the conclusion follows from \eqref{eq:preproj}.
\end{proof}

\chapter{``Pythagoras' theorem'' holds}\label{se:pit}
\section{Preliminary notions}
Let  $(X_1,\sfd_1,\mm_1)$ and   $(X_2,\sfd_2,\mm_2)$  be two metric measure spaces and consider the product space $(X_1\times X_2,\sfd_1\times\sfd_2,\mm_1\times\mm_2)$ where here and in the following the distance $\sfd_1\times\sfd_2$ is given by
\begin{equation}
\label{eq:proddist}
\sfd_1\times\sfd_2\big((x_1,x_2),(y_1,y_2)\big):=\sqrt{\sfd_1^2(x_1,y_2)+\sfd^2_2(x_2,y_2)},\qquad\forall x_1,y_1\in X_1,\ x_2,y_2\in X_2.
\end{equation}
It is quite natural to ask what are the relations between the Sobolev spaces on $X_1,X_2$ and that on $X_1\times X_2$. The general answer is not known, but in \cite{AmbrosioGigliSavare11-2}, \cite{AmbrosioGigliSavare12}  the following (surprisingly non trivial) result has been proved, which asserts that under the $RCD(K,\infty)$ condition it holds the same relation valid in the smooth world.

We shall adopt the following convention: given $f:X_1\times X_2\to\R$ and $x_1\in X_1$, the map $f^{(x_1)}:X_2\to\R$ is given by $f^{(x_1)}(x_2):=f(x_1,x_2)$, and similarly given $x_2\in X_2$, $f^{(x_2)}:X_1\to\R$ is given by $f^{(x_2)}(x_1):=f(x_1,x_2)$. Also, we write $|\nabla f|_{X_1},|\nabla f|_{X_2},|\nabla f|_{X_1\times X_2}$ to denote the minimal weak upper gradient for a Sobolev function defined on $X_1,X_2,X_1\times X_2$ respectively.
\begin{theorem}\label{thm:prod}
Let $(X_1,\sfd_1,\mm_1)$ and $(X_2,\sfd_2,\mm_2)$ be two $RCD(K,\infty)$ spaces and consider the product space $(X_1\times X_2,\sfd_1\times \sfd_2,\mm_1\times\mm_2)$. 

Then the product space is $RCD(K,\infty)$ as well and in particular it has the Sobolev-to-Lipschitz property. Furthermore, the following are equivalent:
\begin{itemize}
\item[i)]  $f\in W^{1,2}(X_1\times X_2)$
\item[ii)] for $\mm_1$-a.e. $x_1$ it holds $f^{(x_1)}\in W^{1,2}(X_2)$, for $\mm_2$-a.e. $x_2$ it holds $f^{(x_2)}\in W^{1,2}(X_1)$ and 
\[
\int_{X_1}\int_{X_2}|\nabla f^{(x_1)}|^2_{X_2}(x_2)\,\d\mm_2(x_2)\,\d\mm_1(x_1)+\int_{X_2}\int_{X_1}|\nabla f^{(x_2)}|^2_{X_1}(x_1)\,\d\mm_1(x_1)\,\d\mm_2(x_2)< \infty.
\]
\end{itemize}
Moreover, if these holds the equality
\begin{equation}
\label{eq:gradprod}
|\nabla f|_{X_1\times X_2}^2(x_1,x_2)=|f^{(x_2)}|^2(x_1)|\nabla f^{(x_1)}|_{X_2}^2(x_2)+|f^{(x_1)}|^2(x_2)|\nabla f^{(x_2)}|_{X_1}^2(x_1),
\end{equation}
is true for $\mm_1\times\mm_2$-a.e. $(x_1,x_2)$.
\end{theorem}

We shall reformulate the second part of  Theorem \ref{thm:prod} above in the following way, more convenient for our purposes:
\begin{corollary}\label{cor:sezioniprod} With the same notations and assumptions of Theorem \ref{thm:prod} the following are true.
\begin{itemize}
\item[i)] Let $f\in \s^2_{\rm loc}(X_1\times X_2)$. Then for $\mm_1$-a.e. $x_1$ it holds $f^{(x_1)}\in \s^2_{\rm loc}(X_2)$, for $\mm_2$-a.e. $x_2$ it holds $f^{(x_2)}\in \s^2_{\rm loc}(X_1)$ and the identity \eqref{eq:gradprod} holds.
\item[ii)] Let $f_1\in \s^2_{\rm loc}(X_1)$ and define $f:X_1\times X_2\to \R$ by $f(x_1,x_2):=f_1(x_1)$. Then $f\in \s^2_{\rm loc}(X_1\times X_2)$ and
\[
|\nabla f|_{X_1\times X_2}(x_1,x_2)=|\nabla f_1|_{X_1}(x_1),\qquad\mm_1\times\mm_2\ae \ (x_1,x_2).
\]
\item[iii)] Let $f_2\in \s^2_{\rm loc}(X_2)$ and define $f:X_1\times X_2\to \R$ by $f(x_1,x_2):=f_2(x_2)$. Then $f\in \s^2_{\rm loc}(X_1\times X_2)$ and
\[
|\nabla f|_{X_1\times X_2}(x_1,x_2)=|\nabla f_2|_{X_2}(x_2),\qquad\mm_1\times\mm_2\ae\ (x_1,x_2).
\]
\end{itemize}
\end{corollary}
\begin{proof}
All the properties follow from Theorem \ref{thm:prod} with a truncation and cut-off argument based on the locality property \ref{eq:localgrad} of minimal weak upper gradients.
\end{proof}

\section{Result}
Let us recall the notations that we shall use from now on
\begin{equation}
\label{eq:not}
\begin{split}
&(X,\sfd,\mm) \textrm{ is an infinitesimally Hilbertian $CD(0,N)$ space},\\
&\bar\gamma:\R\to\supp(\mm)\textrm{ is a line and $\b:=\b^+$ the corresponding Busemann function},\\
&(X',\sfd',\mm')\textrm{ is the quotient space as given by Definitions \ref{def:xprimo} and \ref{def:mmp}},\\
&\textrm{the maps $\mau,\mad$ from $X'\times\R$ to $ \supp(\mm)$ and viceversa  are given in Definition \ref{def:maumad}},\\
&\textrm{in the product space $(X'\times\R,\sfd'\times\sfd_{\rm Eucl},\mm'\times\mathcal L^1)$ the distance is defined by}\\
&\qquad\qquad\sfd'\times\sfd_{\rm Eucl}\big((x',t),(y's)\big)^2:=\sfd'(x',y')^2+|t-s|^2,\qquad\forall x',y'\in X',\ t,s\in\R.
\end{split}
\end{equation}

\bigskip

Aim of this section is to show that $\mau,\mad$ are isomorphisms of metric measure spaces, which according to Proposition \ref{prop:prodmeas} reduces to prove that 
\[
\sfd(\mau(x',t),\mau(y',s))^2=\sfd'(x',y')^2+|t-s|^2,\qquad\forall x',y'\in X',\ t,s\in\R.
\]
We will achieve this result by a duality argument based on Proposition \ref{prop:isom}.

\bigskip

It is a triviality that the standard definition of Sobolev space $W^{1,2}(\R)$ coincides with the one given by the formula \eqref{eq:w12} in the metric measure space $(\R,\sfd_{\rm Eucl},\mathcal L^1)$, and that for $f\in W^{1,2}(\R)$ its minimal weak upper gradient coincides with the modulus of its distributional derivative.  To keep consistency of the notation we shall denote this object by $|\nabla f|_\R$.

Arguing as in the proof of Proposition \ref{prop:sezioni2} we get the following result.
\begin{proposition}\label{prop:sezioni2} With the same notation as in \eqref{eq:not} the following holds.
\begin{itemize}
\item[i)] Let $f\in \s^2_{\rm loc}(X)$ and for $x'\in X'$ let $f^{(x')}:\R\to\R$ be given by $f^{(x')}(t):=f(\mau(x',t))$. Then  for $\mm'$-a.e. $x'$ it holds $f^{(x')}\in \s^2_{\rm loc}(\R)$ and
\[
|\nabla f^{(x')}|_{\R}(t)\leq |\nabla f|_X(\mau(x',t)),\qquad \mm'\times\mathcal L^1\ae\ (x',t)\in X'\times\R.
\]
\item[ii)] Let $h\in \s^2_{\rm loc}(\R)$ and define $f:X\to\R$ by $f(x):=h\circ\b$. Then $f\in \s^2_{\rm loc}(X)$ and
\[
|\nabla f|_{X}(x)=|\nabla h|_{\R}(\b(x)),\qquad\mm\ae \ x\in X.
\]
\end{itemize}
\end{proposition}
\begin{proof}
The same arguments used in the proof of Proposition \ref{prop:sezioni1} can be applied also in this case recalling that the following are true:
\begin{itemize}
\item[-] for any $t,s\in \R$ it holds $|t-s|=\min_{x\in \b^{-1}(t),\ y\in \b^{-1}(s)}\sfd(x,y)$,
\item[-] for any $x\in \supp(\mm)$ the map $t\mapsto \bar\X_t(x)$ provides an isometric embedding of $\R$ in $X$,
\item[-] it holds $\mau_\sharp(\mm'\times\mathcal L^1)=\mm$ and $\mad_\sharp\mm=\mm'\times\mathcal L^1$.
\end{itemize}
We omit the details.
\end{proof}
Propositions \ref{prop:sezioni1}, \ref{prop:sezioni2} and Corollary \ref{cor:sezioniprod} are the basis of our proof of the fact that right composition with $\mad$ provides an isometry from $W^{1,2}(X'\times\R)$ to $W^{1,2}(X)$. In order to clarify the argument we introduce the following class of functions:
\[
\begin{split}
\mathcal G&:=\Big\{g:X'\times\R\to\R\ :\  g(x',t)=\tilde g(x')\textrm{ for some }  \tilde g\in \s^2(X')\cap L^\infty(X') \Big\},\\
\mathcal H&:=\Big\{h:X'\times\R\to\R\ :\  h(x',t)=\tilde h(t)\textrm{ for some }  \tilde h\in \s^2(\R)\cap L^\infty(\R) \Big\}.
\end{split}
\]
Notice that both $\mathcal G$ and $\mathcal H$ are algebras, i.e. are closed w.r.t. linear combinations and products.

Using Corollary \ref{cor:sezioniprod} and Proposition \ref{prop:sezioni1} we get
\begin{equation}
\label{eq:gradg}
g\in \mathcal G\qquad\Rightarrow\qquad\left\{\begin{array}{l} g\in \s^2_{\rm loc}(X'\times\R),\ g\circ\mad\in \s^2_{\rm loc}(X)\textrm{ and }\\
\\
|\nabla g|_{X'\times\R}\circ\mad=|\nabla (g\circ\mad)|_X\quad \mm\ae.
\end{array}\right.
\end{equation}
Similarly,  Corollary \ref{cor:sezioniprod} and Proposition \ref{prop:sezioni2} give
\begin{equation}
\label{eq:gradh}
h\in \mathcal H\qquad\Rightarrow\qquad\left\{\begin{array}{l} h\in \s^2_{\rm loc}(X'\times\R),\ h\circ\mad\in \s^2_{\rm loc}(X)\textrm{ and }\\
\\
|\nabla h|_{X'\times\R}\circ\mad=|\nabla (h\circ\mad)|_X\quad \mm\ae.
\end{array}\right.
\end{equation}
Now we introduce the algebra of functions $\mathcal A$ as:
\[
\mathcal A:=\textrm{ algebra generated by }\mathcal G\cup\mathcal H.
\]
Notice that $\mathcal A\subset \s^2_{\rm loc}(X'\times\R)$.

The proof of the fact that right composition with $\mad$ produces an isometry of $W^{1,2}(X'\times\R)$ into $W^{1,2}(X)$ is based on the following 3 facts:
\begin{itemize}
\item[1)] for $f\in\mathcal A$ it holds $f\circ\mad\in \s^2_{\rm loc}(X)$ and $|\nabla f|_{X'\times\R}\circ\mad=|\nabla(f\circ\mad)|_X$ $\mm$-a.e. (Proposition \ref{prop:astable}),
\item[2)] $\mathcal A\cap W^{1,2}(X'\times\R)$ is dense in $W^{1,2}(X'\times \R)$ (Proposition \ref{prop:approximation}),
\item[3)] the right composition with $\mad$ produces an homeomorphism of the spaces $W^{1,2}(X'\times\R)$ and $W^{1,2}(X)$ (Proposition \ref{prop:almost}).
\end{itemize}
In order to prove point $(1)$ above we shall need the following basic lemma, which asserts that two functions $g\in\mathcal G$ and $h\in\mathcal H$ have `orthogonal gradients' in $W^{1,2}(X'\times\R)$ and - after a right composition with $\mad$ - also in $W^{1,2}(X)$.
\begin{lemma}[Orthogonality relations] With the same notation as above, let $g\in \mathcal G$ and $h\in \mathcal H$.  Then it holds
\begin{equation}
\label{eq:perp1}
\la \nabla g,\nabla h\ra_{X'\times\R}=0,\qquad\mm'\times\mathcal L^1\ae,
\end{equation}
and
\begin{equation}
\label{eq:perp3}
\la\nabla( g\circ\mad),\nabla(h\circ\mad) \ra_X=0,\qquad\mm\ae.
\end{equation}
\end{lemma}
\begin{proof}
Let $\tilde g\in \s^2(X')\cap L^\infty(X')$ and $\tilde h\in \s^2(\R)\cap L^\infty(\R)$ be such that $g(x',t)=\tilde g(x')$ and $h(x',t)=\tilde h(t)$.
By Corollary \ref{cor:sezioniprod} we know that
\[
|\nabla(g+h)|^2_{X'\times\R}(x',t)= |\nabla \tilde g|_{X'}^2(x')+ |\nabla \tilde h|_{\R}^2(t),\qquad\mm'\times\mathcal L^1\ae \ (x',t).
\]
Thus taking into account formula \eqref{eq:gradprod}  we get
\[
\begin{split}
2\la g,h\ra_{X'\times\R}=|\nabla(g+h)|^2_{X'\times\R}-|\nabla g|^2_{X'\times\R}-|\nabla h|^2_{X'\times\R}=0,\qquad \mm'\times\mathcal L^1\ae,
\end{split}
\]
which is \eqref{eq:perp1}.

We pass to \eqref{eq:perp3}. The chain rule \eqref{eq:chainhil} and the trivial identity $ h\circ\mad=\tilde h\circ\b$ yields
\[
\la \nabla ( g\circ\mad ),\nabla( h\circ\mad) \ra_X=\tilde h'\circ\b\la\nabla (g\circ\mad ),\nabla\b\ra_X,\qquad\mm\ae,
\]
hence to conclude  it is sufficient to show that 
\[
\la\nabla (g\circ\mad),\nabla\b\ra_X=0,\qquad\mm\ae.
\]
This identity is a consequence of the derivation rule   \eqref{eq:horver} applied to $f:=g\circ\mad$:  in this case the left hand side of \eqref{eq:horver} is identically 0 (acutally, in formula \eqref{eq:horver} the function $f$ is assumed to be in $\s^2(X)$, while here we only have $f\in \s^2_{\rm loc}(X)$ - the thesis is anyway true as shown by a simple truncation argument, we omit the details).
\end{proof}

\begin{proposition}\label{prop:astable} With the same notation as above, for every $f\in\mathcal A$ it holds $f\circ\mad\in \s^2_{\rm loc}(X,\sfd,\mm)$ with
\[
|\nabla f|_{X'\times\R}\circ\mad=|\nabla (f\circ\mad)|_{X},\qquad \mm\ae.
\]
\end{proposition}
\begin{proof}
A generic element $f$ of $\mathcal A$ can be written as $f=\sum_{i\in I}g_ih_i$ for some finite set $I$ of indexes and functions $g_i\in \mathcal G$, $h_i\in \mathcal H$, $i\in I$.

Pick such $f$ and use  the infinitesimal Hilbertianity of $X'\times\R$ (Theorem \ref{thm:prod}) to get that $\mm'\times\mathcal L^1$-a.e. it holds
\begin{equation}
\label{eq:stanchino}
\begin{split}
|\nabla f|^2_{X'\times\R}&=\sum_{i,j\in I} g_ig_j\la\nabla h_i,\nabla h_j\ra_{X'\times\R}+ g_ih_j\la\nabla h_i,\nabla g_j\ra_{X'\times\R}\\
&\qquad\qquad+ h_ig_j\la\nabla g_i,\nabla h_j\ra_{X'\times\R}+ h_ih_j\la\nabla g_i,\nabla g_j\ra_{X'\times\R}\\
&=\sum_{i,j\in I} g_ig_j\la\nabla h_i,\nabla h_j\ra_{X'\times\R}+h_ih_j\la\nabla g_i,\nabla g_j\ra_{X'\times\R},
\end{split}
\end{equation}
having used the orthogonality relation  \eqref{eq:perp1} in the second step. The identities in \eqref{eq:gradg}  and  \eqref{eq:gradh}  grant
\[
\begin{split}
\la\nabla h_i,\nabla h_j\ra_{X'\times\R}\circ\mad=\la\nabla (h_i\circ\mad),\nabla (h_j\circ\mad)\ra_{X},\\
\la\nabla g_i,\nabla g_j\ra_{X'\times\R}\circ\mad=\la\nabla (g_i\circ\mad),\nabla (g_j\circ\mad)\ra_{X},
\end{split}
\]
$\mm$-a.e. for any $i,j\in I$. Thus writing - to shorten the notation - $\bar g_i,\bar h_i$ in place of $g_i\circ\mad,h_i\circ\mad$ respectively, from \eqref{eq:stanchino} we have
\[
|\nabla f|^2_{X'\times\R}\circ\mad=\sum_{i,j\in I} \bar g_i\bar g_j\la\nabla\bar  h_i,\nabla \bar h_j\ra_{X'\times\R}+\bar h_i\bar h_j\la\nabla \bar g_i,\nabla\bar  g_j\ra_{X'\times\R}.
\]
Using  the orthogonality relation  \eqref{eq:perp3} and the fact that $X$ is infinitesimally Hilbertian we can do the same computations done in \eqref{eq:stanchino} in reverse order to get
\[
\begin{split}
|\nabla f|^2_{X'\times\R}\circ\mad&=\sum_{i,j\in I} \bar g_i\bar g_j\la\nabla\bar  h_i,\nabla\bar  h_j\ra_{X}+ \bar g_i\bar h_j\la\nabla\bar  h_i,\nabla\bar  g_j\ra_{X}\\
&\qquad+ \bar h_i\bar g_j\la\nabla\bar  g_i,\nabla \bar h_j\ra_{X}+ \bar h_i\bar h_j\la\nabla\bar  g_i,\nabla \bar g_j\ra_{X}=|\nabla (f\circ\mad)|_X^2,
\end{split}
\]
$\mm$-a.e., which is the thesis.
\end{proof}

\begin{proposition}\label{prop:approximation} With the same notation as above, the set $\mathcal A\cap W^{1,2}(X'\times\R)$ is dense in $W^{1,2}(X'\times\R)$.
\end{proposition}
\begin{proof} With a diagonalization argument it is sufficient to prove that for $f\in W^{1,2}(X'\times\R)$ bounded with compact support there exists a sequence $(f_n)\subset\mathcal A\cap W^{1,2}(X'\times\R)$ converging to $f$ in $W^{1,2}(X'\times\R)$. Fix such $f$ and for $n\in\N$ and $i\in \Z$ define
\[
g_{i,n}(x'):=n\int_{i/n}^{(i+1)/n}f(x',s)\,\d s,
\]
and
\[
h_{i,n}(t):=\nchi_n(t-i/n),
\]
where $\nchi_n:\R\to\R$ is given by
\[
\nchi_n(t):=\left\{\begin{array}{ll}
0,&\qquad\textrm{ if }t<-1/n,\\
nt+1,&\qquad\textrm{ if }-1/n\leq t< 0,\\
1-nt,&\qquad\textrm{ if }0\leq t<1/n,\\
0,&\qquad\textrm{ if }1/n<t.\\
\end{array}
\right.
\]
Then define $f_n:X'\times\R\to\R$ by
\begin{equation}
\label{eq:deffn}
f_n(x',t):=\sum_{i\in \Z}h_{i,n}(t)g_{i,n}(x').
\end{equation}
We claim that $f_n\in \mathcal A\cap W^{1,2}(X'\times\R)$. Indeed, given that the support of $f$ is compact only a finite number of terms in the the right hand side of \eqref{eq:deffn} is different from 0. By construction it is also obvious that $h_{i,n}$ and $g_{i,n}$ are both bounded with compact support and that $h_{i,n}$ is Lipschitz.  By Theorem \ref{thm:prod} we know that $f^{(t)}\in\s^2(X')$ for a.e. $t$ thus by  the convexity and $L^2$-lower semicontinuity of the $W^{1,2}$-norm (by the lower semicontinuity of minimal weak upper gradients) we get $g_{i,n}\in\s^2(X')$ with
\begin{equation}
\label{eq:venti}
\int_{X'}|\nabla g_{i,n}|^2_{X'}\,\d\mm'\leq n\int_{X'}\int_{i/n}^{(i+1)/n}|\nabla f^{(t)}|^2_{X'}\,\d t\,\d\mm'.
\end{equation}
Thus $f_n\in \mathcal A\cap W^{1,2}(X'\times\R)$, as claimed.

Now we claim that $f_n\to f$ in $L^{2}(X'\times\R)$ as $n\to\infty$. Integrating the inequality
\[
\begin{split}
\big(f_n(x',t)\big)^2&=\left(\sum_{i\in \Z}h_{i,n}(t)g_{i,n}(x')\right)^2\\
&\leq\sum_{i\in \Z}h_{i,n}(t)\big(g_{i,n}(x')\big)^2\leq \sum_{i\in \Z}h_{i,n}(t)n\int_{i/n}^{(i+1)/n}f^2(x',s)\,\d s,
\end{split}
\]
on $x'$ and $t$ we obtain $\|f_n\|_{L^2(X'\times\R)}\leq \|f\|_{L^2(X'\times\R)}$, $\forall n\in\N$.

Hence $L^2$-convergence will follow if we show that
\begin{equation}
\label{eq:weakl2}
\lim_{n\to\infty }\int \varphi f_n \,\d\mm'\,\d \mathcal L^1=\int \varphi f\,\d\mm'\,\d \mathcal L^1,
\end{equation}
for any $\varphi:X'\times\R\to\R$ Lipschitz with compact support. To check this start observing that 
\begin{equation}
\label{eq:perapp}
\begin{split}
\varphi(x',t)  f_n(x',t)&=\sum_{i\in\Z} \varphi(x',t)h_{i,n}(t)g_{i,n}(x')=n\sum_{i\in\Z} \varphi(x',t)h_{i,n}(t)\int_{\frac in}^{\frac{i+1}{n}}f(x',s)\,\d s\\
&=n\sum_{i\in\Z} h_{i,n}(t)\int_{\frac in}^{\frac{i+1}{n}}\varphi(x',s)f(x',s)\,\d s+{\rm Rem}_{n}(x',t),
\end{split}
\end{equation}
where the reminder term ${\rm Rem}_n(x',t)$ is bounded by
\begin{equation}
\label{eq:remapp}
\begin{split}
\Big|{\rm Rem}_{n}(x',t)\Big|&=\left|n\sum_{i\in\Z} h_{n,i}(t)\int_{\frac in}^{\frac{i+1}{n}}(\varphi(x',t)-\varphi(x',s))f(x',s)\,\d s\right|\\
&\leq n\Lip(\varphi)\sum_{i\in\Z} h_{n,i}(t)\int_{\frac in}^{\frac{i+1}{n}}|t-s|\,|f|(x',s)\,\d s\\
&\leq \Lip(\varphi)\sum_{i\in\Z} h_{n,i}(t)\int_{\frac in}^{\frac{i+1}{n}}|f|(x',s)\,\d s.
\end{split}
\end{equation}
Now integrate \eqref{eq:perapp} and use  the identity 
\[
\int_{X'\times\R} n\sum_{i\in\Z} h_{n,i}(t)\int_{\frac in}^{\frac{i+1}{n}}\varphi(x',s)f(x',s)\,\d s\,\d t\,\d \mm'(x')=\int_{X'\times\R}\varphi(x',s)f(x',s)\,\d s\,\d\mm'(x),
\]
to get
\[
\begin{split}
\int_{X'\times\R}\varphi(x',s) f_n(x',s)\,\d s\,\d\mm'(x)=&\int_{X'\times\R}\varphi(x',s)f(x',s)\,\d s\,\d\mm'(x)\\
&+\int_{X'\times\R}{\rm Rem}_n(x',t)\,\d t\,\d\mm'(x').
\end{split}
\]
From \eqref{eq:remapp} we obtain
\[
\left|\int_{X'\times\R}{\rm Rem}_n(x',t)\,\d t\,\d\mm'(x')\right|\leq \int_{X'\times\R}|{\rm Rem}_n(x',t)|\,\d t\,\d\mm'(x')\leq\frac{2\Lip(\varphi)\|f\|_{L^1(X'\times\R)}}{n},
\]
hence \eqref{eq:weakl2} follows.

Taking into account the $L^2$-lower semicontinuity of the $W^{1,2}$-norm again  and the uniform convexity of $W^{1,2}(X'\times\R)$ (consequence of the infinitesimal Hilbertianity stated in Theorem \ref{thm:prod}), to conclude it is sufficient to show that 
\begin{equation}
\label{eq:perfinirew12}
\int |\nabla  f_n|^2_{X'\times \R}\,\d\mathcal L^1\,\d\mm'\leq\int |\nabla  f|^2_{X'\times \R}\,\d\mathcal L^1\,\d\mm',\qquad\forall n\in\N.
\end{equation}
Notice that for $\mathcal L^1$-a.e. $t\in [i/n,(i+1)/n]$ the function $f_n^{(t)}:X'\to\R$ is $\mm'$-a.e. well defined and given by the expression
\begin{equation}
\label{eq:expr1}
f_n^{(t)}=(1+i-nt)g_{i,n}+(nt-i)g_{i+1,n},
\end{equation}
therefore by Theorem \ref{thm:prod} we know that $f^{(t)}_n\in W^{1,2}(X')$ for $\mathcal L^1$-a.e. $t$, and therefore from \eqref{eq:expr1} we get the bound
\[
\begin{split}
|\nabla f_n^{(t)} |_{X'}^2&\leq \big((1+i-nt)|\nabla g_{i,n}|_{X'}+(nt-i)|\nabla g_{i+1,n}|_{X'}\big)^2\\
&\leq(1+i-nt)|\nabla g_{i,n}|_{X'}^2+(nt-i)|\nabla g_{i+1,n}|_{X'}^2,
\end{split}
\]
which, together with \eqref{eq:venti}, gives
\begin{equation}
\label{eq:gradhor}
\begin{split}
\int_{X'\times\R}|\nabla f_n^{(t)} |_{X'}^2(x')\,\d(\mm'\times\mathcal L^1)(x',t)&\leq \frac1n\sum_{i\in\Z}\int_{X'}|\nabla g_{i,n}|_{X'}^2(x')\,\d\mm'( x')\\
&\leq \sum_{i\in\Z}\int_{X'}\int_{i/n}^{(i+1)/n}|\nabla f^{(t)}|_{X'}^2(x')\,\d t\,\d\mm'(x')\\
&=\int_{X'\times\R}|\nabla f^{(t)}|^2_{X'}(x')\,\d(\mm'\times\mathcal L^1)(x',t).
\end{split}
\end{equation}
Similarly, for $\mm'$-a.e. $x'\in X'$ the function $f_n^{(x')}:\R\to\R$ is $\mathcal L^1$-a.e. well defined and given by
\begin{equation}
\label{eq:expr2}
f_n^{(x')}(t)=(1+i-nt)g_{i,n}(x')+(nt-i)g_{i+1,n}(x'),\qquad\mathcal L^1\ae \  t\in[i/n,(i+1)/n].
\end{equation}
Arguing as before we get that $f_n^{(x')}\in W^{1,2}(\R)$ for $\mm'$-a.e. $x'$, so that \eqref{eq:expr2} gives
\[
\begin{split}
\int_{i/n}^{(i+1)/n}|\nabla f_n^{(x')}|^2_{\R}(t)\,\d t&=\int_{i/n}^{(i+1)/n}n^2\big( g_{i+1,n}(x')-g_{i,n}(x')\big)^2\,\d t\\
&=n\big( g_{i+1,n}(x')-g_{i,n}(x')\big)^2\\
&=n^3\left(\int_{(i+1)/n}^{(i+2)/n}f(x',t)\,\d t-\int_{i/n}^{(i+1)/n}f(x',t)\,\d t\right)^2\\
&=n^3\left(\int_{i/n}^{(i+1)/n}f^{(x')}(t+1/n)-f^{(x')}(t)\,\d t\right)^2\\
&\leq n^3\left(\int_{i/n}^{(i+1)/n}\int_t^{t+1/n}|\nabla f^{(x')}|_\R(s)\,\d s\,\d t\right)^2\\
&\leq n\int_{i/n}^{(i+1)/n}\int_t^{t+1/n}|\nabla f^{(x')}|_\R^2(s)\,\d s\,\d t,
\end{split}
\]
which after integration yields
\[
\int_{X'\times\R}|\nabla f_n^{(x')}|^2_{\R}(t)\,\d t\,\d\mm'(x') \leq \int_{X'\times\R}|\nabla f^{(x')}|^2_{\R}(t)\,\d t\,\d\mm'(x').
\]
This inequality, \eqref{eq:gradhor} and Theorem \eqref{thm:prod} give \eqref{eq:perfinirew12} and thus the conclusion.
\end{proof}

\begin{proposition}\label{prop:almost} With the same notation as above, it holds $f\in W^{1,2}(X'\times\R)$ if and only if $f\circ\mad\in W^{1,2}(X)$ and in this case it holds
\begin{equation}
\label{eq:almost1}
\frac1{\sqrt2}\||\nabla f|\|_{L^2(X'\times\R)}\leq\||\nabla(f\circ\mad)|\|_{L^2(X)}\leq\sqrt 2\||\nabla f|\|_{L^2(X'\times\R)}.
\end{equation}
\end{proposition}
\begin{proof} Direct consequence of Lemma \ref{le:contrdual}, the identity   $\mau_\sharp(\mm'\times\mathcal L^1)=\mm$, inequalities \eqref{eq:bilip} and   the fact that if a distance is scaled by a factor $\lambda$, the corresponding gradient part of the Sobolev norm is scaled by $\frac1\lambda$, which is a direct consequence of the definitions.
\end{proof}
The main theorem of this chapter now follows easily.
\begin{theorem}[``Pythagoras' theorem'' holds]\label{thm:pitagora} With the same notation and assumptions as in \eqref{eq:not}, the maps $\mau,\mad$ are isomorphisms of metric measure spaces.
\end{theorem}
\begin{proof} We know by Theorems \ref{thm:link2} and \ref{thm:prod} that both $X$ and $X'\times\R$ have the Sobolev-to-Lipschitz property and thus to conclude  we can  apply  Proposition \ref{prop:isom}. We also already know that  $\mau,\mad$ are measure preserving.  To conclude it is therefore sufficient to prove that   $f\in W^{1,2}(X'\times\R)$ if and only if $f\circ\mad\in W^{1,2}(X)$ and in this case it holds
\begin{equation}
\label{eq:splitting}
\||\nabla (f\circ\mad)|_X\|_{L^{2}(X)}=\||\nabla f|_{X'\times\R}\|_{L^{2}(X'\times\R)}.
\end{equation}
Pick  $f\in W^{1,2}(X'\times\R)$ and notice that by Proposition \ref{prop:approximation} we know that there exists a sequence $(f_n)\subset \mathcal A\cap W^{1,2}(X'\times\R)$ converging to $f$ in $W^{1,2}(X'\times\R)$ and the second inequality in \eqref{eq:almost1} yields that $f_n\circ \mad,f\circ\mad\in W^{1,2}(X)$ with  $(f_n\circ \mad)$ converging to $f\circ\mad$ in $W^{1,2}(X)$.

Proposition \ref{prop:astable} tells that
\[
|\nabla f_n|_{X'\times\R}\circ\mad=|\nabla ( f_n\circ\mad)|_X,\qquad\mm\ae,
\]
hence squaring, integrating and passing to the limit as $n\to\infty$ we get \eqref{eq:splitting}. Viceversa, if $f:X'\times\R$ is such that $f\circ\mad\in W^{1,2}(X)$, the first inequality in \eqref{eq:almost1} grants that $f\in W^{1,2}(X'\times\R)$ and the above argument can be repeated.
\end{proof}

\chapter{The quotient space has dimension $N-1$}\label{se:dim}
\section{Preliminary notions}
We recall the following basic result about Hausdorff dimension on $CD(0,N)$ spaces:
\begin{proposition}\label{prop:hausd}
Let $(X,\sfd,\mm)$ be a $CD(0,N)$ space. Then the Hausdorff dimension of\linebreak $(\supp(\mm),\sfd)$ is bounded above by  $N$.
\end{proposition}
More generally, the same conclusion holds on $CD(K,N)$ spaces, the proof being based on the Bishop-Gromov volume estimate, see \cite{Sturm06II} for a proof.

We shall also make use of the following simple result about the behavior of geodesics in product spaces:
\begin{proposition}\label{prop:geodprod}
Let $(X_1,\sfd_1)$ and $(X_2,\sfd_2)$ be two complete and separable metric spaces, $\mu_i,\nu_i\in\probt{X_i}$ and  $\ppi_i\in\gopt(\mu_i,\nu_i)$, $i=1,2$. 

Define  $\mathcal J:C([0,1],X_1)\times C([0,1],X_2)\to C([0,1],X_1\times X_2)$ by
\[
\mathcal J(\gamma_1,\gamma_2)_t:=(\gamma_{1,t},\gamma_{2,t}),
\]
and the plan $\ppi_1\otimes\ppi_2\in\prob{C([0,1],X_1\times X_2)}$ as $\mathcal J_\sharp(\ppi_1\times\ppi_2)$.

Then $\ppi_1\otimes\ppi_2\in\gopt(\mu_1\times\mu_2,\nu_1\times\nu_2)$, where $X_1\times X_2$ is endowed of the product distance $\sfd_1\times\sfd_2$  defined as in formula \eqref{eq:proddist}.
\end{proposition}
\begin{proof}
It is clear that $(\e_0)_\sharp(\ppi_1\otimes\ppi_2)=\mu_1\times\mu_2$ and $(\e_1)_\sharp(\ppi_1\otimes\ppi_2)=\nu_1\times\nu_2$ and that the map $\mathcal J$ sends $\geo(X_1)\times\geo(X_2)$ into $\geo(X_1\times X_2)$.  Now let $\varphi_1,\varphi_2$ be Kantorovich potentials relative to $(\mu_1,\nu_1)$ and $(\mu_2,\nu_2)$ respectively and define $\varphi:X_1\times X_2\to\R\cup\{-\infty\}$ by $\varphi(x_1,x_2):=\varphi_1(x_1)+\varphi_2(x_2)$. It is immediate to verify that $\partial^c\varphi=\partial^c\varphi_1^c\times\partial^c\varphi_2$ (after the appropriate permutation of coordinates) and therefore  $(\e_0,\e_1)_\sharp(\ppi_1\otimes\ppi_2)$ is concentrated on $\partial^c\varphi$, which is sufficient to conclude. 
\end{proof}

The proof of the dimension reduction is based on the following very simple  statement. Notice that the proposition below is actually a particular case of a more general statement used in \cite{Cavalletti-Sturm12} (see also \cite{Cavalletti12}) that has been used to grant a sort of dimension reduction in the setting of reduced curvature-dimension  bounds.
\begin{proposition}\label{prop:dim}
Let $N\geq 2$, $t\in[0,1]$, and $a,b,c\geq 0$ be three given non-negative numbers. Assume that for every $\alpha,\beta\in\Q$, $\alpha,\beta> 0$ it holds
\begin{equation}
\label{eq:N}
\left(\frac c{(1-t)\alpha+t\beta}\right)^{-\frac1N}\geq (1-t)\left(\frac a{\alpha}\right)^{-\frac1N}+t\left(\frac b{\beta}\right)^{-\frac1N}.
\end{equation}
Then:
\[
c^{-\frac1{N-1}}\geq(1-t)a^{-\frac1{N-1}}+tb^{-\frac1{N-1}}.
\]
\end{proposition}
\begin{proof} If $a$ or $b$ are 0 the thesis is obvious. Also, the terms in  \eqref{eq:N} are continuous in $\alpha,\beta$ positive. Thus if \eqref{eq:N} holds for positive rationals, it also holds for positive reals. Conclude picking  $\alpha:= a^{-\frac1{N-1}}$ and $\beta:=b^{-\frac1{N-1}}$.
\end{proof}

\section{Result}

\begin{theorem}[The quotient space has dimension $N-1$]\label{thm:dimm1} With the same notation and assumptions as in \eqref{eq:not} the following holds.
\begin{itemize}
\item[i)] If $N\geq 2$, then $(X',\sfd',\mm')$ is an infinitesimally Hilbertian $CD(0,N-1)$ space.
\item[ii)] If $N\in[1,2)$, then $X'$ contains exactly one point.
\end{itemize}
\end{theorem}
\begin{proof} $\ $\\
\noindent{$\mathbf{ (i)}$}
We already know that $(X',\sfd',\mm')$ is an  infinitesimally Hilbertian $CD(0,N)$ space and a simple approximation argument ensures that to conclude it is sufficient to check the $CD(0,N-1)$ condition for given   $\mu_0,\mu_1\in\probt{X'}$ with bounded support and absolutely continuous w.r.t. $\mm'$, say $\mu_i=\rho_i\mm'$, $i=0,1$. By Theorem \ref{thm:optmap} we know that there exists a unique $\ppi\in\gopt(\mu_0,\mu_1)$, and by Corollary \ref{cor:cdpunt} that the measures $\mu_t:=(\e_t)_\sharp\ppi$ are absolutely continuous w.r.t. $\mm'$, say $\mu_t=\rho_t\mm'$, for every $t\in[0,1]$.

Let $\alpha,\beta>0$  be arbitrary, put $\nu_0:=\frac1\alpha\mathcal L^1\restr{[0,\alpha]}$, $\nu_1:=\frac1\beta\mathcal L^1\restr{[0,\beta]}$ so that $\nu_0,\nu_1\in\probt \R$, let $t\mapsto\nu_t=\frac{1}{(1-t)\alpha+t\beta}\mathcal L^1\restr{[0,(1-t)\alpha+t\beta]}$ be the unique geodesic connecting $\nu_0$ to $\nu_1$ and $\tilde\ppi$  the unique element of $\gopt(\nu_0,\nu_1)$. 

By Proposition \ref{prop:geodprod},  the plan $\ppi\otimes\tilde\ppi$ belongs to $\gopt(\mu_0\times\nu_0,\mu_1\times\nu_1)$ and by definition  satisfies  $(\e_t)_\sharp(\ppi\otimes\ppi')=\mu_t\times\nu_t$  and thus
\begin{equation}
\label{eq:densprod}
\frac{\d(\e_t)_\sharp(\ppi\otimes\ppi')}{\d(\mm'\times\mathcal L^1)}(\gamma_t,\tilde\gamma_t)=\frac{\rho_t(\gamma_t)}{(1-t)\alpha+t\beta},\qquad\ppi\times\ppi'\ae \ (\gamma,\tilde\gamma).
\end{equation}
By assumption we know that $(X,\sfd,\mm)$ is an infinitesimally Hilbertian $CD(0,N)$ space and by Theorem \ref{thm:pitagora} that it is isomorphic to $(X'\times\R,\sfd'\times\sfd_{\rm Eucl},\mm'\times\mathcal L^1)$. Thus the latter is an infinitesimally Hilbertian $CD(0,N)$ space and Theorem \ref{thm:optmap} and its proof grant that $\ppi\otimes\tilde\ppi$ is concentrated on a set of non-branching geodesics. Thus by \eqref{eq:densprod} and standard means in optimal transport theory (we omit the details), by the $CD(0,N)$ property we get
\[
\left(\frac{\rho_t(\gamma_t)}{(1-t)\alpha+t\beta}\right)^{-\frac1N}\geq(1-t)\left(\frac{\rho_0(\gamma_0)}{\alpha}\right)^{-\frac1N}+t\left(\frac{\rho_1(\gamma_1)}{\beta}\right)^{-\frac1N},\qquad\ppi\ae \ \gamma.
\]
Given that $\alpha,\beta$ were arbitrary positive numbers, we further obtain that
\[
\left(\frac{\rho_t(\gamma_t)}{(1-t)\alpha+t\beta}\right)^{-\frac1N}\geq(1-t)\left(\frac{\rho_0(\gamma_0)}{\alpha}\right)^{-\frac1N}+t\left(\frac{\rho_1(\gamma_1)}{\beta}\right)^{-\frac1N},\qquad\forall \alpha,\beta\in\Q,\ \alpha,\beta>0,
\]
holds for $\ppi$-a.e. $\gamma$. By Proposition \ref{prop:dim} we deduce 
\[
\rho_{t}(\gamma_{t})^{-\frac1{N-1}}\geq (1-t)\rho_0(\gamma_0)^{-\frac1{N-1}}+t\rho_1(\gamma_1)^{-\frac1{N-1}},\qquad\ppi\ae \ \gamma,
\]
which integrated w.r.t. $\ppi$ yields $\u_{N-1}(\mu_t)\leq (1-t)\u_{N-1}(\mu_0)+t\u_{N-1}(\mu_1)$, as desired.

\noindent{$\mathbf{ (ii)}$}  It is clear that $X'$ is non empty. Assume by contradiction that it contains more than one point. Then, since $(X',\sfd')$ is geodesic - Proposition \ref{prop:basexp} - it contains an isometric copy $I\subset X'$ of some non-trivial interval in $\R$. Given that  $X'\times\R\supset I\times\R$, the Hausdorff dimension of $X'\times\R$ is at least 2. This contradicts Proposition \ref{prop:hausd} and the fact  (Theorem \ref{thm:pitagora}) that $(X'\times\R,\sfd'\times\sfd_{\rm Eucl})$ is isometric to $(\supp(\mm),\sfd)$.
\end{proof}

\appendix
%    Include appendix "chapters" here.

\chapter{Infinitesimal Hilbertianity and behavior of gradient flows}\label{app:infhil}
In this section we collect some comments about the relations between the infinitesimal nature of a Finsler/Riemannian manifold and the behavior of gradient flows of $K$-convex functionals defined on them, the discussion being taken from the paper \cite{OhSt12} by Ohta-Sturm and author's works on gradient flows in collaboration with Ambrosio and Savar\'e (in particular \cite{AmbrosioGigliSavare08} and \cite{AmbrosioGigliSavare11-2}). What we want to show is the direct relation between the Riemannian nature of a manifold and the $K$-EVI (=Evolution Variational Inequality) formulation of gradient flows, which is at the basis of the relation between the infinitesimal Hilbertianity of a $CD(K,\infty)$ space and the existence of $K$-EVI gradient flows for the relative entropy. The establishment of such relation is the key result of \cite{AmbrosioGigliSavare11-2}, and the discussion we make here can be used by the interested reader as a guideline for understanding the key point of  such paper.

It is a well known fact of Riemannian geometry that a function is  $K$-convex if and only if its gradient flow $K$-exponentially decreases the distance. We now check how this works in a Finsler context. Let $(F,\|\cdot\|_x)$  be a $C^1$ Finsler manifold, i.e. a $C^1$ differentiable manifold endowed with a norm $\|\cdot\|_x$ on each tangent space and such that in coordinates the squared norms have $C^1$ dependence on the base point. 

Let ${\rm Dual}_x:T_xF\to T^*_xF$ be given by
\[
{\rm Dual}_x(v_1)(v_2):=\lim_{\eps\to 0}\frac{\|v_1+\eps v_2\|_x^2-\|v_1\|_x^2}{2\eps},\qquad\forall v_1,v_2\in T_xF,
\]
i.e. let  ${\rm Dual}_x$ be the differential of $\frac{\|\cdot\|^2_x}2$. Its inverse ${\rm Dual}^{-1}_x:T_x^*F \to T_xF$ is then given by
\[
{\rm Dual}^{-1}_x(\omega_1)(\omega_2):=\lim_{\eps\to 0}\frac{(\|\omega_1+\eps \omega_2\|_x^*)^2-(\|\omega_1\|_x^*)^2}{2\eps},\qquad\forall \omega_1,\omega_2\in T^*_xF,
\]
where $\|\cdot\|^*_x$ is the dual norm of $\|\cdot\|_x$. The fact that $\|\cdot\|_x$ is smooth and strictly convex ensures that both ${\rm Dual}_x$ and ${\rm Dual}^{-1}_x$ are well defined single valued maps. It is crucial for the foregoing discussion to remark that ${\rm Dual}_x$ is linear if and only if the norm $\|\cdot\|_x$ comes from a scalar product. This can be checked by direct computations.

Given a $C^1$ function $f:F \to \R$, its differential $Df$ is the cotangent vector field defined by
\[
Df(x)(v):=\lim_{t\to0}\frac{f(\gamma_t)-f(\gamma_0)}{t},\qquad\textrm{for any $C^1$ curve $\gamma$ such that $\gamma_0=x$ and $\gamma_0'=v$,}
\]
and the gradient $\nabla f$ is the tangent vector field given by
\[
\nabla f(x):={\rm Dual}_x^{-1}(Df(x)).
\]
Noticing that for every $x\in F$ and any $v\in T_xF$ we have the inequality
\begin{equation}
\label{eq:pergr}
Df(x)(v)\leq \frac12(\|Df(x)\|_x^*)^2+\frac12\|v\|_x^2,
\end{equation}
one sees that the gradient $\nabla f$ at the point $x$ can be equivalently characterized as the only vector $v$ for which equality holds in \eqref{eq:pergr}, so that indeed the gradient indicates the direction of maximal increase of $f$.

\medskip

Let $\sfd$ be the distance on $F$ induced by the Finsler structure and $f:F\to\R$ a $C^1$ function. Then $f$ is $K$-convex provided it holds
\begin{equation}
\label{eq:kconv}
f\big(\gamma_t\big)\leq (1-t)f(\gamma_0)+tf(\gamma_1)-\frac K2t(1-t)\sfd^2(\gamma_0,\gamma_1),\qquad\forall \gamma\textrm{ geodesic, } t\in[0,1].
\end{equation}
Comparing the derivatives of $t\mapsto f(\gamma_t)$ at $t=0$ and $t=1$, we see that \eqref{eq:kconv} is equivalent to
\begin{equation}
\label{eq:kconv2}
Df(y)(\gamma_1')-Df(x)(\gamma_0')\geq K\sfd^2(x,y),\qquad\forall x,y\in F,\  \gamma\textrm{ geodesic from $x$ to $y$ }.
\end{equation}
Now define $\mathcal N_x:F\to\R$ as $\mathcal N_x(y):=\tfrac12\sfd^2(x,y)$ and recall (see e.g. Chapter 6 in \cite{BCS00}) that under general assumptions in a neighborhood of $x$ the function $\mathcal N_x$ is $C^1$ and that its differential is given by the formula
\begin{equation}
\label{eq:derdistf}
D\mathcal N_x(y)={\rm Dual}_y(\gamma_1'),
\end{equation}
where $\gamma:[0,1]\to F$ is the unique (minimal) geodesic connecting $x$ to $y$.

Now let $(x_t)$ and $(y_t)$ be gradient flow trajectories of $f$, i.e. assume they solve $x_t'=-\nabla f(x_t)$ and $y_t'=-\nabla f(y_t)$ and starting respectively from $x$ and $y$. Assume that $x,y$ are close enough so that formula \eqref{eq:derdistf} holds, let $\gamma$ be the unique geodesic from $x$ to $y$ and compute  the derivative of the squared distance between the flows:
\begin{equation}
\label{eq:derdistgf}
\begin{split}
-\frac\d{\d t}\tfrac12\sfd^2(x_t,y_t)\restr{t=0}&=-\frac\d{\d t}\mathcal N_x(y_t)\restr{t=0}-\frac\d{\d t}\mathcal N_y(x_t)\restr{t=0}\\
&= -D\mathcal N_x(y)(y_0')-D\mathcal N_y(x)(x_0')\\
&=-{\rm Dual}_y(\gamma_1')(y_0')+{\rm Dual}_x(\gamma_0')(x_0')\\
&={\rm Dual}_y(\gamma_1')(\nabla f(y))-{\rm Dual}_x(\gamma_0')(\nabla f(x))\\
&={\rm Dual}_y(\gamma_1')({\rm Dual}^{-1}_y(Df(y)))-{\rm Dual}_x(\gamma_0')({\rm Dual}^{-1}_x(Df(x))).
\end{split}
\end{equation}
Comparing the last term in this expression with the left-hand side of \eqref{eq:kconv2} amounts to compare $\omega(v)$ and ${\rm Dual}_z(v)({\rm Dual}^{-1}_z(\omega))$ for arbitrary $z\in F$, $v\in T_zF$, $\omega\in T^*_zF$. These two are different in general, because the former is bilinear in $v,\omega$, while, as said, the duality map ${\rm Dual}_z$ is linear if and only if the norm $\|\cdot\|_z$ comes from a scalar product.  If the norm is Hilbertian, then indeed the two are the same, so that we have
\[
\omega(v)={\rm Dual}_z(v)({\rm Dual}^{-1}_z(\omega)),\qquad\forall v\in T_zF,\ \omega\in T^*_zF\qquad\Leftrightarrow\qquad \|\cdot\|_z\textrm{ is Hilbertian}.
\]
Using this fact it is now possible to see that the last term in \eqref{eq:derdistgf} is bounded above by the left-hand side of \eqref{eq:kconv2} for any $f$ and $x,y$ close enough if and only the manifold is Riemannian (in which case the two actually coincide) and in this case we obtain
\[
\frac\d{\d t}\tfrac12\sfd^2(x_t,y_t)\restr{t=0}\leq -K\sfd^2(\gamma_0,\gamma_1),
\]
which leads, after an application of Gronwall's lemma, to the desired contraction estimate $\sfd(x_t,y_t)\leq e^{-Kt}\sfd(x,y)$.

Following these lines of thought, in \cite{OhSt12} it has been shown that if $\|\cdot\|$ does not come from a scalar product, then there exists a convex and smooth function whose gradient flow does not contract distances. It \cite{OhSt12}  has been also shown that the heat flow on $(\R^d,\|\cdot\|,\mathcal L^d)$ contracts the $W_2$-distance if and only if $\|\cdot\|$ comes from the scalar product, in accordance with the above discussion and the following facts:
\begin{itemize}
\item[-] The heat flow is the gradient flow of the relative entropy functional on $(\probt{\R^d},W_2)$
\item[-] For any norm on $\R^d$, the relative entropy functional is geodesically convex on\linebreak $(\probt{\R^d},W_2)$, where $W_2$ is the quadratic transportation distance built over the distance induced by the given norm
\item[-] The space $(\probt{\R^d},W_2)$ inherits several geometric properties of the underlying space $(\R^d,\|\cdot\|)$, so that - heuristically said - if $\|\cdot \|$ comes from a scalar product then $(\probt{\R^d},W_2)$ looks like an infinite dimensional Riemannian manifold, while if $\|\cdot\|$ is a generic norm then $(\probt{\R^d},W_2)$ looks like an infinite dimensional Finslerian manifold.
\end{itemize} 

\medskip

There is another way to look at the  problem  which arises when looking for $K$-contractivity of the gradient flow of  a $K$-convex function on a Finsler setting:  rather then studying  the distance between two gradient flow trajectories, we study the distance between a gradient flow trajectory $(x_t)$ and a fixed point $y$. Assuming points are close enough so that we can use formula \eqref{eq:derdistf} and with the same computations as above we have
\[
\frac{\d}{\d t}\sfd^2(x_t,y)={\rm Dual}_{x_t}(\gamma_{t,0}')(\nabla f(x_t)),
\]
where the curve $s\mapsto \gamma_{t,s}$ is the geodesic connecting $x_t$ to $y$ and the derivation $\gamma_{t,0}'$ in the formula is taken in the $s$ variable. Computing the derivative of $f$ along $s\mapsto\gamma_{t,s}$ we have
\[
\frac{\d}{\d s}f(\gamma_{t,s})\restr{s=0}=Df(x_t)(\gamma_{t,0}').
\]
Recalling formula \eqref{eq:derdistf} we have $D\mathcal N_y(x_t)=-{\rm Dual}_{x_t}(\gamma_{t,0}')$ and thus also $\nabla\mathcal N_y(x_t)=-\gamma_{t,0}'$, hence the above identities can be written as
\begin{equation}
\label{eq:perrcd}
\begin{split}
\frac{\d}{\d t}\sfd^2(x_t,y)&=-D\mathcal N_y(\nabla f)(x_t),\\
\frac{\d}{\d s}f(\gamma_{t,s})\restr{s=0}&=-Df(\nabla\mathcal N_y)(x_t).
\end{split}
\end{equation}
As before, the right-hand sides of these equalities coincide if and only if we can `swap differentials with gradients', which we can do for every smooth $f$, every $y$ and every initial point $x_0$ if and only if the norm comes from a scalar product. In this case we get
\[
\frac{\d}{\d t}\sfd^2(x_t,y)=\frac{\d}{\d s}f(\gamma_{t,s})\restr{s=0}.
\] 
If $f$ is $K$-convex, the right-hand side can be estimated from above by $f(y)-f(x_t)-\frac K2\sfd^2(x_t,y)$ and for the gradient flow trajectory $(x_t)$ we deduce
\[
\frac{\d}{\d t}\frac12\sfd^2(x_t,y)+f(x_t)+\frac K2\sfd^2(x_t,y)\leq f(y),\qquad\forall y\in F,\ t\geq 0.
\]
This inequality is the so called $K$-EVI formulation of gradient flows introduced in \cite{AmbrosioGigliSavare08}, whose basic properties are:
\begin{itemize}
\item[i)] It can be formulated on general metric spaces.
\item[ii)] It has very general stability properties both w.r.t. convergence of the initial datum and w.r.t. $\Gamma$-convergence of functionals (see \cite{AmbrosioGigliSavare08} and \cite{AmbrosioGigliSavare11-2})
\item[iii)] The distance between  two gradient flow trajectories satisfying the $K$-EVI decreasaes $K$-exponentially, as shown by:
\[
\begin{split}
\frac{\d}{\d t}\frac12\sfd^2(x_t,y_t)\restr{t=t_0}&=\frac{\d}{\d t}\frac12\sfd^2(x_t,y_{t_0})\restr{t=t_0}+\frac{\d}{\d t}\frac12\sfd^2(x_{t_0},y_t)\restr{t=t_0}\\
&\leq f(y_{t_0})-f(x_{t_0})-\frac K2\sfd^2(x_{t_0},y_{t_0})+f(x_{t_0})-f(y_{t_0})-\frac K2\sfd^2(x_{t_0},y_{t_0})\\
&= -K\sfd^2(x_{t_0},y_{t_0}).
\end{split}
\]
(written this way, the computation is not rigorous in the metric setting, but the result can be justified in full generality, see Chapter 4 of \cite{AmbrosioGigliSavare08})
\item[iv)] The existence of gradient flow trajectories in the $K$-EVI sense encodes \emph{both} the information about $K$-convexity of the functional (see \cite{DaneriSavare08}) \emph{and} the fact that the local nature of the space resembles that of an Hilbert space (there is no known way to make this statement rigorous in general, but at least point $(iii)$ above grants the $K$-contractivity). 
\end{itemize}
Having in mind Remark \ref{rem:averaging}, we can informally infer that if $(X,\sfd,\mm)$ is infinitesimally Hilbertian, then  the space of probability measures with bounded densities endowed with the distance $W_2$ behaves at the first order like a Riemannian manifold. Thus if we are considering a gradient flow  on $\probt X$ for which the weak maximum principle holds (which is the case for the heat flow) and we start with an initial datum $\mu_0$ such that $\mu_0\leq C\mm$ for some $C>0$, the evolution will take place on the set of measures with bounded densities, and thus following the above computations  we might expect infinitesimal Hilbertianity to be linked to the existence of gradient flows in the $K$-EVI sense.

This is precisely the heuristic idea behind the definition of  $RCD(K,\infty)$ spaces given in \cite{AmbrosioGigliSavare11-2}: as shown by the proof of the main theorem there, the implication `infinitesimal Hilbertianity and $CD(K,\infty)$ yields existence of $K$-EVI gradient flows of the relative entropy' is based precisely on computing the derivative of the squared distance between a flow and a fixed measure and  the derivative of the entropy along geodesics, like in formula \eqref{eq:perrcd}, and then on using infinitesimal Hilbertianity to  `swap differentials and gradients'.

\chapter{Infinitesimal Hilbertianity and behavior of the distance}

A different way of looking at infinitesimal Hilbertianity is to look at the property of the differential of the squared distance only, without referring to any further functional defined on space.

Let $(F,\|\cdot\|_x)$ be a Finsler manifold, $x\in F$, $\Omega\ni x$ be so small that $\sfd^2:\Omega^2\to\R$ is smooth and $(x_t),(y_t)$ two geodesics emanating from $x$. Then by formula \eqref{eq:derdistf} above we see that
\[
\frac{\d}{\d t}\frac12\sfd^2(x_t,y_1)\restr{t=0}=-{\rm Dual}_x(x'_0)(y'_0)\qquad\textrm{ and }\qquad\frac{\d}{\d t}\frac12\sfd^2(x_1,y_t)\restr{t=0}=-{\rm Dual}_x(y'_0)(x'_0).
\]
Arguing as before we obtain  that
\[
{\rm Dual}_x(v)(w)={\rm Dual}_x(w)(v),\quad\forall v,w\in T_xF\qquad\Leftrightarrow\qquad \|\cdot\|_x\textrm{ comes from a scalar product},
\]
indeed $\Leftarrow$ is obvious, and for $\Rightarrow$ notice that the map $v\mapsto {\rm Dual}_x(w)(v)$ is linear for any $w\in T_xF$, while  $v\mapsto {\rm Dual}_x(v)(w)$  is linear for any $w\in T_xF$ if and only if ${\rm Dual}_x$ is linear.

Therefore we see that the norm $\|\cdot\|_x$ comes from a scalar product if and only if the equality
\begin{equation}
\label{eq:distr}
\frac{\d}{\d t}\frac12\sfd^2(x_t,y_1)\restr{t=0}=\frac{\d}{\d t}\frac12\sfd^2(x_1,y_t)\restr{t=0}
\end{equation}
holds for any couple of geodesics $(x_t)$, $(y_t)$ as above.  In other words, a Finsler manifold is Riemannian if and only if the squared distance satisfies \eqref{eq:distr} for all sufficiently short geodesics emanating from the same point. Observe that this provides a link between the infinitesimal information `the norm on the tangent space comes from a scalar product' and the behavior of the squared distance between close points, which is a local information.

It is then natural to ask whether on an infinitesimally Hilbertian space a property like \eqref{eq:distr} holds or not. The answer is positive, provided we reformulate \eqref{eq:distr} via a lift to the space of measures with bounded densities, as we did in Chapter \ref{se:quot} (see in particular Remark \ref{rem:averaging}). Indeed, let $(X,\sfd,\mm)$ be an infinitesimally Hilbertian space and $(\mu_t)$, $(\nu_t)$ two geodesics with time-continuous densities and such that $\supp(\mu_t),\supp(\nu_t)$ have uniformly bounded support (as in Proposition \ref{prop:intdist})   and starting from the same measure $\mu$.  Then we know  that the functions $t\mapsto\frac12 W_2^2(\mu_t,\nu_1)$ and $t\mapsto\frac12 W_2^2(\mu_1,\nu_t)$ are $C^1$ and that it holds
\[
\frac{\d}{\d t}\frac12W_2^2(\mu_t,\nu_1)\restr{t=0}=-\int \la\nabla \psi,\nabla\varphi\ra\d\mu= \frac{\d}{\d t}\frac12W_2^2(\mu_1,\nu_t)\restr{t=0}
\] 
where  $\varphi$, $\psi$ are two Lipschitz Kantorovich potentials inducing $(\mu_t)$, $(\nu_t)$ respectively. 

\medskip

In this direction, it is interesting to remark that identity \eqref{eq:distr} has a sort of self improving property, given that if it holds the manifold is Riemannian and therefore the joint limit
\begin{equation}
\label{eq:angles}
\lim_{t,s\downarrow0}\frac{\sfd^2(x_t,x)+\sfd^2(y_s,x)-\sfd^2(x_t,y_s)}{2ts},
\end{equation}
exists, it  being the scalar product $\la x_0',y_0'\ra$. Whenever one is working on a metric space where the limit $L$ in \eqref{eq:angles} exists (for instance: Alexandrov spaces) it is customary to use it to define the \emph{angle between the geodesics} as $\theta({(x_t),(y_t)}):=\cos^{-1}(L/(|x_0'||y_0'|))$. Notice that at least in the smooth case, \eqref{eq:distr} is a particular case of the existence of the limit in \eqref{eq:angles}, because the former is equivalent to the fact that in \eqref{eq:angles} we can take first the limit as $t\downarrow0$ and then the limit as $s\downarrow0$ and obtain the same result we would get by taking limits in reverse order.

Given that in the smooth world \eqref{eq:distr} directly implies that the limit in \eqref{eq:angles} exists, it is natural to ask whether the same is true in the non-smooth one. We don't know if anything like this holds. A potential `averaged' version of \eqref{eq:angles} could be 
\[
\lim_{t,s\downarrow0}\frac{W_2^2(\mu_t,\mu)+W_2^2(\nu_s,\mu)-W_2^2(\mu_t,\nu_s)}{2ts},
\]
the point being that one would like this joint limit to exists for appropriate geodesics with time-continuous densities $(\mu_t)$, $(\nu_t)$ starting from the same measure $\mu$. Currently, this is known to be true only on finite dimensional Alexandrov spaces (\cite{Gigli11}) and in this case there is no need for any requirement about the absolute continuity of the $\mu_t$'s and $\nu_t$'s. However, the proof makes heavy use of the lower bound on the sectional curvature and cannot be generalized to infinitesimally Hilbertian spaces with a lowed bound on the Ricci curvature.

The only related result we are aware of about non-smooth spaces with a lower bound on the Ricci is in  the recent paper \cite{Honda11} by Honda, where he proved a weakened $\mm$-a.e. version of \eqref{eq:angles} on spaces which are limits of Riemannian manifolds with Ricci curvature uniformly bounded from below. Specifically, he proved that on a such space $(X,\sfd,\mm)$, for $\mm$-a.e. $x$ the following holds: for any two unit speed geodesics $(x_t)$, $(y_t)$ emanating from $x$ there exists the limit of
\begin{equation}
\label{eq:honda}
\lim_{t\downarrow0}\frac{2t^2-\sfd^2(x_t,y_t)}{2t^2}.
\end{equation}
Unfortunately, this seems a bit weaker than \eqref{eq:angles}, becasue  the limit in \eqref{eq:honda} exists even on normed spaces. In this sense \eqref{eq:honda}  does not encode the information about the local Riemannian structure of the space (it is unclear to me if  the arguments in \cite{Honda11}  can also be used to get  existence of angles as in \eqref{eq:angles}).

\chapter[Eulerian and Lagrangian points of view on lower Ricci curvature bounds ]{Eulerian and Lagrangian points of view on lower Ricci curvature bounds }\label{app:bochner}

Here we collect some comments about the links between the Bochner inequality and the synthetic treatment of lower Ricci curvature bounds, the discussion being inspired by Chapter 14 of \cite{Villani09} and by some conversations I had with Sturm. Recall that the reduced curvature dimension condition $CD^*(K,N)$ introduced in \cite{BacherSturm10} is defined as
\begin{definition}[Reduced curvature dimension condition]
Let $(X,\sfd,\mm)$ be a metric measure space, $K\in \R$ and $N\in[1,\infty)$.  $(X,\sfd,\mm)$ is said to be a $CD^*(K,N)$ space provided for any $\mu_0,\mu_1\in\probt X$ with $\supp(\mu_0),\supp(\mu_1)\subset\supp(\mm)$ there exists $\ppi\in\gopt(\mu_0,\mu_1)$ such that
\begin{equation}
\label{eq:cdstar}
\u_{N'}((\e_t)_\sharp\ppi)\leq-\int\sigma^{(1-t)}_{K,N'}\big(\sfd(\gamma_0,\gamma_1)\big)\rho^{-\frac1{N'}}(\gamma_0)+\sigma^{(t)}_{K,N'}\big(\sfd(\gamma_0,\gamma_1)\big)\eta^{-\frac1{N'}}(\gamma_1)\,\d\ppi(\gamma),
\end{equation}
holds for every $t\in[0,1]$ and every $N'\geq N$, where $\mu=\rho\mm+\mu^s$, $\nu=\eta\mm+\nu^s$ with $\mu^s,\nu^s\perp\mm$. Here the distortion coefficients $\sigma^{(t)}_{K,N}$ are given by
\[
\sigma^{(t)}_{K,N}(\theta):=\left\{
\begin{array}{ll}
+\infty,&\qquad\textrm{ if }K\theta^2\geq N\pi^2,\\
\frac{\sin(t\theta\sqrt{K/N})}{\sin(\theta\sqrt{K/N})}&\qquad\textrm{ if }0<K\theta^2 <N\pi^2,\\
t&\qquad\textrm{ if }K\theta^2=0,\\
\frac{\sinh(t\theta\sqrt{-K/N})}{\sinh(\theta\sqrt{-K/N})}&\qquad\textrm{ if }K\theta^2 <0.
\end{array}
\right.
\]
In the case $N=\infty$ the definition is the same as the one of $CD(K,\infty)$ spaces (Definition \ref{def:cdinfty}).
\end{definition}
As explained in \cite{BacherSturm10}, the relations between the $CD^*(K,N)$ condition - which is stable w.r.t. mGH-convergence - and the standard $CD(K,N)$ one are the following: 
\begin{itemize}
\item[i)] A space satisfies $CD^*(K,N)$ locally if and only if it satisfies $CD(K,N)$ locally (at least in the non-branching case).
\item[ii)] For the $CD^*(K,N)$ condition it is possible to prove the local-to-global property (again, at least in the non-branching case). The same is not known for the $CD(K,N)$ condition.
\item[iii)] With the current knowledge, the $CD^*(K,N)$ condition produces the same kind of inequalities given by the $CD(K,N)$ one (like Bishop-Gromov, Bonnet-Myers ecc..) but with slightly suboptimal constants. 
\end{itemize}
Thus apart from the issue mentioned in point $(iii)$, one could use the $CD^*(K,N)$ condition as substitute for the standard $CD(K,N)$ one. By point $(i)$ the problem of showing that actually $CD^*(K,N)$ yields sharp estimates is equivalent to show that the local to global property holds for the $CD(K,N)$ condition. This crucial problem is currently open, but recently Cavalletti in \cite{Cavalletti12} made important progresses in this direction.

As pointed out to me by Sturm, given the above it is not surprising that the Bochner inequality is linked to the $CD^*(K,N)$ condition - as we shall soon see - rather than to the $CD(K,N)$ one: in both cases Ricci curvature `acts in every direction', while  the $CD(K,N)$ condition encodes the fact that `it does not act in the direction of motion' (see e.g. the introduction of \cite{Sturm06II} and Chapter 14 of \cite{Villani09}).

\bigskip

We shall perform some formal computation on  a smooth Finsler manifold $(F,\|\cdot\|_x,\mm)$. Denote by $\sfd$ the distance induced by the family of norms $\{\|\cdot\|_x\}_{x\in F}$ and recall that under general assumptions, for any $\mu_0,\mu_1\in\probt F$ with $\mu_0=\rho_0\mm$ there exists a unique geodesic $(\mu_t)$ from $\mu_0$ to $\mu_1$ and also a unique lifting $\ppi\in\prob{C([0,1],F)}$ of $(\mu_t)$. By  the Brenier-McCann theorem on Finlser setting (the metric Brenier theorem \ref{thm:brenmetr} is sufficient) we know that  for any smooth Kantorovich potential $\varphi$ inducing $(\mu_t)$ we have $\sfd(\gamma_0,\gamma_1)=\|\nabla\varphi(\gamma_0)\|_{\gamma_0}$ for every $\gamma\in\supp(\ppi)$.

Noticing that the distortion coefficients $\sigma^{(t)}_{K,N}(\theta)$ satisfy the differential equation
\[
\frac{\d^2}{\d t^2}\sigma^{(t)}_{K,N}(\theta)+\theta^2\frac KN\sigma^{(t)}_{K,N}(\theta)=0,
\]
taking two derivatives of $t\mapsto\u_N(\mu_t)$ and with a comparison argument we see that \eqref{eq:cdstar} holds if and only if
\begin{equation}
\label{eq:defcds}
\partial_{tt}\u_N(\mu_t)\restr{t=0}\geq\frac KN\int\rho_0^{1-\frac 1N}\|\nabla\varphi\|^2\,\d{\rm vol},
\end{equation}
holds for every geodesic $(\mu_t)$ as before.

Now pick $\varphi\in C^\infty_c(F)$ and observe that with arguments similar to those presented in Theorem 13.5 in \cite{Villani09} and Lemma 1.34 in \cite{AmbrosioGigli11} valid in a Riemannian context, for $\eps>0$ sufficiently small the function $\psi:=\eps\varphi$ is $\frac{\sfd^2}{2}$-concave. Let $\rho$ be a smooth probability density and notice that the $\frac{\sfd^2}{2}$-concavity of $\psi$ ensures that the curve $[0,1]\ni t\mapsto \exp(-t\nabla\psi)_\sharp(\rho\mm)=\rho_t\mm$ is a geodesic. The evolution of $(\rho_t)$ is driven by
\begin{equation}
\label{eq:cont}
\partial_t\rho_t+\nabla\cdot(\rho_t\nabla\psi_t)=0,
\end{equation}
where $[0,1]\ni t\mapsto\psi_t$ solves
\begin{equation}
\label{eq:HJ}
\partial_t\psi_t+\frac{\|\nabla\psi_t\|^2}{2}=0,
\end{equation}
with $\psi_0:=-\psi$ (see for instance Chapter 2 of \cite{AmbrosioGigli11} or Chapter 7 of \cite{Villani09}).

Using \eqref{eq:cont}, \eqref{eq:HJ} one easily gets, by explicit computation, that 
\begin{equation}
\label{eq:perboch}
\partial_{tt}\u_N(\mu_t)=\int {\rm p}_{2,N}(\rho_t)(\Delta\psi_t)^2-{\rm p}_N(\rho_t)D (\Delta\psi_t)(\nabla\psi_t)-{\rm p}_N(\rho_t)\partial_t\Delta\psi_t\,\d{\rm vol},
\end{equation}
where ${\rm p}_N,{\rm p}_{2,N}:[0,\infty)\to[0,\infty)$ are given by ${\rm p}_N(z):=zu_N'(z)-u_N(z)$, ${\rm p}_{2,N}(z):=z{\rm p}'_N(z)-{\rm p}_N(z)$. Hence if \eqref{eq:defcds} holds we must have
\begin{equation}
\label{eq:keyboc}
\int \rho^{1-\frac1N}\left(-\frac{(\Delta\psi)^2}{N^2}-\frac{D(\Delta\psi)(\nabla\psi)}N-\frac{\partial_t\Delta\psi_t\restr{t=0}}N\right)\,\d{\rm vol}\geq \frac KN\int\rho^{1-\frac1N}\|\nabla\psi\|^2\,\d{\rm vol}.
\end{equation}
Using now the fact that  $\rho$ is non negative and chosen independently on $\psi=\eps\varphi$, from \eqref{eq:keyboc} we deduce
\begin{equation}
\label{eq:bocf}
-\partial_t\Delta\varphi_t\restr{t=0}\geq \frac{(\Delta\varphi)^2}N+D(\Delta\varphi)(\nabla\varphi)+K\|\nabla\varphi\|^2,
\end{equation}
where $(\varphi_t)$ evolves according to \eqref{eq:HJ} with initial condition $\varphi_0:=-\varphi$. Similarly for $N=\infty$. This formal argument shows that if the Finsler manifold  is a $CD^*(K,N)$ space, then the Bochner inequality written as in \eqref{eq:bocf} holds for any smooth $\varphi$. The converse implication can also be achieved by integration and recalling the local-to-global properties of the $CD^*(K,N)$ condition. The argument as presented is only formal because we didn't pay attention to the smoothness of the object involved in computations, but at least on a Riemannian framework it is easy to check that there is indeed sufficient regularity.

Notice that inequality \eqref{eq:bocf} is different from the one rigorously proven in a Finser setting in \cite{Ohta-SturmBoch}: in this reference at the left-hand side there is the term $\Delta^{\nabla\varphi}\frac{|\nabla\varphi|^2}{2}$, where $\Delta^{\nabla\varphi}$ is an appropriate linearization of the Laplacian $\Delta$ along the direction $\nabla\varphi$. Notice also that the Finsler manifold is Riemannian if and only if the Laplacian $\Delta$ is a linear operator: in this case \eqref{eq:bocf} assumes the more familiar form
\[
\Delta\frac{|\nabla\varphi|^2}{2}\geq \frac{(\Delta\varphi)^2}N+\nabla\Delta\varphi\cdot\nabla\varphi+K|\nabla\varphi|^2.
\]

\bigskip

Now observe that inequality \eqref{eq:cdstar} is an inequality concerning the distribution of masses at different times along a $W_2$-geodesic. As such, we can think at it as a Lagrangian point of view on Ricci bounds. Opposed to this, there should be a Eulerian point of view which gives the same information read at the level of velocity vector fields. This is exactly the point of view adopted in the `proof' of Bochner inequality just provided: as we learned from Otto's interpretation of the space $(\probt F,W_2)$ as infinite dimensional manifold, for any $\varphi\in C^\infty_c(F)$ and any $\mu\in\probt F$, the vector field $-\nabla\varphi$ can be seen as the initial velocity of a Wasserstein geodesic starting from $\mu$ (this is made rigorous by \eqref{eq:cont} and \eqref{eq:HJ}). From this perspective, Bochner inequality should be regarded as an inequality for gradients of functions,  rather than for functions themselves. 

\medskip

It seems hard to use these ideas to prove the validity of the Bochner inequality in a non-smooth setting, a problem being justifying the second differentiation in \eqref{eq:perboch}. Beside this, there is another subtle issue that we want to emphasize. At least in the Riemannian case, all the objects appearing in Bochner inequality are quadratic forms, and writing it in an infinitesimally Hilbertian $CD(K,N)$ space we certainly want to keep this property. Now, quadratic forms are defined on vector spaces, but the procedure outlined here makes use of Kantorovich potentials and $c$-concavity is an highly non-linear property. Thus, in a sense, even if one is able to get the Bochner inequality for Kantorovich potentials he would still need to prove that there is a `large' vector space of functions whose multiples are $c$-concave in order to be sure that the derived Bochner inequality is made of quadratic forms. In the smooth case this is easy, because as we mentioned $C^\infty_c(M)$ does the job, but in the non-smooth one this seems an issue. We do not really know whether such vector space exists. Perhaps, if one wants to build it, a possibility could be to try with a regularization via the heat flow. This raises the following question concerning regularization of the heat flow in terms of $c$-concavity, which we believe of independent interest:
\begin{quote}
Is it true that there are constants $\mathcal C_{K,N}(t)$ such that the following holds?

Given an infinitesimally Hilbertian $CD(K,N)$ space $(X,\sfd,\mm)$ and $\rho\mm\in\probt X$ with $\rho\leq 1$ the function $\h_t(\rho)$ is $\mathcal C_{K,N}(t)\frac{\sfd^2}{2}$-concave.
\end{quote}
The problem is open also if `infinitesimally Hilbertian $CD(K,N)$ space' is replaced by `smooth Riemannian manifold with ${\rm Ric}\geq K$ and ${\rm dim}\leq N$'. In this direction, recall that in a smooth world the formula $f_{t,\eps}:=\eps\log(\rho_{\eps t}+1)$ brings solutions of the heat equation $\frac\d{\d t}\rho_t=\Delta\rho_t$ into solutions of the viscous approximation of the Hamilton-Jacobi equation:
\begin{equation}
\label{eq:visc}
\frac\d{\d t}f_{t,\eps}=|\nabla f_{t,\eps}|^2+\eps\Delta f_{t,\eps}.
\end{equation}
Notice that the map $[0,1]\ni z\mapsto\eps\log(z+1)$ has derivative bounded from above and below by positive constants, therefore if $\rho_{\eps t}$ is $C\frac{\sfd^2}2$-concave then $f_{t,\eps}$ is $c_1C\frac{\sfd^2}{2}$-concave and viceversa if  $f_{t,\eps}$ is $C\frac{\sfd^2}{2}$-concave then  $\rho_{\eps t}$ is $c_2C\frac{\sfd^2}2$-concave, for some $c_1,c_2>0$. Thus in the smooth case the above problem can also be formulated at the level of solutions of \eqref{eq:visc} rather than for the heat flow.

Now observe that  as $\eps\downarrow0$, the functions $f_{t,\eps}$ converge to the unique viscous solution $t\mapsto f_{t,0}$ of the Hamilton-Jacobi equation, and we know from the Hopf-Lax formula that in such limiting case $f_{t,0}$ is indeed $-\frac{\sfd^2}t$-concave independently on any curvature-dimension bound. By the Oleinik principle we know that on the Euclidean space $\R^d$, also solutions of the viscous approximation \eqref{eq:visc} are $-\frac{\sfd^2}t$-concave. Therefore in the smooth case the question can be reformulated as: 
\begin{quote}
is it true that for a given $\eps>0$ an Oleinik-type principle holds for  \eqref{eq:visc} uniformly under a   curvature-dimension bound?
\end{quote}

\bigskip

We conclude recalling that  the only non-smooth situation where Bochner inequality has been proved is  the case $N=\infty$ and for infinitesimally Hilbertian spaces. The strategy, proposed in \cite{Gigli-Kuwada-Ohta10} and generalized in \cite{AmbrosioGigliSavare11-2}, is different from the one outlined above and uses twice a duality argument based on infinitesimal Hilbertianity. The idea is the following: first one uses the $K$-convexity of the entropy and infinitesimal Hilbertianity to deduce that the gradient flow $\h_t$ of the entropy $K$-contracts the $W_2$-distance (see also Appendix \ref{app:infhil}) , i.e.
\[
W_2(\h_t(\mu),\h_t(\nu))\leq e^{-Kt}W_2(\mu,\nu),\qquad\forall t\geq 0,\ \mu,\nu\in\probt X.
\]
Then one uses once again infinitesimal Hilbertianity to get that the flow $\h_t$, which coincides with the gradient flow of the energy $f\mapsto\frac12\int\weakgrad f^2\,\d\mm$, is linear and, by a general duality principle due to  Kuwada (see \cite{Kuwada10}) to deduce that
\[
\lip(\h_t(f))^2\leq e^{-Kt}\h_t(\lip(f)^2),\qquad\forall f\in\Lip(X)\cap L^2(X,\mm),
\]
then with a relaxation procedure based on the lower semicontinuity of minimal weak upper gradients and the density result in Theorem \ref{thm:stronglip} one deduces
\[
|\nabla\h_t(f)|^2\leq e^{-Kt}\h_t(|\nabla f|^2),\qquad\mm\ae,,\qquad\forall f\in W^{1,2}(X,\sfd,\mm),
\]
which, taking the derivative at $t=0$, is equivalent to the Bochner inequality for $N=\infty$.

\backmatter
%    Bibliography styles amsplain or author-year (using natbib) are
%    also acceptable.
%\bibliographystyle{amsalpha}
%\bibliography{biblio}

\newcommand{\etalchar}[1]{$^{#1}$}
\def\cprime{$'$} \def\cprime{$'$}
\providecommand{\bysame}{\leavevmode\hbox to3em{\hrulefill}\thinspace}
\providecommand{\MR}{\relax\ifhmode\unskip\space\fi MR }
% \MRhref is called by the amsart/book/proc definition of \MR.
\providecommand{\MRhref}[2]{%
  \href{http://www.ams.org/mathscinet-getitem?mr=#1}{#2}
}
\providecommand{\href}[2]{#2}

%    See note above about multiple indexes.
\printindex

\end{document}